\documentclass[10pt, a4paper]{amsart}

\usepackage{xcolor}
\usepackage{fancyhdr}
\usepackage{tikz}
\usetikzlibrary{decorations.pathmorphing}
\usepackage[T1]{fontenc}
\usepackage{ucs}
\usepackage[utf8]{inputenc}
\usepackage{amsfonts}
\usepackage{listings}
\usepackage{amsmath}
\usepackage[mathscr]{euscript}
\usepackage{amssymb}
\usepackage{amsfonts}
\usepackage{dsfont}
\usepackage{amsthm}
\usepackage{fancybox}
\usepackage{enumitem}
\usepackage{dsfont}
\usepackage{algpseudocode}
\usepackage{algorithm}
\usepackage{graphicx}
\usepackage{caption}
\usepackage{subcaption}
\usepackage{float}
\usepackage{authblk}
\usepackage{relsize}
\usepackage{setspace}
\usepackage{mathtools}
\usepackage{comment}

\usepackage{fancybox} 
\usepackage{pifont} 
\usepackage[utf8]{inputenc}
\usepackage[T1]{fontenc}
\usepackage{appendix}
\usepackage{lmodern}
\usepackage{newunicodechar}
\usepackage{amsmath, amsthm, amsfonts,amssymb}
\usepackage{xcolor,mhsetup}
\usepackage[extra,safe]{tipa}
\usepackage{mathtools}
\usepackage{geometry}
\usepackage{mathrsfs}
\usepackage{hyperref}
\usepackage{url}
\usepackage{fancyhdr}
\numberwithin{equation}{section}

\def\my_c{c_\infty}

\newcommand{\mynewtheorem}[2]{
  \newaliascnt{#1}{dummy}
  \newtheorem{#1}[#1]{#2}
  \aliascntresetthe{#1}
  \expandafter\def\csnam #1autorefname\endcsname{#2}
}

\newcommand{\dcb}{\begin{array}{lll}}
\newcommand{\dce}{\end{array}}
\newcommand{\ebe}{\begin{enumerate}\setlength{\baselineskip}{13pt}\setlength{\parskip}{0pt}}
\newcommand{\dbe}{\end{enumerate}}

\newcommand{\VersionLongEx}[1]{}

\newcommand{\VersionLong}[1]{#1}

\newcommand{\VersionCourt}[1]{}

\DeclareMathOperator{\CstePsiouf }{\mathbf{K}}
\DeclareMathOperator{\ExpPsiouf }{\mathbf{p}}
\DeclareMathOperator{\MomZ}{\mathbf{M}}
\DeclareMathOperator{\HormV}{\boldsymbol{\mathcal{V}}}
\DeclareMathOperator{\MatGramV}{\mathbf{G}}
\DeclareMathOperator{\funceta}{\boldsymbol{\eta}}
\DeclareMathOperator{\funclocgen}{\boldsymbol{\phi}}

\DeclareMathOperator{\LLL }{\mbox{\textsc{L}}}

\DeclareMathOperator{\BasemathcalH}{\mathcal{B}_{\mathcal{H}}}
\newcommand{\ElemBaseH}{\mathtt{h}}
\newcommand{\ElemMultidimBaseH}{\mathbf{h}}
\DeclareMathOperator{\EnsT}{\mathcal{T}}
\DeclareMathOperator{\EnsN}{\mathcal{N}}
\DeclareMathOperator{\EnsLocMatMall}{\mathcal{T}_{0}}

\DeclareMathOperator{\TermStepInitCovMal}{\mathscr{V}}
\DeclareMathOperator{\TermStepDeuxCovMal}{\mathscr{W}}

\DeclareMathOperator{\DeriveMal}{\mathbf{D}}
\DeclareMathOperator{\OpOrnsteinUlh}{\mathbf{L}}
\DeclareMathOperator{\MatCovMal}{\boldsymbol{\sigma}}
\DeclareMathOperator{\InvMatCovMal}{\boldsymbol{\gamma}}
\DeclareMathOperator{\IntSkorohod}{\boldsymbol{\delta}}
\DeclareMathOperator{\PoidMal}{\mathbf{H}}

\DeclareMathOperator{\OpeNorrisDrift}{\mathscr{K}}
\DeclareMathOperator{\OpeNorrisMart}{\mathscr{M}}

\DeclareMathOperator{\FuncNorris}{\mathscr{N}}

\DeclareMathOperator{\DeriveDirec}{\mathscr{D}}
\DeclareMathOperator{\Hessiennex}{\nabla^{2}_{\mathit{x}}}

\DeclareMathOperator{\OperPreuveLie }{\mathscr{T}}

\DeclareMathOperator{\CsteBDH}{\mathscr{B}}
\newcommand*{\BornsupNl}{ 50m_{\ast}^{-1}(N+1)^{\frac{13L}{12}}}
\newcommand*{\BornNl}{ 50m_{\ast}^{-1}(N+1)^{\frac{13l}{12}}}

\newcommand{\Gaussian}{G}

\everydisplay={\aftergroup\specindent}
\def\specindent{\global\hangindent=2em \global\hangafter=-1 \global\prevgraf=0 }

\makeatletter
\newcommand*{\inlineequation}[2][]{%
  \begingroup
    \refstepcounter{align}%
    \ifx\\#1\\%
    \else
      \label{#1}%
    \fi
    \relpenalty=100 %
    \binoppenalty= 100 %
    \ensuremath{%
      #2%
    }%
    ~\@eqnnum
  \endgroup
}
\makeatother

\newtheorem{theorem}{Theorem}[section]
\newtheorem{proposition}{Proposition}[section]
\newtheorem{lemme}{Lemma}[section]
\newtheorem{corollary}{Corollary}[section]
\newtheorem{remark}{Remark}[section]
\newtheorem{example}{Example}[section]

\date{}

\usepackage{todonotes}
\begin{document}

\title{H\"ormander Properties of Discrete-Time Markov Processes}



\maketitle

\begin{center}
Cl\'ement Rey 
\footnote{address: CMAP, École Polytechnique, Institut Polytechnique de Paris, Route de Saclay, 91120 Palaiseau, France\\
 e-mail: clement.rey@polytechnique.edu}
\end{center}

\begin{abstract}
We develop a general framework for establishing smoothing properties of inhomogeneous discrete-time Markov processes under a local weak space-time H\"ormander condition imposed at the initial point. Our analysis is carried out directly at the discrete level, combining a Malliavin-inspired differential calculus allowing non-Gaussian increments with covariance non-degeneracy estimates specific to this setting. This yields smooth densities and quantitative derivative bounds for Gaussian-regularized discrete processes, as well as quantitative semigroup regularization and density-derivative bounds for the limiting time-inhomogeneous diffusion. We further establish a quantitative invariance principle in total variation for discrete processes sharing the same infinitesimal drift and covariance. The convergence rate is of order $\delta^{\frac{1}{2}}$, improves to order $\delta$ under a simple moment cancellation condition, and more generally inherits higher weak approximation rates for smooth test functions.


\end{abstract}

%

\tableofcontents{}

\section{Introduction}
\subsection{Context}

H\"ormander's bracket condition provides a geometric mechanism by which a degenerate diffusion can acquire smoothing properties in directions that are not directly forced by the noise. For It\^o diffusions, it underlies hypoellipticity, smoothness of transition densities and regularization of the associated Markov semigroup, usually established through Malliavin calculus; see, among others, \cite{Hormander_1967,Kusuoka_Stroock_1985_AMCII,Nualart_2006}. We investigate analogous properties for discrete-time Markov processes and their convergence as the time step tends to zero. In weak approximation, convergence rates are typically first obtained for smooth test functions, whereas total variation convergence requires estimates for bounded measurable functions. Our aim is to transfer the former estimates to the latter class under a local weak H\"ormander condition imposed only at the initial point and possibly involving brackets of arbitrary order. More precisely, we identify, in total variation, the diffusion limit of a broad class of time-inhomogeneous discrete-time Markov processes and quantify the convergence rate.

Let $(\Omega,\mathcal{F},\mathbb{P})$ be a probability space and, for $\delta\in(0,1]$ and $d,N\in\mathbb{N}^{\ast}$, set $\pi^{\delta}:=\delta\mathbb{N}$ and $\pi^{\delta,\ast}:=\delta\mathbb{N}^{\ast}$. We consider independent random variables $Z^{\delta}_{t}\in\mathbb{R}^{N}$, $t\in\pi^{\delta,\ast}$, which are centered, have identity covariance matrix and finite moments of the required orders. Their distributions are Lebesgue-lower-bounded, namely,  they dominate a positive multiple of the Lebesgue measure restricted to a ball of fixed radius; in particular, they need not be Gaussian. We consider the $\mathbb{R}^{d}$-valued discrete-time Markov process
\begin{align}
\label{eq:scheme_intro}
X^{\delta}_{t+\delta}
=
\psi(X^{\delta}_{t},t,\delta^{\frac{1}{2}}Z^{\delta}_{t+\delta},\delta),
\quad t\in\pi^{\delta},
\quad X^{\delta}_{0}\in\mathbb{R}^{d},
\end{align}
where
$\psi:(x,t,z,y)\mapsto\psi(x,t,z,y)\in
\mathcal{C}^{\infty}(\mathbb{R}^{d}\times\mathbb{R}_{+}\times
\mathbb{R}^{N}\times[0,1];\mathbb{R}^{d})$ and $\psi(x,t,0,0)=x$.
The corresponding candidate diffusion limit solves
\begin{align}
\label{eq:eds_ito_intro}
X_{t}
=
X^{\delta}_{0}
+\int_{0}^{t}V_{0}(X_{s},s)\mbox{d}s
+\sum_{i=1}^{N}\int_{0}^{t}V_{i}(X_{s},s)\mbox{d}W^{i}_{s}, \quad t \geqslant 0,
\end{align}
where $((W^{i}_{t})_{t\geqslant0},i\in\{1,\ldots,N\})$ are independent standard Brownian motions,
$V_{0}:=\partial_{y}\psi(\cdot,\cdot,0,0)
+\frac{1}{2}\sum_{i=1}^{N}\partial^{2}_{z^{i}}\psi(\cdot,\cdot,0,0)$ and
$V_{i}:=\partial_{z^{i}}\psi(\cdot,\cdot,0,0)$.
Thus the limiting diffusion depends on $\psi$ only through these first and second-order derivatives evaluated at $(x,t,0,0)$ and on the increments through their first two moments. This is an invariance principle: a large class of schemes and non-Gaussian increment distributions yields the same limiting diffusion, higher-order characteristics only affecting the approximation rate (see Theorem \ref{th:invariance_main_result}).

Our main contribution is a quantitative regularization result established intrinsically for the discrete dynamics under a local weak H\"ormander condition at the initial point. This regularization allows weak smooth approximation estimates to be transferred to bounded measurable test functions: the total variation rate is of order $\delta^{\frac{1}{2}}$, improves to $\delta$ when the third-order moments vanish, and more generally inherits the underlying weak smooth rate. The general update structure (\ref{eq:scheme_intro}) is close to the abstract framework of \cite{Bally_Rey_2016}, but the non-degeneracy assumptions and regularization mechanism are different. Uniform ellipticity is replaced by a local weak H\"ormander condition of arbitrary order, and time-dependent coefficients with polynomial growth are allowed. 

We now describe the H\"ormander structure. Rewriting (\ref{eq:eds_ito_intro}) in Stratonovich form gives
\begin{align}
\label{eq:eds_strat_intro}
X_{t}
=
X^{\delta}_{0}
+\int_{0}^{t}\bar{V}_{0}(X_{s},s)\mbox{d}s
+\sum_{i=1}^{N}\int_{0}^{t}V_{i}(X_{s},s)\circ\mbox{d}W^{i}_{s},
\end{align}
where
$\bar{V}_{0}=V_{0}-\frac{1}{2}\sum_{i=1}^{N}\nabla_{x}V_{i}V_{i}$.
The vector fields $V_{0},V_{1},\ldots,V_{N}$ and their derivatives have polynomial growth in space, except for their first-order spatial derivatives, which are bounded. The infinitesimal generator is
$\mathcal{A}=\bar{V}_{0}+\frac{1}{2}\sum_{i=1}^{N}(V_{i})^{2}$, and hypoellipticity of $\mathcal{A}+\partial_{t}$ is related to the Lie algebra generated by these fields \cite{Hormander_1967}. Geometrically, successive brackets describe how the noise directions are transported by the dynamics and generate new spatial directions.

Since the coefficients may depend on time, the relevant Lie algebra is generated in space-time. For fixed $t$, regard $x \mapsto \bar{V}_{0}(x,t)$ and $x \mapsto V_{i}(x,t)$ as vector fields on $\mathbb{R}^{d}$ and introduce
the time-extended vector fields $(x,t) \mapsto \bar{V}_{\ast,0}(x,t)=(\bar{V}_{0}(x,t),1)$ and
$(x,t) \mapsto V_{\ast,i}(x,t)=(V_{i}(x,t),0)$ on $\mathbb{R}^{d} \times \mathbb{R}_{+}$.
For $V,W\in\{\bar{V}_{0},V_{1},\ldots,V_{N}\}$ and
$j\in\{1,\ldots,d+1\}$,
\begin{align*}
[V_{\ast},W_{\ast}]^{j}(x,t)
=&\,
[V,W]^{j}(x,t)
+\partial_{t}W_{\ast}^{j}V_{\ast}^{d+1}(x,t)
-\partial_{t}V_{\ast}^{j}W_{\ast}^{d+1}(x,t),
\end{align*}
where $[V,W]=\nabla_{x}WV-\nabla_{x}VW$, with the convention
$[V,W]^{d+1}=0$, so that $[V_{\ast},W_{\ast}]^{d+1}=0$.

Thus the extended brackets incorporate time derivatives absent from the usual spatial brackets. Here and below, by spatial brackets we mean Lie brackets computed with respect to the space variable only, with time treated as a parameter. Starting from the diffusion directions, define
\begin{align*}
\mathbf{V}_{\ast}^{(0)}
=&\,
\{V_{\ast,i},i\in\{1,\ldots,N\}\},\\
\mathbf{V}_{\ast}^{(n+1)}
=&\,
\mathbf{V}_{\ast}^{(n)}
\cup
\{[\bar{V}_{\ast,0},V],[V_{\ast,i},V]:
i\in\{1,\ldots,N\},V\in\mathbf{V}_{\ast}^{(n)}\}.
\end{align*}
Define $\mathbf{V}^{(n)}$ analogously with
$\bar{V}_{0},V_{1},\ldots,V_{N}$. The local weak H\"ormander condition at the initial point $(X_{0}=x,0)$ is
\begin{align*}
\mbox{span}
(\cup_{n=0}^{\infty}\mathbf{V}_{\ast}^{(n)}(x,0))
=
\mathbb{R}^{d}\times\{0\}.
\end{align*}
In the homogeneous setting it reduces to
$\mbox{span}(\cup_{n=0}^{\infty}\mathbf{V}^{(n)})(x)=\mathbb{R}^{d}$;
see \cite{Kusuoka_Stroock_1985_AMCII}. The elliptic case corresponds to
$\mbox{span}(\mathbf{V}^{(0)}(x,0))=\mathbb{R}^{d}$. The condition is local when imposed only at the initial point and
uniform when the corresponding non-degeneracy holds uniformly in
space-time; it is weak because brackets involving the drift are
allowed, whereas the strong condition uses only brackets generated by the diffusion fields; that is, $\bar{V}_{\ast,0}$ is replaced by 0 in the definition of $\mathbf{V}^{(n)}_{\ast}$. The existence and smoothness of transition densities for $(X_{t})_{t\geqslant 0}$ have been studied under several variants of Hörmander's bracket condition. For time-inhomogeneous diffusions, a spatial weak H\"ormander condition is considered in \cite{Chen_Zhou_1991}. Local weak H\"ormander conditions at the starting point are studied in \cite{Cattiaux_Mesnager_2002,Pigato_2022}, whereas \cite{Hopfner_Locherbach_Thieullen_2017} assumes a local weak H\"ormander condition in a neighborhood of the terminal space--time point. A uniform weak H\"ormander condition for a specific chain structure is considered in \cite{Delarue_Menozzi_2010}. In the homogeneous local weak setting,
see \cite{Kusuoka_Stroock_1985_AMCII,Bally_KohatsuHiga_2010,Pigato_2018}.
The space-time extension is essential: for time-dependent coefficients,
a spanning condition involving only spatial brackets need not imply
hypoellipticity; see Remark~2.1 in \cite{Taniguchi_1985}.

Our first main result establishes regularization directly at the discrete level. Under the preceding H\"ormander condition and suitable polynomial growth assumptions on the derivatives of $\psi$, one constructs a process $(\bar{X}^{\delta}_{t})_{t\in\pi^{\delta}}$ arbitrarily close to $(X^{\delta}_{t})_{t\in\pi^{\delta}}$ in total variation at fixed times and satisfying, uniformly in $\delta$,
\begin{align}
\label{eq:reg_prop_intro}
\vert \partial_{x}^{\alpha} \mathbb{E} [\partial^{\beta}f(\bar{X}^{\delta}_{T}) \vert X^{\delta}_{0}=x ] \vert \leqslant C(x,T)\Vert f\Vert_{\infty},
\end{align}
for every $\alpha,\beta\in\mathbb{N}^{d}$ and
$f\in\mathcal{C}^{\infty}_{b}(\mathbb{R}^{d};\mathbb{R})$.
A refined quantitative version is given in Theorem~\ref{th:regul_main_result_intro}; in particular, the regularized law admits a smooth density with quantitative derivative bounds; see Corollary~\ref{coro:borne_densite_reg_gauss_intro}.

These estimates yield the approximation result. Theorem~\ref{th:invariance_main_result} shows that, for
$T\in\pi^{\delta}$, $T\geqslant2\delta$, and
$X_{0}=X^{\delta}_{0}=x$,
\begin{align}
\label{eq:vt_approx_intro}
\nonumber
d_{TV}(\mbox{Law}(X_{T}),\mbox{Law}(X^{\delta}_{T}))=&\,\frac{1}{2}\sup_{\substack{f:\mathbb{R}^{d}\to[-1,1]\\f\text{ measurable}}}\vert\mathbb{E}[f(X_{T})-f(X^{\delta}_{T}) \vert X_{0}=X^{\delta}_{0}=x ] \vert \\ \leqslant&\,\delta^{\frac{1}{2}}\frac{1+\vert x\vert_{\mathbb{R}^{d}}^{c}}{(\HormV_{L}(x)T)^{\eta}}C\exp(CT),
\end{align}
where $c,C,\eta>0$ and $\HormV_{L}(x)\in(0,1]$ is the finite-order quantitative non-degeneracy quantity associated with the local weak H\"ormander condition. Its strict positivity expresses quantitative non-degeneracy of the relevant H\"ormander directions at $(x,0)$ up to order $L$. If the third-order moments of $Z^{\delta}_{t}$ vanish, the rate improves to $\delta$. More generally, higher weak smooth rates transfer to total variation; see Remark~\ref{rmk:main_theo}. The same analysis gives a smooth density for $X_{T}$ and quantitative approximation estimates for this density and its derivatives.

The key point behind these results is that the estimates of Theorem~\ref{th:regul_main_result_intro} are uniform with respect to the time step $\delta$. Combined with a Lindeberg-type decomposition, this uniformity yields the invariance principle of Theorem~\ref{th:invariance_main_result}. Letting $\delta$ tend to zero in the density bounds of Corollary~\ref{coro:borne_densite_reg_gauss_intro} then provides, as a direct consequence, quantitative semigroup regularization and density-derivative bounds for the limiting time-inhomogeneous diffusion under a weak H\"ormander condition imposed only at the initial point (Corollary~\ref{coro:borne_densite_SDE}), including Gaussian upper bounds when the coefficients are bounded. To the best of our knowledge, such bounds are not available in this generality.

The transfer from weak smooth approximation to total variation connects our results with several strands of the literature. The Euler scheme has weak order one under vanishing third-order moments for sufficiently smooth test functions \cite{Talay_Tubaro_1991}, and higher-order weak methods are developed, for example, in \cite{Talay_1990,Ninomiya_Victoir_2008,Alfonsi_2010}. For Gaussian Euler schemes, total variation and density approximation results are obtained in \cite{Bally_Talay_1996_I,Bally_Talay_1996_II} in the homogeneous uniform weak H\"ormander setting or \cite{Guyon_2006} in the uniformly elliptic setting. Under uniform ellipticity, first-order convergence in total variation was established in \cite{Yamada_2023} for a Milstein scheme in a homogeneous small-noise setting, and in \cite{Li_Wang_Wang_2026} for the Euler scheme of time-inhomogeneous SDEs with a sufficiently regular initial density. In a different setting, second-order weak approximation for bounded measurable test functions was obtained in \cite{Iguchi_Yamada_2021} under a weak H\"ormander condition holding at every spatial point, using a weighted Gaussian scheme. The abstract framework of \cite{Bally_Rey_2016} transfers weak smooth rates to total variation under uniform ellipticity, while related abstract regularization mechanisms are developed in \cite{Bally_Caramellino_Poly}. Our approach allows local weak H\"ormander non-degeneracy of arbitrary order, time-dependent coefficients and possibly non-Gaussian increments while retaining the general scheme (\ref{eq:scheme_intro}). Related approximation questions under the different UFG condition are considered, for example, in \cite{Kusuoka_2013}. A different line of work based on parametrix methods provides local limit theorems for Markov-chain approximations of uniformly elliptic diffusions \cite{Konakov_Mammen_2000} and of specific rank-two degenerate diffusions \cite{Konakov_Menozzi_Molchanov_2010}.

Unlike the Gaussian framework of \cite{Bally_Talay_1996_I,Bally_Talay_1996_II}, based on continuous-time interpolation and standard Malliavin estimates, we establish regularization directly for the discrete dynamics. The Lebesgue lower bound yields a Nummelin splitting \cite{Nummelin_1978} and discrete integration-by-parts formulas, while martingale and discrete Gronwall estimates control Sobolev norms. Establishing covariance non-degeneracy requires a careful treatment of specific discrete difficulties: one-step Jacobians need not be invertible, and the evolution of transported vector fields introduces Taylor remainders that must be controlled uniformly in $\delta$. Localization and a discrete counterpart of Norris's lemma \cite{Norris_1986}, proved here by controlling additional terms in discrete martingale expansions, yield non-degeneracy estimates through successive Lie brackets under a local weak space-time H\"ormander condition. The method also yields original quantitative semigroup regularization and density-derivative bounds for the limiting time-inhomogeneous diffusion under this initial-point condition.

\subsection{Organization of the paper}
Section \ref{Sec:Main_Result} presents the main regularization results (Theorem \ref{th:regul_main_result_intro} and Corollary \ref{coro:borne_densite_reg_gauss_intro}), the invariance principle (Theorem \ref{th:invariance_main_result}) and quantitative bounds on the density and its derivatives for the limiting diffusion (Corollary \ref{coro:borne_densite_SDE}). Section \ref{Sec:prove_Regularization_properties} develops a Malliavin-inspired discrete differential calculus and uses a localized regularization estimate (Theorem \ref{th:regularisation}) as a key ingredient in proving these results. The main technical effort is devoted to establishing quantitative non-degeneracy estimates for the discrete Malliavin covariance matrix (Theorem \ref{th:borne_Lp_inv_cov_Mal}), which are developed in Section \ref{Sec:Proof_reg_prop} alongside the necessary Sobolev estimates (Theorem \ref{theo:Norme_Sobolev_borne}). 


\subsection{Notations.} 
We denote by $\mathcal{M}(\mathbb{R}^{d})$ (respectively $\mathcal{M}_{b}(\mathbb{R}^{d})$) the set of measurable (resp. measurable and bounded) functions defined on $\mathbb{R}^{d}$. 
For $q \in \mathbb{N} \cup \{+\infty\}$, $\mathcal{C}_b^q(\mathbb{R}^{d}) $ (resp. $\mathcal{C}_{K}^q(\mathbb{R}^{d}) $, $\mathcal{C}_{pol}^q(\mathbb{R}^{d}) $) is the set of functions belonging to $\mathcal{C}^q(\mathbb{R}^{d}) $ such that all the derivatives (of order 0 to $q$) are bounded (resp. have compact support, have polynomial growth).  
We will also denote $\mathcal{M}
(\mathbb{R}^{d};\mathbb{R})$ for the set of measurable functions  defined on $\mathbb{R}^{d}$  and taking values in $\mathbb{R}$ (and similarly for other sets of functions defined above).

Let $\mathcal{H},\mathcal{H}^{\prime}$ be Hilbert spaces. For $f:\mathcal{H} \to \mathcal{H}^{\prime}$ and $u \in \mathcal{H}$, the directional derivative $\DeriveDirec_{u}f$ of $f$ along $u$ is $\DeriveDirec _{u}f(x) :=   \lim_{\epsilon \to 0}\frac{f(x+\epsilon u)-f(x)}{\epsilon}$ for every $x \in \mathcal{H}$ whenever this limit exists. When $f$ is Fr\'echet differentiable, we recall that its derivative $u \mapsto \DeriveDirec _{u}f(x)$ is a linear map from $\mathcal{H}$ to $\mathcal{H}^{\prime}$ that we simply denote $\DeriveDirec f(x)$. When $\mathcal{H}^{\prime}=\mathbb{R}$,  $\DeriveDirec   f(x)$ is identified with its Riesz representative and $\DeriveDirec _{u}f(x)=\langle \DeriveDirec f(x) , u \rangle_{\mathcal{H}}$.
For $f\in \mathcal{M}_{b}(\mathbb{R}^{d};\mathbb{R}^{d^{\prime}})$, we set $\Vert f \Vert_{\infty} = \sup_{x \in \mathbb{R}^{d}} \vert f(x) \vert_{\mathbb{R}^{d^{\prime}}}$ with  $\vert \cdot \vert_{\mathbb{R}^{d^{\prime}}}$ the standard Euclidean norm on $\mathbb{R}^{d^{\prime}}$. For $z\in\mathbb{R}^{d^{\prime}}$ and $r>0$, we denote by $B_{r}(z)=\{z^{\prime}\in\mathbb{R}^{d^{\prime}},\vert z^{\prime}-z\vert_{\mathbb{R}^{d^{\prime}}}<r\}$ the open ball with center $z$ and radius $r$, and by $\bar{B}_{r}(z)$ its closure. When $f$ takes values in $\mathbb{R}^{d^{\prime} \times d^{\prime}}$, we denote $\Vert f(x) \Vert_{\mathbb{R}^{d^{\prime}}} = \sup_{\xi \in \mathbb{R}^{d^{\prime}}; \vert \xi \vert_{\mathbb{R}^{d^{\prime}}}=1} \vert f(x) \xi \vert_{\mathbb{R}^{d^{\prime}}}$.  For a multi-index $\alpha =(\alpha^{1},\cdots,\alpha^{d})\in \mathbb{N}^{d}$ we denote $\vert \alpha \vert
=\alpha^{1}+...+\alpha^{d}$, $\Vert \alpha \Vert=d$ and if $f\in \mathcal{C}^{\vert \alpha \vert}(\mathbb{R}^{d})$,   $\partial^{\alpha }f=\partial_x^{\alpha}f(x) =\partial
_{x^{1}}^{\alpha^{1}}\ldots\partial_{x^{d}}^{\alpha^{d}}f(x)$. If $\alpha =(0,...,0)$, $\partial^{\alpha
}f=f$.  For $\beta \in \mathbb{N}^{d^{\prime}}$, set $(\alpha,\beta)=(\alpha^{1},\cdots,\alpha^{d},\beta^{1},\ldots,\beta^{d^{\prime}})$. In addition, $\nabla_{x} f =(\partial_{x^{j}} f^{i})_{(i,j) \in \{1,\ldots,d^{\prime}\} \times \{1,\ldots,d\}}$ is the Jacobian of $f$ and $\Hessiennex  f =((\partial_{x^{j}}\partial_{x^{l}} f^{i})_{(l,j) \in \{1,\ldots,d\} \times \{1,\ldots,d\}})_{i \in \{1,\ldots,d^{\prime}\}}$ is its Hessian. In particular, for $v \in \mathbb{R}^{d}$, $v^{\mbox{\textsc{t}}} \Hessiennex f  \in \mathbb{R}^{d^{\prime} \times d}$ and $(v^{\mbox{\textsc{t}}} \Hessiennex f)^{i,j}=\sum_{l=1}^{d} \partial_{x^{j}}\partial_{x^{l}} f^{i} v^{l}$.  

Throughout the article, $C$ denotes a constant that may change from line to line. Its dependence on a parameter $\vartheta=(\vartheta_1,\ldots,\vartheta_l)\in\mathbb R^l$ may be emphasized by writing $C(\vartheta)$, allowing $C(\vartheta)$ to diverge as some $\vert \vartheta_i \vert \to\infty$. When $C$ is said to depend on the moments of $Z_t^\delta$, $t\in\pi^{\delta,*}$, $t\le T$, this means moments up to a finite order independent of $\delta$, $T$, and $X_0^\delta$. 
Also, $\mathbf{1}_{a,b}$ stands for the Kronecker symbol, meaning $\mathbf{1}_{a,b}=1$ if $a=b$ and is zero otherwise.\\
Finally, for a discrete-time process $(Y_{t})_{t \in \pi^{\delta}}$, we denote by $\mathcal{F}^{Y}_{t}  :=  \sigma (Y_{w},w \in \pi^{\delta},w \leqslant t)$ the sigma-algebra generated by $Y$ until time $t$ and for a random variable $F$, we write $\mathbb{E}_{x }[F]=\mathbb{E}_{}[F \vert X_{0}=X^{\delta}_{0}=x]$, $x \in \mathbb{R}^{d}$.

\section{Main results}
\label{Sec:Main_Result}


In this section, we establish regularization properties for discrete-time Markov processes (Theorem \ref{th:regul_main_result_intro} and Corollary \ref{coro:borne_densite_reg_gauss_intro}), a quantitative invariance principle (Theorem \ref{th:invariance_main_result}), and density and semigroup regularization estimates for the limiting diffusion (Corollary \ref{coro:borne_densite_SDE}).

\subsection{A Class of Markov Semigroups} \;

\noindent \textbf{Definition of the semigroups.}
For $\delta \in (0,1]$ and $N\in \mathbb{N}^{\ast}$, we consider a sequence of
independent random variables $Z^{\delta}_{t}\in \mathbb{R}^{N},\; t\in
\pi^{\delta,\ast}$, and we assume that $Z^{\delta}_{t}$ are centered with $ \mathbb{E}[ Z^{\delta,i}_{t} Z^{\delta,j}_{t}]=\mathbf{1}_{i,j}$ for every $i,j \in \EnsN :=   \{1,\ldots,N\}$ and every $t\in
\pi^{\delta,\ast}$. These centering and covariance conditions are assumed throughout. We construct the $\mathbb{R}^{d}$-valued Markov process $(X^{\delta}_{t})_{t \in \pi^{\delta}}$ in the following way:%
\begin{align}
X^{\delta}_{t+\delta}=\psi(X^{\delta}_{t},t,\delta^{\frac{1}{2}} Z^{\delta}_{t+\delta}, \delta) , \quad t \in  \pi^{\delta}, \quad X^{\delta}_{0}=\mbox{\textsc{x}}^{\delta}_{0}\in \mathbb{R}^{d},\label{eq:schema_general}
\end{align}%
where
\begin{align*}
\psi \in \mathcal{C}^{\infty }( \mathbb{R}^{d}\times \mathbb{R}_+\times \mathbb{R}^{N} \times [0,1];\mathbb{R}^{d})\quad \mbox{and} \quad  \forall (x,t) \in \mathbb{R}^{d} \times \mathbb{R}_{+},\psi
(x,t,0,0)=x.  
\end{align*}
Let us now define the inhomogeneous discrete-time Markov semigroup associated with $(X^{\delta}_{t})_{t \in \pi^{\delta}}$. For every measurable function $f$ from $\mathbb{R}^{d}$ to $\mathbb{R}$, and every $x \in \mathbb{R}^{d}$,
\begin{align*}
\forall s,t \in \pi^{\delta}, s \leqslant t \qquad  Q^{\delta}_{s,t}f(x) =\int_{\mathbb{R}^{d}} f(y) Q^{\delta}_{s,t}(x,\mbox{d} y)  :=  \mathbb{E}[f(X^{\delta}_{t}) \vert X^{\delta}_{s}=x] ,
  \end{align*}
 whenever the expectation is well defined. 
We also adopt the notation $Q^{\delta}_{t}:=Q^{\delta}_{0,t}$, $t \in \pi^{\delta}$. We will obtain regularization properties for modifications of this discrete semigroup.  Our approach relies on the assumptions on $\psi$ and $Z^{\delta}$ stated below.  \\

\noindent \textbf{Hypothesis on $\psi$.  Polynomial growth and H\"ormander property.} 
We first impose a polynomial growth assumption concerning the derivatives of $\psi$. Let $(\CstePsiouf_{q})_{q \in \mathbb{N}^{\ast}}$ and $(\ExpPsiouf  _{q})_{q \in \mathbb{N}^{\ast}}$, $\CstePsiouf_{q}\geqslant 1,\ExpPsiouf  _{q}  \in \mathbb{N}$ be some nondecreasing sequences. For $r \in \mathbb{N}^{\ast}$, 
\begin{enumerate}[label=$\mathbf{A}_{1}^{\delta}(r)$.\quad]
\item \label{Hypothese_pol_growth} 
We assume that 
 for every $(x,t,z,y) \in \mathbb{R}^{d} \times \mathbb{R}_{+} \times \mathbb{R}^{N} \times [0,1]$,
\begin{align}
 \sum_{\vert \alpha^{x}
\vert + \vert \alpha^{t} \vert =0}^{r} \sum_{\vert \alpha^{z} \vert  +\vert \alpha^{y} \vert 
=1}^{r-\vert \alpha^{x} \vert - \vert \alpha^{t} \vert}\vert \partial
_{x}^{\alpha^{x}}\partial_{t}^{\alpha^{t}} \partial_{z}^{\alpha^{z}} \partial_{y}^{\alpha^{y}}\psi\vert_{\mathbb{R}^{d}} (x,t,z,y) \leqslant \CstePsiouf_{r}(1+\vert x \vert_{\mathbb{R}^{d}}^{\ExpPsiouf  _{r}}+ \delta^{-\frac{\ExpPsiouf  _{r}}{2}} \vert z \vert_{\mathbb{R}^{N}}^{\ExpPsiouf  _{r}}),
\label{eq:hyp_1_Norme_adhoc_fonction_schema}
\end{align}

\begin{align}
 \sum_{l=1}^{d} (\vert \partial_{x^{l}}  \partial_{y}\psi \vert_{\mathbb{R}^{d}} +\sum_{i \in \EnsN}	\vert \partial_{x^{l}}  \partial_{z^{i}} \psi \vert_{\mathbb{R}^{d}}  +\sum_{i,j \in \EnsN} \vert \partial_{x^{l}}  \partial_{z^{i}} \partial_{z^{j}} \psi \vert_{\mathbb{R}^{d}} )  (x,t,z,y) \leqslant \CstePsiouf_{3} (1 +\delta^{-\frac{\ExpPsiouf_{3}  }{2}} \vert z \vert_{\mathbb{R}^{N}}^{\ExpPsiouf_{3}  }).
\label{eq:hyp_3_Norme_adhoc_fonction_schema}
\end{align}
\end{enumerate}
We denote $\mathbf{A}_{1}^{\delta}(+\infty)$ if $\mathbf{A}_{1}^{\delta}(r)$ is satisfied for every $r \in \mathbb{N}^{\ast}$.
We note that we obtain exactly the same results if we add $\CstePsiouf_{3} \delta^{-1} \vert y \vert$ in the $r.h.s.$ of (\ref{eq:hyp_3_Norme_adhoc_fonction_schema}), or if we add $\CstePsiouf_{r} \delta^{-1} \vert y \vert$ in the $r.h.s.$ of (\ref{eq:hyp_1_Norme_adhoc_fonction_schema}).  This is due to the fact that the function $\psi$ is only used for $y=\delta$ (or $y = C \delta$, $C \leqslant 1$) so the assumptions above are then satisfied replacing $\CstePsiouf_{3} $ (respectively $\CstePsiouf_{r}$) by $2\CstePsiouf_{3}$ (respectively $2\CstePsiouf_{r}$). Also, we do not give the explicit dependence of the $r.h.s$ of (\ref{eq:hyp_1_Norme_adhoc_fonction_schema}) or (\ref{eq:hyp_3_Norme_adhoc_fonction_schema}) $w.r.t.$ the variable $t$ because in our results, $t$ is taken in a compact interval of the form $[0,T]$.\\


In the time-homogeneous case, when $\psi(x,t,z,y)=x+V_0(x)y+\sum_{i=1}^N V_i(x)z^i$, Assumption $\mathbf{A}_1^\delta(+\infty)$ coincides with the regularity assumptions on $V_0,\ldots,V_N$ in \cite{Kusuoka_Stroock_1985_AMCII}. Under these assumptions and a local weak Hörmander condition, they established bounds analogous to (\ref{eq:borne_densite_SDE}) (see Corollary \ref{coro:borne_densite_SDE}) for the limiting diffusion, see \cite{Kusuoka_Stroock_1985_AMCII}, Corollary 3.25.\\

The second assumption on $\psi$ is a local weak H\"ormander condition formulated in terms of the vector fields defined below. We denote the Lie bracket of two $\mathcal{C}^1$ vector fields in $\mathbb{R}^{d}$, $[\cdot,\cdot]:(\mathcal{C}^1(\mathbb{R}^{d},\mathbb{R}^{d}))^{2} \to \mathcal{C}^0(\mathbb{R}^{d},\mathbb{R}^{d})$, $f_1,f_2 \mapsto [f_{1},f_{2}] :=   \nabla_{x}f_{2}f_{1}-\nabla_{x}f_{1}f_{2}$. \\
We denote $V_{0}   = \partial_{y} \psi(\cdot,\cdot,0,0)+\frac{1}{2}\sum_{i \in \EnsN} \partial_{z^{i}}^{2} \psi(\cdot,\cdot,0,0)$, $V_{i} =\partial_{z^{i}} \psi(\cdot,\cdot,0,0)$, $i \in \EnsN$, $\bar{V}_0=V_{0}-\frac{1}{2}\sum_{i \in \EnsN} \nabla_{x} V_{i} V_{i}$. For a multi-index $\alpha \in \{0,\ldots,N\}^{\Vert \alpha \Vert}$ and $V:\mathbb{R}^{d} \times \mathbb{R}_{+} \to \mathbb{R}^{d}$, we define also $V^{[\alpha]}$ using the recurrence relation $V^{[(\alpha,0)]}=[\bar{V}_0,V^{[\alpha]}]+\partial_{t}V^{[\alpha]}+\frac{1}{2}\sum_{i \in \EnsN}[ V_{i},[V_{i},V^{[\alpha]}]]$ and $V^{[(\alpha,j)]} :=   [V_{j},V^{[\alpha]}]$ if $j \in \EnsN$ with the convention $V^{[\emptyset]}=V$.  We are now in a position to introduce our quantitative H\"ormander hypothesis on $\psi$: For $L \in \mathbb{N}$,  the order of our H\"ormander condition,  let us define for every $(x,t) \in \mathbb{R}^{d} \times \mathbb{R}_{+}$,
\begin{align}
\label{def:hormander_func}
\HormV_{L}(x,t)  := 1 \wedge \lambda_{min}( \MatGramV_{L}(x,t))=   1 \wedge \inf_{b \in \mathbb{R}^{d}, \vert b \vert_{\mathbb{R}^{d}} =1} \sum_{i \in \EnsN}  \sum_{l=0}^{L} \sum_{\alpha \in \{0,\ldots,N\} ^{l}} \langle V^{[\alpha]}_{i}(x,t) , b \rangle_{\mathbb{R}^{d}}^{2} ,
\end{align}
where $\MatGramV_{L}(x,t):=\sum_{i \in \EnsN}  \sum_{l=0}^{L} \sum_{\alpha \in \{0,\ldots,N\} ^{l}} V^{[\alpha]}_{i}(x,t) V^{[\alpha]}_{i}(x,t)^{\mbox{\textsc{t}}}$ is the Gram matrix of the H\"ormander fields and, for a symmetric matrix $M\in\mathbb{R}^{d\times d}$, $\lambda_{min}(M)$ denotes its smallest eigenvalue.


We introduce, for $x \in \mathbb{R}^{d}$:
\begin{enumerate}[label=$\mathbf{A}_{2}(x\mbox{,}L)$.]
\item \label{Hypothese_hormander}  A local weak H\"ormander property of order $L \in \mathbb{N}$,
\begin{align}
\label{hyp:loc_hormander}
\HormV_{L}(x,0)>0 .
\end{align}
In particular, this condition is imposed at the initial point for $x=X^{\delta}_{0}$. We will sometimes consider a uniform weak H\"ormander property of order $L$,
\begin{align}
\label{hyp:unif_hormander}
\HormV_{L}^{\infty}   :=  \inf_{t \in \mathbb{R}_{+}}\inf_{u \in \mathbb{R}^{d}}\HormV_{L}(u,t)>0 .
\end{align}
In this case, we denote $\mathbf{A}_{2}^{\infty}(L)$ instead of $\mathbf{A}_{2}(x,L)$. If $\HormV_{L}(x,0)$ is replaced by $\HormV_{L}(x,t)$ in (\ref{hyp:loc_hormander}) we denote $\mathbf{A}_{2}(x,t,L)$. Also, we denote $\HormV_{L}(x)  :=  \HormV_{L}(x,0)$.
\end{enumerate}
It is worth noticing that, with the notations introduced in the Introduction, (\ref{hyp:loc_hormander}) is satisfied for some $L \in \mathbb{N}$ if and only if $\mbox{span}(\cup_{n =0}^{\infty} \mathbf{V}_{\ast}^{(n)})(x,0) = \mathbb{R}^{d} \times \{0\}$, which is why we refer to it as a local weak H\"ormander property. A similar observation holds for (\ref{hyp:unif_hormander}) in the uniform setting. The case $L=0$ corresponds to the elliptic case.\\

\noindent \textbf{Hypothesis on $Z^{\delta}$.  Lebesgue-lower-bounded distributions.}  A first assumption concerns the finiteness of the moments of $Z^{\delta}$: For $p \geqslant 0$,
\begin{enumerate}[label=$\mathbf{A}_{3}^{\delta}(p)$.\quad]
\item \label{hyp:moment_borne_Z_assumption} There exists $C(p)\in \mathbb{R}_{+}$ that depends only on $p$ such that
\begin{align}
\label{eq:hyp:moment_borne_Z} 
\MomZ_{p}(Z^{\delta})   \leqslant C(p), \qquad \MomZ_{p}(Z^{\delta}) := 1\vee \sup_{t \in  \pi^{\delta,\ast}}\mathbb{E}[\vert Z^{\delta}_{t}\vert_{\mathbb{R}^{N}} ^{p}] .
\end{align}
We denote $\mathbf{A}_{3}^{\delta}(+\infty)$ the assumption such that $\mathbf{A}_{3}^{\delta}(p)$ is satisfied for every $p \geqslant 0$. 
\end{enumerate}

A second assumption is made on the distribution of $Z^{\delta}$.  We suppose that the distribution of $Z^{\delta}$ is Lebesgue-lower-bounded:
\begin{enumerate}[label=$\mathbf{A}_{4}^{\delta}$.\quad\quad]
\item \label{hyp:moment_borne_Z} 
There exist $z_{\ast}=(z_{\ast ,t})_{t >0}$ taking values in $\mathbb{R}^{N}$ and $%
\varepsilon _{\ast },r_{\ast }>0$, all independent of $\delta$, such that for every Borel set $A\subset
\mathbb{R}^{N}$ and every $t \in \pi^{\delta,\ast},$%
\begin{align}
\mathbb{P}(Z^{\delta}_{t}\in A)\geqslant
\varepsilon _{\ast }\lambda_{\mathscr{L}} (A\cap B_{r_{\ast }}(z_{\ast ,t})) , \label{hyp:lebesgue_bounded}
\end{align}%
where $\lambda_{\mathscr{L}} $ is the Lebesgue measure on $\mathbb{R}^{N}.$
\end{enumerate}

Let us comment on the assumption $\mathbf{A}_{4}^{\delta}$.  First,  notice that (\ref{hyp:lebesgue_bounded}) holds if and only if there exist nonnegative measures $\mu^{\delta}_{t}$ with total mass $\mu^{\delta}_{t}(\mathbb{R}^{N})<1$ and a lower
semi-continuous function $\varphi \geqslant 0$ such that $\mathbb{P}(Z^{\delta}_{t}\in dz)=\mu^{\delta}_{t}(dz)+\varphi(z-z_{\ast ,t})dz$ for every $t\in \pi^{\delta,\ast}$. We also point out that the random variables $(Z^{\delta}_{t})_{t \in \pi^{\delta,\ast}}$ are not assumed to be identically distributed. However, the fact
that $r_{\ast }>0$ and $\varepsilon _{\ast }>0$ are the same for all $t \in \pi^{\delta,\ast}$
provides a mild uniformity condition in place of identical distribution. In order to construct a candidate for $\varphi$,
 we introduce the following function: 
 For $v>0$, set $\varphi _{v}:{%
\mathbb{R}^N}\rightarrow {\mathbb{R}}$ defined by
\begin{align}
\varphi _{v}(z)=\funclocgen (\frac{\vert z \vert_{\mathbb{R}^{N}} -v}{v}).  \label{def:fonction_regularisante_{d}ensite_Zk}
\end{align}%
where for every $x \in \mathbb{R}$,
\begin{align}
\label{eq:def_func_localization_generic}
\funclocgen (x) = \mathbf{1}_{x \leqslant 0}+  \frac{\exp(\frac{-1}{1-x})}{\exp(\frac{-1}{x})+\exp(\frac{-1}{1-x})}\mathbf{1}_{0 <x < 1}
\end{align}
Then $\varphi _{v}\in \mathcal{C}_{b}^{\infty }(\mathbb{R}^N;\mathbb{R})$, $0\leqslant \varphi
_{v}\leqslant 1$, $\varphi _{v}(z)=1$ if $\vert z \vert_{\mathbb{R}^{N}} \leqslant v$,  $\varphi _{v}(z)=0$ if $\vert z \vert_{\mathbb{R}^{N}} \geqslant 2v$, and we have the following crucial property: For every $p,q\in \mathbb{N}$, every $z\in \mathbb{R}^N$ with $\vert z \vert_{\mathbb{R}^{N}} <2v$,
\begin{align}
\vert \sum_{\underset{\vert \alpha^{z} \vert \in \{1,\ldots, q\}}{\alpha^{z} \in \mathbb{N}^{N}}} \vert \partial_{z}^{\alpha^{z}}  \ln \varphi_{v} (z)  \vert^{2} \vert^{\frac{p}{2}} \varphi_{v}(z) \leqslant 
\frac{C(N,q,p)}{ 1 \wedge v^{pq}},  \label{eq:borne_{d}eriv_log_reg_Zk}
\end{align}%
As an immediate consequence of (\ref{hyp:lebesgue_bounded}), for every nonnegative function $%
f:\mathbb{R}^{N}\rightarrow \mathbb{R}_{+}$ and $t \in \pi^{\delta}$, $t>0$,
\begin{align*}
\mathbb{E}[f(Z^{\delta}_{t})] \geqslant \varepsilon _{\ast }\int_{\mathbb{R}^{N}}\varphi _{r_{\ast
}/2}(\ z-z_{\ast ,t}\ )f(z) \mbox{d} z.  
\end{align*}%
We denote
\begin{align*}
m_{\ast }=\varepsilon _{\ast }\int_{\mathbb{R}^{N}}\varphi _{r_{\ast }/2}(
z )dz=\varepsilon _{\ast }\int_{\mathbb{R}^{N}}\varphi _{r_{\ast
}/2}( z-z_{\ast ,t} )\mbox{d} z  \in (0,1] .
\end{align*}
We assume without loss of generality that $m_{\ast}<1$. We consider a sequence of independent random variables $\chi^{\delta}_{t}\in
\{0,1\},\; U^{\delta}_{t}, V^{\delta}_{t}\in \mathbb{R}^{N}$, $t \in \pi^{\delta,\ast}$ (also independent of $W$, the Brownian motion driving (\ref{eq:eds_ito_intro})), with laws given by%
\begin{align*}
\mathbb{P}(\chi^{\delta} _{t} =1)= &m_{\ast },\qquad \mathbb{P}(\chi^{\delta}_{t}=0)=1-m_{\ast },  
\\
\mathbb{P}(\delta^{-\frac{1}{2}}U^{\delta}_{t} \in \mbox{d} z)=& \frac{\varepsilon _{\ast }}{m_{\ast }}\varphi _{r_{\ast
}/2}( z-z_{\ast ,t} ) \mbox{d}z ,  \nonumber \\
\mathbb{P}(\delta^{-\frac{1}{2}}V^{\delta}_{t} \in \mbox{d} z)=& \frac{1}{1-m_{\ast }}(\mathbb{P}(Z^{\delta}_{t}\in
\mbox{d} z)- \varepsilon_{\ast} \varphi _{\frac{r_{\ast
}}{2}}( z-z_{\ast ,t} ) \mbox{d} z),
\nonumber
\end{align*}
where $\varphi _{\frac{r_{\ast
}}{2}}$ satisfies (\ref{eq:borne_{d}eriv_log_reg_Zk}) with $v=\frac{r_{\ast
}}{2}$. In particular, our candidate for $\varphi$ in the decomposition of the law of $Z^{\delta}_{t}$, $t \in \pi^{\delta,\ast}$, is $z\mapsto  \varepsilon _{\ast }\varphi _{r_{\ast
}/2}( z )$. Note that $\mathbb{P}(V^{\delta}_{t}\in \mbox{d} z)\geqslant 0$ and that $\chi^{\delta}_{t}U^{\delta}_{t}+(1-\chi^{\delta}_{t})V^{\delta}_{t}$ follows the same law as $\delta^{\frac{1}{2}}  Z^{\delta}_{t}$. 
Henceforth, we use this Nummelin splitting
to represent the increments as
\begin{align*}
\delta^{\frac{1}{2}} Z^{\delta}_{t}=\chi^{\delta}_{t}U^{\delta}_{t}+(1-\chi^{\delta}_{t})V^{\delta}_{t}. 
\end{align*}
Throughout this paper, we work on a complete probability space $(\Omega,\mathcal{F},\mathbb{P})$ where $\mathcal{F}$ is the completion of the sigma-field generated by $(\chi^{\delta}_{t},U^{\delta}_{t}, V^{\delta}_{t} )_{t \in \pi^{\delta,\ast}}$ and $W$.
This classical Nummelin splitting has been used for convergence to equilibrium in \cite{Nummelin_1978} and \cite{Locherbach_Loukianova_2008} and to
study the central limit theorem in \cite{Bobkov_Chistyakov_Gennadiy_Gotze_2014_BerryEsseen}, \cite{Bobkov_Chistyakov_Gennadiy_Gotze_2014_FisherCLT} and \cite{Yu_Zaitsev_1995}. A closely related splitting, using a uniform distribution on an interval as the regular component, is used in \cite{Nourdin_Poly_2015} to establish invariance principles in total variation. 
Finally, in \cite{Bally_Rey_2016}, the Nummelin splitting is used to prove regularization properties of Markov semigroups under the uniform ellipticity condition $\HormV_{0}^{\infty}>0 $. 

We introduce a final purely technical assumption ensuring that the time step $\delta$ is small enough. For $\delta \in (0,1]$, when (\ref{eq:hyp_3_Norme_adhoc_fonction_schema}) holds, we first define



\begin{align}
\label{def:eta_delta}
\funceta_{1}(\delta)  :=  & \delta^{-d\frac{44}{91}} (\mathbf{1}_{L=0}+\mathbf{1}_{L>0}\frac{1}{2^{5} (1 \vee T^{3})^{\frac{1}{2}}})^{d13^{(L-1)}} \quad \mbox{and} \\
 \funceta_{2}(\delta)  :=  & \min(\delta^{-\frac{1}{2}} \funceta_{1}(\delta)^{-\frac{1}{d}}, \frac{1}{2} \vert \delta^{\frac{1}{2}} 8  \CstePsiouf_{3}  \vert^{-\frac{1}{\ExpPsiouf_{3}  +1}}) \nonumber ,
\end{align}

with $\ExpPsiouf_3 $ given in (\ref{eq:hyp_3_Norme_adhoc_fonction_schema}).
For $T \in \pi^{\delta,\ast}$, $x \in \mathbb{R}^{d}$, we introduce the following assumption:
\begin{enumerate}[label=$\mathbf{A}_{5}^{\delta}(x\mbox{,}T)$.]
\item \label{hyp:hyp_5_loc_var} Assume that (\ref{eq:hyp_3_Norme_adhoc_fonction_schema}) and $\mathbf{A}_{2}(x,L)$ (see (\ref{hyp:loc_hormander})) hold and that $\delta \in (0,(2^{\ExpPsiouf_{3}  +1}8\CstePsiouf_{3} )^{-4}]$ is small enough so that 
\begin{align}
\label{hyp:hyp_5_loc_var_eq}
\funceta_{1}(\delta) > & 2 \max(2^{\frac{d}{2}}, (\frac{4(L+1) \BornsupNl }{T \HormV_{L}(x) })^{d13^{L}} ,\\
&  \mathbf{1}_{L=0}+ \mathbf{1}_{L>0} ( \frac{10}{m_{\ast}} )^{d} (4^{\frac{13}{6}}  ( 2^{20} (1 \vee T^{6}) )^{143} (N+1)^{\frac{13}{12}(L-1)})^{d13^{L-1}} ) \quad \mbox{and} \nonumber \\
\funceta_{2}(\delta)>&1.\nonumber
\end{align}
Similarly to the assumption $\mathbf{A}_{2}(x,L)$, this condition is imposed at the initial point for $x= X^{\delta}_{0}$.  When it is not considered at initial point but more generally at $t \in \pi^{\delta}$, $t<T$, we denote $\mathbf{A}_{5}^{\delta}(x,t,T)$ defined as above with $T$ replaced by $T-t$ and $\HormV_{L}(x)$ replaced by $\HormV_{L}(x,t)$.
\end{enumerate}

Although the lower bound on $\funceta_{1}(\delta)$ in (\ref{hyp:hyp_5_loc_var_eq}) is independent of $\delta$, it may be very large. It could be reduced by modifying the proof, at the cost of worsening the bounds on the semigroup's derivatives. We instead tailor our proof to limit the dependence of $C(x,T)$ in (\ref{eq:reg_prop_intro}) on $\HormV_{L}(x)^{-1}$ and $T^{-1}$. The resulting exponents $\eta$ appearing in Theorem \ref{th:regul_main_result_intro} and Corollary \ref{coro:borne_densite_reg_gauss_intro} are made explicit in the proofs of these results. Note that the condition $\funceta_{2}(\delta)>1$ is actually implied by $\delta \leqslant (2^{\ExpPsiouf_{3}+1}8\CstePsiouf_{3})^{-4}$.

\subsection{An alternative regularization property}
We establish regularization estimates for a Gaussian perturbation of the discrete process. Passing to a product extension of the probability space, we consider a $d$-dimensional standard (centered with identity covariance matrix) Gaussian random variable 
$\Gaussian$ which is independent of $(\chi_t^\delta,U_t^\delta,V_t^\delta)_{t\in\pi^{\delta,\ast}}$ and $(W_{t})_{t \geqslant 0}$, and for $\theta>0$, 
\begin{align}
\label{def:semigroup_convolution}
Q_{T}^{\delta,\theta }f(x)  =\int_{\mathbb{R}^{d}} f(y) Q^{\delta,\theta}_{T}(x,\mbox{d} y)  :=  \mathbb{E}[f(X^{\delta}_{T}+\delta^{\theta } \Gaussian) \vert X^{\delta}_{0} = x], \quad T \in \pi^{\delta}.
\end{align}
This Gaussian convolution is directly simulable and is therefore suitable for Monte Carlo methods. More generally, for $s,t\in\pi^{\delta}$ with $s\leqslant t$, we set $Q^{\delta,\theta}_{s,t}f(x):=\mathbb{E}[f(X^{\delta}_{t}+\delta^{\theta}\Gaussian)\vert X^{\delta}_{s}=x]$, so that $Q^{\delta,\theta}_{T}=Q^{\delta,\theta}_{0,T}$.
\begin{theorem}
\label{th:regul_main_result_intro}
Let $T \in \pi^{\delta,\ast}$, $L \in \mathbb{N}$ and $f \in \mathcal{M}(\mathbb{R}^{d} ; \mathbb{R})$ satisfy the polynomial growth assumption: there exist $\CstePsiouf_{f} \geqslant 0$ and $\ExpPsiouf  _{f} \in \mathbb{N}$ such that for every $x \in \mathbb{R}^{d}$,
\begin{align*}
\vert f(x) \vert \leqslant \CstePsiouf_{f}(1+ \vert x \vert_{\mathbb{R}^{d}}^{\ExpPsiouf  _{f} }).
\end{align*}
 Then we have the following properties:

\begin{enumerate}[label=\textbf{\Alph*.}, ref=\Alph*]
\item \label{th:reg_gauss_regprop_intro} Let $x \in \mathbb{R}^{d}$, $q \in \mathbb{N}$ and $\alpha ,\beta \in \mathbb{N}^{d}$ such that $\vert \alpha
\vert +\vert \beta \vert \leqslant q$. Assume that $\mathbf{A}^{\delta}_{1}( \max(q+3,2 L + 5))$,  $\mathbf{A}_{2}(x,L)$, $\mathbf{A}_{3}^{\delta}(+\infty)$, $\mathbf{A}_{4}^{\delta}$ and $\mathbf{A}_{5}^{\delta}(x,T)$ hold.  Then,  if $f$ is $\vert \beta \vert$-times differentiable and $\partial^{\beta}f$ has polynomial growth,
\begin{align}
  \label{eq:borne_semigroupe_regularisation_gauss_main} 
 \vert \partial^{\alpha }Q_{T}^{\delta,\theta }\partial^{\beta}f(x) \vert \leqslant &  \CstePsiouf_{f} \frac{(1 + \vert x\vert_{\mathbb{R}^{d}}^{c }   )  C\exp(C T) }{(  \HormV_{L}( x) T )^{\eta }}  ,
\end{align}
where $\eta \geqslant 0$ depends on $d,L,q$ and $\theta$ and $c,C \geqslant 0$ depend on $d,N,L,q$, $\CstePsiouf_{\max(q+3,2 L + 5)},$ $\ExpPsiouf  _{\max(q+3,2 L + 5)},\ExpPsiouf  _{f},\frac{1}{m_{\ast}},\frac{1}{r_{\ast}},\theta$ and on the moments of $Z^{\delta}$.
\item \label{th:reg_gauss_distprop_intro}  Let $x \in \mathbb{R}^{d}$. Assume that $\mathbf{A}^{\delta}_{1}(2 L + 5)$,  $\mathbf{A}_{2}(x,L)$, $\mathbf{A}_{3}^{\delta}(+\infty)$, $\mathbf{A}_{4}^{\delta}$ and $\mathbf{A}_{5}^{\delta}(x,T)$ hold. Then, 
\begin{align}
\label{eq:reg_gauss_distprop_intro}
\vert Q^{\delta}_{T}f(x)-Q_{T}^{\delta,\theta }f(x)\vert 
\leqslant &  \delta^{\theta}   \CstePsiouf_{f} \frac{(1 +  \vert x  \vert_{\mathbb{R}^{d}}^{c }   ) C \exp(C T) }{(  \HormV_{L}(x ) T )^{\eta }},
\end{align}
where $\eta \geqslant 0$ depends on $d,L$ and $\theta$ and $c,C \geqslant 0$ depend on $d,N,L,$ $\CstePsiouf_{2 L + 5},$ $\ExpPsiouf  _{2 L + 5},\ExpPsiouf  _{f}$, $\frac{1}{m_{\ast}},\frac{1}{r_{\ast}},\theta$ and on the moments of $Z^{\delta}$.

\end{enumerate}
\end{theorem}
\begin{remark}
\label{rmk:reg_variants}
\begin{enumerate}
\item If $\ExpPsiouf  _{f}=\ExpPsiouf  _{r}=0$ for every $r \in \mathbb{N}^{\ast}$, one may take $c=0$ in (\ref{eq:borne_semigroupe_regularisation_gauss_main}) and (\ref{eq:reg_gauss_distprop_intro}). The same holds if (\ref{eq:hyp_1_Norme_adhoc_fonction_schema}) holds without the term $\vert x \vert_{\mathbb{R}^{d}}^{\ExpPsiouf_{r}}$ for every $r \in \mathbb{N}^{\ast}$ and $\ExpPsiouf_{f}=0$. Assuming further that $\mathbf{A}_{2}^{\infty}(L)$ holds, the upper bounds established in Theorem \ref{th:regul_main_result_intro} thus become uniform $w.r.t.$ $x$. We refer to Remark \ref{rmk:comparison_density_SDE} for the corresponding results for Corollaries \ref{coro:borne_densite_reg_gauss_intro} and \ref{coro:borne_densite_SDE} and Theorem \ref{th:invariance_main_result}.
\item \label{rmk:reg_shift}  The same estimates hold, for $s<t$ in $\pi^{\delta}$, with $Q^{\delta,\theta}_{T}$, $Q^{\delta}_{T}$, $T$ and $\HormV_{L}(x)$ replaced by $Q^{\delta,\theta}_{s,t}$, $Q^{\delta}_{s,t}$, $t-s$ and $\HormV_{L}(x,s)$, when $\mathbf{A}_{2}(x,L)$ and $\mathbf{A}_{5}^{\delta}(x,T)$ are replaced by $\mathbf{A}_{2}(x,s,L)$ and $\mathbf{A}_{5}^{\delta}(x,s,t)$.
\end{enumerate}
\end{remark}

A consequence of Theorem \ref{th:regul_main_result_intro} concerns the existence of quantitative derivative bounds for the density of $X^{\delta}_{T}+\delta^{\theta } \Gaussian $.  The proof of this result is given in Section \ref{sec:Proof of the main results}. Notice that an explicit value is given for $\eta$. 

\begin{corollary}
\label{coro:borne_densite_reg_gauss_intro} 
Let $T \in \pi^{\delta,\ast}$,  $x \in \mathbb{R}^{d}$ and $L \in \mathbb{N}$. Let $q \in \mathbb{N}$, let $\alpha ,\beta \in \mathbb{N}^{d}$ such that $\vert \alpha
\vert +\vert \beta \vert \leqslant q$. 
Assume that $\mathbf{A}^{\delta}_{1}( \max(q+d+3,2 L + 5))$, $\mathbf{A}_{2}(x,L)$, $\mathbf{A}_{3}^{\delta}(+\infty)$, $\mathbf{A}_{4}^{\delta}$ and $\mathbf{A}_{5}^{\delta}(x,T)$ hold. \\
 Then,  $Q^{\delta,\theta}_{T}(x,\mbox{d}y) = q^{\delta,\theta}_{T}(x ,y)\mbox{d} y$ where $q^{\delta,\theta}_{T} \in \mathcal{C}^{q}(\mathbb{R}^{d} \times \mathbb{R}^{d})$ satisfies, for every $p>0$,
\begin{align}
 \label{eq:borne_densite_reg_gauss_intro} 
 \vert \partial_{x}^{\alpha } \partial_{y}^{\beta } q_{T}^{\delta,\theta }(x,y) \vert \leqslant &  \frac{(1 +\vert x \vert_{\mathbb{R}^{d}}^{c }   )  C \exp(C T) }{( \HormV_{L}(x ) T )^{\eta }(1+ \vert y-x \vert_{\mathbb{R}^{d}}^{p})}  ,
\end{align}
where $\eta \geqslant 0$ depends on $d,L,q$ and $\theta$ and $c,C \geqslant 0$ depend on $d,N,L,q$, $\CstePsiouf_{\max(q+d+3,2 L + 5)}$,  $\ExpPsiouf_{\max(q+d+3,2 L + 5)}$, $\frac{1}{m_{\ast}},\frac{1}{r_{\ast}},\theta,p$ and on the moments of $Z^{\delta}$.

Moreover, if $\ExpPsiouf  _{2}=0$ (see hypothesis $\mathbf{A}^{\delta}_{1}$) and there exists $z^{\infty} \geqslant 1$ such that $a.s.$ $\sup_{t \in \pi^{\delta,\ast}} \vert Z^{\delta}_{t} \vert_{\mathbb{R}^{N}} \leqslant z^{\infty}$, then,
\begin{align}
 \label{eq:borne_densite_reg_gauss_borne_exp} 
 \vert \partial_{x}^{\alpha } \partial_{y}^{\beta } q_{T}^{\delta,\theta }(x,y) \vert \leqslant &  \frac{ (1 +\vert x \vert_{\mathbb{R}^{d}}^{c_1 }   ) C \exp(C T) }{(  \HormV_{L}(x) T )^{\eta }} \exp(- \frac{\vert y - x\vert_{\mathbb{R}^{d}}^{2}}{c_2 (T \vee \delta^{2 \theta})}) ,
\end{align}
where $\eta$ is the same as in (\ref{eq:borne_densite_reg_gauss_intro}), $c_2> 0$ depends on $d$, $\CstePsiouf_{1}$ and $z^{\infty} $, and $C,c_1 \geqslant 0$ depend on $d,N,L,q$, $\CstePsiouf_{\max(q+d+3,2 L + 5)}$, $\ExpPsiouf_{\max(q+d+3,2 L + 5)}$, $\frac{1}{m_{\ast}},\frac{1}{r_{\ast}},\theta$ and $z^{\infty}$.
\end{corollary}

\subsection{An invariance principle}
Let $(X_t)_{t\geqslant 0}$ denote the It\^o diffusion defined by (\ref{eq:eds_ito_intro}). Under $\mathbf{A}_{1}^{\delta}(2)$, (\ref{eq:hyp_1_Norme_adhoc_fonction_schema}) and (\ref{eq:hyp_3_Norme_adhoc_fonction_schema}) with $z=0$ show that the coefficients of (\ref{eq:eds_ito_intro}) are Lipschitz in space with linear growth, uniformly in time. Hence (\ref{eq:eds_ito_intro}) is well posed and, for every $p\geqslant 1$, $t\geqslant 0$ and $C=C(d,N,p,\CstePsiouf_{2},\CstePsiouf_{3})$ (see e.g. \cite[Corollary 2.5.12]{Krylov_1980} and the Burkholder-Davis-Gundy inequality),
\begin{align}
\label{eq:moments_diffusion}
\mathbb{E}_{x}[\vert X_{t}\vert_{\mathbb{R}^{d}}^{p}]^{\frac{1}{p}}&\leqslant C(1+\vert x\vert_{\mathbb{R}^{d}})\exp(Ct), \quad \mbox{and} \quad \mathbb{E}_{x}[\vert X_{t}-x\vert_{\mathbb{R}^{d}}^{p}]^{\frac{1}{p}} \leqslant Ct^{\frac{1}{2}}(1+\vert x\vert_{\mathbb{R}^{d}})\exp(Ct). 
\end{align}
Since the constants are uniform in time, the Markov property yields the same bounds for $\mathbb{E}[\,\cdot\,\vert X_{s}=u]$, $s\leqslant t$, with $(x,t)$ replaced by $(u,t-s)$.
 Theorem~\ref{th:invariance_main_result} establishes quantitative convergence of $X_T^\delta$ to $X_T$ in total variation, with bounds extending to measurable test functions of polynomial growth. This yields an invariance principle: the limit depends on $\psi$ only through its first-order derivative in $y$ and its first- and second-order derivatives in $z$ at $(x,t,0,0)$, and on the increments only through their first two moments. The theorem also provides quantitative density and derivative approximation bounds using $X^{\delta}_{T}+\delta^\theta G$, which admits a density even when $X^{\delta}_{T}$ does not.

\begin{theorem}
\label{th:invariance_main_result}
Let $T \in \pi^{\delta}$, with $T \geqslant 2 \delta$, and $L \in \mathbb{N}$. We have the following properties:
\begin{enumerate}[label=\textbf{\Alph*.}, ref=\Alph*]
\item \label{th:invariance_main_result_VT}
Let $f \in \mathcal{M}(\mathbb{R}^{d} ; \mathbb{R})$ be such that there exist $\CstePsiouf_{f} \geqslant 0$ and $\ExpPsiouf  _{f} \in \mathbb{N}$ such that for every $x \in \mathbb{R}^{d}$,
\begin{align*}
\vert f(x) \vert \leqslant \CstePsiouf_{f}(1+ \vert x \vert_{\mathbb{R}^{d}}^{\ExpPsiouf  _{f} }).
\end{align*}
Let $x \in \mathbb{R}^{d}$. Assume that $\mathbf{A}^{\delta}_{1}(\max(6,2L+5))$,  $\mathbf{A}_{2}(x,L)$, $\mathbf{A}_{3}^{\delta}(+\infty)$, $\mathbf{A}_{4}^{\delta}$ and $\mathbf{A}_{5}^{\delta}(x,T)$ hold.  
Then, 
\begin{align}
\label{eq:invariance_main_result_VT}
\vert \mathbb{E}_{x }[f(X_{T}) -f(X^{\delta}_{T}) ] \vert  \leqslant & \delta^{\frac{1}{2}} \CstePsiouf_{f} \frac{1 + \vert x  \vert_{\mathbb{R}^{d}}^{c }     }{( \HormV_{L}(x) T )^{\eta}} C\exp(C T)  ,
\end{align}
where $\eta \geqslant 0$ depends on $d$ and $L$ and $c,C \geqslant 0$ depend on $d,N,L$, $ \CstePsiouf_{\max(6,2L+5)}$,  $ \ExpPsiouf  _{\max(6,2L+5)}$, $\ExpPsiouf _{f}$,$\frac{1}{m_{\ast}}$, $\frac{1}{r_{\ast}}$, and on the moments of $Z^{\delta}$.
\item \label{th:invariance_main_result_density} Assume that $\mathbf{A}^{\delta}_{1}(+\infty)$, $\mathbf{A}_{3}^{\delta}(+\infty)$ and $\mathbf{A}_{4}^{\delta}$ hold, and denote by $U$ the set of $x \in \mathbb{R}^{d}$ such that $\mathbf{A}_{2}(x,L)$ holds. Then, for every $x \in U$, the law of $X_{T}$ started at $x$ has a density $y \in \mathbb{R}^{d} \mapsto p_{T}(x,y)$ with $p_{T} \in \mathcal{C}^{\infty}(U \times \mathbb{R}^{d})$.\\
Moreover, for every $x \in U$ such that $\mathbf{A}_{5}^{\delta}(x,T)$ holds, 
for every $\theta \geqslant \frac{1}{2}$, $q \in \mathbb{N}$,  $\alpha ,\beta \in \mathbb{N}^{d}$ with $\vert \alpha
\vert +\vert \beta \vert \leqslant q$, $p \geqslant 0$,  $\epsilon >0$ and every $y\in \mathbb{R}^{d}$,
\begin{align}
\label{eq:invariance_main_result_density}
\vert  \partial_{x}^{\alpha } \partial_{y}^{\beta }  p_{T}(x,y) - \partial_{x}^{\alpha } \partial_{y}^{\beta } q_{T}^{\delta,\theta }(x ,y) \vert  \leqslant & \delta^{\frac{1}{2}-\epsilon} \frac{(1 + \vert x  \vert_{\mathbb{R}^{d}}^{c}     )C \exp(C T)}{ (\HormV_{L}(x) T )^{\eta} (1+\vert y-x \vert_{\mathbb{R}^{d}}^{p})}  ,
\end{align}
where $\eta \geqslant 0$ depends on $d,L,q,\theta$ and $\frac{1}{\epsilon}$ and $c,C \geqslant 0$ depend on $d,N,L,q,$ $\frac{1}{m_{\ast}},\frac{1}{r_{\ast}},\theta,p,\frac{1}{\epsilon}$,  on $\CstePsiouf_{r}$, $\ExpPsiouf _{r}$ for some $r \in \mathbb{N}^{\ast}$ large enough, and on the moments of $Z^{\delta}$.
\end{enumerate}
\end{theorem}

\begin{remark}
\label{rmk:main_theo}
\begin{enumerate}
\item 
For two probability measures $\mu$ and $\nu$ on the Borel $\sigma$-algebra of $\mathbb{R}^{d}$, the total variation distance between $\mu$ and $\nu$ satisfies
\begin{align*}
d_{TV}(\mu , \nu )
= & \sup_{f \in \mathcal{M}(\mathbb{R}^{d};\mathbb{R}), \Vert f \Vert_{\infty} \leqslant 1} \frac{1}{2} \vert \mu(f) - \nu(f) \vert  
=  \sup_{f \in \mathcal{C}^{\infty}_{K}(\mathbb{R}^{d};\mathbb{R}), \Vert f \Vert_{\infty} \leqslant 1} \frac{1}{2} \vert \mu(f) - \nu(f) \vert   ,
\end{align*}
where $\mu(f) = \int_{\mathbb{R}^{d}} f(x) \mu( \mbox{d} x)$ and similarly for $\nu(f)$. The last equality above is a direct consequence of Lusin's theorem.

In particular, (\ref{eq:invariance_main_result_VT}) provides a bound on the total variation distance between the law of $X_{T}$ starting from $x  \in \mathbb{R}^{d}$ (denoted $P_{T}(x,.)$) and that of $X^{\delta}_{T}$ also starting from $x $ (denoted $Q^{\delta}_{T}(x ,.)$). Under the hypotheses of Theorem \ref{th:invariance_main_result} \ref{th:invariance_main_result_VT}, 
\begin{align}
\label{eq:invariance_main_result_VT_Lusin}
d_{TV}( P_{T}(x ,.),Q^{\delta}_{T}(x ,.) )\leqslant & \delta^{\frac{1}{2}}  \frac{1 +  \vert x \vert_{\mathbb{R}^{d}}^{c}     }{(  \HormV_{L}(x) T )^{\eta}} C \exp(C T) .
\end{align}
\item \label{rem:improve_rate_invariance} If we suppose in addition that $\theta\geqslant  2$ and that for every $t \in \pi^{\delta,\ast}$, $i,j,l \in \EnsN$, $\mathbb{E}[Z_{t}^{\delta,i} Z_{t}^{\delta,j} Z_{t}^{\delta,l}]=0$ and we replace $\mathbf{A}^{\delta}_{1}(\max(6,2L+5))$ by $\mathbf{A}^{\delta}_{1}(\max(7,2L+5))$, then Theorem \ref{th:invariance_main_result} (and also (\ref{eq:invariance_main_result_VT_Lusin})) holds with $\delta^{\frac{1}{2}}$ (resp.  $\delta^{\frac{1}{2}-\epsilon}$) replaced by $\delta$ (resp.  $\delta^{1-\epsilon}$) and $(\CstePsiouf_{\max(6,2 L + 5)},\ExpPsiouf  _{\max(6,2 L + 5)})$ replaced by $(\CstePsiouf_{\max(7,2 L + 5)},\ExpPsiouf  _{\max(7,2 L + 5)})$
 in the $r.h.s.$ of (\ref{eq:invariance_main_result_VT}) (resp. (\ref{eq:invariance_main_result_density})).
\item \label{rem:higher_order_rate} More generally, let $m>0$. For every $\delta'\in\{\frac{\delta}{2^{k}},k\in\mathbb{N}\}$, we define $(X^{\delta'}_{t})_{t\in\pi^{\delta'}}$ as in (\ref{eq:schema_general}), with $\delta$ replaced by $\delta'$, using the same function $\psi$ and independent random variables $(Z^{\delta'}_{t})_{t\in\pi^{\delta',\ast}}$ such that $Z^{\delta'}_{t}$ has the same law as $Z^{\delta}_{\delta\lceil t/\delta\rceil}$.  Assume that, in addition to the hypotheses of Theorem \ref{th:invariance_main_result}, $\mathbf{A}_1^\delta(+\infty)$ holds and there exists $q(m)\in\mathbb{N}$ such that, for every $f \in \mathcal{C}^{\infty}_{pol}(\mathbb{R}^{d};\mathbb{R})$ satisfying for every $\alpha \in \mathbb{N}^{d}$ and every $x \in \mathbb{R}^{d}$, $
\vert \partial^{\alpha}f(x) \vert \leqslant \CstePsiouf_{f,\alpha}(1+ \vert x \vert^{p(\alpha)}),
$
with $\CstePsiouf_{f,\alpha} \geqslant 1$ and $p(\alpha) \geqslant 0$, for every $\delta'\in\{\frac{\delta}{2^{k}},k\in\mathbb{N}\}$ and $t \in \pi^{\delta'}$ such that $t+\delta'\leqslant T$, we have
\begin{align}
\label{eq:short_time_error}
\vert  \mathbb{E}[f(X^{\delta'}_{t+\delta'})-f(X_{t+\delta'})\vert X_{t}=X^{\delta'}_{t}=x] \vert \leqslant (\delta')^{m+1} \sum_{\vert \alpha \vert \leqslant q(m)} \CstePsiouf_{f,\alpha}  C (1+ \vert x \vert^{p}),
\end{align}
where $C$ and $p$ depend on $f$ only through $(p(\alpha))_{\vert\alpha\vert\leqslant q(m)}$, and do not depend on $x$, $t$ or $\delta'$. Assume also that $\theta \geqslant m+1$.
Then, Theorem \ref{th:invariance_main_result} holds with $\delta^{\frac{1}{2}}$ (resp. $\delta^{\frac{1}{2}-\epsilon}$)  replaced by $\delta^{m}$ (resp. $\delta^{m-\epsilon}$) and $(\CstePsiouf_{\max(6,2 L + 5)},\ExpPsiouf  _{\max(6,2 L + 5)})$ replaced by $(\CstePsiouf_{r},\ExpPsiouf  _{r})$ for $r$ large enough in the $r.h.s.$ of (\ref{eq:invariance_main_result_VT}) (resp. (\ref{eq:invariance_main_result_density})) and also of (\ref{eq:invariance_main_result_VT_Lusin}). In this case $\eta,c$ and $C$ may depend on $m$. 

Under the assumptions of point \ref{rem:improve_rate_invariance}, (\ref{eq:short_time_error}) holds with $m=1$.
\end{enumerate}
\end{remark}

Combining Corollary \ref{coro:borne_densite_reg_gauss_intro} with the density approximation estimate (\ref{eq:invariance_main_result_density}), we obtain the following quantitative bounds for the limiting diffusion. For $q \in \mathbb{N} $, we introduce the following class of functions with polynomially bounded derivatives: $\mathcal{C}^{q}_{\mbox{pol}}(\mathbb{R}^{d} \times \mathbb{R}_{+};\mathbb{R}^{d} )  :=    \{f  \in \mathcal{C}^{q}(\mathbb{R}^{d} \times \mathbb{R}_{+};\mathbb{R}^{d} ), \exists \CstePsiouf_{f,q} \geqslant 1,\ExpPsiouf  _{f,q} \in \mathbb{N},  \forall (x,t)  \in \mathbb{R}^{d} \times \mathbb{R}_{+}, \sum_{l+\vert \alpha \vert \leqslant q}\vert \partial_{x}^{\alpha}  \partial_{t}^{l}  f(x,t) \vert_{\mathbb{R}^{d}} \leqslant \CstePsiouf_{f,q} (1+\vert x \vert_{\mathbb{R}^{d}}^{\ExpPsiouf  _{f,q} } ) \}$ and $\mathcal{C}^{\infty}_{\mbox{pol}}(\mathbb{R}^{d} \times \mathbb{R}_{+};\mathbb{R}^{d} ) = \cap_{q \in \mathbb{N}} \mathcal{C}^{q}_{\mbox{pol}}(\mathbb{R}^{d} \times \mathbb{R}_{+};\mathbb{R}^{d} ) $.

\begin{corollary}
\label{coro:borne_densite_SDE}
Let $T>0$ and $L \in \mathbb{N}$. Let $q \in \mathbb{N}$, let $\alpha ,\beta \in \mathbb{N}^{d}$ such that $\vert \alpha
\vert +\vert \beta \vert \leqslant q$. Assume that for every $i \in \{0,\ldots ,N\}$, $V_{i} \in \mathcal{C}^{\infty}_{\mbox{pol}}(\mathbb{R}^{d} \times \mathbb{R}_{+};\mathbb{R}^{d} ) $ and that their first-order spatial derivatives are uniformly bounded by $ \CstePsiouf_{V_{i},1} $.

Let $U$ be the set of $x \in \mathbb{R}^{d}$ such that $\mathbf{A}_{2}(x,L)$ holds. Then, for every $x \in U$, the law of $X_T$ started at $x$ admits a density $p_T(x,\cdot)$, with $p_T \in \mathcal{C}^{\infty}(U\times\mathbb{R}^{d})$. Moreover, for every $p>0$, $x \in U$ and $y \in \mathbb{R}^{d}$,
\begin{align}
\label{eq:borne_densite_SDE}
\vert \partial_x^{\alpha}\partial_y^{\beta}p_T(x,y)\vert
\leqslant &
\frac{(1+\vert x\vert_{\mathbb{R}^{d}}^{c})C\exp(CT)}
{(\HormV_L(x)T)^{\eta}(1+\vert y-x\vert_{\mathbb{R}^{d}}^{p})},
\end{align}
where $\eta \geqslant 0$ depends on $d,L$ and $q$ and $c,C \geqslant 0$ depend on $d,N,L,q,p$ and on $\CstePsiouf_{V_{i},r}$, $\ExpPsiouf _{V_{i},r}$, $i \in \{ 0,\ldots,N \}$, for some $r \in \mathbb{N}^{\ast}$ large enough. In particular, the regularization estimate (\ref{eq:borne_semigroupe_regularisation_gauss_main}) holds with $Q_T^{\delta,\theta}$ replaced by the diffusion semigroup $P_T$, where $P_{T}f(x):=\mathbb{E}[f(X_{T})\vert X_0=x]$.\\
\end{corollary}

\begin{remark}
\label{rmk:comparison_density_SDE}
The preceding corollary can be proved by taking $Z^{\delta}$ to be standard Gaussian and $\psi$ to be the Euler scheme. 
To the best of our knowledge, quantitative bounds of the form (\ref{eq:borne_densite_SDE}) are not available in this generality for time-inhomogeneous diffusions under a local weak space--time H\"ormander condition imposed only at the initial point. In the homogeneous setting, similar quantitative estimates are obtained in \cite{Kusuoka_Stroock_1985_AMCII}; see also \cite{Pigato_2018} for local Gaussian density estimates in a two-dimensional weak H\"ormander setting. For time-inhomogeneous diffusions, quantitative density and derivative estimates under a spatial weak H\"ormander condition are obtained in \cite{Chen_Zhou_1991}, where the corresponding non-degeneracy is required at the starting spatial point for all initial times. In \cite{Cattiaux_Mesnager_2002}, smoothness is considered under a local weak space--time H\"ormander condition at the starting point, while quantitative density estimates are obtained under a local strong spatial H\"ormander condition. Two-sided Gaussian density bounds under a uniform weak H\"ormander condition are obtained in \cite{Delarue_Menozzi_2010} for a specific chain structure, while quantitative density and derivative estimates under a local weak H\"ormander condition at the initial point are established in \cite{Pigato_2022} for specific chain structures.\\
Finally, under the conditions of Remark \ref{rmk:reg_variants}, one may also take $c=0$ in (\ref{eq:borne_densite_reg_gauss_intro}) and (\ref{eq:invariance_main_result_VT}), but not in (\ref{eq:invariance_main_result_density}), and $c_{1}=0$ in (\ref{eq:borne_densite_reg_gauss_borne_exp}). Similarly, one may take $c=0$ in (\ref{eq:borne_densite_SDE}) when $V_{0},\ldots,V_{N}\in\mathcal{C}^{\infty}_{b}(\mathbb{R}^{d}\times\mathbb{R}_{+};\mathbb{R}^{d})$. In this case, $p_{T}$ moreover satisfies the Gaussian upper bound
\begin{align*}
\vert \partial_{x}^{\alpha}\partial_{y}^{\beta}p_{T}(x,y)\vert \leqslant \frac{C\exp(CT)}{(\HormV_{L}(x)T)^{\eta}}\exp\Big(-\frac{\vert y-x\vert_{\mathbb{R}^{d}}^{2}}{c_{2}T}\Big),
\end{align*}
where $\eta$ is as in (\ref{eq:borne_densite_SDE}), $c_{2}>0$ depends on $d$, $N$ and $\sup_{i \in \EnsN }\Vert V_{i}\Vert_{\infty}$ and $C>0$ does not depend on $x$, $y$ and $T$. This Gaussian bound relies on (\ref{eq:borne_densite_reg_gauss_borne_exp}) and can be obtained using bounded increments, for instance $Z^{\delta}_{t}$ uniformly distributed on $[-\sqrt{3},\sqrt{3}]^{N}$, which are centered with identity covariance, satisfy $\mathbf{A}_{3}^{\delta}(+\infty)$ and $\mathbf{A}_{4}^{\delta}$, and are bounded by $z^{\infty}=\sqrt{3N}$.
\end{remark}

\begin{example}
\begin{enumerate}
\item
Let us consider the two-dimensional process $X=(X^{1},X^{2})$ started from the point $x=(x^{1},x^{2}) \in \mathbb{R}^{2}$ and satisfying the SDE
\begin{align*}
d X^{1}_{t} =& \mu(X^{1}_{t},t) \mbox{d} t +  \sigma(X^{1}_{t},t)  \mbox{d} W_{t} \\
d X^{2}_{t} =& X^{1}_{t} \mbox{d} t  ,
\end{align*}
where $(W_{t})_{t \geqslant 0}$ is a one-dimensional standard Brownian motion, and $\mu$ and $\sigma$ are smooth with bounded derivatives of order one and polynomial bounds for higher orders.  In the setting from (\ref{eq:eds_ito_intro}), we have $V_{0}:(u,t) \mapsto (\mu(u^{1},t),u^{1})$ and $V_{1}:(u,t) \mapsto (\sigma(u^{1},t),0)$. In this example local ellipticity holds for $X^{1}$ as long as $\sigma(x^{1},t) \neq 0$. However, ellipticity does not hold for $X$ since $\mbox{dim} ( \mbox{span} ((\sigma,0)))(x,0) \leqslant 1$. Nevertheless, let us compute the Lie brackets. In particular 
\begin{align*}
[\bar{V}_{0},V_{1}]:(u,t) \mapsto (\partial_{u^{1}} \sigma (u^{1},t) \mu (u^{1},t)- \partial_{u^{1}} \mu(u^{1},t) \sigma(u^{1},t) +\frac{1}{2} \sigma(u^{1},t)^{2}  \partial_{u^{1}}^{2}  \sigma(u^{1},t)  , -\sigma(u^{1},t) 
),
\end{align*}
and,  for $\sigma(x^{1},0) \neq 0$, $\mbox{span} ((\sigma,0),(\partial_{x^{1}} \sigma  \mu - \partial_{x^{1}} \mu \sigma+\frac{1}{2} \sigma^{2}  \partial_{x^{1}}^{2}  \sigma+\partial_{t} \sigma , - \sigma ))(x ,0) = \mathbb{R}^{2}$ so that local weak H\"ormander condition holds. Now, let us consider the Euler scheme of 
$X$, given by $(X^{\delta,1}_{0} ,X^{\delta,2}_{0} )=x$ and  for $t \in \pi^{\delta}$,
\begin{align*}
X^{\delta,1}_{t+\delta} =& X^{\delta,1}_{t} +\mu(X^{\delta,1}_{t},t) \delta +  \sigma(X^{\delta,1}_{t},t) \sqrt{\delta} Z^{\delta}_{t+\delta}\\
X^{\delta,2}_{t+\delta} =& X^{\delta,2}_{t}+ X^{\delta,1}_{t} \delta ,
\end{align*}
with $Z^{\delta}_{t} \in \mathbb{R}$, $t \in \pi^{\delta,\ast} $ satisfying $\mathbf{A}_{3}^{\delta}(+\infty)$ and $\mathbf{A}_{4}^{\delta}$ and such that for every $t \in \pi^{\delta,\ast}$, $\mathbb{E}[(Z^{\delta} _{t})^{3}]=0$.  With notations introduced in (\ref{def:hormander_func}), for $\sigma(x^{1},0) \neq 0$,
\begin{align*}
\HormV_{1}(x) =&  1 \wedge \inf_{b \in \mathbb{R}^{2}, \vert b \vert_{\mathbb{R}^{2}} =1} (  \langle V_{1}(x,0) , b \rangle_{\mathbb{R}^{2}}^{2} + \langle [V_{0}- \frac{1}{2} \nabla_{x}V_{1} V_{1},V_{1}](x,0) +\partial_{t}V_{1}(x,0) , b \rangle_{\mathbb{R}^{2}}^{2}  )\\
=& 1 \wedge \inf_{b \in \mathbb{R}^{2}, \vert b \vert_{\mathbb{R}^{2}} =1} ( \langle (\sigma,0) , b \rangle_{\mathbb{R}^{2}}^{2} + \langle (\partial_{u^{1}} \sigma \mu-\partial_{u^{1}} \mu \sigma+\frac{1}{2} \sigma^{2} \partial_{u^{1}}^{2}  \sigma+\partial_{t}\sigma,-\sigma) , b \rangle_{\mathbb{R}^{2}}^{2})(x^{1},0)  \\
>&0 .
\end{align*}
For every $f \in \mathcal{M}(\mathbb{R}^{2};\mathbb{R})$ satisfying the assumptions from Theorem \ref{th:invariance_main_result}, \ref{th:invariance_main_result_VT}, we have, for $T \in \pi^{\delta}$ with $T \geqslant 2\delta$ and $\delta$ sufficiently small so that $\mathbf{A}_{5}^{\delta}(x,T)$ holds,
\begin{align*}
\vert \mathbb{E}_{x}[f(X_{T}) -f(X^{\delta}_{T})  ] \vert  \leqslant & \delta \CstePsiouf_{f} \frac{1 + \vert x\vert_{\mathbb{R}^{2}}^{c }     }{(  \HormV_{1}(x) T )^{\eta}} C\exp(C T)  .
\end{align*}
Moreover, by Corollary \ref{coro:borne_densite_SDE}, for every $T>0$, the law of $X_{T}$ started from $x$ admits a $\mathcal{C}^{\infty}(\mathbb{R}^{2})$ density $p_{T}(x,\cdot)$ satisfying (\ref{eq:borne_densite_SDE}).

\item The same argument yields a total variation central limit theorem for Brownian motion and its iterated time integrals.\\
We consider $Z_{l} \in \mathbb{R}$, $l \in \mathbb{N}^{\ast}$, which are $i.i.d.$, centered with variance one, and have a Lebesgue-lower-bounded distribution and finite moments of any order. We then define $S^{(0)}_{l}  := n^{-\frac{1}{2}}  \sum_{k=1}^{l} Z_{k}$, $n \in \mathbb{N}^{\ast}$, $l \in \mathbb{N}$, and for $h \in \mathbb{N}^{\ast}$, $S^{(h)}_{l}  :=  n^{-1} \sum_{k=1}^{l} S^{(h-1)}_{k}$. \\
Then, for every $h \in \mathbb{N}$, $( S^{(0)}_{l}, \ldots, S^{(h)}_{l})_{l \in \mathbb{N}}$ satisfies the uniform weak H\"ormander property with $L=h$ and $( S^{(0)}_{n}, \ldots, S^{(h)}_{n})$ converges in total variation distance, as $n$ tends to infinity, toward the random variable $(W_{1},\int_{0}^{1} W_{s} \mbox{d} s, \ldots, \int_{0}^{1} \ldots \int_{0}^{s_{2}} W_{s_{1}} \mbox{d} s_{1} \ldots \mbox{d} s_{h})$ where $(W_{t})_{t \geqslant 0}$ is a one-dimensional standard Brownian motion.
\item Let $d=2$, $N=1$ and $\psi(x,t,z,y)=x+z(\cos t,\sin t)$. Let $(Z_t^{\delta})_{t\in\pi^{\delta,\ast}}$ be independent, centered random variables with unit variance, satisfying $\mathbf{A}^{\delta}_3(+\infty)$, $\mathbf{A}^{\delta}_4$ and $\mathbb{E}[(Z_t^{\delta})^3]=0$ (for example, uniformly distributed on $[-\sqrt{3},\sqrt{3}]$). The discrete process and its limiting diffusion, starting from $x\in\mathbb{R}^2$, satisfy
\begin{align*}
X_{t+\delta}^{\delta} &= X_t^{\delta}+\sqrt{\delta}Z_{t+\delta}^{\delta}(\cos t,\sin t), \qquad dX_t=(\cos t,\sin t)\mbox{d}W_t.
\end{align*}
Here $V_0=0$, $V_1(x,t)=(\cos t,\sin t)$ and $V_1^{[(0)]}(x,t)=\partial_t V_1(x,t)=(-\sin t,\cos t)$. Thus $\HormV_1(x,t)=1$ for every $(x,t)\in\mathbb{R}^2\times\mathbb{R}_+$, whereas all spatial Lie brackets vanish. Since $\mathbf{A}^{\delta}_{1}(+\infty)$ holds, for every $f \in \mathcal{M}(\mathbb{R}^{2};\mathbb{R})$ satisfying the assumptions from Theorem \ref{th:invariance_main_result}, \ref{th:invariance_main_result_VT}, we have, for $T \in \pi^{\delta}$ with $T \geqslant 2\delta$ and $\delta$ sufficiently small so that $\mathbf{A}_{5}^{\delta}(x,T)$ holds,
\begin{align*}
\vert \mathbb{E}_{x}[f(X_{T})-f(X_{T}^{\delta})] \vert
\leqslant \delta \CstePsiouf_{f}\frac{1+|x|_{\mathbb{R}^{2}}^{c}}{(\HormV_{1}(x)T)^{\eta}}C\exp(CT).
\end{align*}
\end{enumerate}
\end{example}

\section{A Malliavin-inspired approach to prove smoothing properties}
\label{Sec:prove_Regularization_properties}
In this section, we establish Theorem \ref{th:regularisation}, which provides a regularization result for a localized modification of $Q^{\delta}$. Although this modification is not readily simulable, it is a key ingredient in the proofs of Theorem \ref{th:regul_main_result_intro}, Corollary \ref{coro:borne_densite_reg_gauss_intro} and Theorem \ref{th:invariance_main_result} (see Section \ref{sec:Proof of the main results}).

The proof of Theorem \ref{th:regularisation} relies on discrete integration-by-parts formulas and estimates for the associated Malliavin weights (Theorem \ref{theoreme:IPP_Malliavin}). These estimates follow from bounds on Malliavin-Sobolev norms (Theorem \ref{theo:Norme_Sobolev_borne}) and localized moments of the inverse Malliavin covariance matrix (Theorem \ref{th:borne_Lp_inv_cov_Mal}).

\subsection{A generic discrete-time Malliavin calculus}
In this section, we present the discrete Malliavin calculus tailored to our framework and derive integration-by-parts formulas and estimates on the Malliavin weights.\\

We work on the probability space $(\Omega,\mathcal{F}^{\delta},\mathbb{P})$, where $\mathcal{F}^{\delta}$ is generated by $(\chi^{\delta}_{t},U^{\delta}_{t},V^{\delta}_{t})_{t\in\pi^{\delta,\ast}}$. Since $Z^{\delta}$ need not be Gaussian, standard Malliavin calculus is not directly suited to our setting. We therefore develop a discrete-time differential calculus inspired by Malliavin calculus, adapted to innovations with a regular component, as ensured by the Lebesgue lower bound. Throughout this section, we assume that $\mathbf{A}_{4}^{\delta}$ (see (\ref{hyp:lebesgue_bounded})) holds. For $\EnsT\subset\pi^{\delta,\ast}$, we denote $\vert\EnsT\vert=\mbox{Card}(\EnsT)$.
%

In the following, we will denote $\chi^{\delta} =(\chi^{\delta} _{t})_{t \in \pi^{\delta,\ast}}, \; U^{\delta}=(U^{\delta}_{t})_{t \in \pi^{\delta,\ast}}$ and $V^{\delta}=(V^{\delta} _{t})_{t \in \pi^{\delta,\ast}}$. Given a finite-dimensional Hilbert space $(\mathcal{H},\langle.,.\rangle_{\mathcal{H}})$ with dimension $d_{\mathcal{H}}$, equipped with an orthonormal basis $\BasemathcalH :=   (\ElemBaseH_{n})_{n \in \{1,\ldots,d_{\mathcal{H}}\}}$, we will consider the class of random variables: 
\begin{align*}
\mathcal{S}^{\delta}(\mathcal{H})=\{ & F=f(\chi^{\delta}_{t} ,U^{\delta}_{t} ,V^{\delta}_{t} ,t \in \EnsT), \EnsT \subset  \pi^{\delta,\ast}, \vert \EnsT \vert <+\infty 
\\
& \forall  \chi \in \{0,1\}^{\EnsT} , (u,v) \mapsto f(\chi,u,v) \in \mathcal{C}^{\infty}_{b}(\mathbb{R}^{\EnsN \times \EnsT}\times \mathbb{R}^{\EnsN \times \EnsT};\mathcal{H}) \}.  \nonumber
\end{align*}

When $\mathcal{H}=\mathbb{R}$, we simply denote $\mathcal{S}^{\delta}$. For $T \geqslant 0$, we introduce $\mathcal{S}^{\delta}_{T}(\mathcal{H})$ defined as in $\mathcal{S}^{\delta}(\mathcal{H})$ but with the additional constraint that $\sup\{t \in \EnsT\} \leqslant T$. Our applications will be limited to cases where $\mathcal{H} = \mathbb{R}^{l_{1} \times \ldots \times l_{m}}$ with $l_{1} , \ldots ,l_{m},m \in \mathbb{N}^{\ast}$. However, we consider $\mathcal{H}$ as an abstract finite-dimensional Hilbert space since this introduces no additional difficulties.  

We now construct a differential calculus based on the laws of the random
variables $U^{\delta}$ which mimics the Malliavin calculus, following the ideas from \cite{Bally_Clement_2011}, \cite{Bally_Caramellino_2014_distance}, \cite{Bally_Caramellino_2016_CLT} or \cite{Bally_Rey_2016}. We begin by introducing the basic element of our differential calculus. 

\noindent \textbf{Derivative operator and Malliavin covariance matrix.} 

 For $F\in \mathcal{S}^{\delta}(\mathcal{H})$, we define the Malliavin derivatives $\DeriveMal F  :=   (\DeriveMal _{(t,i)}F)_{(t,i) \in \pi^{\delta,\ast} \times \EnsN} \in \mathcal{S}^{\delta}(\mathcal{H})^{\pi^{\delta,\ast}\times \EnsN}$ by
\begin{align*}
\DeriveMal_{(t,i)}F
 :=   \chi^{\delta}_{t}\partial _{u_{t}^{i}} f(\chi^{\delta},U^{\delta},V^{\delta}), \quad (t,i) \in \pi^{\delta,\ast} \times \EnsN.
\end{align*}
We recall that there always exists a finite $\EnsT  \subset \pi^{\delta,\ast}$ such that $\DeriveMal_{(t,i)}F=0$ for $t \notin \EnsT$. In particular $\DeriveMal F $ is finite dimensional and we use the notation $\DeriveMal F :=   (\DeriveMal _{(t,i)}F)_{(t,i) \in \EnsT \times \EnsN} \in \mathcal{S}^{\delta}(\mathcal{H})^{\EnsT\times \EnsN}$.
We also extend the derivative outside of the time grid $\pi^{\delta,\ast}$. For $s \in (t-\delta,t]$, with $t \in \pi^{\delta,\ast} $ we define 
\begin{align*}
\DeriveMal_{(s,i)}F :=  \DeriveMal_{(t,i)}F ,
\end{align*}
and $\DeriveMal_{(0,i)}F=0$.
The higher order derivatives are defined by iterating $\DeriveMal$. In particular, for $m \in \mathbb{N}$, we introduce $\DeriveMal^{m}$. Let $\alpha=(\alpha^{1},\ldots,\alpha^{m}) \in (\pi^{\delta,\ast} \times \EnsN)^{m}$, $m \in \mathbb{N}$. In the sequel, we will use the notation $\Vert \alpha \Vert = m$. We define
\begin{align*}
\DeriveMal_{\alpha }F=\DeriveMal_{\alpha^{1}}\cdots \DeriveMal_{\alpha^{m}}F
\end{align*}
when $m>0$ and $\DeriveMal_{\alpha }F=\DeriveMal_{\emptyset}F=F$ if $m=0$. In particular, we will denote
\begin{align*}
\DeriveMal^{m}F=(\DeriveMal_{\alpha }F)_{\alpha \in  (\pi^{\delta,\ast} \times \EnsN)^{m} }=(\DeriveMal_{\alpha }F)_{\alpha \in  (\EnsT \times \EnsN)^{m} }
\end{align*}
The Malliavin covariance matrix of $F \in \mathcal{S}^{\delta}(\mathcal{H})$ is the matrix defined for every $\ElemBaseH,\ElemBaseH^{\prime} \in \BasemathcalH$ by 
\begin{align}
\MatCovMal _{F }[\ElemBaseH,\ElemBaseH^{\prime}]
&=\delta \langle \DeriveMal  \langle F, \ElemBaseH\rangle_{\mathcal{H}}  , \DeriveMal\langle F, \ElemBaseH^{\prime}\rangle_{\mathcal{H}} \rangle_{\mathbb{R}^{ \pi^{\delta,\ast} \times \EnsN}}  \nonumber\\
&  :=    \delta
\sum_{t \in \pi^{\delta,\ast}} \sum_{l \in \EnsN}\DeriveMal_{(t,l)}\langle F, \ElemBaseH\rangle_{\mathcal{H}} \DeriveMal_{(t,l)}\langle F, \ElemBaseH^{\prime}\rangle_{\mathcal{H}}   .
\label{def:matrice_covariance_Malliavin}
\end{align}
Notice that this quantity is well defined as the sum has only finitely many nonzero terms. Moreover, we have the common integral representation formula,
\begin{align*}
\MatCovMal_{F}[\ElemBaseH,\ElemBaseH^{\prime}]= \int_{0}^{+\infty} \sum_{l \in \EnsN}\DeriveMal_{(t,l)}\langle F, \ElemBaseH\rangle_{\mathcal{H}} \DeriveMal_{(t,l)}\langle F, \ElemBaseH^{\prime}\rangle_{\mathcal{H}}    \mbox{d} t .
\end{align*}

We remark that $\MatCovMal_{F}$ can be seen as a linear operator on $\mathcal{H}$ such that for every $\ell \in \mathcal{H}$, $\MatCovMal_{F}\ell :=   \sum_{\ElemBaseH,\ElemBaseH^{\prime} \in \BasemathcalH} \MatCovMal _{F}[\ElemBaseH,\ElemBaseH^{\prime}] \langle \ell, \ElemBaseH^{\prime}  \rangle_{\mathcal{H}} \ElemBaseH $, that is the standard matrix product. We then introduce the inverse Malliavin covariance matrix. In particular, when the determinant $\det \MatCovMal 
_{F}$ of $ \MatCovMal 
_{F}$ is not zero, we define $\InvMatCovMal_{F}$, 
the inverse of the Malliavin covariance matrix $ \MatCovMal 
_{F}$ satisfying, for every $\ell \in \mathcal{H}$, $\MatCovMal_{F}\InvMatCovMal_{F} \ell=\InvMatCovMal_{F} \MatCovMal_{F} \ell =\ell$. We also use the convention $ \det \InvMatCovMal_F:=+\infty$ on the set $\{ \det \MatCovMal_F=0\}$.

\noindent \textbf{Divergence and Ornstein-Uhlenbeck operators.}
Let $G=\left(G_t \right)_{t \in \pi^{\delta,\ast}}$ with $G_t \in \mathcal{S}^{\delta}(\mathcal{H}^{\mathcal{N}})$ and $G_t \neq 0$ for a finite number of $t$. The divergence operator $\IntSkorohod $ is given by
\begin{align*}
\IntSkorohod (G) =- \delta \sum_{t \in  \pi^{\delta,\ast}}\sum_{i \in \EnsN} G^{i}_{t} \DeriveMal_{(t,i)} \Gamma_{t}+\DeriveMal_{(t,i)}G^{i}_{t} \in \mathcal{H}
\end{align*}
with, for $t\in \pi^{\delta,\ast}$, 
\begin{align*}
\Gamma_{t}= \ln \varphi_{r_{\ast }/2} ( \delta^{-\frac{1}{2}} U^{\delta}_{t}-z_{\ast ,t})\in \mathbb{R},
\end{align*}
which is $a.s.$ well defined since the density of $ \delta^{-\frac{1}{2}} U^{\delta}_{t}-z_{\ast ,t}$ is proportional to $\varphi_{r_{\ast }/2} $ that is defined in (\ref{def:fonction_regularisante_{d}ensite_Zk}). In particular, for $i \in \EnsN$,
\begin{align*}
\DeriveMal_{(t,i)}\Gamma_{t}=  \delta^{-\frac{1}{2}} \chi^{\delta}_{t} \partial_{z^{i}} \ln \varphi_{r_{\ast }/2} ( \delta^{-\frac{1}{2}} U^{\delta}_{t}-z_{\ast ,t}) ,
\end{align*}
and using (\ref{eq:borne_{d}eriv_log_reg_Zk}), it belongs to $\cap_{p \geqslant 1} \LLL_{p}(\Omega;\mathbb{R})$  and so $\IntSkorohod (G) \in \cap_{p \geqslant 1}  \LLL_{p}(\Omega;\mathcal{H})$.

Finally, we define the Ornstein-Uhlenbeck operator, for $F \in \mathcal{S}^{\delta}(\mathcal{H})$,%
\begin{align*}
\OpOrnsteinUlh F=-\IntSkorohod(\DeriveMal F)=  \delta \sum_{t \in \pi^{\delta,\ast} }\sum_{i \in \EnsN} \DeriveMal _{(t,i)} \DeriveMal_{(t,i)}F  +  \DeriveMal_{(t,i)}F \DeriveMal_{(t,i)}\Gamma_{t} . 
\end{align*}

We observe that in our framework, the duality formula reads as
follows: For each $F\in \mathcal{S}^{\delta}(\mathcal{H})$ and $G=\left(G_t \right)_{t \in \pi^{\delta,\ast}}$ with $G_t \in \mathcal{S}^{\delta}(\mathcal{H}^{\mathcal{N}})$
\begin{align}
\mathbb{E}[\langle F, \IntSkorohod( G) \rangle_{\mathcal{H}}]=\delta \mathbb{E}[\langle \DeriveMal F,  G \rangle_{\mathcal{H}^{\pi^{\delta,\ast} \times \EnsN}} ]   :=   & \delta \sum_{t \in \pi^{\delta,\ast}} \sum_{i \in \EnsN}\mathbb{E}  [ \langle \DeriveMal_{(t,i)}F ,  G_{t}^{i} \rangle_{\mathcal{H}}  ] .  \label{formule_dualite_pure}
\end{align}
This leads to the following relationship: For each $F,G\in \mathcal{S}^{\delta}(\mathcal{H})$,
\begin{align}
\mathbb{E}[\langle F,\OpOrnsteinUlh G \rangle_{\mathcal{H}}]= \mathbb{E}[\langle G,  \OpOrnsteinUlh F \rangle_{\mathcal{H}}]=& - \delta \mathbb{E} [ \langle \DeriveMal F , \DeriveMal G \rangle_{\mathcal{H}^{\pi^{\delta,\ast} \times \EnsN}} ]  \nonumber \\
 =  &- \delta \sum_{t \in \pi^{\delta,\ast}} \sum_{i \in \EnsN}\mathbb{E}  [ \langle \DeriveMal_{(t,i)}F ,  \DeriveMal_{(t,i)}G \rangle_{\mathcal{H}}  ].    \label{formule_{d}ualite}
\end{align}%

We prove directly (\ref{formule_{d}ualite}). The duality relationship (\ref{formule_dualite_pure}) is proven by similar arguments. This follows immediately using the independence structure and standard
integration-by-parts on $\mathbb{R}^{N}$: Indeed, if $f,g\in \mathcal{C}^{2}(\mathbb{R}^{N};\mathbb{R})$ and $t \in \pi^{\delta,\ast}$, using that the density of $U^{\delta}_{t}$ is $u \mapsto \frac{\varepsilon _{\ast }}{m_{\ast }}\delta^{-\frac{N}{2}} \varphi_{r_{\ast }/2} (\delta^{-\frac{1}{2}} u-z_{\ast ,t})$, we obtain 
\begin{align*}
\sum_{i \in \EnsN}\mathbb{E}[\partial_{u^i}f(U^{\delta}_{t})\partial_{u^i}g(U^{\delta}_{t})] 
=& -\mathbb{E} [f(U^{\delta}_{t})\sum_{i \in \EnsN} \partial_{u^i}^{2}g(U^{\delta}_{t})+\partial_{u^i}g(U^{\delta}_{t})\delta^{-\frac{1}{2}}\partial_{z^i}\ln \varphi_{r_{\ast }/2}(\delta^{-\frac{1}{2}} U^{\delta}_{t}-z_{\ast ,t}) ].
\end{align*}%

Now, we consider $F,G\in \mathcal{S}^{\delta}(\mathcal{H})$, 
so that $F=f(\chi^{\delta}_{t_{i}} ,U^{\delta}_{t_{i}},V^{\delta}_{t_{i}},i \in \{1,\ldots,n\})$ and $G=g(\chi^{\delta}_{t_{i}} ,U^{\delta}_{t_{i}},V^{\delta}_{t_{i}},i \in \{1,\ldots,n\})$.
Moreover, we introduce the functions $f_{l} :=   \langle f,\ElemBaseH_{l} \rangle_{\mathcal{H}},g_{l} :=   \langle g,\ElemBaseH_{l} \rangle_{\mathcal{H}}$, $l \in \{1,\ldots,d_{\mathcal{H}}\}$, which belong to $\mathcal{C}^{\infty}((\mathbb{R}^{N})^{n};\mathbb{R})$.  It follows from the calculus above that
\begin{align*}
\mathbb{E} [ \langle & \DeriveMal F , \DeriveMal G \rangle_{\mathcal{H}^{ \pi^{\delta,\ast} \times \EnsN}} ]=\sum_{l=1}^{d_{\mathcal{H}}}\sum_{t \in  \pi^{\delta,\ast}  }\sum_{i \in \EnsN}\mathbb{E}[\chi^{\delta}_{t} \partial _{u_{t}^{i}}f_{l}(\chi^{\delta} ,U^{\delta},V^{\delta}) \partial_{u_{t}^{i}}g_{l}(\chi^{\delta} ,U^{\delta} ,V^{\delta}) ] \\
=&- \mathbb{E} [\sum_{t \in \pi^{\delta,\ast} } \sum_{i \in \EnsN} \chi^{\delta}_{t}\sum_{l=1}^{d_{\mathcal{H}}} f_{l}  
\partial^{2}_{u_{t}^{i}}g_{l}(\chi^{\delta} ,U^{\delta},V^{\delta})+f_{l}\partial_{u_{t}^{i}}g_{l}(\chi^{\delta}
,U^{\delta},V^{\delta}) \delta^{-\frac{1}{2}}  \partial_{z^{i}}\ln \varphi_{r_{\ast }/2}(\delta^{-\frac{1}{2}}  U^{\delta}_{t}-
z_{\ast ,t}) ]\\
=& - \mathbb{E} [\langle F,\sum_{t \in \pi^{\delta,\ast}}\sum_{i \in \EnsN}\DeriveMal_{(t,i)}\DeriveMal_{(t,i)}G+\DeriveMal_{(t,i)}G\DeriveMal_{(t,i)}\Gamma _{t} \rangle_{\mathcal{H}}] = - \delta^{-1} \mathbb{E}[\langle F,\OpOrnsteinUlh  G \rangle_{\mathcal{H}}],
\end{align*}
which is exactly (\ref{formule_{d}ualite}). We have the following standard chain rule: Let $\phi \in \mathcal{C}^{\infty}(\mathcal{H};\mathcal{H}^{\prime})$ with $\mathcal{H}^{\prime}$ a finite-dimensional Hilbert space and $F\in \mathcal{S}^{\delta}(\mathcal{H})$. Then
\begin{align}
 \DeriveMal \phi(F)= \DeriveDirec _{\DeriveMal F}\phi(F) \in \mathcal{S}^{\delta}(\mathcal{H}^{\prime})^{\pi^{\delta,\ast} \times \EnsN} .
 \label{formule_chain_rule}
\end{align}
More specifically, when $\mathcal{H}^{\prime}=\mathbb{R}$ we have 
\begin{align}
 \DeriveMal \phi(F)=\langle \DeriveDirec  \phi(F), \DeriveMal F \rangle_{\mathcal{H}} \in \mathcal{S}^{\delta}(\mathbb{R})^{\pi^{\delta,\ast} \times \EnsN} .
 \label{formule_chain_rule_R}
\end{align}

Moreover, it follows from (\ref{formule_chain_rule}) and the duality relation (or
direct computation), that 
\begin{align}
\OpOrnsteinUlh \phi(F)=\langle \DeriveDirec  \phi(F), \OpOrnsteinUlh F \rangle_{\mathcal{H}} + \delta
\sum_{\ElemBaseH,\ElemBaseH^{\prime} \in \BasemathcalH} \DeriveDirec _{\ElemBaseH^{\prime}}\DeriveDirec _{\ElemBaseH} \phi(F)  \langle  \DeriveMal \langle F, \ElemBaseH \rangle_{\mathcal{H}},\DeriveMal \langle F, \ElemBaseH^{\prime} \rangle_{\mathcal{H}} \rangle_{\mathbb{R}^{\pi^{\delta,\ast}  \times \EnsN}}   .
\label{formule_chain_rule_Operateur_OU}
\end{align}

We next introduce the Malliavin-Sobolev norms used to estimate the weights appearing in the integration-by-parts formulas of Theorem \ref{theoreme:IPP_Malliavin}.

\noindent \textbf{Sobolev-Malliavin norm.}
 For $F \in \mathcal{S}^{\delta}(\mathcal{H})$, $q \in \mathbb{N}$,
	we begin by introducing the Malliavin-Sobolev norms:%
\begin{align}
\vert F\vert _{\mathcal{H} ,1,q}^{2}=\sum_{\underset{j \in \{1,\ldots,q\}}{
\alpha \in (\pi^{\delta,\ast}  \times \EnsN )^{j}}
 }
 \delta^{j}\vert \DeriveMal_{\alpha}F\vert_{\mathcal{H}}^{2} ,\qquad \vert
F\vert_{\mathcal{H} ,0,q}^{2}=\vert
F\vert_{\mathcal{H} ,q}^{2}=\vert F\vert_{\mathcal{H}}^{2}+\vert F\vert_{\mathcal{H},1,q}^{2} 
 \label{def:Norme_Sobolev}
\end{align}%
and for $p\geqslant 1$
\begin{align}
\label{def:Norme_Sobolev_moment}
\Vert F\Vert_{\mathcal{H} ,1,q,p}  =\Vert \vert F\vert _{\mathcal{H},1,q} \Vert_{p} \qquad \Vert F\Vert_{\mathcal{H} ,0,q,p}=\Vert F\Vert_{\mathcal{H} ,q,p}  =\Vert F \Vert_{\mathcal{H},p}+\Vert F\Vert_{\mathcal{H}, 1,q,p}  ,
\end{align}


where $\Vert F \Vert_{\mathcal{H},p}=\mathbb{E}[\vert F\vert _{\mathcal{H}}^{p}]^{\frac{1}{p}}$, which we simply denote $\Vert F \Vert_{p}$ when $\mathcal{H}=\mathbb{R}$. Owing to the duality relationship (\ref{formule_dualite_pure}), the operator $\DeriveMal$, initially defined on $\mathcal{S}^{\delta}(\mathcal{H})$ is closable from $\LLL_{p}(\Omega;\mathcal{H})$ into $\LLL_{p}(\Omega;\mathcal{H}^{\pi^{\delta,\ast}  \times \EnsN})$ for $p \geqslant 1$. More specifically, let us consider $(F_{n})_{n \in \mathbb{N}}$ taking values in $ \mathcal{S}^{\delta}(\mathcal{H})$ which converges to $0$ in $\LLL_{p}(\Omega;\mathcal{H})$ and such that $(\DeriveMal F_{n})_{n \in \mathbb{N}}$ converges to $\zeta$ in $\LLL_{p}(\Omega;\mathcal{H}^{\pi^{\delta,\ast} \times \EnsN})$. For every $G \in  \mathcal{S}^{\delta}(\mathcal{H}^{\pi^{\delta,\ast} \times \EnsN})$ with $G_t \neq 0$ for a finite number of $t$, let us denote
\begin{align*}
\tilde{G}_{t}=G_{t}  \exp(\Gamma_{t}) , \quad t \in \pi^{\delta,\ast} .
\end{align*}
Since $\varphi_{r_{\ast }/2} \in \mathcal{C}^{\infty}_{b}(\mathbb{R}^{N};\mathbb{R})$, we have $\tilde{G} \in  \mathcal{S}^{\delta}(\mathcal{H}^{\pi^{\delta,\ast} \times \EnsN})$  and using (\ref{formule_dualite_pure}), we have
\begin{align*}
\delta  \mathbb{E}[ \langle \DeriveMal F_{n}  ,\tilde{G} \rangle_{\mathcal{H}^{\pi^{\delta,\ast} \times \EnsN}} ]  
 =& \mathbb{E}[\langle F_{n}  , \IntSkorohod (\tilde{G}) \rangle_{\mathcal{H}} ] .
\end{align*}
Since $\tilde{G}$ is bounded, the $l.h.s.$ above tends to $\delta \mathbb{E}[ \langle \zeta ,\tilde{G} \rangle_{\mathcal{H}^{\pi^{\delta,\ast} \times \EnsN}} ] $. Moreover, exploiting (\ref{eq:borne_{d}eriv_log_reg_Zk}), for every $t  \in \pi^{\delta,\ast} $ and $i \in \EnsN$,  $ \exp(\Gamma_{t}) \DeriveMal_{(t,i)} \Gamma_{t}$ is bounded. Therefore $ \IntSkorohod (\tilde{G}) $ is bounded and the $r.h.s.$ tends to $0$. Consequently, for every $G \in  \mathcal{S}^{\delta}(\mathcal{H})^{\pi^{\delta,\ast} \times \EnsN}$ with $G_t \neq 0$ for a finite number of $t$,  
\begin{align*}
 \mathbb{E}[ \langle \zeta,\tilde{G} \rangle_{\mathcal{H}^{\pi^{\delta,\ast} \times \EnsN}} ] = \mathbb{E}[\sum_{t \in  \pi^{\delta,\ast}} \sum_{i \in \EnsN} \exp(\Gamma_{t}) \langle \zeta_{t,i} , G_{t}^{i} \rangle_{\mathcal{H}} ] =0.
\end{align*} 
It follows that $\zeta=0$, which proves the closability. For higher orders, the result is similar.

%
%
%

In particular, we introduce $\mathbb{D}^{q,p}(\mathcal{H})$ (resp. $\mathbb{D}^{q,p}_{T}(\mathcal{H})$, $T \geqslant 0$) the completion of $\mathcal{S}^{\delta}(\mathcal{H})$ (resp. $\mathcal{S}^{\delta}_{T}(\mathcal{H})$) by the norm $\Vert \cdot \Vert_{\mathcal{H},q,p}$ and extend continuously the definition of the norms (\ref{def:Norme_Sobolev}) (resp. (\ref{def:Norme_Sobolev_moment})) from $F \in \mathcal{S}^{\delta}(\mathcal{H})$ to $F \in \mathbb{D}^{q,1}(\mathcal{H})$ (resp. $\mathbb{D}^{q,p}(\mathcal{H})$). Then, for $F \in \mathbb{D}^{q,p}(\mathcal{H})$, we define $\DeriveMal ^{q} F$ as the limit in $  \LLL_{p}(\Omega;\mathcal{H}^{(\pi^{\delta,\ast} \times \EnsN)^{q}})$ of $\DeriveMal ^{q} F_{n}$ where $(F_{n})_{n \in \mathbb{N}}$ is a $ \mathcal{S}^{\delta}(\mathcal{H})$-valued sequence converging to $F$ in the norm $\Vert \cdot \Vert_{\mathcal{H},q,p}$.  Similarly, we define $ \IntSkorohod (G) $ for every $G $ that belongs to
\begin{align*}
\mbox{Dom}(\IntSkorohod ) \coloneqq \{u \in \LLL_{2}(\Omega;\mathcal{H}^{\pi^{\delta,\ast} \times \EnsN}),  \exists C \geqslant 0, \forall F \in \mathbb{D}^{1,2}(\mathcal{H}),\vert  \mathbb{E}[ \langle \DeriveMal F  ,u \rangle_{\mathcal{H}^{\pi^{\delta,\ast} \times \EnsN}} ]   \vert \leqslant C \Vert F \Vert_{\mathcal{H},0,2}    \},
\end{align*}
which satisfies (\ref{formule_dualite_pure}) for every $F \in  \mathbb{D}^{1,2}(\mathcal{H})$. Moreover, defining $\mbox{Dom}(\OpOrnsteinUlh):=\{F\in\mathbb{D}^{1,2}(\mathcal{H}):\DeriveMal F\in\mbox{Dom}(\IntSkorohod)\}$, the identity (\ref{formule_{d}ualite}) extends to every $F,G\in\mbox{Dom}(\OpOrnsteinUlh)$, by the duality relation (\ref{formule_dualite_pure}) and the definition $\OpOrnsteinUlh=-\IntSkorohod\DeriveMal$.\\
Moreover, by density and H\"older's inequality, the chain rule
(\ref{formule_chain_rule}) extends to every $F \in \cap_{p \geqslant 1}\mathbb{D}^{1,p}(\mathcal{H})$ and $\phi \in \mathcal{C}^{1}_{pol}(\mathcal{H};\mathcal{H}^{\prime})$ and $\phi(F) \in \cap_{p \geqslant 1} \mathbb{D}^{1,p}(\mathcal{H}^{\prime})$.
For fixed $T\geqslant0$ and $\delta>0$, the explicit divergence formula, H\"older's inequality and (\ref{eq:borne_{d}eriv_log_reg_Zk}) show that $\IntSkorohod$ extends continuously from $\mathbb{D}^{q+1,2p}_{T}(\mathcal{H}^{((0,T]\cap\pi^{\delta,\ast})\times\EnsN})$ to $\mathbb{D}^{q,p}_{T}(\mathcal{H})$, $q\in\mathbb{N}$ and $p\geqslant1$.
By density, the duality relation (\ref{formule_dualite_pure}) remains valid for $F\in\mathbb{D}^{1,2}_{T}(\mathcal{H})$ and $G\in\mathbb{D}^{1,4}_{T}(\mathcal{H}^{((0,T]\cap\pi^{\delta,\ast})\times\EnsN})$.
Similarly, using (\ref{eq:borne_{d}eriv_log_reg_Zk}) and H\"older's inequality, the operator $\OpOrnsteinUlh$ extends continuously from $\mathbb{D}^{q+2,2p}_{T}(\mathcal{H})$ into $\mathbb{D}_{T}^{q,p}(\mathcal{H})$, for every $q \in \mathbb{N}$, $p\geqslant 1$ and $T \geqslant 0$. Consequently, formula (\ref{formule_chain_rule_Operateur_OU}) remains valid in $\LLL_{p}(\Omega;\mathcal{H}^{\prime})$, $p \geqslant 1$ for $F \in \cap_{p \geqslant 1} \mathbb{D}^{2,p}_{T}(\mathcal{H})$ and $\phi \in \mathcal{C}^{2}_{pol}(\mathcal{H};\mathcal{H}^{\prime})$.

Finally, for a random variable $F=f(\chi^{\delta},U^{\delta},V^{\delta})$ such that $u \mapsto f(\chi,u,v)$ is $\mathcal{C}^{\infty}$ (possibly only on an open set, with derivatives which are not necessarily bounded), we define $\DeriveMal F$ and the norms (\ref{def:Norme_Sobolev}) pointwise, exactly as for $F \in \mathcal{S}^{\delta}(\mathcal{H})$. When $F \in \cap_{p \geqslant 1}\mathbb{D}^{q,p}(\mathcal{H})$, they coincide almost surely with the ones defined by closure. 

\begin{remark}
\label{rmk:deriv_mal_Z}
The basic random variables in our calculus are $Z^{\delta}_{t}, t \in \pi^{\delta,\ast}$, so we specify the way in which the differential operators act on them.  Since $\delta^{\frac{1}{2}}Z^{\delta}_{t}=\chi^{\delta}_{t}U^{\delta}_{t}+(1-\chi^{\delta}_{t})V^{\delta}_{t}$, we may have $Z^{\delta}_{t} \notin \mathcal{S}^{\delta}(\mathbb{R}^{N})$. However using a similar localization approach as in the proof of Theorem \ref{theoreme:IPP_Malliavin} below, we can deduce from $\mathbf{A}_{3}^{\delta}(+\infty)$ and $\mathbf{A}_{4}^{\delta}$ that $Z^{\delta}_{t} \in \cap_{q \in \mathbb{N}}\cap_{p\geqslant 1}\mathbb{D}^{q,p}(\mathbb{R}^{N})$ for every $t \in \pi^{\delta,\ast}$. Moreover, for every $w\in \pi^{\delta,\ast}$,  $i,j \in \EnsN$,
\begin{align}
\delta^{\frac{1}{2}} \DeriveMal_{(w,i)}Z^{\delta,j}_{t} =&\chi^{\delta}_{t}\mathbf{1}_{w,t} \mathbf{1}_{i,j},
\label{eq:derivee_Zkj} 
\end{align}
and
\begin{align}
 \OpOrnsteinUlh Z^{\delta,i}_{t} =&  \chi^{\delta}_{t} \partial_{z^{i}} \ln \varphi_{r_{\ast }/2} (\delta^{-\frac{1}{2}} U^{\delta}_{t}-z_{\ast ,t})  .  \label{eq:L_Zkj} 
\end{align}
Furthermore, when $q\geqslant 1$, under assumptions $\mathbf{A}_{1}^{\delta}(q)$ (see (\ref{eq:hyp_1_Norme_adhoc_fonction_schema}) and (\ref{eq:hyp_3_Norme_adhoc_fonction_schema})), and $\mathbf{A}_{3}^{\delta}(+\infty)$ (see (\ref{eq:hyp:moment_borne_Z})), $X^{\delta}_{t} \in \mathbb{D}_{t}^{q,p}(\mathbb{R}^{d})$ for every $p \geqslant 1$, $t \in \pi^{\delta}$ and we have similar results for the tangent process introduced in (\ref{eq:tangent_flow_def}).
\end{remark}


\noindent \textbf{Integration by parts formula.}
Below, we define the Malliavin weights that appear in our integration-by-parts formulas. 
Let $\ElemBaseH \in \BasemathcalH$. Assume that $F \in \mathbb{D}^{2,p}_{T}(\mathcal{H})$, $G \in \mathbb{D}_{T}^{1,p}(\mathbb{R})$,  $T\geqslant 0$, and $\mathbb{E}[\vert \det \InvMatCovMal_{F } \vert^{p}\mathbf{1}_{ \vert
G \vert>0}]<+\infty$ for every $p \geqslant 1$. We define
\begin{align*}
\PoidMal (F,G)[\ElemBaseH]  :=    & - \langle G \InvMatCovMal_{F }\OpOrnsteinUlh   F, \ElemBaseH \rangle_{\mathcal{H}} - \delta \sum_{\ElemBaseH^{\prime}\in \BasemathcalH}  \langle \DeriveMal( G \InvMatCovMal_{F}[\ElemBaseH,\ElemBaseH^{\prime}]), \DeriveMal\langle F,\ElemBaseH^{\prime}\rangle_{\mathcal{H}} \rangle_{\mathbb{R}^{\pi^{\delta,\ast} \times \EnsN}}.
\end{align*}
We also use the convention $\PoidMal (F,G)[\emptyset]=G$. Note that (see the proof of Theorem \ref{theoreme:IPP_Malliavin} below), almost surely, $\PoidMal (F,G)[\ElemBaseH]  =0$ if $G=0$. Considering higher order integration-by-parts formulas, for $
F \in \mathbb{D}^{q+1,p}_{T}(\mathcal{H})$, $G \in \mathbb{D}^{q,p}_{T}(\mathbb{R})$ and $\ElemMultidimBaseH=(\ElemBaseH^{1},\ldots,\ElemBaseH^{q})\in
(\BasemathcalH)^{q}$ we define $\PoidMal(F,G)[\ElemMultidimBaseH]$ by the recurrence
\begin{align}
\PoidMal(F,G)[\ElemMultidimBaseH] :=   \PoidMal(F,\PoidMal(F,G)[\ElemBaseH^{1},\ldots,\ElemBaseH^{q-1}]) [\ElemBaseH^{q}].
 \label{eq:IPP_Malliavin_degre_superieur_weights}
\end{align}

Exploiting the relationships (\ref{formule_{d}ualite}) and (\ref{formule_chain_rule_R}), we obtain the following integration-by-parts formula and an estimate of the Sobolev norms of the weights. 

\begin{theorem}
\label{theoreme:IPP_Malliavin}

Let $q \in \mathbb{N}^{\ast}$, $T\geqslant 0$, $\phi \in \mathcal{C}^{q}_{pol}(\mathcal{H};\mathbb{R})$. Let $F$ and $G$ be such that, for every $p\geqslant 1$ $
F \in \mathbb{D}_{T}^{q+1,p}(\mathcal{H})$, $G \in \mathbb{D}^{q,p}_{T}(\mathbb{R})$ and
$\mathbb{E}[\vert \det \InvMatCovMal_{F } \vert^{p}\mathbf{1}_{ \vert
G \vert>0}]<+\infty$.
Then, for every $\ElemMultidimBaseH=(\ElemBaseH^{1},\ldots,\ElemBaseH^{q})\in
(\BasemathcalH)^{q}$, 
\begin{align}
\mathbb{E}[\DeriveDirec _{\ElemMultidimBaseH}\phi (F)G]= \mathbb{E}[\phi (F) \PoidMal(F,G)[\ElemMultidimBaseH]] ,
 \label{eq:IPP_Malliavin_degre_superieur}
\end{align}%
with $\PoidMal(F,G)[\ElemMultidimBaseH]$ defined in (\ref{eq:IPP_Malliavin_degre_superieur_weights}). Moreover, for every $m \in \mathbb{N}$, if $F \in \mathbb{D}^{q+m+1,p}_{T}(\mathcal{H})$ and $G \in \mathbb{D}^{q+m,p}_{T}(\mathbb{R})$ for every $p \geqslant 1$, then
\begin{align}
\vert \PoidMal(F,G)[\ElemMultidimBaseH]\vert_{\mathbb{R},m}\leqslant & C(d_{\mathcal{H}} ,q,m) \mathbf{c}(d_{\mathcal{H}} ,q,m,F,G), \label{eq:borne_norme_sobolev_poids_malliavin_th}
\end{align}
with, 
\begin{align*}
\mathbf{c}(d_{\mathcal{H}} ,q,m, F,G)=& \mathbf{1}_{\vert G \vert > 0}(1\vee \det \InvMatCovMal_{F}
)^{q(m+q+1)} (1+\vert F\vert
_{\mathcal{H},1,q+m+1}^{2 d_{\mathcal{H}}  q(q+m+2)}+
\vert \OpOrnsteinUlh F\vert_{\mathcal{H},q+m-1}^{2q})\vert
G\vert_{\mathbb{R}, m+q} .
\end{align*}
\end{theorem}

\begin{proof}[Proof of Theorem \ref{theoreme:IPP_Malliavin}]
We first introduce a regularized version of the indicator function.  For $v>1$, we consider $\Psi_{v} \in \mathcal{C}_{b}^{\infty}(\mathbb{R}; [0,1])$ such that $\Psi_{v}(x)=1$ if $\vert x \vert \leqslant v -\frac{1}{2} $ and $0$ if $\vert x \vert \geqslant v $ and such that the function $z \in \mathbb{R}^{N} \mapsto \Psi_{v}(\vert z \vert_{\mathbb{R}^{N}})$ belongs to $ \mathcal{C}_{b}^{\infty}(\mathbb{R}^{N}; [0,1])$. In particular we choose $  
\Psi_{v}(x)=\funclocgen (2\vert x \vert -2v+1) $ for all $x \in \mathbb{R}$ with $\funclocgen $ defined in (\ref{eq:def_func_localization_generic}). We also use the convention $\Psi_v(+\infty)=0$ below. This regularization is first used to prove that
\begin{align}
\label{eq:conserv_0_mall}
\mathbf{1}_{\{G=0\}}\DeriveMal^{j}G=0,\qquad j  \in \{1,\ldots,q\}.
\end{align}
 Let $h_{\varepsilon}: y \in \mathbb{R} \mapsto y\Psi_2(\frac{y}{\varepsilon})$, $\varepsilon \in (0,1]$. Note that $h_{\varepsilon}$ and $h_{\varepsilon}'$ are uniformly bounded in $\varepsilon \in (0,1]$. Since, almost surely, $h_{\varepsilon}(G)\to0$, and $\DeriveMal h_{\varepsilon}(G)=h'_{\varepsilon}(G)\DeriveMal G\to\mathbf{1}_{\{G=0\}}\DeriveMal G$ in $\LLL_p$  as $\varepsilon$ tends to $0$, the closability of $\DeriveMal$ implies $\mathbf{1}_{\{G=0\}}\DeriveMal G=0$. Applying the same argument successively yields (\ref{eq:conserv_0_mall}). 
 
 We now use localization to prove that for every $\ElemBaseH,\ElemBaseH'\in\BasemathcalH$,
 \begin{align}
 \label{eq:reg_loc_Mall}
G\InvMatCovMal_F[\ElemBaseH,\ElemBaseH'] \in\mathbb{D}^{q,p}_T(\mathbb{R}),
\end{align}
where $G\InvMatCovMal_F=0$ on $\{G=0\}$. For $n\geqslant2$, consider $G\InvMatCovMal_F\Psi_n(\det\InvMatCovMal_F)$, extended by zero on $\{\det\MatCovMal_F=0\}$ by using $\Psi_n(+\infty)=0$. For every $\ElemBaseH,\ElemBaseH'\in\BasemathcalH$, the chain rule shows that $G\InvMatCovMal_F[\ElemBaseH,\ElemBaseH']\Psi_n(\det\InvMatCovMal_F)\in\bigcap_{p\geqslant1}\mathbb{D}^{q,p}_T(\mathbb{R})$. Moreover, the chain rule and locality yield, for $0\leqslant j\leqslant q$,
\begin{align*}
\big\vert\DeriveMal^j\big(G\InvMatCovMal_F[\ElemBaseH,\ElemBaseH']\Psi_n(\det\InvMatCovMal_F)\big)\big\vert
\leqslant C\mathbf{1}_{\{G\neq0\}}(1\vee\det\InvMatCovMal_F)^{C}
(1+\vert F\vert_{\mathcal{H},q+1}+\vert G\vert_{\mathbb{R},q})^{C},
\end{align*}
where $C \geqslant 1$ is independent of $n$. By H\"older's inequality and the localized inverse-moment assumption, it belongs to $\LLL_p$ for every $p \geqslant 1$. Since $\Psi_n(\det\InvMatCovMal_F)\to1$ on $\{G\neq0\}$, dominated convergence and closability give, for every $p\geqslant1$,
\begin{align*}
G\InvMatCovMal_F[\ElemBaseH,\ElemBaseH']\Psi_n(\det\InvMatCovMal_F)\to G\InvMatCovMal_F[\ElemBaseH,\ElemBaseH']
\quad\text{in }\mathbb{D}^{q,p}_T(\mathbb{R}),
\end{align*}
where $G\InvMatCovMal_F=0$ on $\{G=0\}$ and (\ref{eq:reg_loc_Mall}) follows. We now can conclude the proof by using the duality formula argument. Using the chain rule (\ref{formule_chain_rule_R}), we have for every $\ElemBaseH \in \BasemathcalH$,
\begin{align*}
\langle \DeriveMal \phi(F) ,\DeriveMal \langle F, \ElemBaseH\rangle_{\mathcal{H}} \rangle_{\mathbb{R}^{\pi^{\delta,\ast}  \times \EnsN}} = & \sum_{\ElemBaseH^{\prime} \in \BasemathcalH} \langle \DeriveDirec  \phi(F),\ElemBaseH^{\prime} \rangle_{\mathcal{H}}  \langle  \DeriveMal \langle F, \ElemBaseH^{\prime} \rangle_{\mathcal{H}},\DeriveMal \langle F, \ElemBaseH \rangle_{\mathcal{H}} \rangle_{\mathbb{R}^{\pi^{\delta,\ast}   \times \EnsN}}  \\
=&   \delta^{-1} \sum_{\ElemBaseH^{\prime} \in \BasemathcalH} \DeriveDirec _{\ElemBaseH^{\prime}}\phi(F) \MatCovMal_{F }[\ElemBaseH,\ElemBaseH^{\prime}]  =   \delta^{-1}  \langle \MatCovMal_{F} \ElemBaseH , \DeriveDirec \phi (F) \rangle_{\mathcal{H}} .
\end{align*}
Now, set, $\EnsT =(0,T]\cap\pi^{\delta,\ast}$ and $B=(B_t^i)_{t \in \EnsT,i \in \EnsN}$, where $B_t^i=G\InvMatCovMal_F\DeriveMal_{(t,i)}F$. By (\ref{eq:reg_loc_Mall}), the chain rule and H\"older's inequality, $B \in\bigcap_{p\geqslant1}\mathbb{D}^{q,p}_T(\mathcal{H}^{\EnsT \times\EnsN})$. Hence continuity of $\IntSkorohod$ gives $\IntSkorohod(B)\in\bigcap_{p\geqslant1}\mathbb{D}^{q-1,p}_T(\mathcal{H})$. Using the extended duality (\ref{formule_dualite_pure}) with $F=\phi(F)$ with $G \InvMatCovMal_{F} =0$ on $\{G=0\}$ and $\OpOrnsteinUlh=-\IntSkorohod \DeriveMal $, we compute
\begin{align*}
\mathbb{E}& [\DeriveDirec _{\ElemBaseH}\phi(F)  G] = \delta \sum_{\ElemBaseH^{\prime} \in \BasemathcalH}   \mathbb{E}[ G \InvMatCovMal _{F}[\ElemBaseH,\ElemBaseH^{\prime}] \langle \DeriveMal\phi(F) ,\DeriveMal \langle F, \ElemBaseH^{\prime}\rangle_{\mathcal{H}} \rangle_{\mathbb{R}^{\pi^{\delta,\ast}   \times \EnsN}}]  \\
=&  \sum_{\ElemBaseH^{\prime} \in \BasemathcalH}   \mathbb{E}[  \phi(F) \IntSkorohod( G \InvMatCovMal _{F}[\ElemBaseH,\ElemBaseH^{\prime}] \DeriveMal \langle F, \ElemBaseH^{\prime}\rangle_{\mathcal{H}}) ]  \\
 =& - \sum_{\ElemBaseH^{\prime} \in \BasemathcalH}   \mathbb{E}[  \phi(F)( G \InvMatCovMal_{F }[\ElemBaseH,\ElemBaseH^{\prime}] \OpOrnsteinUlh  \langle F, \ElemBaseH^{\prime}\rangle_{\mathcal{H}} + \delta \langle \DeriveMal ( G \InvMatCovMal_{F}[\ElemBaseH,\ElemBaseH^{\prime}]), \DeriveMal \langle F,\ElemBaseH^{\prime}\rangle_{\mathcal{H}} \rangle_{\mathbb{R}^{\pi^{\delta,\ast}   \times \EnsN}})]  .
\end{align*}

This proves (\ref{eq:IPP_Malliavin_degre_superieur}) for $q=1$. Furthermore, on $\{G=0\}$, $B=0$ and by localizing as above, $\DeriveMal B=0$. It follows that $\mathbf{1}_{\{G=0\}}\PoidMal(F,G)[\ElemBaseH]=0$ and
\begin{align*}
\mathbb{E}\big[\vert\det\InvMatCovMal_F\vert^p\mathbf{1}_{\{\PoidMal(F,G)[\ElemBaseH]\neq0\}}\big]
\leqslant\mathbb{E}\big[\vert\det\InvMatCovMal_F\vert^p\mathbf{1}_{\{G\neq0\}}\big]<+\infty.
\end{align*}
Moreover $\PoidMal(F,G)[\ElemBaseH]\in\bigcap_{p\geqslant1}\mathbb{D}^{q-1,p}_T$. Hence iterating the integration-by-parts formula as above proves (\ref{eq:IPP_Malliavin_degre_superieur}) for every $q\geqslant1$. Finally, (\ref{eq:borne_norme_sobolev_poids_malliavin_th}) follows directly from \cite{Bally_Caramellino_2014_distance}, Theorem 3.4 (see also \cite{Bally_Clement_2011}). 
 \end{proof}

In the sequel, we are interested in obtaining an estimate of the weights $\PoidMal$ which appear in the
integration-by-parts formulas (\ref{eq:IPP_Malliavin_degre_superieur}) when $G$ is replaced by $G \Theta$ with $\Theta \in [0,1]$ the localizing random weight.  We first provide a bound on the Sobolev norms of $G \Theta$.

\begin{lemme}
\label{lemme:borne_norm_sob_prod_bis}
Let $q \in \mathbb{N}$. Let $G \in \mathbb{D}^{q,2}(\mathcal{H})$ and $\Theta \in \mathbb{D}^{q,2}(\mathbb{R})  $. Then
\begin{align}
\label{eq:borne_norm_sob_prod_bis}
\vert G \Theta  \vert _{\mathcal{H},q} \leqslant  C(q) \sum_{m=0}^{q} \vert G \vert _{\mathcal{H} ,m} \vert \Theta \vert _{\mathbb{R},q-m}.
\end{align}
\end{lemme}

\begin{proof}
We limit ourselves to the case $G \in \mathcal{S}^{\delta}(\mathcal{H})$ and $\Theta \in \mathcal{S}^{\delta}$. The extension to $\mathbb{D}^{q,2}$ follows by a Cauchy sequence approach. We prove the result by recurrence. We assume, without loss of generality, that $G=g(\chi^{\delta}_{t_{i}} ,U^{\delta}_{t_{i}},V^{\delta}_{t_{i}},t_{i} \in \EnsT, i \in \{1,\ldots,n\})$ for a finite set $\EnsT \subset \pi^{\delta,\ast}$. For $q \in \mathbb{N}$, we define $\mathcal{H}_{0}=\mathcal{H}$ and $\mathcal{H}_{q+1}=(\mathcal{H}_{q})^{\EnsT  \times \EnsN}$.  The result is true for $q=0$. Assume that the result holds at order $q \in \mathbb{N}$ and let us show it still holds for $q+1$. We have
\begin{align*}
\vert G \Theta \vert_{\mathcal{H},q+1}^{2} = \vert G \Theta \vert_{\mathcal{H}}^{2}  + \sum_{l=0}^{q} \delta^{l+1} \vert \DeriveMal^{l} (\Theta \DeriveMal G + G \DeriveMal\Theta) \vert_{\mathcal{H}_{l+1}}^{2} ,
\end{align*}
with
\begin{align*}
\vert \DeriveMal^{l}  (\Theta \DeriveMal G ) \vert_{\mathcal{H}_{l+1}} \leqslant &  \delta^{-\frac{l}{2}} \vert \Theta \DeriveMal G  \vert_{\mathcal{H}^{\EnsT  \times \EnsN},l}  \\
\leqslant &  C \delta^{-\frac{l}{2}}  \sum_{m=0}^{l}  \vert \Theta \vert_{\mathbb{R},l-m}   \vert \DeriveMal G \vert_{\mathcal{H}_{1} ,m}  = \delta^{-\frac{l+1}{2}}  C(l)  \sum_{m=0}^{l}  \vert \Theta \vert_{\mathbb{R} ,l-m}   \vert G \vert_{\mathcal{H},1,m+1} ,
\end{align*}
where we have applied (\ref{eq:borne_norm_sob_prod_bis}) with $G$ replaced by $\DeriveMal G$, $q=l$ and $\mathcal{H}=\mathcal{H}_{1}$. To complete the proof, we observe that by a similar approach
\begin{align*}
\vert \DeriveMal^{l} ( G \DeriveMal \Theta) \vert_{\mathcal{H}_{l+1}} =& \vert \sum_{\Vert \alpha \Vert = l} \sum_{\Vert \beta \Vert =1} \vert \DeriveMal_{\alpha} (G \DeriveMal_{\beta} \Theta) \vert_{\mathcal{H}}^{2}  \vert^{\frac{1}{2}} = \vert  \sum_{\Vert \beta \Vert =1} \vert \DeriveMal^{l}(G \DeriveMal_{\beta} \Theta) \vert_{\mathcal{H}_{l}}^{2}  \vert^{\frac{1}{2}} \\
\leqslant & \vert  \sum_{\Vert \beta \Vert =1} \delta^{-l}  \vert G \DeriveMal_{\beta} \Theta \vert_{\mathcal{H},l}^{2}  \vert^{\frac{1}{2}} \leqslant  C \delta^{-\frac{l}{2}}  \sum_{m=0}^{l}  \vert G \vert_{\mathcal{H},m} \vert \vert \DeriveMal \Theta \vert_{\mathbb{R}^{\EnsT \times \EnsN} ,l-m}^{2}  \vert^{\frac{1}{2}}\\
\leqslant & C\delta^{-\frac{l+1}{2}}  \sum_{m=0}^{l}   \vert G \vert_{\mathcal{H}, m}  \vert \Theta \vert_{\mathbb{R},l+1-m} .
\end{align*}
\end{proof}
\subsection{Sobolev Norms}

Before we state our results, we recall that $\partial_{\mbox{\textsc{x}}^{\delta}_{0}}X^{\delta}_{t} $, $t \in \pi^{\delta}$, is the tangent flow and is introduced in (\ref{eq:tangent_flow_def}). In a similar way, for $\alpha \in \mathbb{N}^{d}$, $\partial_{\mbox{\textsc{x}}^{\delta}_{0}}^{\alpha }X^{\delta}_{t} $ denotes the derivatives of $X^{\delta}_{t} $ of order $\vert \alpha \vert$ $w.r.t.$ $X^{\delta}_{0}$ and is given by $\partial_{(\mbox{\textsc{x}}^{\delta}_{0})^{1}}^{\alpha^{1} }\ldots \partial_{(\mbox{\textsc{x}}^{\delta}_{0})^{d}}^{\alpha^{d} }X^{\delta}_{t} $. The following result provides an upper bound for the Sobolev norms needed to estimate the Malliavin weights appearing in Theorem \ref{theoreme:IPP_Malliavin}.

 \begin{theorem}
\label{theo:Norme_Sobolev_borne}
Let $T \in \pi^{\delta,\ast}$, $\EnsT =(0,T] \cap \pi^{\delta}$ and $x \in \mathbb{R}^{d}$. Let $q\in \mathbb{N}$, $q^{\diamond} \in \{0,1\}$,  $p\geqslant 1$ and $\alpha \in \mathbb{N}^{d}$ a multi-index. Assume that $\mathbf{A}_{1}^{\delta}(q+\vert \alpha \vert +2)$,  $\mathbf{A}_{3}^{\delta}(+\infty)$ and $\mathbf{A}_{4}^{\delta}$ hold.  Then
\begin{align}
\label{eq:borne_Sob_flot_tang}
 \mathbb{E}_{x}[\sup_{t \in \EnsT } \vert   \partial_{\mbox{\textsc{x}}^{\delta}_{0}}^{\alpha }X^{\delta}_{t}  \vert_{\mathbb{R}^{d},q^{\diamond},q}^{p}]^{\frac{1}{p}}  \leqslant  &  ( 1+ \mathbf{1}_{\ExpPsiouf  _{q+ \vert \alpha \vert +2} > 0 \cup q^{\diamond} = \vert \alpha \vert = 0}  \vert x  \vert_{\mathbb{R}^{d}}^{C}  )    \CstePsiouf_{q+\vert \alpha \vert+2}^{C} \exp (C(1+T)\MomZ_{C}(Z^{\delta})\CstePsiouf_{3} ^{2})   ,
\end{align}
with $C= C(d,N,\frac{1}{r_{\ast}},\ExpPsiouf_{\max(3,q+\vert \alpha \vert+2)},q,p,\vert \alpha \vert) $. Moreover, if we replace the assumption $\mathbf{A}_{1}^{\delta}(q+\vert \alpha \vert +2)$, by the assumption $\mathbf{A}_{1}^{\delta}(q+4)$, then
\begin{align}
\label{eq:theo_borne_{d}eriv_mall_schema_annex_L}
 \mathbb{E}_{x}[\sup_{t \in \EnsT } \vert  \OpOrnsteinUlh X^{\delta}_{t} \vert_{\mathbb{R}^{d} ,q}^{p}]^{\frac{1}{p}}  \leqslant  &  (1+ \mathbf{1}_{\ExpPsiouf  _{q+4} > 0}   \vert x \vert_{\mathbb{R}^{d}}^{C}   )  \CstePsiouf_{q+4}^{C} \exp (C(1+T)\MomZ_{C}(Z^{\delta}) \CstePsiouf_{3} ^{2})   ,
\end{align}
with $C =C(d,N,\frac{1}{r_{\ast}},\ExpPsiouf  _{q+4},q,p) $.

%

\end{theorem}
\begin{remark}
This result was obtained in \cite{Bally_Rey_2016} (see Theorem 4.2) in the case $\ExpPsiouf  _{r}=0$ for $r$ large enough in the assumption $\mathbf{A}_{1}^{\delta}(r)$ (see (\ref{eq:hyp_1_Norme_adhoc_fonction_schema})).  
\end{remark}

\subsection{Malliavin covariance matrix}

In this section, we provide an upper bound for the localized moments of the inverse of the Malliavin covariance matrix of $(X^{\delta}_{t})_{t \in \pi^{\delta}}$ defined in (\ref{def:matrice_covariance_Malliavin}).  In the following, we will not work under $\mathbb{P}$, but under a localized measure which we define now.  
The technique consists in localizing the random variables $Z^{\delta}$ and the Malliavin covariance matrix $\MatCovMal_{F}$. The localization controls the size of the former and prevents excessive degeneracy of the latter.
Given $\EnsT  \subset \pi^{\delta,\ast}$, and $\eta=(\eta_{1},\eta_{2}) \in (1,+\infty)^{2}$ and recalling that for $v>1$, $ \Psi_{v}$ is defined in the proof of Theorem \ref{theoreme:IPP_Malliavin}, we introduce
\begin{align}
\label{def:loc_theta}
&\Theta_{F,G,\eta,\EnsT }=\Theta_{F,G,\eta_{1} } \Theta_{\eta_{2},\EnsT } \quad \mbox{with}  \\
&\Theta_{F,G,\eta_{1}}=\Psi_{\eta_{1}}(G \det \InvMatCovMal
_{F}) , \quad \mbox{and}\quad \Theta_{\eta_{2},\EnsT ,t}= \prod_{w \in ((0,t] \cap \EnsT )}\Psi_{\eta_{2}}(\vert Z^{\delta}_{w} \vert_{\mathbb{R}^{N}} ), \quad t \in \pi^{\delta},  \nonumber
\end{align}
where $\Theta_{\eta_{2},\EnsT }=\Theta_{\eta_{2},\EnsT ,\infty}:= \prod_{w \in   \EnsT }\Psi_{\eta_{2}}(\vert Z^{\delta}_{w} \vert_{\mathbb{R}^{N}} )$.  Note that we extend $\Psi_v$ to $\mathbb{R}\cup\{+\infty\}$ by setting
$\Psi_v(+\infty)=0$. Moreover, we introduce the tangent flow process $(\dot{X}^{\delta}_{t})_{t \in \pi^{\delta}}$ defined by $ \dot{X}^{\delta}_{0}=I_{d \times d}$ and
\begin{align}
\label{eq:tangent_flow_def}
\dot{X}^{\delta}_{t} :=   \partial_{\mbox{\textsc{x}}^{\delta}_{0}} X^{\delta}_{t} ,
\end{align}
the Jacobian matrix of derivatives of $X^{\delta}$ $w.r.t.$ the initial value $X^{\delta}_{0}$. In the following, for $\EnsT =(0,T] \cap \pi^{\delta}$,  $T \in \pi^{\delta}$,  we will employ the localization random variable
\begin{align}
\label{eq:Theta(delta)}
\Theta^{\ast}_{T}=\Theta_{X^{\delta}_{T},\det(\dot{X}^{\delta}_{T})^{2},(\funceta_{1}(\delta),\funceta_{2}(\delta)),\EnsT },
\end{align}
where $\funceta_{1}(\delta)$ and $\funceta_{2}(\delta)$ are defined in (\ref{def:eta_delta}). In the definition of $\Theta_T^*$, the first
localization factor is only relevant on the support of
$\Theta_{\funceta_{2}(\delta),\EnsT}$. On this set, (\ref{hyp:delta_eta_inversibilite})
ensures that $\dot{X}^{\delta}_{T}$ is invertible. We set
$\Theta^{\ast}_{T}=0$ outside this support.

 \begin{theorem}
  \label{th:borne_Lp_inv_cov_Mal}
Let $T \in \pi^{\delta,\ast}$, $x \in \mathbb{R}^{d}$ and $p \geqslant 0$. Assume that $\mathbf{A}^{\delta}_{1}(2 L + 5)$,  $\mathbf{A}_{2}(x,L)$, $\mathbf{A}_{3}^{\delta}(+\infty)$, $\mathbf{A}_{4}^{\delta}$ and $\mathbf{A}_{5}^{\delta}(x,T)$ hold.  Then
  \begin{align}
 \label{eq:borne_Lp_inv_cov_Mal}
 \mathbb{E}_{x}[ \vert  \det \InvMatCovMal
_{X^{\delta}_{T}}  \vert^{p} \mathbf{1}_{\Theta^{\ast}_{T}>0} ] \leqslant   \frac{1+ \mathbf{1}_{\ExpPsiouf  _{2 L +5}>0}  \vert x \vert_{\mathbb{R}^{d}}^{C }     }{  \HormV_{L}(x)^{13^{L} 3dp+1} T^{13^{L}dp}}   \CstePsiouf_{2 L + 5}^{C}    \exp(C (1+T) \MomZ_{C}(Z^{\delta}) \CstePsiouf_{3}^{4}),
 \end{align}
 with $C=C(d,N,L,\frac{1}{m_{\ast}},\ExpPsiouf  _{2 L + 5},p)$. 
 Moreover, if $p \geqslant 4$, for every $a>0$
  \begin{align}
\label{eq:th_cov_mal_proba_espace_comp}
 \mathbb{P}_{x}(\Theta^{\ast}_{T} <1) \leqslant  & \delta^{-1}  \funceta_{2}(\delta)^{-a} T C(a) \MomZ_{\lceil a \rceil}(Z^{\delta}) \\
&+ \funceta_{1}(\delta)^{-p+2} \frac{1+ \mathbf{1}_{\ExpPsiouf  _{2 L +5}>0}  \vert x  \vert_{\mathbb{R}^{d}}^{C }  }{ \HormV_{L}(x)^{ 13^{L} 3dp}  }    \CstePsiouf_{2 L + 5}^{C }     \exp(C(1+T) \MomZ_{C}(Z^{\delta}) \CstePsiouf_{3} ^{4}) ,\nonumber
 \end{align}
 with $C=C(d,N,L,\frac{1}{m_{\ast}},\ExpPsiouf  _{2 L + 5},p)$.
 \end{theorem}

\begin{remark} 
\label{remark:sg_reg}
 The terms $13^{L}$ in the $r.h.s.$ of both (\ref{eq:borne_Lp_inv_cov_Mal}) and (\ref{eq:th_cov_mal_proba_espace_comp}) can be replaced by $(12+b)^{L}$, $b>0$, but the miscellaneous constants $C(.)$ may explode when $b$ tends to zero or to infinity.
\end{remark}

\subsection{Regularization properties for approximations of the semigroup}



We introduce $(Q_{t}^{\delta,\Theta^{\ast}_{T}})_{t \in \pi^{\delta}}$ such that, for every $x \in \mathbb{R}^{d}$,
\begin{align}
\label{def:semigroupe_regularisant}
\forall T \in \pi^{\delta}, \quad Q_{T}^{\delta,\Theta^{\ast}_{T} }f(x) :=   \mathbb{E}[\Theta^{\ast}_{T}f(X^{\delta}_{T}) \vert X^{\delta}_{0}=x],
\end{align}
with $\Theta^{\ast}_{T}$ defined in (\ref{eq:Theta(delta)}).
Notice that $(Q_{t}^{\delta,\Theta^{\ast}_{T}})_{t\in\pi^{\delta}}$ is not a semigroup. We do not establish smoothing properties directly for $Q^{\delta}$, but rather for the localized operators $Q^{\delta,\Theta^{\ast}_{T}}$. To this end, we apply the integration-by-parts formulas of Theorem \ref{theoreme:IPP_Malliavin} and bound the moments of the associated Malliavin weights using Theorem \ref{theo:Norme_Sobolev_borne} and Theorem \ref{th:borne_Lp_inv_cov_Mal}.

\begin{theorem}
\label{th:regularisation}
Let $T \in \pi^{\delta,\ast}$ and $f \in \mathcal{M}(\mathbb{R}^{d} ; \mathbb{R})$ be such that there exist $\CstePsiouf_{f} \geqslant 0$ and $\ExpPsiouf  _{f} \in \mathbb{N}$ such that for every $x \in \mathbb{R}^{d}$,
\begin{align*}
\vert f(x) \vert \leqslant \CstePsiouf_{f}(1+ \vert x \vert_{\mathbb{R}^{d}}^{\ExpPsiouf  _{f} }).
\end{align*}
 Then we have the following properties:
\begin{enumerate}[label=\textbf{\Alph*.}, ref=\Alph*]
\item \label{th:reg_gen_regprop} Let $x \in \mathbb{R}^{d}$, $q \in \mathbb{N}$ and $\alpha ,\beta \in \mathbb{N}^{d}$ such that $\vert \alpha \vert +\vert \beta \vert \leqslant q$. Assume that $\mathbf{A}^{\delta}_{1}( \max(q+3,2 L + 5))$,  $\mathbf{A}_{2}(x,L)$, $\mathbf{A}_{3}^{\delta}(+\infty)$, $\mathbf{A}_{4}^{\delta}$ and $\mathbf{A}_{5}^{\delta}(x,T)$ hold.  Then,  if $f$ is $\vert \beta \vert$-times differentiable and $\partial^{\beta}f$ has polynomial growth
\begin{align}
 \label{eq:borne_semigroupe_regularisation} 
 \vert \partial^{\alpha }Q_{T}^{\delta,\Theta^{\ast}_{T}}\partial^{\beta}f(x) \vert \leqslant & \CstePsiouf_{f} \frac{1 + \mathbf{1}_{\ExpPsiouf  _{\max(q+3,2 L + 5)} + \ExpPsiouf  _{f} > 0} \vert x \vert_{\mathbb{R}^{d}}^{C }}{ (\HormV_{L}(x) T)^{13^{L} 3 dq(q+3)+1} }\\
&\times  \CstePsiouf_{\max(q+3,2 L + 5)}^{C}     \exp(C (1+T) \MomZ_{C }(Z^{\delta}) \CstePsiouf_{3} ^{4}) \nonumber ,
\end{align}

with $C=C(d,N,L,q ,\ExpPsiouf  _{\max(q+3,2 L + 5)},\ExpPsiouf  _{f},\frac{1}{m_{\ast}},\frac{1}{r_{\ast}})\geqslant 0$.

\item \label{th:reg_gen_distprop}  Let $x \in \mathbb{R}^{d}$ and $h>0$. Assume that $\mathbf{A}^{\delta}_{1}(2 L + 5)$, $\mathbf{A}_{2}(x,L)$, $\mathbf{A}_{3}^{\delta}(+\infty)$, $\mathbf{A}_{4}^{\delta}$ and $\mathbf{A}_{5}^{\delta}(x,T)$ hold. Then,
\begin{align}
\label{eq:reg_gen_distprop}
\vert Q^{\delta}_{T}f(x)-Q_{T}^{\delta,\Theta^{\ast}_{T}}f(x)\vert
\leqslant &  \delta^{h}  \CstePsiouf_{f}   \frac{1+\mathbf{1}_{\ExpPsiouf  _{2 L +5} + \ExpPsiouf  _{f}>0} \vert x \vert_{\mathbb{R}^{d}}^{C}  }{\HormV_{L}(x)^{13^{L} 3 d \max(2,1+\frac{91 h}{44d})} } \CstePsiouf_{2 L + 5}^{C } \exp(C (1+T) \MomZ_{C}(Z^{\delta}) \CstePsiouf_{3} ^{4})   , 
\end{align}
with $C=C(d,N,L,\ExpPsiouf  _{2 L + 5},\ExpPsiouf  _{f},\frac{1}{m_{\ast}},h) \geqslant 0$.
\end{enumerate}
\end{theorem}

\begin{proof} 
For the sake of clarity, in the proof of this result we simply denote by $C$ the constant appearing in (\ref{eq:borne_semigroupe_regularisation}) (respectively (\ref{eq:reg_gen_distprop})) without displaying the dependence on the parameters. We adopt this convention throughout the proofs.\\

We prove the result for $f \in  \mathcal{C}^{\infty}_{pol}(\mathbb{R}^{d};\mathbb{R})$. The extension to our setting follows from arguments similar to those in the proof of Proposition 3.2 in \cite{Fournie_Lasry_Lions_Lebuchoux_Touzi_1999} by introducing a smooth approximating sequence for $f$.\\

 Let $q \in \mathbb{N}^{\ast}$. The case $q=0$ follows directly from Lemma \ref{lemme:borne:moments_X}. The first step of our proof consists in estimating the Sobolev norm of  $\partial_{\mbox{\textsc{x}}^{\delta}_{0}}^{\gamma} \Theta^{\ast}_{T} $ for a multi-index $\gamma \in \mathbb{N}^{d}$. We begin with a useful lemma establishing bounds on the derivatives of the inverse covariance matrix. 

\begin{lemme}
\label{lem:borne_inv_mat_cov_deriv}
Let $T \in \pi^{\delta,\ast}$, $m \in \mathbb{N}$ and $\gamma \in \mathbb{N}^{d}$ with $m+\vert \gamma \vert >0$. Then, on the set $\{\det \MatCovMal_{X^{\delta}_{T}} \neq 0\}$,
\begin{align*}
\vert \partial_{\mbox{\textsc{x}}^{\delta}_{0}}^{\gamma} \det( \InvMatCovMal_{X^{\delta}_{T}}) \vert_{\mathbb{R},m} \leqslant C(d,m,\vert \gamma \vert) \det( \InvMatCovMal_{X^{\delta}_{T}})^{m+\vert \gamma \vert+1} (\sum_{\vert \tilde{\gamma} \vert \leqslant \vert \gamma \vert} \vert \partial_{\mbox{\textsc{x}}^{\delta}_{0}}^{\tilde{\gamma}} X^{\delta}_{T} \vert_{\mathbb{R}^{d},1,m+1} )^{2d(m+\vert \gamma \vert)}.
\end{align*}
\end{lemme}

\begin{proof}
Let $\EnsT =(0,T] \cap \pi^{\delta}$. Since $\partial_{\mbox{\textsc{x}}^{\delta}_{0}}$ commutes with $\DeriveMal$, the same approach as in the proof of Lemma \ref{lemme:borne_norm_sob_prod_bis}, combined with the Leibniz rule for $\partial_{\mbox{\textsc{x}}^{\delta}_{0}}$, yields, for $j \in \mathbb{N}$, $\zeta \in \mathbb{N}^{d}$ and $F,G$ some $\mathcal{C}^{\infty}$ functions of $(X^{\delta}_{0},(U^{\delta}_{t})_{t \in \EnsT})$ on an open set (the following inequality then holding on this set), which is the case of the entries of $\DeriveMal X^{\delta}_{T}$, of $\partial^{\rho}_{\mbox{\textsc{x}}^{\delta}_{0}}X^{\delta}_{T}$, $\rho \in \mathbb{N}^{d}$, $\vert \rho \vert \geqslant 1$, of $\MatCovMal_{X^{\delta}_{T}}$, of $\det(\InvMatCovMal_{X^{\delta}_{T}})$ on $\{\det \MatCovMal_{X^{\delta}_{T}} \neq 0\}$, of $\Theta_{\eta_{2},\EnsT}$ and of smooth functions of these quantities (such as $\Theta_{1}$ and $\mathscr{P}_{\gamma,\gamma'}(X^{\delta}_{T})$ below),
\begin{align}
\label{eq:leibniz_tens}
\vert \partial_{\mbox{\textsc{x}}^{\delta}_{0}}^{\zeta}(FG) \vert_{\mathbb{R},j} \leqslant C \sum_{i=0}^{j}\sum_{\zeta' \leqslant \zeta} \vert \partial_{\mbox{\textsc{x}}^{\delta}_{0}}^{\zeta'}F \vert_{\mathbb{R},i} \vert \partial_{\mbox{\textsc{x}}^{\delta}_{0}}^{\zeta-\zeta'}G \vert_{\mathbb{R},j-i},
\end{align}
where for $\zeta,\zeta' \in \mathbb{N}^{d}$, we write $\zeta' \leqslant \zeta$ if $\zeta^{\prime,i} \leqslant \zeta^{i}$ for every $i \in \{1,\ldots,d\}$.

Let $l,l' \in \{1,\ldots,d\}$, $r \leqslant m$ and $\zeta \leqslant \gamma$. Since $\MatCovMal_{X^{\delta}_{T}}[l,l']=\delta\sum_{\upsilon \in \EnsT \times \EnsN}\DeriveMal_{\upsilon}X^{\delta,l}_{T}\DeriveMal_{\upsilon}X^{\delta,l'}_{T}$, the triangle inequality, (\ref{eq:leibniz_tens}) applied with $F=\DeriveMal_{\upsilon}X^{\delta,l}_{T}$ and $G=\DeriveMal_{\upsilon}X^{\delta,l'}_{T}$, and the Cauchy-Schwarz inequality in $\upsilon$ yield
\begin{align*}
\vert \partial_{\mbox{\textsc{x}}^{\delta}_{0}}^{\zeta}\MatCovMal_{X^{\delta}_{T}}[l,l'] \vert_{\mathbb{R},r} \leqslant & C \sum_{i=0}^{r}\sum_{\zeta' \leqslant \zeta} (\delta \sum_{\upsilon \in \EnsT \times \EnsN}\vert \partial_{\mbox{\textsc{x}}^{\delta}_{0}}^{\zeta'}\DeriveMal_{\upsilon}X^{\delta,l}_{T} \vert_{\mathbb{R},i}^{2})^{\frac{1}{2}}(\delta \sum_{\upsilon \in \EnsT \times \EnsN}\vert \partial_{\mbox{\textsc{x}}^{\delta}_{0}}^{\zeta-\zeta'}\DeriveMal_{\upsilon}X^{\delta,l'}_{T} \vert_{\mathbb{R},r-i}^{2})^{\frac{1}{2}} \\
\leqslant & C \sum_{\zeta' \leqslant \zeta} \vert \partial_{\mbox{\textsc{x}}^{\delta}_{0}}^{\zeta'} X^{\delta}_{T} \vert_{\mathbb{R}^{d},1,r+1} \vert \partial_{\mbox{\textsc{x}}^{\delta}_{0}}^{\zeta-\zeta'} X^{\delta}_{T} \vert_{\mathbb{R}^{d},1,r+1},
\end{align*}
where we used that $\delta \sum_{\upsilon \in \EnsT \times \EnsN}\vert \DeriveMal_{\upsilon}F \vert_{\mathbb{R},i}^{2}=\vert F\vert_{\mathbb{R},1,i+1}^{2}$. Since $\det(\MatCovMal_{X^{\delta}_{T}})$ is a polynomial of degree $d$ in these entries, (\ref{eq:leibniz_tens}) yields, for $r \leqslant m$ and $\zeta \leqslant \gamma$,
\begin{align*}
\vert \partial_{\mbox{\textsc{x}}^{\delta}_{0}}^{\zeta}\det(\MatCovMal_{X^{\delta}_{T}}) \vert_{\mathbb{R},r} \leqslant C (\sum_{\vert \tilde{\gamma} \vert \leqslant \vert \gamma \vert} \vert \partial_{\mbox{\textsc{x}}^{\delta}_{0}}^{\tilde{\gamma}} X^{\delta}_{T} \vert_{\mathbb{R}^{d},1,m+1})^{2d}.
\end{align*}
Now let $j \leqslant m$ and $\alpha \in (\EnsT \times \EnsN)^{j}$ with $j+\vert\gamma\vert>0$. By the Fa\`a di Bruno formula applied to $x \in \mathbb{R}_{+}^{\ast} \mapsto x^{-1}$, $\DeriveMal_{\alpha}\partial_{\mbox{\textsc{x}}^{\delta}_{0}}^{\gamma}\det( \InvMatCovMal_{X^{\delta}_{T}})$ is a linear combination, with coefficients bounded by $C$, of terms
\begin{align*}
\det( \InvMatCovMal_{X^{\delta}_{T}})^{k+1}\prod_{i=1}^{k}\DeriveMal_{\alpha_{i}}\partial_{\mbox{\textsc{x}}^{\delta}_{0}}^{\gamma_{i}}\det(\MatCovMal_{X^{\delta}_{T}}),\quad k \in \{1,\ldots,j+\vert\gamma\vert\},
\end{align*}
where $\alpha_{1},\ldots,\alpha_{k}$ are complementary subsequences of $\alpha$ and $\gamma_{1}+\ldots+\gamma_{k}=\gamma$. For a fixed repartition of the components of $\alpha$, the map $\alpha \mapsto (\alpha_{1},\ldots,\alpha_{k})$ is one-to-one. Multiplying the squares of these terms by $\delta^{j}$, summing over $\alpha$ and $j$, and recalling that the term $j=0$ is $\det( \InvMatCovMal_{X^{\delta}_{T}})$ when $\gamma=0$, we obtain
\begin{align*}
\vert \partial_{\mbox{\textsc{x}}^{\delta}_{0}}^{\gamma} \det( \InvMatCovMal_{X^{\delta}_{T}}) \vert_{\mathbb{R},m} \leqslant C\sum_{k=0}^{m+\vert \gamma \vert} \det( \InvMatCovMal_{X^{\delta}_{T}})^{k+1}(\sum_{\vert \tilde{\gamma} \vert \leqslant \vert \gamma \vert} \vert \partial_{\mbox{\textsc{x}}^{\delta}_{0}}^{\tilde{\gamma}} X^{\delta}_{T} \vert_{\mathbb{R}^{d},1,m+1})^{2dk}.
\end{align*}
Finally, for $k<m+\vert \gamma \vert$, we write $\det( \InvMatCovMal_{X^{\delta}_{T}})^{k+1}=\det( \InvMatCovMal_{X^{\delta}_{T}})^{m+\vert \gamma \vert+1}\det( \MatCovMal_{X^{\delta}_{T}})^{m+\vert \gamma \vert-k}$ and we use the arithmetic-geometric mean inequality:
\begin{align*}
\det(\MatCovMal_{X^{\delta}_{T}}) \leqslant (\frac{1}{d}\mbox{Tr}(\MatCovMal_{X^{\delta}_{T}}))^{d}\leqslant (\delta\sum_{\upsilon\in\EnsT\times\EnsN}\vert \DeriveMal_{\upsilon}X^{\delta}_{T}\vert_{\mathbb{R}^{d}}^{2})^{d} =\vert X^{\delta}_{T}\vert_{\mathbb{R}^{d},1,1}^{2d} \leqslant \vert X^{\delta}_{T}\vert_{\mathbb{R}^{d},1,m+1}^{2d}.
\end{align*}
Since $\vert X^{\delta}_{T}\vert_{\mathbb{R}^{d},1,m+1}$ is one of the terms of the sum appearing in the statement, this completes the proof.
\end{proof}

Now, let $\mathcal{T}=\pi^{\delta,\ast} \cap (0,T]$. On the set $\{\Theta_{\eta_{2}(\delta),\mathcal{T}}>0\}$, we write $\Theta_{1} =\Theta_{X^{\delta}_{T},\det(\dot{X} ^{\delta}_{T})^{2},\eta_{1}}=\Psi_{\eta_{1}}( \det( \InvMatCovMal_{X^{\delta}_{T}})\det(\dot{X} ^{\delta}_{T})^{2})$.  
We estimate the Sobolev norms of $ \partial_{\mbox{\textsc{x}}^{\delta}_{0}}^{\gamma} \Theta_{1} $ for $\gamma \in \mathbb{N}^{d}$. Let $m \in \mathbb{N}$. Since $\det(\dot{X} ^{\delta}_{T})^{2}$ is a polynomial of degree $2d$ in the entries $\partial_{\mbox{\textsc{x}}^{\delta,k}_{0}}X^{\delta,i}_{T}$, combining (\ref{eq:leibniz_tens}) with Lemma \ref{lem:borne_inv_mat_cov_deriv} (applied with $(m,\gamma)$ replaced by $(r',\zeta')$, $r' \leqslant r$, $\zeta' \leqslant \zeta$, the case $r'+\vert \zeta' \vert=0$ being trivial) yields, for $r \leqslant m$ and $\zeta \leqslant \gamma$,
\begin{align}
\label{eq:borne_prod_det}
\vert \partial_{\mbox{\textsc{x}}^{\delta}_{0}}^{\zeta}(\det( \InvMatCovMal_{X^{\delta}_{T}})\det(\dot{X} ^{\delta}_{T})^{2}) \vert_{\mathbb{R},r} \leqslant & C(1 \vee \det( \InvMatCovMal_{X^{\delta}_{T}})^{r+\vert\zeta\vert+1}) \\
& \times (1+\sum_{\vert \tilde{\gamma} \vert \leqslant \vert \gamma \vert+1}\vert  \partial_{\mbox{\textsc{x}}^{\delta}_{0}}^{\tilde{\gamma}} X^{\delta}_{T}   \vert_{\mathbb{R}^{d},\mathbf{1}_{\vert \tilde{\gamma} \vert=0},m+\mathbf{1}_{\vert \tilde{\gamma}\vert < \vert \gamma \vert+1}})^{2d(r+\vert\zeta\vert+1)}. \nonumber
\end{align}
Let $l \leqslant m$ and $\alpha \in (\EnsT \times \EnsN)^{l}$ with $l+\vert\gamma\vert>0$. As in the proof of Lemma \ref{lem:borne_inv_mat_cov_deriv}, the Fa\`a di Bruno formula shows that $\DeriveMal_{\alpha}\partial_{\mbox{\textsc{x}}^{\delta}_{0}}^{\gamma}\Theta_{1}$ is a linear combination of terms
\begin{align*}
\Psi^{(k)}_{\eta_{1}}(\det( \InvMatCovMal_{X^{\delta}_{T}})\det(\dot{X} ^{\delta}_{T})^{2})\prod_{i=1}^{k}\DeriveMal_{\alpha_{i}}\partial_{\mbox{\textsc{x}}^{\delta}_{0}}^{\gamma_{i}}(\det( \InvMatCovMal_{X^{\delta}_{T}})\det(\dot{X} ^{\delta}_{T})^{2}),\quad k \in \{1,\ldots,l+\vert\gamma\vert\},
\end{align*}
which vanish on $\{\Theta_{1}=0\}$. Since $\sum_{i=1}^{k}(\Vert\alpha_{i}\Vert+\vert\gamma_{i}\vert+1)=l+\vert\gamma\vert+k\leqslant 2(m+\vert\gamma\vert)$, we deduce from (\ref{eq:borne_prod_det}), as in the proof of Lemma \ref{lem:borne_inv_mat_cov_deriv} and using $\vert  \Theta_{1} \vert_{\mathbb{R}} \leqslant \mathbf{1}_{\Theta_{1}>0}$, that
\begin{align*}
\vert   \partial_{\mbox{\textsc{x}}^{\delta}_{0}}^{\gamma} \Theta_{1} \vert_{\mathbb{R},m}^{2} \leqslant & C\mathbf{1}_{\Theta_{1}>0}\Vert \Psi_{\eta_{1}} \Vert_{\infty, m+\vert \gamma \vert}^{2}(1\vee \det( \InvMatCovMal_{X^{\delta}_{T}})^{4(m+\vert \gamma \vert)}) \\
& \times (1+\sum_{\vert \tilde{\gamma} \vert \leqslant \vert \gamma \vert+1}\vert  \partial_{\mbox{\textsc{x}}^{\delta}_{0}}^{\tilde{\gamma}} X^{\delta}_{T}   \vert_{\mathbb{R}^{d},\mathbf{1}_{\vert \tilde{\gamma} \vert=0},m+\mathbf{1}_{\vert \tilde{\gamma}\vert < \vert \gamma \vert+1}})^{8d(m+\vert \gamma \vert)},
\end{align*}
where, for a smooth function $g$ and $k \in \mathbb{N}$, $\Vert g \Vert_{\infty,k}:=\sum_{\vert \upsilon \vert \leqslant k} \Vert \partial^{\upsilon} g \Vert_{\infty}$. Note that, since $\Psi_{v}=\funclocgen(2\vert\cdot\vert-2v+1)$, $\Vert \Psi_{v} \Vert_{\infty,k}$ (and also $\Vert \Psi_{v}(\vert\cdot\vert_{\mathbb{R}^{N}}) \Vert_{\infty,k}$) is bounded uniformly in $v>1$.

Now, we focus on the estimate of $\vert \Theta_{\eta_{2},\EnsT } \vert_{\mathbb{R},m}^{2}= \sum_{j=0}^{m} \delta^{j}\sum_{\alpha \in (\mathcal{T}  \times \EnsN )^{j}} \vert \DeriveMal_{\alpha}   \prod_{t \in \EnsT }\Psi_{\eta_{2}}(\vert Z^{\delta}_{t} \vert_{\mathbb{R}^{N}} )  \vert_{\mathbb{R}}^{2} $.
Let us denote $g=\Psi_{\eta_{2}}(\vert \cdot \vert_{\mathbb{R}^{N}})$ and, for $\alpha=((t_{1},i_{1}),\ldots,(t_{j},i_{j}))$, $j \leqslant m$, let $s_{1},\ldots,s_{j'}$ be the distinct elements of $\{t_{1},\ldots,t_{j}\}$, so that $j' =\vert \{t_{1},\ldots,t_{j}\} \vert$. We have
\begin{align*}
\DeriveMal_{\alpha}   \prod_{t \in \EnsT }\Psi_{\eta_{2}}(\vert Z^{\delta}_{t} \vert_{\mathbb{R}^{N}} ) 
=&  \prod_{l=1}^{j'}  \DeriveMal_{\alpha \vert s_{l}}g(Z^{\delta}_{s_{l}})  \prod_{t \notin \{t_{1},\ldots, t_{j}\}} g(Z^{\delta}_{t}) ,
\end{align*}
where $\DeriveMal_{\alpha \vert s_{l}}$ is a short notation for $(\prod_{k=1}^{j}(\DeriveMal_{s_{l},i_{k}})^{\mathbf{1}_{t_{k}=s_{l}}}) $. Using (\ref{eq:derivee_Zkj} ) and the chain rule (\ref{formule_chain_rule}),  we have $
 \DeriveMal_{\alpha \vert s_{l}}g(Z^{\delta}_{s_{l}}) =\delta^{-\frac{r_{l}}{2}}\chi^{\delta}_{s_{l}} \partial_{z \vert s_{l}} g(Z^{\delta}_{s_{l}})$ where $\partial_{z \vert s_{l}} $ is a short notation for $\prod_{k=1}^{j}(\partial_{z}^{i_{k} } )^{\mathbf{1}_{t_{k}=s_{l}}} $ and $r_{l}=\vert \{k\in \{1,\ldots,j\},s_{l}=t_{k}\} \vert$. It thus follows that
\begin{align*}
\delta^{j} \vert \DeriveMal_{\alpha}   \prod_{t \in \EnsT }\Psi_{\eta_{2}}(\vert Z^{\delta}_{t} \vert_{\mathbb{R}^{N}} )   \vert^2 =  \prod_{l=1}^{j'} \chi^{\delta}_{s_{l}} \partial_{z \vert s_{l}} g(Z^{\delta}_{s_{l}})^2  \prod_{t \notin \{t_{1},\ldots, t_{j}\}} g(Z^{\delta}_{t})^{2} .
\end{align*}
Moreover, studying (\ref{eq:def_func_localization_generic}) we can prove that $\vert \partial_{z \vert s_{l}} g(z) \vert^2 \leqslant C(1-g(z)^{2}),z \in \mathbb{R}^{N}$ so that 
\begin{align*}
\vert \Theta_{\eta_{2},\EnsT } \vert_{\mathbb{R},m}^{2} 
\leqslant  &  C  \mathbf{1}_{\Theta_{\eta_{2},\EnsT } >0} \sum_{j'=0}^{m} \sum_{\{s_{1} ,\dots, s_{j'}\} \subset \EnsT}  \prod_{t \in \{s_{1},\ldots, s_{j'}\}}(1- g(Z^{\delta}_{t})^2 ) \prod_{t \in \EnsT \setminus \{s_{1},\ldots, s_{j'}\}  } g(Z^{\delta}_{t})^{2} \\
\leqslant  &  C \mathbf{1}_{\Theta_{\eta_{2},\EnsT } >0} \sum_{j'=0}^{m}  \prod_{t \in \EnsT}((1- g(Z^{\delta}_{t})^2 ) + g(Z^{\delta}_{t})^{2} ) \leqslant C \mathbf{1}_{\Theta_{\eta_{2},\EnsT } >0}.
\end{align*}
Since $\Theta_{\eta_{2},\EnsT}$ does not depend on $X^{\delta}_{0}$, applying (\ref{eq:leibniz_tens}) with $F=\Theta_{1}$, $G=\Theta_{\eta_{2},\EnsT}$ and $\zeta=\gamma$, combining the two preceding estimates and taking $(\eta_{1},\eta_{2})=(\funceta_{1}(\delta),\funceta_{2}(\delta))$, we obtain
  \begin{align}
  \label{eq:borne_derive_theta}
\vert   \partial_{\mbox{\textsc{x}}^{\delta}_{0}}^{\gamma} \Theta^{\ast}_{T} \vert_{\mathbb{R},m} 
\leqslant &C  \mathbf{1}_{\Theta^{\ast}_{T}>0} \Vert \Psi_{\funceta_{1}(\delta)} \Vert_{\infty, m+ \vert \gamma \vert}  \Vert \Psi_{\funceta_{2}(\delta)}(\vert \cdot \vert_{\mathbb{R}^{N}}) \Vert_{\infty, m} \nonumber \\
& \times ( 1 \vee \vert \det \InvMatCovMal
_{X^{\delta}_{T}} \vert^{2(m+\vert \gamma \vert) } )(1+\sum_{\vert  \tilde{\gamma} \vert \leqslant \vert \gamma \vert+1}\vert  \partial_{\mbox{\textsc{x}}^{\delta}_{0}}^{\tilde{\gamma}} X^{\delta}_{T}   \vert_{\mathbb{R}^{d},\mathbf{1}_{\vert \tilde{\gamma} \vert = 0},m+\mathbf{1}_{\vert \tilde{\gamma}  \vert < \vert \gamma \vert+1}}^{4 d( m+\vert \gamma \vert)}).
 \end{align}

We now focus on establishing (\ref{eq:borne_semigroupe_regularisation}). We remark that,
\begin{align}
\partial^{\alpha }Q_{T}^{\delta,\Theta^{\ast}_{T} }\partial^{\beta}f(x)=&
\sum_{\vert \beta \vert \leqslant  \vert \gamma \vert  \leqslant q} \sum_{0 \leqslant   \vert \gamma^{\prime} \vert \leqslant q-\vert \gamma \vert } \mathbb{E}_{x}[  \partial^{\gamma }f (X^{\delta}_{T})   \mathscr{P}_{\gamma,\gamma' }(X^{\delta}_{T}) \partial_{\mbox{\textsc{x}}^{\delta}_{0}}^{\gamma^{\prime} }\Theta^{\ast}_{T} ],  \label{eq:derivee_semigroup_regularisation}
\end{align}
where $\mathscr{P}_{\gamma,\gamma^{\prime}}(X^{\delta}_{T})$ is a universal polynomial of $ \partial_{\mbox{\textsc{x}}^{\delta}_{0}}^{\rho }X^{\delta}_{T} ,1 \leqslant \vert \rho \vert \leqslant q- \vert \gamma  \vert +1$, $\gamma,\gamma^{\prime}  \in \mathbb{N}^{d}$.  We do not give details about the validity of the formula above. It is actually a consequence of the dominated convergence theorem and of the Gronwall lemma to control the increments of the function $X^{\delta}_{0} \mapsto X^{\delta}_{T}$ and its derivatives. 
Notice that $\{\mathscr{P}_{\gamma,\gamma'}(X^{\delta}_{T})\partial_{\mbox{\textsc{x}}^{\delta}_{0}}^{\gamma'}\Theta^{\ast}_{T}\neq 0\}\subset\{\Theta^{\ast}_{T}>0\}$, so that, owing to (\ref{eq:borne_derive_theta}), Theorem \ref{theo:Norme_Sobolev_borne} and Theorem \ref{th:borne_Lp_inv_cov_Mal}, the assumptions of Theorem \ref{theoreme:IPP_Malliavin} are satisfied with $F=X^{\delta}_{T}$ and $G=\mathscr{P}_{\gamma,\gamma'}(X^{\delta}_{T})\partial_{\mbox{\textsc{x}}^{\delta}_{0}}^{\gamma'}\Theta^{\ast}_{T}$.
Therefore, using the integration-by-parts formula (\ref{eq:IPP_Malliavin_degre_superieur}) and the estimate (\ref{eq:borne_norme_sobolev_poids_malliavin_th}) obtained in Theorem \ref{theoreme:IPP_Malliavin}, we derive
\begin{align*}
\vert   \mathbb{E}_{x}[  \partial^{\gamma }f (X^{\delta}_{T})   \mathscr{P}_{\gamma,\gamma' }(X^{\delta}_{T}) \partial_{\mbox{\textsc{x}}^{\delta}_{0}}^{\gamma^{\prime} } \Theta^{\ast}_{T}  ] \vert =&\vert  \mathbb{E}_{x} [f(X^{\delta}_{T}) \PoidMal
 (X^{\delta}_{T}  ,\mathscr{P}_{\gamma,\gamma' }(X^{\delta}_{T}) \partial_{\mbox{\textsc{x}}^{\delta}_{0}}^{\gamma^{\prime} } \Theta^{\ast}_{T} )[\gamma]  ]  \vert  \\
\leqslant & \CstePsiouf_{f}  \mathbb{E}_{x}[(1+\vert X^{\delta}_{T} \vert_{\mathbb{R}^{d}}^{\ExpPsiouf  _{f}}) \vert  \PoidMal
 (X^{\delta}_{T},  \mathscr{P}_{\gamma,\gamma' }(X^{\delta}_{T}) \partial_{\mbox{\textsc{x}}^{\delta}_{0}}^{\gamma^{\prime} } \Theta^{\ast}_{T}  )[\gamma] \vert ]  \nonumber \\
\leqslant &C \CstePsiouf_{f} A_{1} A_{2} A_{3}  A_{4} , \nonumber 
\end{align*}%
where, using (\ref{eq:borne_derive_theta}) and (\ref{eq:leibniz_tens}) combined with the Cauchy-Schwarz inequality, we have
\begin{align*}
A_{1} =& 1\vee \mathbb{E}_{x}[\vert  \det \InvMatCovMal_{X^{\delta}_{T} } \vert^{4q(q+3)} \mathbf{1}_{\Theta^{\ast}_{T} >0} ]^{\frac{1}{4}} ,  &A_{2} =&1+ \mathbb{E}_{x}[\vert X^{\delta}_{T} \vert _{\mathbb{R}^{d},1,q+1}^{8dq(q+2)} ]^{\frac{1}{4}} + \mathbb{E}_{x}[\vert
\OpOrnsteinUlh X^{\delta}_{T} \vert _{\mathbb{R}^{d},q-1}^{8 q} ]^{\frac{1}{4}} , \\
A_{3}=& 1+ \mathbb{E}_{x}[\sum_{\vert \tilde{\gamma} \vert \leqslant q- \vert \gamma \vert+1}\vert  \partial_{\mbox{\textsc{x}}^{\delta}_{0}}^{\tilde{\gamma}} X^{\delta}_{T}   \vert_{\mathbb{R}^{d},\mathbf{1}_{\vert \tilde{\gamma} \vert=0 },\vert \gamma \vert + \mathbf{1}_{\vert \tilde{\gamma} \vert \leqslant q- \vert \gamma \vert}}^{16 dq } ]^{\frac{1}{4}},   &A_{4} =& \mathbb{E}_{x}[(1+\vert X^{\delta}_{T} \vert_{\mathbb{R}^{d}}^{\ExpPsiouf  _{f}})^{4}\vert \mathscr{P}_{\gamma,\gamma' }(X^{\delta}_{T})  \vert _{\mathbb{R} , \vert \gamma \vert}^{4} ]^{\frac{1}{4}} .
\end{align*}
In order to bound the first term, we use the estimates on the moments of inverse Malliavin covariance matrix derived in Theorem \ref{th:borne_Lp_inv_cov_Mal} and obtain
\begin{align*}
A_{1} \leqslant & \frac{1+ \mathbf{1}_{\ExpPsiouf  _{2 L +5}>0}  \vert x  \vert_{\mathbb{R}^{d}}^{C}}{( \HormV_{L}(x) T)^{13^{L} 3 dq(q+3)+1} }   \CstePsiouf_{2 L + 5}^{C}   \exp(C(1+T) \MomZ_{C }(Z^{\delta}) \CstePsiouf_{3} ^{4})  . \nonumber
\end{align*}
Moreover, using the results from Theorem \ref{theo:Norme_Sobolev_borne}, we obtain 
\begin{align*}
A_{2}A_{3}A_{4} \leqslant & (1 +  \mathbf{1}_{\ExpPsiouf  _{q+3} + \ExpPsiouf  _{f} > 0} \vert x\vert_{\mathbb{R}^{d}}) ^{C}\CstePsiouf_{q+3} ^{C}   \exp (C(T+1)\MomZ_{C}(Z^{\delta})\CstePsiouf_{3} ^{2}) .
\end{align*}
We gather all the terms together and the proof of (\ref{eq:borne_semigroupe_regularisation}) is completed.

Now, let us prove \ref{th:reg_gen_distprop}. We have
\begin{align*}
\vert Q^{\delta}_{T}f(x)-Q_{T}^{\delta,\Theta^{\ast}_{T} }f(x)\vert  = & \vert \mathbb{E}_{x } [f(X^{\delta}_{T})(1-\Theta^{\ast}_{T} ) ] \vert
\leqslant   \CstePsiouf_{f} \mathbb{E}_{x}[(1+\vert X^{\delta}_{T} \vert_{\mathbb{R}^{d}}^{\ExpPsiouf  _{f}} )^{2} ]^{\frac{1}{2}}\mathbb{E}_{x }[1-\Theta^{\ast}_{T} 
]^{\frac{1}{2}}.
\end{align*}
We then obtain an upper bound for $\mathbb{E}_{x} [1-\Theta^{\ast}_{T} 
]$ by using (\ref{eq:th_cov_mal_proba_espace_comp}).  The upper bound of $\mathbb{E}_{x}[\vert X^{\delta}_{T} \vert^{2 \ExpPsiouf  _{f}} ]$ is obtained using Lemma \ref{lemme:borne:moments_X}. It follows that, for every $a>0$ and every $p \geqslant 4$,

\begin{align*}
\vert Q^{\delta}_{T}f(x)- & Q_{T}^{\delta,\Theta^{\ast}_{T} }f(x)\vert
\leqslant   (\delta^{-1}  \funceta_{2}(\delta)^{-a}\MomZ_{\lceil a \rceil}(Z^{\delta}) + \funceta_{1}(\delta)^{-p+2} (1+ \HormV_{L}(x)^{-13^{L} 3 d p}  ))^{\frac{1}{2}}\nonumber \\
& \times  \CstePsiouf_{f}  (1 + \mathbf{1}_{\ExpPsiouf  _{2 L +5}+\ExpPsiouf  _{f}>0}   \vert x \vert_{\mathbb{R}^{d}}^{C}   )\CstePsiouf_{2 L + 5}^{C }  \exp(C (1+T) \MomZ_{C}(Z^{\delta}) \CstePsiouf_{3} ^{4})  ,\nonumber 
\end{align*}
%
%
%

with $C$ depending also on $a$. We 
fix $p=p(h)=\max(4,2+\frac{91 h}{22d})$ so that $\funceta_{1}(\delta)^{-p(h)+2}\leqslant \delta^{2h} C (1+T^{C})$. Similarly we choose $a=a(h)=2(2h+1)\max(\ExpPsiouf_{3}  +1,91)$ so that $\funceta_{2}(\delta)^{-a(h)} \delta^{-1} \leqslant  \delta^{2h} C  \CstePsiouf_{3}^{C} (1+T^{C})$ (here we use directly the definition (\ref{def:eta_delta})) and
\begin{align*}
\vert Q^{\delta}_{T}f(x)-  Q_{T}^{\delta,\Theta^{\ast}_{T}  }f(x)\vert
\leqslant  &  \delta^{h}     \CstePsiouf_{f}  (1+ \HormV_{L}(x)^{-\frac{3}{2}13^{L}dp(h)}  )  (1 + \mathbf{1}_{\ExpPsiouf  _{2 L +5}+\ExpPsiouf  _{f}>0} \vert x \vert_{\mathbb{R}^{d}}^{C}   )  \nonumber \\
& \times \CstePsiouf_{2 L + 5}^{C }   C\exp(C (1+T) \MomZ_{C}(Z^{\delta}) \CstePsiouf_{3} ^{4})  ,\nonumber 
\end{align*}
and the proof of (\ref{eq:reg_gen_distprop}) is completed.

\end{proof}

\subsection{Proof of the main results}
\label{sec:Proof of the main results}

We now deduce Theorem \ref{th:regul_main_result_intro} from Theorem \ref{th:regularisation}, showing that the Gaussian-regularized operator $Q^{\delta,\theta}$ enjoys the same regularization properties as the localized operator $Q^{\delta,\Theta^{\ast}_{T}}$. Unlike the latter, which is not readily simulable, $Q^{\delta,\theta}$ can be directly simulated by adding an independent Gaussian variable to $X^{\delta}_{T}$.

\begin{proof}[Proof of Theorem \ref{th:regul_main_result_intro}]
As in the proof of Theorem \ref{th:regularisation}, it is sufficient to prove our result for $f \in \mathcal{C}^{\infty}_{pol}(\mathbb{R}^{d};\mathbb{R})$. Let us prove (\ref{eq:borne_semigroupe_regularisation_gauss_main}). Assume that $q \in \mathbb{N}^{\ast}$, the case $q=0$ being immediate. As in (\ref{eq:derivee_semigroup_regularisation}), we write  
\begin{align*}
\partial^{\alpha }Q_{T}^{\delta,\theta}\partial^{\beta}f(x)=\sum_{\vert \beta \vert \leqslant
\vert \gamma \vert \leqslant q}\mathbb{E}_{x}[\partial^{\gamma }f(\delta^{\theta}
\Gaussian +X^{\delta}_{T})\mathscr{P}_{\gamma }(X^{\delta}_{T}) ],
\end{align*}%

where $\mathscr{P}_{\gamma }(X^{\delta}_{t})$ is a universal polynomial of $ \partial_{\mbox{\textsc{x}}^{\delta}_{0}}^{\rho }X^{\delta}_{t} ,1 \leqslant \vert \rho \vert \leqslant q- \vert \gamma \vert +1$.  We decompose 
\begin{align*}
\mathbb{E}_{x}[\partial^{\gamma }f(\delta^{\theta} \Gaussian +X^{\delta}_{T})\mathscr{P}_{\gamma }(X^{\delta}_{T}) ]=A_{1}+A_{2},
\end{align*}

with $A_{1} =\mathbb{E}_{x}[\Theta^{\ast}_{T} \partial^{\gamma }f (\delta^{\theta} \Gaussian
+X^{\delta}_{T})\mathscr{P}_{\gamma }(X^{\delta}_{T}) ]$ and $A_{2} =\mathbb{E}_{x}[\partial^{\gamma }f (\delta^{\theta} \Gaussian+X^{\delta}_{T}) \mathscr{P}_{\gamma
}(X^{\delta}_{T})(1-\Theta^{\ast}_{T}  ) ]$.
Applying the reasoning from the proof of Theorem \ref{th:regularisation} (with $\alpha=\emptyset$) 
 we derive
\begin{align*}
\vert A_{1} \vert \leqslant & \CstePsiouf_{f} \frac{1 + \mathbf{1}_{\ExpPsiouf  _{\max(q+3,2 L + 5)} + \ExpPsiouf  _{f} > 0} \vert x\vert_{\mathbb{R}^{d}}^{C } }{ (\HormV_{L}(x)T)^{13^{L} 3 dq(q+3)+1} } \CstePsiouf_{\max(q+3,2 L + 5)}^{C}     \exp(C (1+T) \MomZ_{C }(Z^{\delta}) \CstePsiouf_{3}^{4}) .
\end{align*}

Moreover, since $G$ follows the standard Gaussian distribution and is independent of $X^{\delta}$ and $\Theta^{\ast}_{T}$, we have
\begin{align*}
A_{2} =\mathbb{E}_{x} [\mathscr{P}_{\gamma
}(X^{\delta}_{T})(1-\Theta^{\ast}_{T}) \int_{\mathbb{R}^{d}} \partial^{\gamma }f (\delta^{\theta} u+X^{\delta}_{T}) (2 \pi )^{-\frac{d}{2}} e^{-\frac{\vert u \vert^{2}}{2}} du  ].
\end{align*}
Now, notice that
$
 \partial^{\gamma }f (\delta^{\theta} u+X^{\delta}_{T})=\delta^{- \vert \gamma \vert \theta} \partial_{u}^{\gamma}(f(\delta^{ \theta} u+ X^{\delta}_{T})),
$
so that, using standard integration-by-parts, we have
\begin{align*}
A_{2}=\delta^{-\vert \gamma \vert \theta } \mathbb{E}_{x} [\mathscr{P}_{\gamma
}(X^{\delta}_{T})(1-\Theta^{\ast}_{T})\int_{\mathbb{R}^d} f(\delta^{\theta} u+ X^{\delta}_{T}) \mathscr{H}_{\gamma}(u)  (2 \pi )^{-\frac{d}{2}} e^{-\frac{\vert u \vert^{2}}{2}} \mbox{d} u ] ,
\end{align*}
where $ \mathscr{H}_{\gamma} $ is the Hermite polynomial corresponding to the multi-index $\gamma$.  Finally, using the results from Theorem \ref{theo:Norme_Sobolev_borne} (see (\ref{eq:borne_Sob_flot_tang})), we obtain 
\begin{align*}
\vert A_{2} \vert \leqslant & \delta^{-\vert \gamma \vert \theta }  \CstePsiouf_{f} \mathbb{E}_{x}[1-\Theta^{\ast}_{T}  ]^{\frac{1}{2}}  (1+\mathbf{1}_{\ExpPsiouf  _{q+3}+\ExpPsiouf  _{f} > 0} \vert x \vert_{\mathbb{R}^{d}}^{C})\CstePsiouf_{q+3} ^{C}  \exp (C(1+T)\MomZ_{C}(Z^{\delta})\CstePsiouf_{3} ^{2}) .
\end{align*}
As in the proof of Theorem \ref{th:regularisation}, we estimate $ \mathbb{E}_{x}[1-\Theta^{\ast}_{T} ]$ using Theorem \ref{th:borne_Lp_inv_cov_Mal} (see (\ref{eq:th_cov_mal_proba_espace_comp})) with $p=p(q \theta)= \max(4,2+\frac{91 q \theta}{22d}) $ and $a=a(q \theta)=2(2 q \theta+1)\max(\ExpPsiouf_3  +1,91)$ and deduce that
\begin{align*}
 \vert \partial^{\alpha }Q_{T}^{\delta,\theta }\partial^{\beta}f(x) \vert \leqslant & \CstePsiouf_{f}\frac{(1+ \mathbf{1}_{\ExpPsiouf  _{\max(q+3,2 L + 5)}+\ExpPsiouf  _{f} > 0} \vert x \vert_{\mathbb{R}^{d}}^{C } )\CstePsiouf_{\max(q+3,2 L + 5)}^{C}  }{( \HormV_{L}(x) T)^{\eta}  }      \exp(C (1+T) \MomZ_{C }(Z^{\delta}) \CstePsiouf_{3}^{4}) \nonumber ,
\end{align*}
with $\eta= \max( 13^{L} 3 d(1+\frac{91 q \theta}{44d}),13^{L} 3 dq(q+3)+1)  $ and the proof of (\ref{eq:borne_semigroupe_regularisation_gauss_main}) is completed.  


Let us prove (\ref{eq:reg_gauss_distprop_intro}).  Since $f$ has polynomial growth, it follows that
\begin{align*}
\vert Q^{\delta}_{T}f(x)- & Q_{T}^{\delta,\theta}f(x)\vert  \leqslant  \vert  \mathbb{E}_{x}[\Theta^{\ast}_{T} ( f(X^{\delta}_{T})-f(X^{\delta}_{T}+\delta^{\theta} \Gaussian))  ] \vert \\
&+ \CstePsiouf_{f} C (1+\mathbb{E}_{x}[\vert X^{\delta}_{T} \vert_{\mathbb{R}^{d}}^{2 \ExpPsiouf  _{f}} ]^{\frac{1}{2}}+ \delta^{\theta \ExpPsiouf  _{f}} \mathbb{E}[ \vert G \vert_{\mathbb{R}^{d}}^{2 \ExpPsiouf  _{f}} ]^{\frac{1}{2}}) \mathbb{E}_{x}[1-\Theta^{\ast}_{T}   ]^{\frac{1}{2}}\\
  \leqslant &  \delta^{\theta} \sum_{j=1}^d \int_0^1 \vert \mathbb{E}_{x}[\Theta^{\ast}_{T}   \partial^{(j)} f(X^{\delta}_{T}+ \lambda \delta^{\theta} \Gaussian) G^j  ] \vert \mbox{d} \lambda \\
&+\CstePsiouf_{f} C (1+ \vert x  \vert_{\mathbb{R}^{d}}^{\ExpPsiouf  _{f}} ) \exp(C T \CstePsiouf_{3} ^{2} \MomZ_{C}(Z^{\delta})  )
\mathbb{E}_{x}[1-\Theta^{\ast}_{T}  ]^{\frac{1}{2}} .
\end{align*}
Using Theorem \ref{th:regularisation} (see (\ref{eq:borne_semigroupe_regularisation}) with $q=1$) and the same estimate of $\mathbb{E}_{x}[1-\Theta^{\ast}_{T} ]$ as in the proof of (\ref{eq:borne_semigroupe_regularisation_gauss_main}) with $p=p(\theta)= \max(4,2+\frac{91  \theta}{22d}) $ and $a=a( \theta)=2(2 \theta+1)\max(\ExpPsiouf_{3}  +1,91)$, we obtain
\begin{align*}
\vert Q^{\delta}_{T}f(x )-Q_{T}^{\delta,\theta }f(x )\vert
\leqslant &  \delta^{\theta} \CstePsiouf_{f}   \frac{1 + \mathbf{1}_{\ExpPsiouf  _{2 L + 5} + \ExpPsiouf  _{f} > 0} \vert x  \vert_{\mathbb{R}^{d}}^{C } }{(\HormV_{L}(x )T)^{\max(3d 13^{L} (1+ \frac{91 \theta}{44 d}),13^{L}12d+1) }} \CstePsiouf_{2 L + 5}^{C}  \exp(C (1+T) \MomZ_{C }(Z^{\delta}) \CstePsiouf_{3} ^{4}) .
\end{align*}
\end{proof}
We now derive Corollary \ref{coro:borne_densite_reg_gauss_intro} from Theorem \ref{th:regul_main_result_intro}, obtaining quantitative derivative bounds on the density of $X^{\delta}_{T}+\delta^{\theta}\Gaussian$. The polynomial-decay estimate relies on a density representation formula, while the Gaussian-type bound follows under additional assumptions using the Azuma-Hoeffding inequality.



\begin{proof}[Proof of Corollary \ref{coro:borne_densite_reg_gauss_intro}]


We first establish the existence of $q^{\delta,\theta}_{T}$ and derive a representation formula for it. Let $f \in \mathcal{C}_{b}^{\infty}(\mathbb{R}^{d};\mathbb{R})$.  Let us define $g:\mathbb{R}^{d} \to \mathbb{R}$ such that for every $u \in \mathbb{R}^{d}$,
\begin{align*}
g(u)  :=  \int_{0}^{u^{d}}\ldots \int_{0}^{u^{1}}  f(y) \mbox{d} y^1 \ldots \mbox{d} y^{d}.
\end{align*}
Then $g \in \mathcal{C}_{pol}^{\infty}(\mathbb{R}^{d};\mathbb{R})$ and for $\gamma_{0}=(1,\ldots,1) \in \mathbb{N}^{d}$,  we have $ \partial^{\gamma_{0}}g = f$. In order to state our representation formula, we introduce the function $\boldsymbol{\Gamma}:(u,v)\in (\mathbb{R}^{d})^{2} \mapsto \prod_{i=1}^{d} ( \mathbf{1}_{ 0< v^{i} <  u^{i}} - \mathbf{1}_{u^{i} < v^{i} <   0} )$ which plays a fundamental role in the representation formula that will be derived.  \\

In particular, applying Theorem \ref{theoreme:IPP_Malliavin} to the test function $g$ and using Fubini's theorem, we obtain, with the notation of the proof of Theorem \ref{th:regul_main_result_intro},
\begin{align}
\label{eq:represent_form_derivation}
\partial^{\alpha}  Q^{\delta,\theta}_{T}  \partial^{\beta} f(x ) =&  \partial^{\alpha} Q^{\delta,\theta}_{T} \partial^{(\beta+ \gamma_{0})} g(x )  
= \sum_{0 \leqslant \vert\gamma \vert \leqslant q+d}       \mathbb{E}_{x}[\partial^{\gamma }g (\delta^{\theta} \Gaussian+ X^{\delta}_{T})\mathscr{P}_{\gamma }(X^{\delta}_{T})( \Theta^{\ast}_{T} +1-\Theta^{\ast}_{T}  )   ]  \nonumber \\
=& \sum_{0 \leqslant \vert\gamma \vert \leqslant q+d} \mathbb{E}_{x}[g( \delta^{\theta} \Gaussian + X^{\delta}_{T}) \PoidMal ( X^{\delta}_{T},\Theta^{\ast}_{T}  \mathscr{P}_{\gamma }(X^{\delta}_{T}))[\gamma]]   \nonumber \\
& \quad  + \mathbb{E}_{x}[ g( \delta^{\theta} \Gaussian + X^{\delta}_{T}) \delta^{-\vert \gamma \vert \theta } \mathscr{P}_{\gamma}(X^{\delta}_{T})(1-\Theta^{\ast}_{T}  )  \mathscr{H}_{\gamma}(\Gaussian)  ]  \nonumber \\
=&   \int_{\mathbb{R}^{d}}f(y)\mathbb{E}_{x}[H(\alpha,\beta)\boldsymbol{\Gamma}(\delta^{\theta} \Gaussian+ X^{\delta}_{T} , y)  ]  \mbox{d}y,  
\end{align}%
where 
\begin{align*}
H(\alpha,\beta) =& \sum_{0 \leqslant \vert\gamma \vert \leqslant q+d}  \PoidMal
 (X^{\delta}_{T},\Theta^{\ast}_{T}   \mathscr{P}_{\gamma }(X^{\delta}_{T}))[\gamma]  + \delta^{-\vert \gamma \vert \theta } \mathscr{P}_{\gamma
}(X^{\delta}_{T})(1-\Theta^{\ast}_{T} )  \mathscr{H}_{\gamma}(\Gaussian)  .
\end{align*}
Notice that using the H\"older inequality, it follows from $\mathbf{A}^{\delta}_{1}( \max(q+d+3,2 L + 5))$ and $\mathbf{A}_{3}^{\delta}(+\infty)$ that $\delta^{\theta} \Gaussian^{i}+ X^{\delta,i}_{T} \in \LLL_{p}(\Omega,\mathbb{P})$ for any $p \geq 1$ (see Lemma \ref{lemme:borne:moments_X}), and that $H(\alpha,\beta) \boldsymbol{\Gamma}(\delta^{\theta} \Gaussian+ X^{\delta}_{T} , \cdot)  \in \LLL_{1}(\Omega \times \mathbb{R}^{d},\mathbb{P} \otimes \mbox{d} y)$ (which makes the use of Fubini's theorem valid in the previous computation). 
Since $X^{\delta}_{T}+\delta^{\theta}\Gaussian$ is a Gaussian convolution, its density is given by $q^{\delta,\theta}_{T}(x,y)=\mathbb{E}_{x}[g_{\delta^{\theta}}(y-X^{\delta}_{T})]$ with $g_{\delta^{\theta}}(u)=(2\pi\delta^{2\theta})^{-\frac{d}{2}}\exp(-\frac{\vert u\vert_{\mathbb{R}^{d}}^{2}}{2\delta^{2\theta}})$. Using Theorem \ref{theo:Norme_Sobolev_borne} (see (\ref{eq:borne_Sob_flot_tang}) with $q=q^{\diamond}=0$), the moments of $\partial^{\rho}_{\mbox{\textsc{x}}^{\delta}_{0}}X^{\delta}_{T}$, $\vert\rho\vert\leqslant q$, are locally bounded in $x$ and differentiation under the expectation yields $q^{\delta,\theta}_{T}\in\mathcal{C}^{q}(\mathbb{R}^{d}\times\mathbb{R}^{d})$ whenever $\mathbf{A}^{\delta}_{1}(q+2)$, $\mathbf{A}_{3}^{\delta}(+\infty)$ and $\mathbf{A}_{4}^{\delta}$ hold.
Therefore,  owing to (\ref{eq:represent_form_derivation}) when $f$ has compact support and to the dominated convergence theorem, we have the following representation formula for $q^{\delta,\theta}_{T}$ and its derivatives:
\begin{align*}
\partial_{x}^{\alpha} \partial_{y}^{\beta} q^{\delta,\theta}_{T}(x,y)=(-1)^{\vert \beta \vert} \mathbb{E}_{x}[\boldsymbol{\Gamma}(\delta^{\theta} \Gaussian+ X^{\delta}_{T} , y)  H(\alpha,\beta) ] ,
\end{align*}

which holds for almost every $y$.  The preceding reasoning still applies if we replace $\boldsymbol{\Gamma}$ by $\boldsymbol{\Gamma}_{x}:(u,v)\in (\mathbb{R}^{d})^{2} \mapsto \prod_{i=1}^{d} ( \mathbf{1}_{ x^{i}< v^{i} <  u^{i}} - \mathbf{1}_{u^{i} < v^{i} <    x^{i}})$.  In particular
\begin{align*}
\partial_{x}^{\alpha} \partial_{y}^{\beta} q^{\delta,\theta}_{T}(x,y)=(-1)^{\vert \beta \vert} \mathbb{E}_{x}[\boldsymbol{\Gamma}_{x}(\delta^{\theta} \Gaussian+ X^{\delta}_{T} , y)  H(\alpha,\beta) ] ,
\end{align*}
for almost every $y$. Extending those derivatives continuously on $\mathbb{R}^{d} \times \mathbb{R}^{d}$, the estimate (\ref{eq:borne_densite_reg_gauss_intro}) then follows using $\vert \boldsymbol{\Gamma}_{x}(u,v) \vert \leq \prod_{i=1}^{d}\mathbf{1}_{\vert u^{i}-x^{i} \vert \geq \vert v^{i}-x^{i} \vert} \leq \mathbf{1}_{\vert u-x \vert_{\mathbb{R}^{d}} \geq \vert v-x \vert_{\mathbb{R}^{d}}}$, the Cauchy-Schwarz inequality, Lemma \ref{lemme:borne:moments_X} on the shifted process $X^{\delta}-x$, combined with Markov's inequality and a similar approach as in the proof of the previous results to bound the moments of $H(\alpha,\beta)$. We thus obtain 
\begin{align*}
 \vert \partial_{x}^{\alpha } \partial_{y}^{\beta } q_{T}^{\delta,\theta }(x,y) \vert \leqslant &   \frac{(1 + \mathbf{1}_{\ExpPsiouf  _{\max(q+d+3,2 L + 5)} > 0} \vert x \vert_{\mathbb{R}^{d}}^{C }   )  C \exp(C T) }{(  \HormV_{L}(x) T )^{\eta }(1+ \vert y-x \vert_{\mathbb{R}^{d}}^{p})}  ,
\end{align*}

where $\eta=13^{L}3d \max(1+\frac{91 (d+q) \theta}{44d},(d+q)^{2}+3(d+q)+1) $. 
Now let us prove (\ref{eq:borne_densite_reg_gauss_borne_exp}). 
Using Taylor expansions of $\psi$ and recalling that $\psi(u,t,0,0)=u$ for every $(u,t) \in \mathbb{R}^{d} \times \mathbb{R}_{+}$,  an exponential estimate on the tail probability of $\Gaussian$ and the Azuma-Hoeffding inequality yield, for $ \vert y   -x \vert_{\mathbb{R}^{d}} \geqslant  12 T \CstePsiouf_{2} (1+ \vert z^{\infty} \vert^{2}) $,
\begin{align*}
\mathbb{E}_{x}[\vert \boldsymbol{\Gamma}_{x}(X^{\delta}_{T} &+\delta^{\theta} \Gaussian,y)  \vert  ] \leqslant   \mathbb{P}_{x}(\vert y   -x \vert_{\mathbb{R}^{d}} -  \delta^{\theta} \vert G \vert_{\mathbb{R}^{d}} \leqslant  \vert X^{\delta}_{T}   -x \vert_{\mathbb{R}^{d}} )  \\
 \leqslant &  \mathbb{P}_{x}(\vert y   -x \vert_{\mathbb{R}^{d}} -  \delta^{\theta} \vert G \vert_{\mathbb{R}^{d}}  \leqslant 3 T \CstePsiouf_{2} (1+ \vert z^{\infty} \vert^{2})  + \delta^{\frac{1}{2}}  \vert \sum_{t \in \pi^{\delta},t<T} \sum_{i \in \EnsN} Z^{\delta,i}_{t+\delta} \partial_{z^{i}} \psi(X^{\delta}_{t},t,0,0)  \vert_{\mathbb{R}^{d}}  ) \\
  \leqslant &  \mathbb{P}_{x}(\frac{1}{2} \vert y   -x \vert_{\mathbb{R}^{d}} - 3 T \CstePsiouf_{2} (1+ \vert z^{\infty} \vert^{2})  \leqslant  \delta^{\frac{1}{2}}  \vert \sum_{t \in \pi^{\delta},t<T} \sum_{i \in \EnsN} Z^{\delta,i}_{t+\delta} \partial_{z^{i}} \psi(X^{\delta}_{t},t,0,0)  \vert_{\mathbb{R}^{d}}  ) \\
  &+\mathbb{P}(\vert \Gaussian \vert_{\mathbb{R}^{d}}  \geqslant \frac{1}{2} \delta^{-\theta} \vert y   -x \vert_{\mathbb{R}^{d}}  ) \\
\leqslant & C \exp(CT-  \frac{  \vert y -x \vert_{\mathbb{R}^{d}}^{2}}{c T})+ C\exp(-\frac{  \vert y   -x \vert_{\mathbb{R}^{d}}^{2}}{c \delta^{2\theta}}  ) .
\end{align*}
Finally, we observe that when $ \vert y   -x \vert_{\mathbb{R}^{d}} < 12 T \CstePsiouf_{2} (1+ \vert z^{\infty} \vert^{2}) $, the estimate still holds by remarking that
\begin{align*}
\exp(-\frac{  \vert y   -x \vert_{\mathbb{R}^{d}}^{2}}{c( T \vee \delta^{2\theta})}  ) \geqslant \exp(-\frac{  ( 12 \CstePsiouf_{2} (1+ \vert z^{\infty} \vert^{2}))^{2}}{ c}T  ) ,
\end{align*}
so that 
\begin{align*}
\mathbb{E}_{x}[\vert \boldsymbol{\Gamma}_{x}(X^{\delta}_{T} +\delta^{\theta} \Gaussian,y) \vert] \leqslant \exp(-\frac{  \vert y   -x \vert_{\mathbb{R}^{d}}^{2}}{c( T \vee \delta^{2\theta})}  ) \exp(\frac{  (12  \CstePsiouf_{2} (1+ \vert z^{\infty} \vert^{2}))^{2}}{ c}T  ) .
\end{align*}
The Cauchy-Schwarz inequality, together with the representation formula, completes the proof.

\end{proof}


We conclude this section by proving the invariance principle stated in Theorem \ref{th:invariance_main_result}. Our strategy is to decompose the approximation error using the Lindeberg approach and the semigroup properties, focusing on the one-step error over an interval of length $\delta$. We then replace $Q^{\delta}$ by $Q^{\delta,\theta}$ and apply Taylor expansions. This yields a representation of the error involving slight modifications of $X^{\delta,\theta}$ for which analogous regularization estimates hold. These estimates give the desired result. The same strategy can be adapted to obtain higher-order convergence.

\begin{proof}[Proof of Theorem \ref{th:invariance_main_result}] 
For $u \in \mathbb{R}^{d}$, $s,t\in \pi^{\delta}$, $s \leqslant t$, $\theta > 0$, we define, for measurable $f:\mathbb{R}^{d} \to \mathbb{R}$ such that the expectations are well defined $Q^{\delta}_{s,t}f(u):= \mathbb{E}[f(X^{\delta}_{t}) \vert X^{\delta}_{s}=u]=\mathbb{E}[f(X^{\delta}_{t}(u,s))]$ ($X^{\delta}_{t}(u,s)$ built as $X^{\delta}_{t}$ but starting from $u$ at time $s$), $Q^{\delta,\theta}_{s,t}f(u):= \mathbb{E}[f(X^{\delta,\theta}_{t} )\vert X^{\delta}_{s}=u]$ with $X^{\delta,\theta}_{t}\coloneqq X^{\delta}_{t}+\delta^{\theta} \Gaussian$, $P_{s,t}f(u):= \mathbb{E}[f(X_{t}) \vert X_{s}=u]= \mathbb{E}[f(X_{t}(u,s))]$ ($X_{t}(u,s)$ being the solution of (\ref{eq:eds_strat_intro}) at time $t$ and started from $u$ at time $s$) and $\Delta^{\delta}_{t,t+\delta} f(u):= Q^{\delta}_{t,t+\delta}f(u)-Q^{\frac{\delta}{2}}_{t,t+\delta}f(u)$.  
By point (\ref{rmk:reg_shift}) of Remark \ref{rmk:reg_variants}, the results from Theorem \ref{th:regul_main_result_intro} remain true between any times $s < t$ in $\pi^{\delta}$ under $\mathbf{A}_{2}(x,s,L)$ and $\mathbf{A}_{5}^{\delta}(x,s,t)$. 
We first prove part \ref{th:invariance_main_result_VT}, for which we fix $\theta = \frac{3}{2}$. 
We adopt a Cauchy sequence approach: setting $\delta_{l}=\frac{\delta}{2^{l}}$, we compare $Q^{\delta_{l}}$ and $Q^{\delta_{l+1}}$ step by step and show that $\vert Q^{\delta_{l}}_{T}f(x)-Q^{\delta_{l+1}}_{T}f(x)\vert$ is of order $\delta_{l}^{\frac{1}{2}}$, so that summing over $l$ yields the result. For $l\geqslant 1$, the processes with time step $\delta_l$ are auxiliary. We construct them by choosing the same update function $\psi$ and the random variables $(Z_t^{\delta_l})_{t\in\pi^{\delta_l,\ast}}$ independent, where $Z_t^{\delta_l}$ has the same law as $Z_{\delta\lceil \frac{t}{\delta} \rceil}^\delta$.  In particular, $\mathbf{A}_{3}^{\delta_{l}}(+\infty)$ and $\mathbf{A}_{4}^{\delta_{l}}$ hold with the same moment bounds and the same constants $\varepsilon_{\ast}$ and $r_{\ast}$ as for $Z^\delta$. The corresponding variables $(\chi^{\delta_l}_t,U^{\delta_l}_t,V^{\delta_l}_t)_{t\in\pi^{\delta_l,\ast}}$ are built as for $Z^{\delta}$, and these families, $l\in\mathbb{N}$, are taken mutually independent and independent of $\Gaussian$. Moreover, since $\delta_{l}\leqslant\delta$, $\mathbf{A}_1^{\delta_l}(\max(6,2L+5))$ and $\mathbf{A}_5^{\delta_l}(x,T)$ follow from $\mathbf{A}_1^\delta(\max(6,2L+5))$ and $\mathbf{A}_5^\delta(x,T)$, respectively.


We prove the result for $f \in \mathcal{C}^{\infty}_{pol}(\mathbb{R}^{d})$. The extension to $f$ simply measurable with polynomial growth follows from Lusin's theorem. We present the main steps, omitting lengthy calculations which can be dealt with using arguments similar to those already developed to derive Theorem \ref{th:regul_main_result_intro}. Using the Markov property satisfied by $Q^{\delta_{l}}$ and $Q^{\delta_{l+1}}$ on the common grid $\pi^{\delta_{l}}$, we have
\begin{align}
Q^{\delta_{l}}_{T}f(x)-& Q^{\delta_{l+1}}_{T}f(x) =\sum_{t \in \pi^{\delta_{l}},t < T} Q^{\delta_{l}}_{0,t} \Delta_{t,t+\delta_{l}}^{\delta_{l}}Q^{\delta_{l+1}}_{t+\delta_{l},T}f(x) \nonumber \\
=& \sum_{t \in \pi^{\delta_{l}},t < \frac{T}{2}} Q^{\delta_{l},\theta}_{0,t} \Delta^{\delta_{l}}_{t,t+\delta_{l}}Q^{\delta_{l+1},\theta}_{t+\delta_{l},T}f(x) +  Q^{\delta_{l}}_{0,t}\Delta^{\delta_{l}}_{t,t+\delta_{l}} (Q^{\delta_{l+1}}_{t+\delta_{l},T}-Q^{\delta_{l+1},\theta}_{t+\delta_{l},T})f(x) \nonumber \\
&\qquad \qquad \qquad+(Q^{\delta_{l}}_{0,t}-Q^{\delta_{l},\theta}_{0,t})\Delta^{\delta_{l}}_{t,t+\delta_{l}}Q^{\delta_{l+1},\theta}_{t+\delta_{l},T}f(x) \nonumber \\
&+  \sum_{t \in \pi^{\delta_{l}}, \frac{T}{2} \leqslant t < T} Q^{\delta_{l},\theta}_{0,t} \Delta^{\delta_{l}}_{t,t+\delta_{l}}Q^{\delta_{l+1}}_{t+\delta_{l},T}f(x) + ( Q^{\delta_{l}}_{0,t} -Q^{\delta_{l},\theta}_{0,t}) \Delta^{\delta_{l}}_{t,t+\delta_{l}} Q^{\delta_{l+1}}_{t+\delta_{l},T}f(x).\label{eq:lind_loc}
\end{align}
First note that, by the polynomial moment estimates provided in Lemma \ref{lemme:borne:moments_X}, $Q^{\delta_{l}}_{s,t}$, $Q^{\delta_{l},\theta}_{s,t}$, $Q^{\delta_{l+1}}_{s,t}$ and $Q^{\delta_{l+1},\theta}_{s,t}$ preserve polynomial growth, with constants controlled by $C\exp(CT)$.  Consequently, the same holds for the terms appearing in the sums in (\ref{eq:lind_loc}), with growth constants bounded by $C\exp(CT)\CstePsiouf_f$. We call main terms the first term of each sum in (\ref{eq:lind_loc}) and correction terms the others, and we show that each of them satisfies the required bound $\delta_{l}^{\frac{3}{2}}\CstePsiouf_{f}(1+\vert x\vert_{\mathbb{R}^{d}}^{c})(\HormV_{L}(x)T)^{-\eta}C\exp(CT)$. If $\HormV_{L}(x)T\leqslant C(1\vee T^{3})^{\frac{1}{26}}\delta_{l}^{\frac{44}{13^{L}91}}$, then $1\leqslant\delta_{l}^{\frac{3}{2}}C\exp(CT)(\HormV_{L}(x)T)^{-\eta}$ for $\eta$ large enough and this bound follows from polynomial growth. We thus assume the reverse inequality, with $C$ large enough so that $\mathbf{A}_{5}^{\delta_{l}}(x,t)$ holds for $t\geqslant\frac{T}{2}$, and $\mathbf{A}_{5}^{\delta_{l+1}}(u,t+\delta_{l},T)$ holds for $t<\frac{T}{2}$ and $\HormV_{L}(u,t+\delta_{l})\geqslant\frac{\HormV_{L}(x)}{2}$.



Let $g \in \mathcal{C}^{3}_{pol}(\mathbb{R}^{d};\mathbb{R})$. Before expanding $\Delta^{\delta_{l}}_{t,t+\delta_{l}}g(u)=\mathbb{E}[g(X^{\delta_{l}}_{t+\delta_{l}}(u,t))-g(X^{\delta_{l+1}}_{t+\delta_{l}}(u,t))]$, we multiply the difference inside the expectation by $\Psi=\Psi_{\funceta_{2}(\delta_{l})}(\vert Z^{\delta_{l}}_{t+\delta_{l}}\vert_{\mathbb{R}^{N}})\Psi_{\funceta_{2}(\delta_{l})}(\vert Z^{\delta_{l+1}}_{t+\delta_{l+1}}\vert_{\mathbb{R}^{N}})\Psi_{\funceta_{2}(\delta_{l})}(\vert Z^{\delta_{l+1}}_{t+\delta_{l}}\vert_{\mathbb{R}^{N}})$. The term involving $1-\Psi$ is not expanded. By the polynomial growth of $g$ and the Markov inequality, and since $\funceta_{2}(\delta_{l})\geqslant\min(\delta_{l}^{-\frac{3}{182}},\delta_{l}^{-\frac{1}{4(1+\ExpPsiouf_{3})}})$, it is bounded by $\delta_{l}^{\frac{3}{2}}(1+\vert u\vert_{\mathbb{R}^{d}}^{c})$ times the growth constant of $g$. Its contribution to each term of (\ref{eq:lind_loc}) thus satisfies the required bound, and we no longer mention it. The expansions of $\Delta^{\delta_{l}}_{t,t+\delta_{l}}$ below are understood up to this term. By Taylor's formula at order $3$ of $g$ at $u$, the term involving $\Psi$ in $\Delta^{\delta_{l}}_{t,t+\delta_{l}}g(u)$ can be written as a finite sum of terms of the form
\begin{align*}
\mathbb{E}[\int_{0}^{1} \partial^{\alpha}g(Y^{\lambda}_{t+\delta_{l}}(u))B(u,t,\delta_{l},Y^{1}_{t+\delta_{l}}(u),\lambda)  \mbox{d} \lambda ], \quad \alpha \in \mathbb{N}^{d},\vert \alpha \vert \leqslant 3,
\end{align*}
where $Y^{\lambda}_{t+\delta_{l}}(u)$ takes values in $\{u,u+ \lambda (X^{\delta_{l}}_{t+\delta_{l}}(u,t)-u),u+ \lambda (X^{\delta_{l+1}}_{t+\delta_{l}}(u,t)-u)\}$. 

For $\vert\alpha\vert\leqslant 2$, $Y^{\lambda}_{t+\delta_{l}}(u)=u$ and, since $\partial^{\alpha}g(u)$ does not depend on the increments on $(t,t+\delta_{l}]$, we take $B(u,t,\delta_{l},u,\lambda)=\frac{1}{\alpha!}\mathbb{E}[\Psi((X^{\delta_{l}}_{t+\delta_{l}}(u,t)-u)^{\alpha}-(X^{\delta_{l+1}}_{t+\delta_{l}}(u,t)-u)^{\alpha})]$, which is a deterministic function of $u$. For $\vert\alpha\vert=3$, $B$ is proportional to $\Psi(1-\lambda)^{2}(Y^{1}_{t+\delta_{l}}(u)-u)^{\alpha}$.
By a second-order Taylor expansion of the one-step and two-half-step transitions, the centering and covariance assumptions on $Z$ ensure the cancellation of the leading-order terms after expectation,  up to errors due to $\Psi$ which are bounded by $\mathbb{E}[\vert Z\vert^{2}_{\mathbb{R}^{N}}\mathbf{1}_{\vert Z\vert_{\mathbb{R}^{N}}\geqslant\funceta_{2}(\delta_{l})-\frac{1}{2}}]\leqslant C\funceta_{2}(\delta_{l})^{-a}$, $a>0$ arbitrary,  while, since $\psi(u,t,0,0)=u$, $\mathbf{A}_{1}^{\delta}(\max(6,2L+5))$ shows that the remaining terms and their derivatives in $u$ are of order $\delta_{l}^{\frac{3}{2}}$. Using (\ref{eq:borne_Sob_flot_tang}) for $X^{\delta_{l}}_{t}$, we get, for every $p\geqslant 1$,
\begin{align}
\label{eq:boud_term_proof_invariance_principle_dvpt}
\sup_{t \in \pi^{\delta_{l}},t < T}\mathbb{E}_{x}[ \vert B(X^{\delta_{l},\theta}_{t}, t,\delta_{l},Y^{1}_{t+\delta_{l}}(X^{\delta_{l},\theta}_{t}),\lambda)  \vert_{\mathbb{R},3}^{p} ]^{\frac{1}{p}}\leqslant \delta_{l}^{\frac{3}{2}}  (1+\vert x \vert _{\mathbb{R}^{d}}^{c}) C \exp(CT).
\end{align}
Taking $u=X^{\delta_{l},\theta}_{t}$, $Q^{\delta_{l},\theta}_{0,t}\Delta^{\delta_{l}}_{t,t+\delta_{l}}g(x)$ is thus a finite sum of terms of the form
\begin{align*}
\mathbb{E}_{x}[\int_{0}^{1} \partial^{\alpha}g(Y^{\lambda}_{t+\delta_{l}}(X^{\delta_{l},\theta}_{t}))B(X^{\delta_{l},\theta}_{t}, t,\delta_{l},Y^{1}_{t+\delta_{l}}(X^{\delta_{l},\theta}_{t}),\lambda) \mbox{d} \lambda ].
\end{align*}
At this point, we observe that, for $t>0$ such that $\mathbf{A}_{5}^{\delta_{l}}(x,t)$ holds (in particular for $t\geqslant\frac{T}{2}$), the approach used in the proof of Theorem \ref{th:regul_main_result_intro}, with $X_t^{\delta_l,\theta}$ replaced by $Y^\lambda_{t+\delta_l}(X_t^{\delta_l,\theta})$,  yields the integration-by-parts estimates needed to bound the terms above, uniformly in $\lambda\in[0,1]$. Conditionally on $(\Gaussian,Z^{\delta_{l}}_{t+\delta_{l}},Z^{\delta_{l+1}}_{t+\delta_{l+1}},Z^{\delta_{l+1}}_{t+\delta_{l}})$, which is independent of $(\chi^{\delta_{l}}_{w},U^{\delta_{l}}_{w},V^{\delta_{l}}_{w})_{w \in \pi^{\delta_l,\ast},w\leqslant t}$, $Y^{\lambda}_{t+\delta_{l}}(.)$ is smooth with polynomially bounded derivatives. First, $\DeriveMal Y^{\lambda}_{t+\delta_{l}}(X^{\delta_{l},\theta}_{t})=\nabla_{x}Y^{\lambda}_{t+\delta_{l}}(X^{\delta_{l},\theta}_{t})\DeriveMal X^{\delta_{l}}_{t}$ and, by (\ref{formule_chain_rule_Operateur_OU}), the Sobolev norms of $Y^{\lambda}_{t+\delta_{l}}(X^{\delta_{l},\theta}_{t})$ and $\OpOrnsteinUlh Y^{\lambda}_{t+\delta_{l}}(X^{\delta_{l},\theta}_{t})$ are controlled by Theorem \ref{theo:Norme_Sobolev_borne}. Second, since $\Psi$ multiplies the whole difference, the terms where $Y^{\lambda}_{t+\delta_{l}}(.)$ is not the identity vanish outside $\{\Psi>0\}$. On this set, $\vert Z^{\delta_{l}}_{t+\delta_{l}}\vert_{\mathbb{R}^{N}}$, $\vert Z^{\delta_{l+1}}_{t+\delta_{l+1}}\vert_{\mathbb{R}^{N}}$ and $\vert Z^{\delta_{l+1}}_{t+\delta_{l}}\vert_{\mathbb{R}^{N}}$ are smaller than $\funceta_{2}(\delta_{l})$, and (\ref{borne_grad_psi-Id}) and (\ref{hyp:delta_eta_inversibilite}) give $\Vert\nabla_{x}Y^{\lambda}_{t+\delta_{l}}(.)-I_{d\times d}\Vert_{\mathbb{R}^{d}}\leqslant\frac{5}{6}$, so that $\Vert(\nabla_{x}Y^{\lambda}_{t+\delta_{l}}(.))^{-1}\Vert_{\mathbb{R}^{d}}\leqslant 6$ and $\det\InvMatCovMal_{Y^{\lambda}_{t+\delta_{l}}(X^{\delta_{l},\theta}_{t})}\leqslant 6^{2d}\det\InvMatCovMal_{X^{\delta_{l}}_{t}}$ which is controlled by Theorem \ref{th:borne_Lp_inv_cov_Mal}. Third, the bound on $(\nabla_{x}Y^{\lambda}_{t+\delta_{l}}(.))^{-1}$ is also used to estimate the term involving $1-\Theta^{\ast}_{t}$, where, as in the proof of Theorem \ref{th:regul_main_result_intro}, the derivatives of a test function $g$ are removed by integration by parts $w.r.t.$ $\Gaussian$, using $(\nabla_{x}g)(Y^{\lambda}_{t+\delta_{l}}(X^{\delta_{l}}_{t}+\delta_{l}^{\theta}u))=\delta_{l}^{-\theta}\nabla_{u}[g(Y^{\lambda}_{t+\delta_{l}}(X^{\delta_{l}}_{t}+\delta_{l}^{\theta}u))](\nabla_{x}Y^{\lambda}_{t+\delta_{l}}(X^{\delta_{l}}_{t}+\delta_{l}^{\theta}u))^{-1}$. Iterating this identity for the derivatives of order at most $3$, the Hermite polynomials of the proof of Theorem \ref{th:regul_main_result_intro} are replaced by polynomials in $\Gaussian$, $(\nabla_{x} Y^{\lambda}_{t+\delta_{l}}(.))^{-1}$ and the derivatives of $Y^{\lambda}_{t+\delta_{l}}(.)$, which are controlled as above.

For $t \geqslant \frac{T}{2}$, these estimates allow us to apply the integration-by-parts formula of Theorem \ref{theoreme:IPP_Malliavin} with $F=Y^{\lambda}_{t+\delta_{l}}(X^{\delta_{l},\theta}_{t})$ and the polynomially bounded test function $\phi = Q^{\delta_{l+1}}_{t+\delta_{l},T}f$. Arguing as in the proof of Theorem \ref{th:regul_main_result_intro}, with the Gaussian integration by parts modified as above, and using (\ref{eq:boud_term_proof_invariance_principle_dvpt}), we obtain
\begin{align*}
\vert \mathbb{E}_{x}[\int_{0}^{1} \partial^{\alpha}Q^{\delta_{l+1}}_{t+\delta_{l},T}f(Y^{\lambda}_{t+\delta_{l}}(X^{\delta_{l},\theta}_{t}))B(X^{\delta_{l},\theta}_{t}, t,\delta_{l},Y^{1}_{t+\delta_{l}}(X^{\delta_{l},\theta}_{t}),\lambda) \mbox{d} \lambda ] \vert  \leqslant \delta_{l}^{\frac{3}{2}} \CstePsiouf_{f} \frac{1 + \vert x \vert_{\mathbb{R}^{d}}^{c }     }{(  \HormV_{L}(x) T )^{\eta}} C\exp(C T).
\end{align*}
Now let $t \in \pi^{\delta_{l}}$ with $t < \frac{T}{2} $ so that $T-t-\delta_{l}  \geqslant  \frac{1}{3}T$. We write $ Q^{\delta_{l},\theta}_{0,t} \Delta^{\delta_{l}}_{t,t+\delta_{l}}Q^{\delta_{l+1},\theta}_{t+\delta_{l},T}f(x)$ as a finite sum of localized terms plus a complementary term,
\begin{align*}
\mathbb{E}_{x}[& \int_{0}^{1}\partial^{\alpha}(\phi_{\HormV_{L}(x)}Q^{\delta_{l+1},\theta}_{t+\delta_{l},T}f)(Y^{\lambda}_{t+\delta_{l}}(X^{\delta_{l},\theta}_{t}))B(X^{\delta_{l},\theta}_{t}, t,\delta_{l},Y^{1}_{t+\delta_{l}}(X^{\delta_{l},\theta}_{t}),\lambda) \mbox{d} \lambda ]\\
&+Q^{\delta_{l},\theta}_{0,t} \Delta^{\delta_{l}}_{t,t+\delta_{l}}((1-\phi_{\HormV_{L}(x)})Q^{\delta_{l+1},\theta}_{t+\delta_{l},T}f)(x) ,
\end{align*}
where $\phi_{\HormV_{L}(x)}$ is a smooth localizing function satisfying, for every $u \in \mathbb{R}^{d}$,
\begin{align*}
\mathbf{1}_{\vert \HormV_{L}(u,t+\delta_{l}) - \HormV_{L}(x) \vert \leqslant \frac{\HormV_{L}(x)}{4}} \leqslant \phi_{\HormV_{L}(x)}(u) \leqslant \mathbf{1}_{\vert \HormV_{L}(u,t+\delta_{l}) - \HormV_{L}(x) \vert \leqslant \frac{\HormV_{L}(x)}{2}} ,
\end{align*} 

whose derivatives have polynomial growth in $u$, depending polynomially on $\HormV_{L}(x)^{-1}$. 

Since $\mathbf{A}_{5}^{\delta_{l+1}}(u,t+\delta_{l},T)$ holds for every $u$ in the support of $\phi_{\HormV_{L}(x)}$, the regularization estimate (\ref{eq:borne_semigroupe_regularisation_gauss_main}) for $Q^{\delta_{l+1},\theta}_{t+\delta_{l},T}f$ and (\ref{eq:boud_term_proof_invariance_principle_dvpt}) yield the required bound for the localized terms, without integration by parts. 
To bound the complementary term, we use the polynomial-growth estimates stated above and proceed as previously in the proof, writing it as a sum of terms of the form
 \begin{align*}
\mathbb{E}_{x}[\int_{0}^{1} \partial^{\alpha}((1-\phi_{\HormV_{L}(x)}) Q^{\delta_{l+1},\theta}_{t+\delta_{l},T}f)(Y^{\lambda}_{t+\delta_{l}}(X^{\delta_{l},\theta}_{t}))B(X^{\delta_{l},\theta}_{t}, t,\delta_{l},Y^{1}_{t+\delta_{l}}(X^{\delta_{l},\theta}_{t}),\lambda) \mbox{d} \lambda ], \quad \alpha \in \mathbb{N}^{d},\vert \alpha \vert \leqslant 3.
\end{align*}
Assume first that $\mathbf{A}_{5}^{\delta_{l}}(x,t)$ holds. Applying the integration-by-parts formula of Theorem \ref{theoreme:IPP_Malliavin} with $F=Y^{\lambda}_{t+\delta_{l}}(X^{\delta_{l},\theta}_{t})$ and the test function $\phi=(1-\phi_{\HormV_{L}(x)})Q^{\delta_{l+1},\theta}_{t+\delta_{l},T}f$, each of these terms can be written as
\begin{align*}
\mathbb{E}_{x}[\int_{0}^{1} ((1-\phi_{\HormV_{L}(x)})Q^{\delta_{l+1},\theta}_{t+\delta_{l},T}f)(Y^{\lambda}_{t+\delta_{l}}(X^{\delta_{l},\theta}_{t}))H_{B,\alpha}(X^{\delta_{l},\theta}_{t}, t,\delta_{l},Y^{1}_{t+\delta_{l}}(X^{\delta_{l},\theta}_{t}),\lambda) \mbox{d} \lambda ] ,
\end{align*}
with $H_{B,\alpha}(X^{\delta_{l},\theta}_{t}, t,\delta_{l},Y^{1}_{t+\delta_{l}}(X^{\delta_{l},\theta}_{t}),\cdot)$ bounded in any $\LLL_p(\Omega \times (0,1))$, $p \geqslant 1$ and satisfying
\begin{align*}
\mathbb{E}_{x}[\vert H_{B,\alpha}(X^{\delta_{l},\theta}_{t}, t,\delta_{l},Y^{1}_{t+\delta_{l}}(X^{\delta_{l},\theta}_{t}),\lambda) \vert^{p} ]^{\frac{1}{p}}\leqslant  \delta_{l}^{\frac{3}{2}} \frac{1 + \vert x \vert_{\mathbb{R}^{d}}^{c }     }{(  \HormV_{L}(x) t )^{\eta}} C\exp(C T),
\end{align*}
for every $\lambda \in (0,1)$. In addition we remark that
\begin{align*}
\vert (1-\phi_{\HormV_{L}(x)}(u))Q^{\delta_{l+1},\theta}_{t+\delta_{l},T}f(u) \vert \leqslant   \CstePsiouf_{f}(1+\vert u \vert_{\mathbb{R}^{d}}^{c}) \mathbf{1}_{\vert \HormV_{L}(u,t+\delta_{l}) - \HormV_{L}(x) \vert > \frac{\HormV_{L}(x)}{4}}  C\exp(CT).
\end{align*}
Moreover, we can combine a Taylor-type expansion of $\HormV_{L}$ (more specifically of the Gram matrix $\MatGramV_{L}(x,t)$ of the H\"ormander fields and exploiting that $\lambda_{min}$ is $1$-Lipschitz on the set of symmetric matrices for the norm $\Vert\cdot\Vert_{\mathbb{R}^{d}}$), Markov's inequality and (\ref{eq:Doob_martin_expo_ineq}) with Theorem \ref{theo:Norme_Sobolev_borne}, to derive, for any $p,q\geqslant 1$
\begin{align}
\label{eq:estime_moment_incre_unloc}
\mathbb{E}_{x}[\vert (1+\vert Y^{\lambda}_{t+\delta_{l}}(X^{\delta_{l},\theta}_{t}) \vert_{\mathbb{R}^{d}}^{c}) \mathbf{1}_{\vert \HormV_{L}(Y^{\lambda}_{t+\delta_{l}}(X^{\delta_{l},\theta}_{t}),t+\delta_{l}) - \HormV_{L}(x) \vert > \frac{\HormV_{L}(x)}{4}} \vert^{p} ]^{\frac{1}{p}}  \leqslant (t+\delta_{l})^{\frac{q}{2}} \frac{1 + \vert x \vert_{\mathbb{R}^{d}}^{c }     }{ \HormV_{L}(x) ^{q}} C\exp(C T).
\end{align}
When $\mathbf{A}_{5}^{\delta_{l}}(x,t)$ holds, choosing $q$ sufficiently large and applying the Cauchy-Schwarz inequality yields the required bound. Otherwise, that is if $t=0$ or if $\mathbf{A}_{5}^{\delta_{l}}(x,t)$ fails, we return to the unexpanded complementary term. In both cases, $(t+\delta_{l})\HormV_{L}(x)\leqslant C(1\vee T^{3})^{\frac{1}{26}}\delta_{l}^{\frac{44}{13^{L}91}}$: this is trivial if $t=0$, since $\HormV_{L}(x)\leqslant 1$ and $\delta_{l}\leqslant 1$, and it follows from the failure of $\mathbf{A}_{5}^{\delta_{l}}(x,t)$ otherwise, since $\mathbf{A}_{5}^{\delta_{l}}(x,T)$ holds. The same estimates as in (\ref{eq:estime_moment_incre_unloc}) yield a bound proportional to $(t+\delta_{l})^{\frac q2}\HormV_{L}(x)^{-q}\leqslant C(1\vee T^3)^{\frac q{52}}\delta_{l}^{\frac{22q}{13^{L}91}}\HormV_{L}(x)^{-\frac{3q}{2}}$. Choosing $q$ sufficiently large and increasing $\eta$ gives the required bound of order $\delta_{l}^{\frac32}$, without Malliavin integration by parts.

It remains to study the correction terms. The bound on $ ( Q^{\delta_{l}}_{0,t} -Q^{\delta_{l},\theta}_{0,t}) \Delta^{\delta_{l}}_{t,t+\delta_{l}} Q^{\delta_{l+1}}_{t+\delta_{l},T}f(x)$, $\frac{T}{2} \leqslant t<T$,  is obtained by Theorem \ref{th:regul_main_result_intro} (see (\ref{eq:reg_gauss_distprop_intro})). 
Let $t<\frac{T}{2}$ and recall that $\theta = \frac{3}{2}$. 
As for the main terms, for the two terms involving $1-\phi_{\HormV_{L}(x)}$ below, we assume that $t>0$ and $\mathbf{A}_{5}^{\delta_{l}}(x,t)$ holds. Otherwise, $(t+\delta_{l})\HormV_{L}(x)\leqslant C(1\vee T^{3})^{\frac{1}{26}}\delta_{l}^{\frac{44}{13^{L}91}}$ (trivially if $t=0$), and the required bound follows, after increasing $\eta$, from (\ref{eq:estime_moment_incre_unloc}) applied to the unexpanded terms.


We write $Q^{\delta_{l+1}}_{t+\delta_{l},T}f=Q^{\delta_{l+1},\theta}_{t+\delta_{l},T}f+(Q^{\delta_{l+1}}_{t+\delta_{l},T}-Q^{\delta_{l+1},\theta}_{t+\delta_{l},T})f$ and split both functions with $\phi_{\HormV_{L}(x)}$. We also write $Q^{\delta_{l}}_{0,t}=Q^{\delta_{l},\theta}_{0,t}+(Q^{\delta_{l}}_{0,t}-Q^{\delta_{l},\theta}_{0,t})$ in front of $(1-\phi_{\HormV_{L}(x)})(Q^{\delta_{l+1}}_{t+\delta_{l},T}-Q^{\delta_{l+1},\theta}_{t+\delta_{l},T})f$. The two remaining correction terms then sum up to
\begin{align*}
&Q^{\delta_{l}}_{0,t}\Delta^{\delta_{l}}_{t,t+\delta_{l}}(\phi_{\HormV_{L}(x)}(Q^{\delta_{l+1}}_{t+\delta_{l},T}-Q^{\delta_{l+1},\theta}_{t+\delta_{l},T})f)(x)+(Q^{\delta_{l}}_{0,t}-Q^{\delta_{l},\theta}_{0,t})\Delta^{\delta_{l}}_{t,t+\delta_{l}}(\phi_{\HormV_{L}(x)}Q^{\delta_{l+1},\theta}_{t+\delta_{l},T}f)(x) \\
&+Q^{\delta_{l},\theta}_{0,t}\Delta^{\delta_{l}}_{t,t+\delta_{l}}((1-\phi_{\HormV_{L}(x)})(Q^{\delta_{l+1}}_{t+\delta_{l},T}-Q^{\delta_{l+1},\theta}_{t+\delta_{l},T})f)(x)+(Q^{\delta_{l}}_{0,t}-Q^{\delta_{l},\theta}_{0,t})\Delta^{\delta_{l}}_{t,t+\delta_{l}}((1-\phi_{\HormV_{L}(x)})Q^{\delta_{l+1}}_{t+\delta_{l},T}f)(x).
\end{align*}
For $u$ in the support of $\phi_{\HormV_{L}(x)}$, we have $\HormV_{L}(u,t+\delta_{l})\geqslant\frac{\HormV_{L}(x)}{2}$ and $T-t-\delta_{l}\geqslant\frac{T}{3}$. Hence (\ref{eq:reg_gauss_distprop_intro}), applied with $\delta_{l+1}$ on $[t+\delta_{l},T]$, gives $\vert(Q^{\delta_{l+1}}_{t+\delta_{l},T}-Q^{\delta_{l+1},\theta}_{t+\delta_{l},T})f(u)\vert\leqslant\delta_{l}^{\theta}\CstePsiouf_{f}(1+\vert u\vert_{\mathbb{R}^{d}}^{c})(\HormV_{L}(x)T)^{-\eta}C\exp(CT)$. Moreover, every measurable $g$ with polynomial growth satisfies $\vert\Delta^{\delta_{l}}_{t,t+\delta_{l}}g(u)\vert\leqslant\mathbb{E}[\vert g\vert(X^{\delta_{l}}_{t+\delta_{l}}(u,t))]+\mathbb{E}[\vert g\vert(X^{\delta_{l+1}}_{t+\delta_{l}}(u,t))]$. Lemma \ref{lemme:borne:moments_X} then yields the required bound for the first term.

In the same way, (\ref{eq:borne_semigroupe_regularisation_gauss_main}) with $q=1$, together with (\ref{eq:borne_Sob_flot_tang}) for the derivatives of $X^{\delta_{l}}_{t+\delta_{l}}(u,t)$ and $X^{\delta_{l+1}}_{t+\delta_{l}}(u,t)$ $w.r.t.$ $u$, shows that the gradient of $\Delta^{\delta_{l}}_{t,t+\delta_{l}}(\phi_{\HormV_{L}(x)}Q^{\delta_{l+1},\theta}_{t+\delta_{l},T}f)$ at $u$ is bounded by $\CstePsiouf_{f}(1+\vert u\vert_{\mathbb{R}^{d}}^{c})(\HormV_{L}(x)T)^{-\eta}C\exp(CT)$. Since $\Gaussian$ is independent of $X^{\delta_{l}}_{t}$, every $g\in\mathcal{C}^{1}_{pol}(\mathbb{R}^{d};\mathbb{R})$ satisfies
\begin{align*}
(Q^{\delta_{l}}_{0,t}-Q^{\delta_{l},\theta}_{0,t})g(x)=-\delta_{l}^{\theta}\int_{0}^{1}\mathbb{E}_{x}[\langle\nabla_{x}g(X^{\delta_{l}}_{t}+\lambda\delta_{l}^{\theta}\Gaussian),\Gaussian\rangle_{\mathbb{R}^{d}}]\mbox{d}\lambda.
\end{align*}
Together with Lemma \ref{lemme:borne:moments_X}, this gives the required bound for the second term, without Malliavin integration by parts.

The third term is bounded exactly as the complementary term of $Q^{\delta_{l},\theta}_{0,t}\Delta^{\delta_{l}}_{t,t+\delta_{l}}Q^{\delta_{l+1},\theta}_{t+\delta_{l},T}f(x)$ above. Indeed, the only properties of $(1-\phi_{\HormV_{L}(x)})Q^{\delta_{l+1},\theta}_{t+\delta_{l},T}f$ used there are its smoothness and the bound $\CstePsiouf_{f}(1+\vert u\vert_{\mathbb{R}^{d}}^{c})\mathbf{1}_{\vert\HormV_{L}(u,t+\delta_{l})-\HormV_{L}(x)\vert>\frac{\HormV_{L}(x)}{4}}C\exp(CT)$, and both also hold for $(1-\phi_{\HormV_{L}(x)})(Q^{\delta_{l+1}}_{t+\delta_{l},T}-Q^{\delta_{l+1},\theta}_{t+\delta_{l},T})f$.

The last term requires more care, since $Q^{\delta_{l}}_{0,t}$ involves no Gaussian regularization. We therefore localize with $\Theta^{\ast}_{t}$ before expanding. Let $g=\Delta^{\delta_{l}}_{t,t+\delta_{l}}((1-\phi_{\HormV_{L}(x)})Q^{\delta_{l+1}}_{t+\delta_{l},T}f)$, and let $\Theta^{\ast}_{t}$ be given by (\ref{eq:Theta(delta)}) with $(\delta,T)$ replaced by $(\delta_{l},t)$. 
Since $\Theta^{\ast}_{t}$ only depends on $(\chi^{\delta_{l}}_{w},U^{\delta_{l}}_{w},V^{\delta_{l}}_{w})_{w \in \pi^{\delta_{l},\ast},w\leqslant t}$, which is independent of $\Gaussian$, we may apply the Taylor formula and then (\ref{eq:IPP_Malliavin_degre_superieur}) with $F=X^{\delta_{l}}_{t}$, conditionally on $\Gaussian$. This gives
\begin{align*}
(Q^{\delta_{l}}_{0,t}-Q^{\delta_{l},\theta}_{0,t})g(x)=&-\delta_{l}^{\theta}\sum_{j=1}^{d}\int_{0}^{1}\mathbb{E}_{x}[g(X^{\delta_{l}}_{t}+\lambda\delta_{l}^{\theta}\Gaussian)\langle\Gaussian,\ElemBaseH_{j}\rangle_{\mathbb{R}^{d}}\PoidMal(X^{\delta_{l}}_{t},\Theta^{\ast}_{t})[\ElemBaseH_{j}]]\mbox{d}\lambda \\
&+\mathbb{E}_{x}[(1-\Theta^{\ast}_{t})(g(X^{\delta_{l}}_{t})-g(X^{\delta_{l}}_{t}+\delta_{l}^{\theta}\Gaussian))],
\end{align*}
where $(\ElemBaseH_{j})_{j\in\{1,\ldots,d\}}=\BasemathcalH$ is the canonical basis of $\mathbb{R}^{d}$. We apply (\ref{eq:reg_gen_distprop}) with $(\delta,T,h)$ replaced by $(\delta_{l},t,\frac{3}{2})$, conditionally on $\Gaussian$ for the second contribution. This bounds the term involving $1-\Theta^{\ast}_{t}$ by $\delta_{l}^{\frac{3}{2}}\CstePsiouf_{f}(1+\vert x\vert_{\mathbb{R}^{d}}^{c})\HormV_{L}(x)^{-\eta}C\exp(CT)$. As in the proof of Theorem \ref{th:regularisation} \ref{th:reg_gen_regprop} with $q=1$, we have $\mathbb{E}_{x}[\vert\PoidMal(X^{\delta_{l}}_{t},\Theta^{\ast}_{t})[\ElemBaseH_{j}]\vert^{p}]^{\frac{1}{p}}\leqslant(1+\vert x\vert_{\mathbb{R}^{d}}^{c})(\HormV_{L}(x)t)^{-\eta}C\exp(CT)$ for every $p\geqslant1$. Besides, (\ref{eq:estime_moment_incre_unloc}) still holds when $Y^{\lambda}_{t+\delta_{l}}(X^{\delta_{l},\theta}_{t})$ is replaced by $X^{\delta^{\prime}}_{t+\delta_{l}}(X^{\delta_{l}}_{t}+\lambda\delta_{l}^{\theta}\Gaussian,t)$, $\delta^{\prime}\in\{\delta_{l},\delta_{l+1}\}$. By the Cauchy-Schwarz inequality, applied conditionally on $(X^{\delta_{l}}_{t},\Gaussian)$, it follows that $\mathbb{E}_{x}[\vert g(X^{\delta_{l}}_{t}+\lambda\delta_{l}^{\theta}\Gaussian)\vert^{2}]^{\frac{1}{2}}\leqslant\CstePsiouf_{f}(t+\delta_{l})^{\frac{q}{2}}(1+\vert x\vert_{\mathbb{R}^{d}}^{c})\HormV_{L}(x)^{-q}C\exp(CT)$. The H\"older inequality then gives
\begin{align*}
\vert(Q^{\delta_{l}}_{0,t}-Q^{\delta_{l},\theta}_{0,t})g(x)\vert\leqslant\delta_{l}^{\theta}\CstePsiouf_{f}\frac{(t+\delta_{l})^{\frac{q}{2}}}{t^{\eta}}\frac{1+\vert x\vert_{\mathbb{R}^{d}}^{c}}{\HormV_{L}(x)^{q+\eta}}C\exp(CT)+\delta_{l}^{\frac{3}{2}}\CstePsiouf_{f}\frac{1+\vert x\vert_{\mathbb{R}^{d}}^{c}}{\HormV_{L}(x)^{\eta}}C\exp(CT),
\end{align*}
which is the required bound for $q\geqslant2\eta$, since $t\geqslant\delta_{l}$.

Gathering all the estimates, each of the at most $T\delta_{l}^{-1}$ terms of (\ref{eq:lind_loc}) is bounded by $\delta_{l}^{\frac{3}{2}}\CstePsiouf_{f}(1+\vert x\vert_{\mathbb{R}^{d}}^{c})(\HormV_{L}(x)T)^{-\eta}C\exp(CT)$, so that
\begin{align*}
\vert Q^{\delta_{l}}_{T}f(x)-Q^{\delta_{l+1}}_{T}f(x)\vert\leqslant\delta_{l}^{\frac{1}{2}}\CstePsiouf_{f}\frac{1+\vert x\vert_{\mathbb{R}^{d}}^{c}}{(\HormV_{L}(x)T)^{\eta}}C\exp(CT).
\end{align*}
Moreover, $\sum_{l\in\mathbb{N}}\delta_{l}^{\frac{1}{2}}=(1-2^{-\frac{1}{2}})^{-1}\delta^{\frac{1}{2}}$ and $Q^{\delta_{l}}_{T}f(x)$ converges to $P_{T}f(x)$ as $l$ tends to $\infty$, for $f\in\mathcal{C}^{\infty}_{pol}(\mathbb{R}^{d})$. This follows from the telescoping identity $P_{T}f-Q^{\delta_{l}}_{T}f=\sum_{t\in\pi^{\delta_{l}},t<T}P_{0,t}(P_{t,t+\delta_{l}}-Q^{\delta_{l}}_{t,t+\delta_{l}})Q^{\delta_{l}}_{t+\delta_{l},T}f$. Indeed, by (\ref{eq:borne_Sob_flot_tang}) and Lemma \ref{lemme:borne:moments_X}, $Q^{\delta_{l}}_{t+\delta_{l},T}f\in\mathcal{C}^{3}_{pol}(\mathbb{R}^{d};\mathbb{R})$, with growth constants bounded uniformly in $t$ and $l$, and, by the same moment matching as above, using It\^o’s formula and (\ref{eq:moments_diffusion}) for $P$, $\vert(P_{t,t+\delta_{l}}-Q^{\delta_{l}}_{t,t+\delta_{l}})g(u)\vert\leqslant\delta_{l}^{\frac{3}{2}}C(1+\vert u\vert_{\mathbb{R}^{d}}^{c})$ for $g\in\mathcal{C}^{3}_{pol}(\mathbb{R}^{d};\mathbb{R})$, where $C$ depends on the growth constants of $g$. Summing the $T\delta_{l}^{-1}$ terms and using (\ref{eq:moments_diffusion}) gives $\vert P_{T}f(x)-Q^{\delta_{l}}_{T}f(x)\vert\leqslant\delta_{l}^{\frac{1}{2}}C(1+\vert x\vert_{\mathbb{R}^{d}}^{c})\exp(CT)$, which tends to $0$. We thus obtain (\ref{eq:invariance_main_result_VT}).


As a remark, we observe that the methodology developed in \cite{Bally_Caramellino_Poly} can be adapted to derive a similar result. Nevertheless, since their approach is based on an optimization argument, the resulting convergence rate is, at best, of order $\delta^{\frac{1}{2}-\varepsilon}$ for $\varepsilon>0$ (and the final constant $C$ in the estimate may explode as $\varepsilon$ tends to zero). \\

\VersionCourt{
For part \ref{th:invariance_main_result_density}, since $x \mapsto \HormV_L(x)$ is continuous, $U$ is open and, for every compact $K \subset U$, $\mathbf{A}_5^{\delta_l}(x,T)$ holds uniformly for $x \in K$ and all sufficiently large $l$, so that the previous estimates are uniform for $x\in K$. Approximation (\ref{eq:invariance_main_result_density}) follows from an application of Theorem 2.6 in \cite{Bally_Caramellino_2017} as in the proof of Proposition 2.9 in \cite{Bally_Rey_2016}. Notice that this application is also a reason why the convergence rate is $\delta^{\frac{1}{2}-\epsilon}$ instead of $\delta^{\frac{1}{2}}$.
}

\VersionLong{

Let us prove part \ref{th:invariance_main_result_density}. Let $x \in U$. Recall that $\HormV_{L}(x)=1\wedge\lambda_{min}(\MatGramV_{L}(x,0))$. Since $\mathbf{A}^{\delta}_{1}(+\infty)$ holds, the entries of $\MatGramV_{L}(\cdot,0)$ have polynomially bounded derivatives. Moreover, $\lambda_{min}$ is $1$-Lipschitz on symmetric matrices. Hence, there exist $\mathbf{c}\geqslant 0$ depending only on $L$ and $\ExpPsiouf_{2L+2}$ and $\mathbf{C} \geqslant 0$ depending on $d,N,L$,  $\CstePsiouf_{2L+2}$ and $\ExpPsiouf_{2L+2}$ such that 
\begin{align*}
\vert \HormV_{L}(x')-\HormV_{L}(x)\vert\leqslant  \Vert \MatGramV_{L}(x',0)-\MatGramV_{L}(x,0) \Vert_{\mathbb{R}^{d} } \leqslant \mathbf{C}(1+\vert x \vert_{\mathbb{R}^{d}}^{\mathbf{c}})\vert x'-x\vert_{\mathbb{R}^{d}},
\end{align*}
 for every $x' \in B_{1}(x)$. Now, setting $r(x)=1\wedge\HormV_{L}(x)(2\mathbf{C}(1+\vert x \vert_{\mathbb{R}^{d}}^{\mathbf{c}}))^{-1}$, we have $\HormV_{L}(x') \geqslant \frac{1}{2}\HormV_{L}(x)$ for every $x' \in B_{r(x)}(x)$.  Let $\delta_{l}=\delta 2^{-l}$, $l\in\mathbb{N}$, with the auxiliary processes introduced in the proof of part \ref{th:invariance_main_result_VT}.  There exists $l_{0}$ such that $\mathbf{A}_{5}^{\delta_{l}}(x',T)$ holds for every $x' \in B_{r(x)}(x)$ and every $l \geqslant l_{0}$. Since $\vert x' \vert_{\mathbb{R}^{d}}\leqslant\vert x \vert_{\mathbb{R}^{d}}+1$ and $\HormV_{L}(x')\geqslant\frac{1}{2}\HormV_{L}(x)$ for $x'\in B_{r(x)}(x)$, (\ref{eq:borne_densite_reg_gauss_intro}), (\ref{eq:invariance_main_result_VT}) and (\ref{eq:reg_gauss_distprop_intro}) (with $\delta_{l}^{\theta}\leqslant\delta_{l}^{\frac{1}{2}}$) yield, for $l\geqslant l_{0}$, $x'\in B_{r(x)}(x)$, $y\in\mathbb{R}^{d}$, $\alpha,\beta\in\mathbb{N}^{d}$, $p'>0$ and $f$ as in part \ref{th:invariance_main_result_VT},
\begin{align}
\label{eq:unif_ball}
\begin{split}
\vert\partial^{\alpha}_{x}\partial^{\beta}_{y}q^{\delta_{l},\theta}_{T}(x',y)\vert\leqslant&\frac{(1+\vert x \vert_{\mathbb{R}^{d}}^{c})C\exp(CT)}{(\HormV_{L}(x)T)^{\eta}(1+\vert y-x \vert^{p'}_{\mathbb{R}^{d}})},\\
\vert P_{T}f(x')-Q^{\delta_{l},\theta}_{T}f(x')\vert\leqslant&\delta_{l}^{\frac{1}{2}}\CstePsiouf_{f}\frac{(1+\vert x \vert_{\mathbb{R}^{d}}^{c})C\exp(CT)}{(\HormV_{L}(x)T)^{\eta}},
\end{split}
\end{align}
where $\eta$ also depends on $\vert\alpha\vert+\vert\beta\vert$, and $c,C$ on $\vert\alpha\vert+\vert\beta\vert$, $p'$ and $\ExpPsiouf_{f}$.

Let $q\in\mathbb{N}$, $p\geqslant 0$, $\epsilon\in(0,\frac{1}{2})$ and $m=\lceil\frac{q+2d+1}{4\epsilon}\rceil$. Let $\phi_{\frac{r(x)}{4}}=\funclocgen(\frac{4 \vert \cdot-x \vert_{\mathbb{R}^{d}}}{r(x)}-1) $ with $\funclocgen$ defined in (\ref{eq:def_func_localization_generic}) so that $\phi_{\frac{r(x)}{4}}$ equals $1$ on $B_{\frac{r(x)}{4}}(x)$, vanishes outside $B_{\frac{r(x)}{2}}(x)$ and has derivatives of order $j$ bounded by $Cr(x)^{-j}$. We consider the following measures on $\mathbb{R}^{2d}$, which are finite by (\ref{eq:unif_ball}),
\begin{align*}
\mu(\mbox{d}x',\mbox{d}y)=\phi_{\frac{r(x)}{4}}(x')(1+\vert y-x \vert^{2}_{\mathbb{R}^{d}})^{\frac{p}{2}}P_{T}(x',\mbox{d}y)\mbox{d}x' \quad\mbox{and}\quad\mu_{l}(\mbox{d}x',\mbox{d}y)=h_{l}(x',y)\mbox{d}x'\mbox{d}y,\quad l\geqslant l_{0},
\end{align*}
with $h_{l}(x',y)=\phi_{\frac{r(x)}{4}}(x')(1+\vert y-x\vert^{2}_{\mathbb{R}^{d}})^{\frac{p}{2}}q^{\delta_{l},\theta}_{T}(x',y)$. Let $g \in\mathcal{C}^{\infty}(\mathbb{R}^{2d};\mathbb{R})$ with $\Vert g\Vert_{\infty}\leqslant 1$. Applying the second estimate in (\ref{eq:unif_ball}) to $f=g(x',\cdot)(1+\vert\cdot-x \vert^{2}_{\mathbb{R}^{d}})^{\frac{p}{2}}$, for which $\CstePsiouf_{f}\leqslant C(1+\vert x \vert^{\lceil p\rceil}_{\mathbb{R}^{d}})$ and $\ExpPsiouf_{f}=\lceil p\rceil$, and integrating against $\phi_{\frac{r(x)}{4}}(x')\mbox{d}x'$, we obtain
\begin{align}
\label{eq:dist_loc}
\vert\mu(g)-\mu_{l}(g)\vert\leqslant\delta_{l}^{\frac{1}{2}}\frac{(1+\vert x \vert_{\mathbb{R}^{d}}^{c})C\exp(CT)}{(\HormV_{L}(x)T)^{\eta}}.
\end{align}

Moreover, $h_{l}\in\mathcal{C}^{\infty}(\mathbb{R}^{2d};\mathbb{R})$ and, for $x'\in B_{r(x)}(x)$, $1+\vert(x',y)\vert_{\mathbb{R}^{2d}}\leqslant 2(1+\vert x\vert_{\mathbb{R}^{d}})(1+\vert y-x\vert_{\mathbb{R}^{d}})$. Using the Leibniz rule, the bound on $r(x)^{-1}$ and the first estimate in (\ref{eq:unif_ball}) with $p'=p+2m+d+1$, we get, for $\gamma\in\mathbb{N}^{2d}$ with $\vert\gamma\vert\leqslant 2m+q+1$, $(1+\vert(x',y)\vert_{\mathbb{R}^{2d}})^{2m}\vert\partial^{\gamma}h_{l}(x',y)\vert\leqslant(1+\vert x\vert_{\mathbb{R}^{d}}^{c})(\HormV_{L}(x)T)^{-\eta}C\exp(CT)\mathbf{1}_{B_{r(x)}(x)}(x')(1+\vert y-x\vert_{\mathbb{R}^{d}})^{-d-1}$, so that
\begin{align}
\label{eq:sob_loc}
\sup_{l\geqslant l_{0}}\sum_{\vert\gamma\vert\leqslant 2m+q+1}\Vert(1+\vert\cdot\vert_{\mathbb{R}^{2d}})^{2m}\partial^{\gamma}h_{l}\Vert_{L^{2d+1}(\mathbb{R}^{2d})} \leqslant\frac{(1+\vert x  \vert_{\mathbb{R}^{d}}^{c})C\exp(CT)}{(\HormV_{L}(x)T)^{\eta}}.
\end{align}
We now use Theorem 2.6 (ii) in \cite{Bally_Caramellino_2017} on $\mathbb{R}^{2d}$, with (in the notation therein) $\mathbf{e}:s \in \mathbb{R} \mapsto \vert s\vert^{2d+1}$, $k=0$ and $q$ replaced by $q+1$. After bounding from above the weights appearing there, it asserts the following: if $\nu$ is a finite signed measure on $\mathbb{R}^{2d}$ and $\nu_{n}(\mbox{d}z)=w_{n}(z)\mbox{d}z$, $n\in\mathbb{N}$, with $w_{n}\in\mathcal{C}^{\infty}(\mathbb{R}^{2d};\mathbb{R})$, are such that
\begin{align*}
\sum_{n\in\mathbb{N}}2^{n(q+2d+1)}\sup_{\Vert g \Vert_{\infty}\leqslant 1}\vert\nu(g)-\nu_{n}(g)\vert+2^{-2nm}\sum_{\vert\gamma\vert\leqslant 2m+q+1}\Vert(1+\vert\cdot\vert_{\mathbb{R}^{2d}})^{2m}\partial^{\gamma}w_{n}\Vert_{L^{2d+1}(\mathbb{R}^{2d})}<+\infty,
\end{align*}
then $\nu(\mbox{d}z)=w(z)\mbox{d}z$, and the norm $\Vert w \Vert_{W^{q+1,2d+1}(\mathbb{R}^{2d})}$ of $w$ in the Sobolev space $W^{q+1,2d+1}(\mathbb{R}^{2d})$ is bounded by $C(d,q,m)$ times this sum. Let $l\geqslant l_{0}$ and let $n_{\ast}$ be the smallest element of $\mathbb{N}$ such that $2^{-n_{\ast}(2m+q+2d+1)}\leqslant\delta_{l}^{\frac{1}{2}}$. We apply this result to $\nu=\mu-\mu_{l}$, $\nu_{n}=0$ for $n<n_{\ast}$ and $\nu_{n}=\mu_{l+2(q+2d+2)n}-\mu_{l}$ for $n\geqslant n_{\ast}$. Owing to (\ref{eq:dist_loc}), $\vert\nu(g)-\nu_{n}(g)\vert$ is bounded by $\delta_{l}^{\frac{1}{2}}$ if $n<n_{\ast}$ and by $\delta_{l}^{\frac{1}{2}}2^{-(q+2d+2)n}$ if $n\geqslant n_{\ast}$, up to the factor $(1+\vert x\vert_{\mathbb{R}^{d}}^{c})(\HormV_{L}(x)T)^{-\eta}C\exp(CT)$, and, owing to (\ref{eq:sob_loc}), the weighted Sobolev norms of $w_{n}$ are bounded by this factor. Therefore the above sum is bounded by
\begin{align*}
(\delta_{l}^{\frac{1}{2}}2^{n_{\ast}(q+2d+1)}+\delta_{l}^{\frac{1}{2}}+2^{-2mn_{\ast}} )\frac{(1+\vert x \vert_{\mathbb{R}^{d}}^{c})C\exp(CT)}{(\HormV_{L}(x)T)^{\eta}} \leqslant \delta_{l}^{\frac{1}{2}-\epsilon} \frac{(1+\vert x \vert_{\mathbb{R}^{d}}^{c})C\exp(CT)}{(\HormV_{L}(x)T)^{\eta}}.
\end{align*}
Hence $\mu-\mu_{l}$ has a density $w$ with $\Vert w \Vert_{W^{q+1,2d+1}(\mathbb{R}^{2d})}\leqslant\delta_{l}^{\frac{1}{2}-\epsilon}(1+\vert x \vert_{\mathbb{R}^{d}}^{c})(\HormV_{L}(x)T)^{-\eta}C\exp(CT)$. Since $2d+1>2d$, Morrey's inequality gives $\sup_{\vert\gamma\vert\leqslant q}\Vert\partial^{\gamma} w \Vert_{\infty}\leqslant C\Vert w\Vert_{W^{q+1,2d+1}(\mathbb{R}^{2d})}$. Therefore $\mu$ has a density $h$, which does not depend on $l$, belongs to $\mathcal{C}^{\infty}(\mathbb{R}^{2d};\mathbb{R})$ ($q$ being arbitrary) and satisfies, for every $l\geqslant l_{0}$,
\begin{align}
\label{eq:conv_dens_loc}
\sup_{\vert\gamma\vert\leqslant q}\Vert  \partial^{\gamma} h- \partial^{\gamma}h_{l} \Vert_{\infty}\leqslant\delta_{l}^{\frac{1}{2}-\epsilon}\frac{(1+\vert x \vert_{\mathbb{R}^{d}}^{c})C\exp(CT)}{(\HormV_{L}(x)T)^{\eta}}.
\end{align}

We can now conclude. Let $f$ be bounded and continuous and let $x' \in B_{\frac{r(x)}{4}}(x)$. As $l$ tends to infinity $Q^{\delta_{l},\theta}_{T}f(x')=\int_{\mathbb{R}^{d}}f(y)(1+\vert y-x \vert^{2}_{\mathbb{R}^{d}})^{-\frac{p}{2}}h_{l}(x',y)\mbox{d}y$ converges to $P_{T}f(x')$ by (\ref{eq:unif_ball}), and to the same integral with $h_{l}$ replaced by $h$ owing to (\ref{eq:conv_dens_loc}) and to the dominated convergence theorem (with (\ref{eq:unif_ball}) for $p'=d+1$). Hence $P_{T}(x',\mbox{d}y)=p_{T}(x',y)\mbox{d}y$ with $p_{T}(x',y)=(1+\vert y-x \vert^{2}_{\mathbb{R}^{d}})^{-\frac{p}{2}}h(x',y)$, and $p_{T}\in\mathcal{C}^{\infty}(B_{\frac{r(x)}{4}}(x)\times\mathbb{R}^{d})$. Since $x\in U$ is arbitrary, $p_{T}\in\mathcal{C}^{\infty}(U\times\mathbb{R}^{d})$. Moreover, the derivatives of $y\mapsto(1+\vert y-x \vert^{2}_{\mathbb{R}^{d}})^{-\frac{p}{2}}$ are bounded by $C(1+\vert y-x \vert^{2}_{\mathbb{R}^{d}})^{-\frac{p}{2}}\leqslant 2C(1+\vert y-x\vert^{p}_{\mathbb{R}^{d}})^{-1}$. Thus (\ref{eq:conv_dens_loc}) and (\ref{eq:unif_ball}) (with $l=l_{0}$ for the second estimate below) yield, for $x'\in B_{\frac{r(x)}{4}}(x)$, $y\in\mathbb{R}^{d}$, $\vert\alpha\vert+\vert\beta\vert\leqslant q$ and $l\geqslant l_{0}$,
\begin{align}
\label{eq:local_unif_B}
\begin{split}
\vert\partial^{\alpha}_{x}\partial^{\beta}_{y} p_{T}(x',y)- \partial^{\alpha}_{x}\partial^{\beta}_{y} q^{\delta_{l},\theta}_{T}(x',y)\vert\leqslant&\delta_{l}^{\frac{1}{2}-\epsilon}\frac{(1+\vert x \vert_{\mathbb{R}^{d}}^{c})C\exp(CT)}{(\HormV_{L}(x )T)^{\eta}(1+\vert y-x \vert^{p}_{\mathbb{R}^{d}})},\\
\vert\partial^{\alpha}_{x}\partial^{\beta}_{y}p_{T}(x',y)\vert\leqslant&\frac{(1+\vert x \vert_{\mathbb{R}^{d}}^{c})C\exp(CT)}{(\HormV_{L}(x )T)^{\eta}(1+\vert y-x \vert^{p}_{\mathbb{R}^{d}})}.
\end{split}
\end{align}
Now, assume that $\mathbf{A}_{5}^{\delta}(x,T)$ holds. If $\mathbf{A}_{5}^{\delta}(x',T)$ holds for every $x' \in B_{r(x)}(x)$, one can take $l_{0}=0$ and (\ref{eq:invariance_main_result_density}) is the first estimate in (\ref{eq:local_unif_B}) with $l=0$ and $x'=x$.
Otherwise, there exists $x' \in B_{r(x)}(x)$ such that $\mathbf{A}_{5}^{\delta}(x',T)$ fails. Since $\mathbf{A}_{2}(x',L)$ holds and the only condition in (\ref{hyp:hyp_5_loc_var_eq}) (with $x$ replaced by $x'$) which depends on $x'$ is the second one, we have $\funceta_{1}(\delta)\leqslant 2(4(L+1)\BornsupNl (T\HormV_{L}(x'))^{-1})^{d13^{L}}$. Using $\HormV_{L}(x')\geqslant\frac{1}{2}\HormV_{L}(x)$ and the definition of $\funceta_{1}$ (see (\ref{def:eta_delta})), it follows that $1\leqslant\delta^{\frac{1}{2}}C\exp(CT)(\HormV_{L}(x)T)^{-\frac{13^{L}91}{88}}$. Since (\ref{eq:borne_densite_reg_gauss_intro}) at $x$ and the second estimate in (\ref{eq:local_unif_B}) give 
\begin{align*}
\vert\partial^{\alpha}_{x}\partial^{\beta}_{y}p_{T}(x,y)- \partial^{\alpha}_{x}\partial^{\beta}_{y} q^{\delta,\theta}_{T} (x,y)\vert\leqslant \frac{1+\vert x \vert_{\mathbb{R}^{d}}^{c}}{(\HormV_{L}(x)T)^{\eta}}\frac{C\exp(CT)}{1+\vert y-x\vert^{p}_{\mathbb{R}^{d}}},
\end{align*}
 (\ref{eq:invariance_main_result_density}) follows after increasing $\eta$. These two estimates also give (\ref{eq:invariance_main_result_density}) when $\epsilon\geqslant\frac{1}{2}$. Notice that the loss $\epsilon$ in the rate comes from the interpolation approach and that the same argument, with (\ref{eq:invariance_main_result_VT}) replaced by its improved versions from Remark \ref{rmk:main_theo}, yields the corresponding improved rates.
}

\end{proof}

\VersionLong{
\begin{proof}[Proof of point \ref{rem:higher_order_rate} of Remark \ref{rmk:main_theo}]
We follow the proof of Theorem \ref{th:invariance_main_result} and only indicate the changes. In the proof of part \ref{th:invariance_main_result_VT}, the order of the one-step error only enters through (\ref{eq:boud_term_proof_invariance_principle_dvpt}). The other estimates are of order $\delta_{l}^{\theta}$ (see (\ref{eq:reg_gauss_distprop_intro})), of order $\delta_{l}^{h}$ for an arbitrary $h>0$ (see (\ref{eq:reg_gen_distprop})), or rely on the failure of $\mathbf{A}_{5}$, in which case any power of $\delta_{l}$ is obtained by increasing $\eta$ (and $q$ in (\ref{eq:estime_moment_incre_unloc})). The localization errors due to $\Psi$ must also be added. They are of order $\funceta_{2}(\delta_{l})^{-a}$ for an arbitrary $a>0$, also for their derivatives in $u$, which is why $\Psi$ does not appear in the coefficients $c^{\delta_{l}}_{\alpha}$ below. Let $n=\lceil 2m\rceil+2$, so that $\frac{n}{2}\geqslant m+1>\frac{n-1}{2}$. Taking $\theta=m+1$ instead of $\frac{3}{2}$ in the proof of Theorem \ref{th:invariance_main_result} \ref{th:invariance_main_result_VT}, it is sufficient to write, for $g\in\mathcal{C}^{n}_{pol}(\mathbb{R}^{d};\mathbb{R})$, $\Delta^{\delta_{l}}_{t,t+\delta_{l}}g$ as a finite sum of terms of the same form as in that proof, with $\vert\alpha\vert\leqslant n$, such that (\ref{eq:boud_term_proof_invariance_principle_dvpt}) holds with $\delta_{l}^{\frac{3}{2}}$ and $\vert\cdot\vert_{\mathbb{R},3}$ replaced by $\delta_{l}^{m+1}$ and $\vert\cdot\vert_{\mathbb{R},n}$. Theorem \ref{theoreme:IPP_Malliavin} and (\ref{eq:borne_semigroupe_regularisation_gauss_main}) are then used with derivatives of order at most $n$, which is possible under $\mathbf{A}_{1}^{\delta}(r)$ for $r$ large enough.  Each of the at most $T\delta_{l}^{-1}$ terms of (\ref{eq:lind_loc}) is then bounded by $\delta_{l}^{m+1}$, up to the same factor as before, and $\sum_{l\in\mathbb{N}}\delta_{l}^{m}=(1-2^{-m})^{-1}\delta^{m}$. 
Part \ref{th:invariance_main_result_density} for $m>\frac{1}{2}$ then follows as in the proof of Theorem \ref{th:invariance_main_result}, with $\delta_{l}^{\frac{1}{2}}$ replaced by $\delta_{l}^{m}$ in the definition of $n_{\ast}$ and the integer $\lceil\frac{q+2d+1}{4\epsilon}\rceil$ of that proof replaced by $\lceil\frac{m(q+2d+1)}{2\epsilon}\rceil$.  
Let $t\in\pi^{\delta_{l}}$, $t<T$, and $u\in\mathbb{R}^{d}$. For $l' \in \mathbb{N}$, $l' \geqslant l$ and $1\leqslant\vert\alpha\vert\leqslant n-1$, we set
\begin{align*}
c^{\delta_{l'}}_{\alpha}(t,u)=\frac{1}{\alpha!}\mathbb{E}[(X^{\delta_{l'}}_{t+\delta_{l'}}(u,t)-u)^{\alpha}-(X^{\frac{\delta_{l'}}{2}}_{t+\delta_{l'}}(u,t)-u)^{\alpha}].
\end{align*}
The Taylor formula at order $n$ shows that $\Delta^{\delta_{l}}_{t,t+\delta_{l}}g(u)$ is the sum of $\sum_{1\leqslant\vert\alpha\vert\leqslant n-1}c^{\delta_{l}}_{\alpha}(t,u)\partial^{\alpha}g(u)$ and of a finite sum of terms $\mathbb{E}[\int_{0}^{1}\partial^{\alpha}g(Y^{\lambda}_{t+\delta_{l}}(u))B(u,t,\delta_{l},Y^{1}_{t+\delta_{l}}(u),\lambda)\mbox{d}\lambda]$, $\vert\alpha\vert=n$, with $B$ proportional to $\Psi(1-\lambda)^{n-1}(Y^{1}_{t+\delta_{l}}(u)-u)^{\alpha}$.  Since $\psi(u,t,0,0)=u$, $\mathbf{A}_{1}^{\delta}(+\infty)$ shows that the increments $Y^{1}_{t+\delta_{l}}(u)-u$ and their derivatives in $u$ are of order $\delta_{l}^{\frac{1}{2}}$, so that, using (\ref{eq:borne_Sob_flot_tang}) for $X^{\delta_{l}}_{t}$, the latter terms satisfy the improved version of (\ref{eq:boud_term_proof_invariance_principle_dvpt}). The former ones have the same form, with $Y^{\lambda}_{t+\delta_{l}}(u)=u$ and $B=c^{\delta_{l}}_{\alpha}(t,u)$. By Theorem \ref{theo:Norme_Sobolev_borne}, it is thus sufficient to show that, for every $\gamma\in\mathbb{N}^{d}$, $\vert\partial^{\gamma}_{u}c^{\delta_{l}}_{\alpha}(t,u)\vert\leqslant\delta_{l}^{m+1}C(1+\vert u\vert_{\mathbb{R}^{d}}^{c})$. Writing the transitions with $\psi$, we have
\begin{align*}
c^{\delta_{l'}}_{\alpha}(t,u)=\frac{1}{\alpha!}\mathbb{E}\big[&(\psi(u,t,\delta_{l'}^{\frac{1}{2}}Z^{\delta_{l'}}_{t+\delta_{l'}},\delta_{l'})-u)^{\alpha} -(\psi(\psi(u,t,(\frac{\delta_{l'}}{2})^{\frac{1}{2}}Z^{\frac{\delta_{l'}}{2}}_{t+\frac{\delta_{l'}}{2}},\frac{\delta_{l'}}{2}),t+\frac{\delta_{l'}}{2},(\frac{\delta_{l'}}{2})^{\frac{1}{2}}Z^{\frac{\delta_{l'}}{2}}_{t+\delta_{l'}},\frac{\delta_{l'}}{2})-u)^{\alpha}\big],
\end{align*}
where $x^{\alpha} \coloneqq \prod_{i=1}^{d}(x^{i})^{\alpha^{i}}$ and, since $\delta_{l'}\leqslant\delta_{l}$, the three random variables involved have the law of $Z^{\delta}_{\delta\lceil (t+\delta_{l})/\delta\rceil}$.  
Hence, by Taylor's formula, $\psi(u,t,0,0)=u$, $\mathbf{A}_{1}^{\delta}(+\infty)$ and $\mathbf{A}_{3}^{\delta}(+\infty)$, $c^{\delta_{l'}}_{\alpha}(t,u)$ is the sum of a polynomial in $\delta_{l'}^{\frac{1}{2}}$ of degree at most $n-1$, without constant term and with coefficients which do not depend on $l'$, and of a remainder whose derivatives in $u$ are bounded by $\delta_{l'}^{\frac{n}{2}}C(1+\vert u\vert_{\mathbb{R}^{d}}^{c})\leqslant\delta_{l'}^{m+1}C(1+\vert u\vert_{\mathbb{R}^{d}}^{c})$. It remains to show that this polynomial vanishes. Since its degree is at most $n-1<2m+2$, it vanishes as soon as it is bounded by $\delta_{l'}^{m+1}C(1+\vert u\vert_{\mathbb{R}^{d}}^{c})$ for every $l'\geqslant l$. Otherwise, its non-zero term of lowest degree would dominate it as $l'$ tends to infinity. Since the remainder satisfies this bound, it is sufficient to show that, for every $l'\geqslant l$,
\begin{align*}
\vert c^{\delta_{l'}}_{\alpha}(t,u)\vert\leqslant\delta_{l'}^{m+1}C(1+\vert u\vert_{\mathbb{R}^{d}}^{c}).
\end{align*}
Notice that $\alpha!\,c^{\delta_{l'}}_{\alpha}(t,u)=\Delta^{\delta_{l'}}_{t,t+\delta_{l'}}m_{\alpha}^{u}(u)$, where the growth constants of $m_{\alpha}^{u}:y \mapsto (y-u)^{\alpha}$ are bounded by $C(1+\vert u\vert_{\mathbb{R}^{d}}^{\vert\alpha\vert})$, and that $\Delta^{\delta_{l'}}_{t,t+\delta_{l'}}=(Q^{\delta_{l'}}_{t,t+\delta_{l'}}-P_{t,t+\delta_{l'}})-(Q^{\frac{\delta_{l'}}{2}}_{t,t+\delta_{l'}}-P_{t,t+\delta_{l'}})$ with
\begin{align*}
Q^{\frac{\delta_{l'}}{2}}_{t,t+\delta_{l'}}-P_{t,t+\delta_{l'}}=&(Q^{\frac{\delta_{l'}}{2}}_{t,t+\frac{\delta_{l'}}{2}}-P_{t,t+\frac{\delta_{l'}}{2}})Q^{\frac{\delta_{l'}}{2}}_{t+\frac{\delta_{l'}}{2},t+\delta_{l'}}\\
&+P_{t,t+\frac{\delta_{l'}}{2}}(Q^{\frac{\delta_{l'}}{2}}_{t+\frac{\delta_{l'}}{2},t+\delta_{l'}}-P_{t+\frac{\delta_{l'}}{2},t+\delta_{l'}}).
\end{align*}
The contribution of $Q^{\delta_{l'}}_{t,t+\delta_{l'}}-P_{t,t+\delta_{l'}}$ is bounded by (\ref{eq:short_time_error}) with $\delta'=\delta_{l'}$. By (\ref{eq:borne_Sob_flot_tang}) and Lemma \ref{lemme:borne:moments_X}, $Q^{\frac{\delta_{l'}}{2}}_{t+\frac{\delta_{l'}}{2},t+\delta_{l'}}$ maps $\mathcal{C}^{\infty}_{pol}(\mathbb{R}^{d};\mathbb{R})$ into itself with controlled growth constants, so that the first term of the right-hand side above is bounded by (\ref{eq:short_time_error}) with $\delta'=\frac{\delta_{l'}}{2}$. The second one is bounded by (\ref{eq:short_time_error}) with $\delta'=\frac{\delta_{l'}}{2}$, applied at every point, together with (\ref{eq:moments_diffusion}). This proves the above bound on $c^{\delta_{l'}}_{\alpha}(t,u)$. Hence the polynomial vanishes, $c^{\delta_{l'}}_{\alpha}(t,u)$ reduces to the remainder, and taking $l'=l$ concludes the proof.
\end{proof}
}

\section{Estimates on Sobolev norms and on the Malliavin covariance matrix}
\label{Sec:Proof_reg_prop}

In this section, we establish Theorems \ref{theo:Norme_Sobolev_borne} and \ref{th:borne_Lp_inv_cov_Mal}, which provide the Sobolev and inverse Malliavin covariance estimates needed in the proofs of Theorems \ref{th:regularisation} and \ref{th:regul_main_result_intro}.

We begin by introducing for every $(x,t,z,y) \in \mathbb{R}^{d} \times \pi^{\delta} \times \mathbb{R}^{N} \times [0,1]$ and $i,j \in \EnsN$, 
\begin{align}
\label{def:notation_{d}vpt_X}
& A_{1}^{i}(x,t) =  \partial_{z^{i}} \psi(x,t,0,0) , \quad A_{2}^{i,j}(x,t,z) = \int_{0}^{1} (1-\lambda)  \partial_{z^{i}} \partial_{z^{j}} \psi(x,t,\lambda z,0) \mbox{d} \lambda \\
& A_{3}(x,t,z,y) = \int_{0}^{1}\partial_{y} \psi(x,t,z,\lambda y) \mbox{d} \lambda \nonumber.
\end{align}
A first-order Taylor expansion in the fourth variable, followed by a second-order expansion in the third variable, yields
\begin{align}
\label{eq:representation_psi_A}
 \psi(x,t,z,y)=x+\sum_{i \in \EnsN}z^{i} A_{1}^{i}(x,t)+\sum_{i,j \in \EnsN}z^{i} z^{j} A_{2}^{i,j}(x,t,z)  + yA_{3}(x,t,z,y) .
\end{align}
We will also denote $A_{1} :=   (A_{1}^{i})_{i \in \EnsN}$ and $A_{2} :=   (A_{2}^{i,j})_{(i,j) \in \EnsN^{2}}$. Before we treat the Sobolev norms of $X^{\delta}$ and $\OpOrnsteinUlh X^{\delta}$ we establish some preliminary results. The first one gives an estimate of the Sobolev norms of $\OpOrnsteinUlh Z^{\delta}$.

\begin{lemme}
\label{lemme:estim_LZ}
Let $t \in \pi^{\delta,\ast}$. For every $i \in \EnsN$, we have $\mathbb{E}[\OpOrnsteinUlh Z^{\delta,i}_{t}]=0$. Moreover, for every $q\in \mathbb{N}$ and $p\geqslant 1$, 
\begin{align}
\Vert \OpOrnsteinUlh Z^{\delta}_{t} \Vert _{\mathbb{R}^{N}  ,q,p}\leqslant  \frac{C(N,p,q) m_{\ast}^{\frac{1}{p}} }{1 \wedge r_{\ast}^{q+1}} .\label{eq:deriv_LZ}
\end{align}
\end{lemme}

\begin{proof}
 Using the duality relation (\ref{formule_{d}ualite}) with $F=Z^{\delta,i}_{t}$, $G=1$ and $\mathcal{H}=\mathbb{R}$, we obtain $ \mathbb{E}[
\OpOrnsteinUlh Z^{\delta,i}_{t}]=-\delta \sum_{(w,j) \in  \pi^{\delta,\ast}  \times \EnsN}\mathbb{E}[\DeriveMal_{(w,j)}1 \DeriveMal_{(w,j)} Z^{\delta,i}_{t}]=0.$
Let us prove (\ref{eq:deriv_LZ}). For $\alpha =(\alpha^{1},\ldots,\alpha^{q})$ with $\alpha^{j}=(t_{j},i_{j})$, $t_{j} \in \pi^{\delta}, t_{j}>0$, $i_{j} \in \EnsN$, it follows from (\ref{eq:L_Zkj}) that
\begin{align*}
\DeriveMal _{\alpha} \OpOrnsteinUlh Z^{\delta,i}_{t}=\delta^{-\frac{\Vert \alpha \Vert}{2}} \chi^{\delta}_{t} \partial_{z}^{\alpha^{z}_{i}} \ln \varphi_{r_{\ast }/2} (\delta^{-\frac{1}{2}} U^{\delta}_{t}-z_{\ast ,t})  \mathbf{1}_{ \cap_{j=1}^q \{t = t_j\}} ,
\end{align*}
with $\alpha^{z}_{i} :=   ((\alpha^{z}_{i})^{j})_{j \in \EnsN}$, $(\alpha^{z}_{i})^{j}= \mathbf{1}_{i=j}+\sum_{l=1}^{q} \mathbf{1}_{i_{l}=j}$.  In particular,
\begin{align*}
\sum_{\underset{j \leqslant q}{\alpha \in (\pi^{\delta,\ast}  \times \EnsN})^{j}}  \delta^{j} \vert   \DeriveMal _{\alpha}\OpOrnsteinUlh Z^{\delta}_{t} \vert_{\mathbb{R}^{N}}^{2} \leqslant & C\chi^{\delta}_{t} \sum_{\underset{\vert \alpha^{z} \vert \in \{1,\ldots, q+1\}}{\alpha^{z} \in \mathbb{N}^{N}}}  \vert \partial_{z}^{\alpha^{z}}  \ln \varphi_{r_{\ast }/2} ( \delta^{-\frac{1}{2}} U^{\delta}_{t}-z_{\ast ,t})  \vert^{2} .
\end{align*}
Since the function $\varphi _{r_{\ast }/2}$ is constant on $ B_{r_{\ast}/2}(0)$ and on $\mathbb{R}^{N} \setminus \bar{B}_{r_{\ast}}(0) $, using (\ref{eq:borne_{d}eriv_log_reg_Zk}) with $v=\frac{r_{\ast}}{2}$, we obtain 
\begin{align*}
\mathbb{E}[ \vert & \sum_{\underset{j \leqslant q}{\alpha \in (\pi^{\delta,\ast} \times \EnsN})^{j}}    \delta^{j} \vert  \DeriveMal_{\alpha}\OpOrnsteinUlh Z^{\delta}_{t} \vert_{\mathbb{R}^{N}}^{2} \vert^{\frac{p}{2}}] \\
\leqslant &  C \frac{\varepsilon_{\ast}  \mathbb{E}\left[ \vert \chi^{\delta}_{t}  \vert^{p} \right]}{m_{\ast}}\int_{ \mathbb{R}^{N}} \delta^{-\frac{N}{2}}  \vert  \sum_{\underset{\vert \alpha^{z} \vert \in \{1,\ldots, q+1\}}{\alpha^{z} \in \mathbb{N}^{N}}} \vert \partial_{z}^{\alpha^{z}}  \ln \varphi_{\frac{r_{\ast }}{2}} ( \delta^{-\frac{1}{2}} u-z_{\ast ,t})  \vert^{2} \vert^{\frac{p}{2}}  \varphi _{\frac{r_{\ast }}{2}}( \delta^{-\frac{1}{2}} u-z_{\ast ,t}) du \\
\leqslant & C   \varepsilon_{\ast} \int_{r_{\ast}/2 \leqslant \vert u \vert  \leqslant r_{\ast}}  \vert \sum_{\underset{\vert \alpha^{z} \vert \in \{1,\ldots, q+1\}}{\alpha^{z} \in \mathbb{N}^{N}}}   \vert \partial_{z}^{\alpha^{z}}  \ln \varphi_{\frac{r_{\ast }}{2}} (u)  \vert^{2} \vert^{\frac{p}{2}} \varphi _{\frac{r_{\ast }}{2}}(u) du 
 \leqslant  \frac{C  \varepsilon_{\ast}  \vert \pi^{\frac{1}{2}} r_{\ast} \vert^{N}}{ 1 \wedge r_{\ast}^{p(q+1)}} .
\end{align*}
In order to derive (\ref{eq:deriv_LZ}), we observe that $m_{\ast} \geqslant  \varepsilon_{\ast} \lambda_{\mathscr{L}}(B_{\frac{r_{\ast}}{2} }(0))$ so that $  \varepsilon_{\ast} \vert \pi^{\frac{1}{2}} \frac{r_{\ast}}{2} \vert^{N} \leqslant C m_{\ast}$.
\end{proof}

Before we estimate the Sobolev norms, we recall the Burkholder inequality for Hilbert-space-valued martingales.  We consider a separable Hilbert space $\mathcal{H}$ and denote $\vert
.\vert _{\mathcal{H}}$ the norm of $\mathcal{H}$.  
 Then, for a martingale $M_{n}\in \mathcal{H}, \;n\in \mathbb{N}$ with $M_{0}=0$, the Burkholder inequality gives: For each $p\geqslant 1$,
there exists a constant $\CsteBDH_{p}\geqslant 1$ such that%
\begin{align}
\forall n \in  \mathbb{N}, \quad \Vert \sup_{k \in \{0,\ldots,n\}}   \vert M_{k} \vert_{\mathcal{H}} \Vert _{p}\leqslant \CsteBDH_{p}\mathbb{E}[(\sum_{k=1}^{n}\vert
M_{k}-M_{k-1}\vert _{\mathcal{H}}^{2})^{\frac{p}{2}}]^{\frac{1}{p}}.  \label{eq:burkholder_inequality}
\end{align}

As an immediate consequence, for $p\geqslant 2$,
\begin{align}
 \Vert \sup_{k \in \{0,\ldots,n\}}   \vert M_{k} \vert_{\mathcal{H}} \Vert _{p} \leqslant \CsteBDH_{p} \vert \sum_{k=1}^{n}\Vert
M_{k}-M_{k-1}\Vert _{\mathcal{H},p}^{2} \vert^{\frac{1}{2}}.  \label{eq:burholder_inequality_bis}
\end{align}
Now, we establish a bound on the moments of $(X^{\delta}_{t})_{t \in \pi^{\delta}}$.

\begin{lemme}
\label{lemme:borne:moments_X}
Let $T\in \pi^{\delta,\ast}$, $\EnsT =(0,T] \cap \pi^{\delta}$, $x \in \mathbb{R}^{d}$ and $p \geqslant 1$.  Assume that $\mathbf{A}_{1}^{\delta}(2)$ and $\mathbf{A}_{3}^{\delta}((\ExpPsiouf_{3}  +2) (p \vee 2))$ hold.  Then
\begin{align}
\label{eq:borne_moment_X}
\mathbb{E}_{x}[\sup_{t \in \EnsT }\vert X^{\delta}_{t} \vert_{\mathbb{R}^{d}}^{p} ]^{\frac{1}{p}} \leqslant &  (1 + \vert x \vert_{\mathbb{R}^{d}} )  \exp(C(p) (1+T) \MomZ_{(\ExpPsiouf_{3}  +2) (p \vee 2)}(Z^{\delta})^2 \CstePsiouf_{3}^{2}  ) .
\end{align}

\end{lemme}

\begin{proof}

Let us begin by proving the result for $p \geqslant 2$. In a first step, we show that, for every $(x,t,z,y) \in \mathbb{R}^{d}\times \pi^{\delta} \times \mathbb{R}^{N} \times [0,1]$,
\begin{align}
\label{eq:borne_deriv_moment_X}
\{ \vert  \partial_{y}\psi \vert_{\mathbb{R}^{d}}+& \sum_{i \in \EnsN}  	\vert  \partial_{z^{i}} \psi \vert_{\mathbb{R}^{d}} +\sum_{i,j \in \EnsN} \vert  \partial_{z^{i}} \partial_{z^{j}} \psi \vert_{\mathbb{R}^{d}}  \} (x,t,z,y) \leqslant \mathbf{C}(x,z,\delta).
\end{align}
with $\mathbf{C}(x,z,\delta)=2 \CstePsiouf_{2}(2+ \delta^{-\frac{\ExpPsiouf  _{2}}{2}} \vert z \vert_{\mathbb{R}^{N}}^{\ExpPsiouf  _{2}}) + \CstePsiouf_{3}  \vert x  \vert_{\mathbb{R}^{d}}  (1+ \delta^{-\frac{\ExpPsiouf_{3} }{2}} \vert z \vert_{\mathbb{R}^{N}}^{\ExpPsiouf_{3}  })$. 
To prove (\ref{eq:borne_deriv_moment_X}), we simply combine the following two observations. First, using a first-order Taylor expansion, we write
\begin{align*}
\partial_{y}\psi (x,t,z,y)= & \partial_{y}\psi (0,t,z,y)+ \sum_{l=1}^{d} x^{l} \int_{0}^{1} \partial_{x^{l}}\partial_{y}\psi (\lambda x,t,z,y) \mbox{d} \lambda ,
\end{align*}
with similar formulas for the derivatives $w.r.t.$ $z$. In addition, it follows from assumption $\mathbf{A}_{1}^{\delta}(2)$,  (\ref{eq:hyp_1_Norme_adhoc_fonction_schema}) that 
\begin{align*}
\{ \vert  \partial_{y}\psi \vert_{\mathbb{R}^{d}}+\sum_{i \in \EnsN} 	\vert  \partial_{z^{i}} \psi \vert_{\mathbb{R}^{d}} +\sum_{i,j \in \EnsN} \vert  \partial_{z^{i}} \partial_{z^{j}} \psi \vert_{\mathbb{R}^{d}}  \} (0,t,z,y) 
\leqslant & 2 \CstePsiouf_{2}(2+ \delta^{-\frac{\ExpPsiouf  _{2}}{2}} \vert z \vert_{\mathbb{R}^{N}}^{\ExpPsiouf  _{2}}).
\end{align*}
We bound the integral remainder using (\ref{eq:hyp_3_Norme_adhoc_fonction_schema}) which is part of
$\mathbf{A}_{1}^{\delta}(2)$. Gathering these estimates yields (\ref{eq:borne_deriv_moment_X}). Moreover, owing again to the Taylor expansion at order two applied to the function $\mathbf{m}_{p}:x \in \mathbb{R}^{d} \mapsto \vert x \vert_{\mathbb{R}^{d}}^{p}$, we rely on the following decomposition
\begin{align*}
\vert X^{\delta}_{t} \vert_{\mathbb{R}^{d}}^{p}-&\vert X^{\delta}_{t-\delta} \vert_{\mathbb{R}^{d}}^{p}=\sum_{i=1}^{d}  (X^{\delta,i}_{t} - X^{\delta,i}_{t-\delta} - \mathbb{E}[X^{\delta,i}_{t} - X^{\delta,i}_{t-\delta}  \vert X^{\delta}_{t-\delta} ])  \partial_{x_{i}}\mathbf {m}_{p} (X^{\delta}_{t-\delta} )\\
&+ \sum_{i=1}^{d}  \mathbb{E}[(X^{\delta,i}_{t} - X^{\delta,i}_{t-\delta})  \partial_{x_{i}}\mathbf {m}_{p} (X^{\delta}_{t-\delta} )  \vert X^{\delta}_{t-\delta} ]\\
&+\sum_{i,j=1}^{d} (X^{\delta}_{t} - X^{\delta}_{t-\delta})_{i } (X^{\delta}_{t} - X^{\delta}_{t-\delta})_{j}  \int_{0}^{1} (1-\lambda) \partial_{x_{i},x_{j}}\mathbf {m}_{p} (X^{\delta}_{t-\delta}+ \lambda 
	(X^{\delta}_{t} - X^{\delta}_{t-\delta}) ) \mbox{d} \lambda ,
\end{align*}
To estimate the first term on the $r.h.s.$ above, we use the fact that $Z^{\delta}_{t}$ is centered, the decomposition (\ref{eq:representation_psi_A}), the inequalities (\ref{eq:burkholder_inequality}) and (\ref{eq:borne_deriv_moment_X}), and the Cauchy-Schwarz inequality. Since $\CstePsiouf_{3} \geqslant \CstePsiouf_{2} \geqslant 1$ and $\ExpPsiouf_{3} \geqslant \ExpPsiouf_{2}$, we obtain
\begin{align*}
\mathbb{E}_{x}[&\sup_{t \in \EnsT } \vert  \sum_{w \in \EnsT;w \leqslant t} \sum_{i=1}^{d} (X^{\delta,i}_{w} - X^{\delta,i}_{w-\delta} - \mathbb{E}[X^{\delta,i}_{w} - X^{\delta,i}_{w-\delta}  \vert X^{\delta}_{w-\delta} ])  \partial_{x_{i}}\mathbf {m}_{p} (X^{\delta}_{w-\delta} ) \vert] \\
\leqslant & \CsteBDH_1 \mathbb{E}_{x} [  ( \sum_{t \in \EnsT}  \vert \sum_{i=1}^{d}(X^{\delta,i}_{t} - X^{\delta,i}_{t-\delta} - \mathbb{E}[X^{\delta,i}_{t} - X^{\delta,i}_{t-\delta}  \vert X^{\delta}_{t-\delta} ])  \partial_{x_{i}}  \mathbf {m}_{p} (X^{\delta}_{t-\delta} ) \vert^2 ) ^{\frac{1}{2}} ] \\
\leqslant &  C(p)\mathbb{E}_{x}[ (\delta \sum_{t \in \EnsT} (\vert X^{\delta}_{t-\delta} \vert_{\mathbb{R}^{d}}^{p-1}(\mathbf{C} (X^{\delta}_{t-\delta}, \delta^{\frac{1}{2}}Z^{\delta}_{t},\delta)(1+\vert Z^{\delta}_{t} \vert_{\mathbb{R}^{N}}^{2})+\mathbb{E}_{x}[\mathbf{C} (X^{\delta}_{t-\delta}, \delta^{\frac{1}{2}}Z^{\delta}_{t},\delta)(1+\vert Z^{\delta}_{t} \vert_{\mathbb{R}^{N}}^{2}) \vert X^{\delta}_{t-\delta}] ))^{2} )^{\frac{1}{2}} ]\\
\leqslant &   C(p) \MomZ_{2(\ExpPsiouf_{3}  +2) }(Z^{\delta})^{\frac{1}{2}}\CstePsiouf_{3}  \mathbb{E}_{x}[1+\sup_{t \in \EnsT \cup \{0\}}\vert X^{\delta}_{t} \vert_{\mathbb{R}^{d}}^{p}]^{\frac{1}{2}}   \mathbb{E}_{x}[\delta \sum_{t \in \EnsT} 1+\vert X^{\delta}_{t-\delta} \vert_{\mathbb{R}^{d}}^{p}   ]^{\frac{1}{2}}  \\
 \leqslant &  \frac{1}{2}  \mathbb{E}_{x}[\sup_{t \in \EnsT}\vert X^{\delta}_{t} \vert_{\mathbb{R}^{d}}^{p}] +1+\vert x  \vert_{\mathbb{R}^{d}}^{p}+C(p) \MomZ_{2(\ExpPsiouf_{3}  +2) }(Z^{\delta})\CstePsiouf_{3}^2   \mathbb{E}_{x}[\delta \sum_{t \in \EnsT} 1+\vert X^{\delta}_{t-\delta} \vert_{\mathbb{R}^{d}}^{p}   ]  ,
\end{align*}
where we have used Young's inequality for the last estimate.
To estimate the second term of the $r.h.s.$ above, we consider the decomposition (\ref{eq:representation_psi_A}) together with (\ref{eq:borne_deriv_moment_X}) and
obtain 
\begin{align*}
\vert \mathbb{E}[ \sum_{i=1}^{d}  (X^{\delta,i}_{t} - X^{\delta,i}_{t-\delta})  \partial_{x_{i}}\mathbf {m}_{p} (X^{\delta}_{t-\delta})  \vert X^{\delta}_{t-\delta}] \vert  \leqslant & C(p) \delta \vert X^{\delta}_{t-\delta} \vert_{\mathbb{R}^{d}}^{p-1}\mathbb{E} [\mathbf{C} (X^{\delta}_{t-\delta},\delta^{\frac{1}{2}} Z^{\delta}_{t},\delta)(1+\vert Z^{\delta}_{t} \vert_{\mathbb{R}^{N}}^{2}) \vert X^{\delta}_{t-\delta}]  \\
\leqslant & C(p) \delta \MomZ_{\ExpPsiouf_{3}  +2 }(Z^{\delta})\CstePsiouf_{3} (1+ \vert X^{\delta}_{t-\delta} \vert_{\mathbb{R}^{d}}^{p}).
\end{align*}
To bound the third term, we apply similar arguments.
After computations, we finally obtain for $p\geqslant 2$,
\begin{align*}
\mathbb{E}_{x}[\sup_{t \in \EnsT}\vert X^{\delta}_{t} \vert_{\mathbb{R}^{d}}^{p}  ]  \leqslant  & C(p) \MomZ_{p(\ExpPsiouf_{3}  +2) }(Z^{\delta}) \CstePsiouf_{3}^2(1+\vert x  \vert_{\mathbb{R}^{d}}^{p}+T+\delta \sum_{t \in \EnsT} \mathbb{E}_{x}[ \vert X^{\delta}_{t-\delta} \vert_{\mathbb{R}^{d}}^{p}  ])  \\
& +  C(p) \MomZ_{p(\ExpPsiouf_{3}  +2) }(Z^{\delta}) \CstePsiouf_{3}^{p}	(1+T +\delta^{\frac{p}{2}}	 \sum_{t \in \EnsT} \mathbb{E}_{x}[ \vert X^{\delta}_{t-\delta} \vert_{\mathbb{R}^{d}}^{p}  ]) .
\end{align*}

%
%

Set $a=\delta \MomZ_{p(\ExpPsiouf_{3}+2)}(Z^{\delta})^{2}\CstePsiouf_{3}^{2}$.
Since $\MomZ_{p(\ExpPsiouf_{3}+2)}(Z^{\delta}),\CstePsiouf_{3}\geqslant 1$, the coefficients of the discrete sums above are bounded by $C(p)(a+a^{\frac{p}{2}})$.
The discrete Gronwall lemma yields
\begin{align*}
\mathbb{E}_{x}[\sup_{t\in\EnsT}\vert X_t^{\delta}\vert_{\mathbb{R}^{d}}^{p} ]
\leqslant & C(p)(1+\vert x\vert_{\mathbb{R}^{d}}^{p}+T) \MomZ_{p(\ExpPsiouf_{3}+2)}(Z^{\delta}) (\CstePsiouf_{3}^{2}+\CstePsiouf_{3}^{p}) (1+C(p)(a+a^{\frac{p}{2}}) )^{\frac{T}{\delta}} \\ 
\leqslant & (1+\vert x\vert_{\mathbb{R}^{d}})^{p} \exp (C(p)(1+T) \MomZ_{p(\ExpPsiouf_{3}+2)}(Z^{\delta})^{2} \CstePsiouf_{3}^{2} ).
\end{align*}
Indeed, $\log(1+C(p)(a+a^{\frac{p}{2}}))\leqslant C(p)a$ for every $a\geqslant 0$, and the remaining polynomial prefactor can be absorbed into the exponential. This proves (\ref{eq:borne_moment_X}). For $p \in [1,2)$, it simply remains to use the Cauchy-Schwarz inequality. 
%
%

\end{proof}

In order to obtain estimates of the Sobolev norms which appear in Theorem \ref{theo:Norme_Sobolev_borne}, we derive some estimates for a generic class of processes which involves the Malliavin derivatives of $\partial_{\mbox{\textsc{x}}^{\delta}_{0}}^{\alpha }X^{\delta}$ and $\OpOrnsteinUlh X^{\delta}_{t}$.  We first recall that for $t \in \pi^{\delta}$, $X^{\delta}_{t+\delta} 
=\psi(X^{\delta}_{t},t,\delta^{\frac{1}{2}}Z^{\delta}_{t+\delta},\delta) $ with $\psi$ satisfying the decomposition (\ref{eq:representation_psi_A}).
%
We then introduce the $\mathbb{R}^{d \times d}$-valued process $(B_{t})_{t \in \pi^{\delta}}$ such that for every $t \in \pi^{\delta}$,
\begin{align*}
B_{t}=  & \delta^{\frac{1}{2}} \sum_{i \in \EnsN}  Z^{\delta,i}_{t+\delta} \nabla_{x} A_{1}^{i}(X^{\delta}_{t}, t)  +  \delta  \sum_{i,j \in \EnsN}  Z^{\delta,i}_{t+\delta} Z^{\delta,j}_{t+\delta} \nabla_{x} A_{2}^{i,j}(X^{\delta}_{t},t,\delta^{\frac{1}{2}}Z^{\delta}_{t+\delta}) + \delta \nabla_{x} A_{3}(X^{\delta}_{t},t,\delta^{\frac{1}{2}}Z^{\delta}_{t+\delta},\delta).
\end{align*}

We now consider a finite-dimensional Hilbert space $\mathcal{H}$ and introduce some $\mathcal{H}^{d}$-valued processes $
(B^{1,i}_{t})_{t \in \pi^{\delta}}$, $(B^{2,i}_{t})_{t \in \pi^{\delta}}$ which are both adapted to the filtration $(\mathcal{F}^{\delta}_{t})_{t \in \pi^{\delta}}:=(\sigma(X^{\delta}_{0},U^{\delta}_{\delta},V^{\delta}_{\delta},\chi^{\delta}_{\delta},\ldots,U^{\delta}_{t},V^{\delta}_{t},\chi^{\delta}_{t}))_{t \in \pi^{\delta}}$ and $ (B^{3}_{t})_{t\in \pi^{\delta}}$ which is adapted to the filtration $(\mathcal{F}^{\delta,+}_t)_{t \in \pi^{\delta}}:=(\sigma(X^{\delta}_{0},U^{\delta}_{\delta},V^{\delta}_{\delta},\chi^{\delta}_{\delta},\ldots,U^{\delta}_{t+\delta},V^{\delta}_{t+\delta},\chi^{\delta}_{t+\delta}))_{t \in \pi^{\delta}}$.
To establish the estimates on Sobolev norms, we will consider a $\mathcal{H}^{d}$-valued generic process $(Y_{t})_{t \in \pi^{\delta}}$, adapted to $(\mathcal{F}^{\delta}_{t})_{t \in \pi^{\delta}}$, which satisfies, for every $t \in \pi^{\delta}$,
\begin{align}
\label{eq:processus_general_Y_sobolev}
Y_{t+ \delta}=  & Y_{t}+B_{t}Y_{t}+ \delta^{\frac{1}{2}} \sum_{i \in \EnsN}  Z^{\delta,i}_{t+\delta} B^{1,i}_{t}  +  \delta^{\frac{1}{2}}   \sum_{i \in \EnsN}  \OpOrnsteinUlh Z^{\delta,i}_{t+\delta} B^{2,i}_{t} +B^{3}_{t}.
\end{align}
Moreover for $q\in \mathbb{N}$ and $p\geqslant 1$,  $\EnsT =(0,T] \cap \pi^{\delta}$ with $T>0$, $B^{1,i}_{t-\delta},B^{2,i}_{t-\delta},B^{3}_{t-\delta} \in  \mathbb{D}^{q,p}(\mathcal{H}^{d})$, $t \in \EnsT$, $i \in \EnsN$, we denote
\begin{align*}
\mathscr{C}_{\mathcal{H}^{d},q,p} & (B^{1},B^{2},B^{3})= 1  + \sup_{t \in \EnsT }(\Vert
B^{1,.}_{t-\delta} \Vert _{(\mathcal{H}^{d})^{\EnsN} ,q,p}+\Vert
B^{2,.}_{t-\delta} \Vert _{(\mathcal{H}^{d})^{\EnsN} ,q,p})+ \Vert \sup_{t \in \EnsT} \vert \sum_{w\in \pi^{\delta},w<t}  B^{3}_{w}  \vert _{\mathcal{H}^{d} ,q} \Vert_{p}   \nonumber
\end{align*}
where for $(B(i,l))_{(i,l) \in \EnsN \times \{1,\ldots,d\}}$ taking values in $\mathcal{H}$, $\vert B \vert_{(\mathcal{H}^{d})^{\EnsN}}= \vert \sum_{i \in \EnsN} \sum_{l=1}^{d} \vert B(i,l) \vert_{\mathcal{H}}^{2} \vert^{\frac{1}{2}}$. We denote $\mathscr{C}^{x}_{\mathcal{H}^{d},q,p}  (B^{1},B^{2},B^{3})$ when the expectations in the norm above are conditional to $X^{\delta}_{0}=X_{0}=x$.  
This first result below gives an estimate of the Sobolev norms of
$(X^{\delta}_{t})_{t \in \EnsT }$ and $(Y_{t})_{t \in \EnsT }$ w.r.t. the quantities above.
The proof of Lemma \ref{lemme:borne:moments_X} illustrates the basic argument used below:
a martingale--drift decomposition combined with Burkholder's inequality and a discrete
Gronwall argument. Proposition \ref{prop:borne_Sob_generique} extends this strategy to
the Malliavin derivatives and to the generic processes satisfying
(\ref{eq:processus_general_Y_sobolev}).

\begin{proposition}
\label{prop:borne_Sob_generique}
Let $T\in \pi^{\delta,\ast}$, $\EnsT =(0,T] \cap \pi^{\delta}$ and $x \in \mathbb{R}^{d}$.  Let $q\in \mathbb{N}$ and $p\geqslant 2$. Assume that $\mathbf{A}_{1}^{\delta}(q+2)$ (see (\ref{eq:hyp_1_Norme_adhoc_fonction_schema}) and (\ref{eq:hyp_3_Norme_adhoc_fonction_schema})),  $\mathbf{A}_{3}^{\delta}(+\infty)$ (see (\ref{eq:hyp:moment_borne_Z})) and $\mathbf{A}_{4}^{\delta}$ (see (\ref{hyp:lebesgue_bounded})) hold. Then, when $q \geqslant 1$,

\begin{align}
\label{eq:norme_Sobolev_X_theo}
\mathbb{E}_{x}[\sup_{t \in \EnsT } \vert  X^{\delta}_{t}  \vert_{\mathbb{R}^{d},1,q}^{p}]^{\frac{1}{p}} \leqslant  &  (1+\mathbf{1}_{\ExpPsiouf  _{q+2}> 0} \vert x \vert_{\mathbb{R}^{d}}  ^{C}) \CstePsiouf_{q+2}^{C} \exp (C(1+T)\MomZ_{C}(Z^{\delta}) \CstePsiouf_{3} ^{2})  .
\end{align}
with $C=C(d,\ExpPsiouf  _{q+2},q,p)$.
Moreover, let $(Y_{t})_{t \in \pi^{\delta}}$ satisfy (\ref{eq:processus_general_Y_sobolev}). Assume in addition that $\mathbf{A}_{1}^{\delta}(q+3)$ holds and that $Y_{0},B^{1,i}_{t-\delta},B^{2,i}_{t-\delta},B^{3}_{t-\delta} \in  \mathbb{D}^{q,2^{q}p}(\mathcal{H}^{d})$, $t \in \EnsT$, $i \in \EnsN$. Then
\begin{align}
\label{eq:borne_norme_sobolev_generique}
 \mathbb{E}_{x}[\sup_{t \in \EnsT }  \vert   Y_{t}  \vert_{\mathcal{H}^{d} ,q}^{p}]^{\frac{1}{p}}  \leqslant  &  ( \mathbb{E}_{x}[\vert   Y_{0}  \vert_{\mathcal{H}^{d} ,q}^{2^{q} p}]^{\frac{1}{2^{q} p}}  + \mathscr{C}^{x}_{\mathcal{H}^{d},q,2^{q}p}(B^{1},B^{2},B^{3})   )\\
 & \times (1+\mathbf{1}_{\ExpPsiouf  _{q+3}> 0}  \vert x \vert_{\mathbb{R}^{d}} ^{C} )\CstePsiouf_{q+3} ^{C}  \exp (C(1+T)\MomZ_{C}(Z^{\delta}) \CstePsiouf_{3} ^{2})   .  \nonumber
\end{align}
with $C=C(d,N,\frac{1}{r_{\ast}} ,\ExpPsiouf  _{q+3},q,p)$.
\end{proposition}

\begin{proof}
\VersionCourt{
The proof follows the same martingale--drift decomposition and Gronwall argument
as in Lemma~\ref{lemme:borne:moments_X}, applied recursively to the Malliavin
derivatives. The additional terms are handled as in Proposition 6.1 of
\cite{Bally_Rey_2016}. We omit the details.
}
\VersionLong{
Our strategy is to derive estimates suitable for applying Gronwall's lemma, first for the generic process $Y$ (\textbf{Step 1}), then for $X$ and its Malliavin derivatives (\textbf{Step 2}), and finally for $Y$ and its Malliavin
derivative (\textbf{Step 3}).

\textbf{Step 1.} We begin by proving (\ref{eq:borne_norme_sobolev_generique}) in the specific case $q=0$.
We aim to use the Gronwall lemma so we study the moments of the terms which appear in the $r.h.s.$ of (\ref{eq:processus_general_Y_sobolev}). We consider $i,j \in \EnsN$. Notice that for every $t \in \pi^{\delta}$,
$\mathbb{E}[\OpOrnsteinUlh Z^{\delta,i}_{t+\delta}]=0$ (see Lemma \ref{lemme:estim_LZ}) and $B^{2,i}_{t} $ is $\mathcal{F}^{\delta}_{t}$-measurable.  It follows from (\ref{eq:burholder_inequality_bis}) (with $\mathcal{H}$ replaced by $\mathcal{H}^{d}$) and (\ref{eq:deriv_LZ}) that
\begin{align*}
\mathbb{E}_{x}[ & \sup_{t \in \EnsT }  \vert  \delta^{\frac{1}{2}} \sum_{i \in \EnsN}\sum_{\underset{w< t}{w\in \pi^{\delta}}}   \OpOrnsteinUlh Z^{\delta,i}_{w+\delta} B^{2,i}_{w}  \vert_{\mathcal{H}^{d}}^{p}]^{\frac{2}{p}}  \leqslant  \CsteBDH_{p}^{2} \delta \sum_{\underset{t<T}{t\in \pi^{\delta}}}  \mathbb{E}_{x}[\vert \sum_{i \in \EnsN} \ \OpOrnsteinUlh Z^{\delta,i}_{t+\delta} B^{2,i}_{t}  \vert_{\mathcal{H}^{d}}^{p}]^{\frac{2}{p}} \\
\leqslant & C\delta \sum_{\underset{t<T}{t\in \pi^{\delta}}}  \mathbb{E}_{x}[\vert \sum_{i \in \EnsN} \vert B^{2,i}_{t}  \vert_{\mathcal{H}^{d}}^{2} \vert^{\frac{p}{2}}]^{\frac{2}{p}}  \leqslant  C T \sup_{\underset{t<T}{t\in \pi^{\delta}}} \mathbb{E}_{x}[\vert B^{2,.}_{t}  \vert_{(\mathcal{H}^{d})^{\EnsN}}^{p}]^{\frac{2}{p}} .
\end{align*}
In the same way, 
\begin{align*}
\mathbb{E}_{x}[\sup_{t \in \EnsT }  \vert  \delta^{\frac{1}{2}} \sum_{i \in \EnsN} \sum_{\underset{w< t}{w\in \pi^{\delta} }}   Z^{\delta,i}_{w+\delta} B^{1,i}_{w}  \vert_{\mathcal{H}^{d}}^{p}]^{\frac{2}{p}}  \leqslant & C \MomZ_{C}(Z^{\delta}) T \sup_{\underset{t<T}{t\in \pi^{\delta}}} \mathbb{E}_{x}[\vert B^{1,.}_{t}  \vert_{(\mathcal{H}^{d})^{\EnsN}}^{p}]^{\frac{2}{p}} .
\end{align*}

We now study the term $B_{t}Y_{t}$. Arguing as above and noting that $Y_{w}$ is $\mathcal{F}^{\delta}_{w}$-measurable, we deduce from $\mathbf{A}_{1}^{\delta}$ together with (\ref{eq:burholder_inequality_bis}) (with $\mathcal{H}$ replaced by $\mathcal{H}^{d}$) that

\begin{align*}
\mathbb{E}_{x}[\sup_{t \in \EnsT }  \vert  \delta^{\frac{1}{2}} \sum_{i \in \EnsN} \sum_{\underset{w< t}{w\in \pi^{\delta}}}   Z^{\delta,i}_{w+\delta} \nabla_{x} A_{1}^{i}(X^{\delta}_{w}, w) Y_{w}    \vert_{\mathcal{H}^{d}}^{p}]^{\frac{2}{p}} \leqslant &\CsteBDH_{p}^{2} \delta \sum_{\underset{t<T}{t\in \pi^{\delta}}}  \mathbb{E}_{x}[ \vert \sum_{i \in \EnsN}  Z^{\delta,i}_{t+\delta} \nabla_{x} A_{1}^{i}(X^{\delta}_{t}, t) Y_{t}    \vert_{\mathcal{H}^{d}}^{p}]^{\frac{2}{p}} \\
 \leqslant & C  \MomZ_{C}(Z^{\delta}) \CstePsiouf_{3} ^{2}  \delta \sum_{\underset{t<T}{t\in \pi^{\delta}}}  \mathbb{E}_{x}[\vert Y_{t}     \vert_{\mathcal{H}^{d}}^{p}]^{\frac{2}{p}} .
\end{align*}


Combining $\mathbf{A}_{1}^{\delta}$ (see (\ref{eq:hyp_3_Norme_adhoc_fonction_schema})) with the triangle inequality also gives

\begin{align*}
\mathbb{E}_{x}[\sup_{t \in \EnsT }  & \vert  \delta \sum_{i,j \in \EnsN} \sum_{\underset{w< t}{w\in \pi^{\delta}}}    Z^{\delta,i}_{w+\delta} Z^{\delta,j}_{w+\delta} \nabla_{x} A_{2}^{i,j}(X^{\delta}_{w},w,\delta^{\frac{1}{2}}Z^{\delta}_{w+\delta}) Y_{w}   \vert_{\mathcal{H}^{d}}^{p}]^{\frac{1}{p}}  \\
\leqslant &\delta \sum_{\underset{t<T}{t\in \pi^{\delta}}}  \mathbb{E}_{x}[\vert Z^{\delta}_{t+\delta} \vert_{\mathbb{R}^{N}}^{2p} (\sum_{i,j \in \EnsN} \Vert \nabla_{x} A_{2}^{i,j}(X^{\delta}_{t}, t,\delta^{\frac{1}{2}}Z^{\delta}_{t+\delta}) \Vert_{\mathbb{R}^{d}} )^{p} \vert Y_{t}    \vert_{\mathcal{H}^{d}}^{p}]^{\frac{1}{p}} \leqslant  C \MomZ_{C}(Z^{\delta}) \CstePsiouf_{3}  \delta \sum_{\underset{t<T}{t\in \pi^{\delta}}}  \mathbb{E}_{x}[\vert Y_{t}  \vert_{\mathcal{H}^{d}}^{p}]^{\frac{1}{p}} . 
\end{align*}
Moreover, a similar estimate holds for the terms involving $A_{3}$.

Combining the preceding estimates and applying the triangle and Cauchy-Schwarz inequalities, we obtain
\begin{align*}
	\mathbb{E}_{x}[ & \sup_{\underset{t \leqslant T}{t\in \pi^{\delta}}}  \vert Y_{t}\vert_{\mathcal{H}^{d}}^{p}]^{\frac{1}{p}} \leqslant 
 \mathbb{E}_{x}[ \vert Y_{0}\vert_{\mathcal{H}^{d}}^{p}]^{\frac{1}{p}}\\
&+C\MomZ_{C}(Z^{\delta})^{\frac{1}{p}} ((1+T^{\frac{1}{2}}) \mathscr{C}^{x}_{\mathcal{H}^{d},0,p}(B^{1},B^{2},B^{3}) +\CstePsiouf_{3}  (\delta \sum_{\underset{t <T}{t\in \pi^{\delta}}}   \mathbb{E}_{x}[\vert Y_{t}  \vert_{\mathcal{H}^{d}}^{p}]^{\frac{1}{p}}+(\delta \sum_{\underset{t <T}{t\in \pi^{\delta}}}   \mathbb{E}_{x}[\vert Y_{t}  \vert_{\mathcal{H}^{d}}^{p}]^{\frac{2}{p}})^{\frac{1}{2}})) .
\end{align*}
Hence, Gronwall's lemma yields (\ref{eq:borne_norme_sobolev_generique}) when $q=0$. Notice that $N$ and $r_{\ast}$ only enter the constant $C$ through the moments of $\OpOrnsteinUlh Z^{\delta}_{t+\delta}$ (see (\ref{eq:deriv_LZ})), which multiply $B^{2}$ in (\ref{eq:processus_general_Y_sobolev}). In particular, when $B^{2}=0$, $C$ depends neither on $N$ nor on $r_{\ast}$.

\textbf{Step 2. } 
Let us prove (\ref{eq:norme_Sobolev_X_theo}). For $q \in \mathbb{N}$, we define $\mathcal{R}_{0}=\mathbb{R}$ and $\mathcal{R}_{q+1}=(\mathcal{R}_{q})^{\EnsT  \times \EnsN}$ and we have
\begin{align*}
\mathbb{E}_{x}[\sup_{t \in \EnsT } \vert  X^{\delta}_{t}  \vert_{\mathbb{R}^{d} ,1,q}^{p}]^{\frac{1}{p}} = & \mathbb{E}_{x}[\sup_{t \in \EnsT }( \sum_{q^{\diamond}=1}^{q} \delta^{q^{\diamond}} \vert \DeriveMal^{q^{\diamond}}_{} X^{\delta}_{t} \vert_{\mathcal{R}_{q^{\diamond}}^{d}}^{2} )^{\frac{p}{2}} ]^{\frac{1}{p}} .
\end{align*}
First, we focus on the case $q=1$.
We remark that for every $t \in \pi^{\delta}$, $w \in \EnsT $, and every $i \in \EnsN$,
\begin{align*}
 \delta^{\frac{1}{2}}\DeriveMal_{(w,i)} X^{\delta}_{t+\delta}=&(I_{d\times d}+B_{t}) \delta^{\frac{1}{2}} \DeriveMal_{(w,i)} X^{\delta}_{t}+(B^{3}_{1,t})_{w,i} ,
\end{align*}
with,  for $(w,i) \in \EnsT  \times \EnsN$,
\begin{align*}
(B^{3}_{1,t})_{w,i}=& \chi^{\delta}_{t+\delta} \mathbf{1}_{w=t+\delta}( \delta^{\frac{1}{2}} A_{1}^{i}(X^{\delta}_{t}, t) + \delta \sum_{j \in \EnsN} 2 Z^{\delta,j}_{t+\delta}  A_{2}^{i,j}(X^{\delta}_{t}, t,\delta^{\frac{1}{2}}Z^{\delta}_{t+\delta})   \\
&+\delta^{\frac{3}{2}} \sum_{j,l \in \EnsN} Z^{\delta,j}_{t+\delta} Z^{\delta,l}_{t+\delta}  \partial_{z^{i}}A_{2}^{j,l}(X^{\delta}_{t}, t,\delta^{\frac{1}{2}}Z^{\delta}_{t+\delta})+\delta^{\frac{3}{2}}\partial_{z^{i}}A_{3}(X^{\delta}_{t}, t,\delta^{\frac{1}{2}}Z^{\delta}_{t+\delta},\delta) ).
\end{align*}
In particular, $ \delta^{\frac{1}{2}} \DeriveMal X^{\delta}_{t} = ( \delta^{\frac{1}{2}} \DeriveMal_{(w,i)} X^{\delta}_{t})_{(w,i) \in \EnsT   \times \EnsN}$ is a $\mathcal{R}_{1}^{d}$-valued random variable and, for $t \in \pi^{\delta}$, we have $ \delta^{\frac{1}{2}} \DeriveMal X^{\delta}_{t+\delta}=(I_{d\times d}+B_{t})  \delta^{\frac{1}{2}} \DeriveMal X^{\delta}_{t}+B^{3}_{1,t}.$
For $q=1$, (\ref{eq:norme_Sobolev_X_theo}) follows from Lemma \ref{lemme:borne:moments_X} (see (\ref{eq:borne_moment_X})) and (\ref{eq:borne_norme_sobolev_generique}) with $q=0$ and $Y= \delta^{\frac{1}{2}} \DeriveMal X^{\delta}$, $\mathcal{H}=\mathcal{R}_{1}$, and  $B^{3}$ thus defined since the assumption $\mathbf{A}_{1}^{\delta}(3)$ (see (\ref{eq:hyp_1_Norme_adhoc_fonction_schema})) implies that

\begin{align*}
\mathscr{C}^{x}_{\mathcal{R}_{1}^{d},0,p}(&0,0,B^{3}_{1,.}) 
=  1+\mathbb{E}_{x}[ \sup_{t \in \EnsT} ( \sum_{\underset{w<t}{w\in \pi^{\delta}}}  \sum_{i \in \EnsN}   \vert (B^{3}_{1,w})_{w+\delta,i}  \vert_{\mathbb{R}^{d}}^{2})^{\frac{p}{2}}]^{\frac{1}{p}}  \\
\leqslant  &  1+  \mathbb{E}_{x}[  \vert \sum_{\underset{t<T}{t\in \pi^{\delta}}} \vert (B^{3}_{1,t})_{t+\delta,.}  \vert_{(\mathbb{R}^{d})^{\EnsN}}^{2} \vert^{\frac{p}{2}}  ]^{\frac{1}{p}}  
\leqslant    1+T^{\frac{1}{2}}  \delta^{-\frac{1}{2}}  \sup_{\underset{t<T}{t\in \pi^{\delta}}} \mathbb{E}_{x}[  \vert (B^{3}_{1,t})_{t+\delta,.}  \vert_{(\mathbb{R}^{d})^{\EnsN}}^{p} ]^{\frac{1}{p}} \\
\leqslant & 1+ C T^{\frac{1}{2}}  \CstePsiouf_{3} \MomZ_{C}(Z^{\delta})(1+\mathbb{E}_{x}[\sup_{t \in \EnsT }\vert X^{\delta}_{t-\delta} \vert_{\mathbb{R}^{d}}^{p \ExpPsiouf  _{3}} ]^{\frac{1}{p}}) .
\end{align*}

Now let us focus on the case $q \in \mathbb{N}$, $q \geqslant 2$. We argue by induction on $q$ and assume that
(\ref{eq:norme_Sobolev_X_theo}) holds up to order $q-1$. Similarly to the case $q=1$,  $\delta^{\frac{q}{2}}\DeriveMal^{q} X^{\delta}_{t}$ is a $\mathcal{R}_{q}^{d}$-valued random variable and, for $t \in \pi^{\delta}$, we have
\begin{align*}
\delta^{\frac{q}{2}} \DeriveMal^{q} X^{\delta}_{t+\delta}=&(I_{d\times d}+B_{t}) \delta^{\frac{q}{2}} \DeriveMal^{q} X^{\delta}_{t}+ \delta^{\frac{1}{2}} \sum_{i \in \EnsN}  Z^{\delta,i}_{t+\delta} B^{1,i}_{q,t}  +B^{3}_{q,t} ,
\end{align*}
with,  $B^{1,i}_{1,.}=0$,  $B^{3}_{1,.}$ defined in the beginning of \textbf{Step 2}, and for $q \geqslant 2$,
\begin{align*}
B^{1,i}_{q,t}=& \delta^{\frac{q}{2}}   (\DeriveMal X^{\delta}_{t})^{\mbox{\textsc{t}}} \Hessiennex  A_{1}^{i}(X^{\delta}_{t}, t)\DeriveMal^{q-1} X^{\delta}_{t}+\delta^{\frac{1}{2}} \DeriveMal B^{1,i}_{q-1,t}   \\
B^{3}_{q,t}=& \delta^{\frac{q-1}{2}}  (B^{3,1}_{t}+B^{3,2}_{t})\DeriveMal^{q-1} X^{\delta}_{t}+ \delta^{\frac{1}{2}}  \DeriveMal B^{3}_{q-1,t}+\delta^{} \sum_{i \in \EnsN}B^{1,i}
_{q-1,t}\DeriveMal Z^{\delta,i}_{t+\delta},
\end{align*}
with, for $(w,v) \in \EnsT  \times \EnsN$,
\begin{align}
\label{eq:def_B12_B13}
B^{3,1}_{t}= & \delta  \sum_{i,j \in \EnsN}  Z^{\delta,i}_{t+\delta} Z^{\delta,j}_{t+\delta} (  \delta^{\frac{1}{2}}  \DeriveMal X^{\delta}_{t})^{\mbox{\textsc{t}}}  \Hessiennex  A_{2}^{i,j}(X^{\delta}_{t},t,\delta^{\frac{1}{2}}Z^{\delta}_{t+\delta}) \\
&+ \delta ( \delta^{\frac{1}{2}}  \DeriveMal X^{\delta}_{t})^{\mbox{\textsc{t}}} \Hessiennex  A_{3}(X^{\delta}_{t},t,\delta^{\frac{1}{2}}Z^{\delta}_{t+\delta},\delta)  \nonumber\\
(B^{3,2}_{t} )_{w,v}  = & \chi^{\delta}_{t+\delta} \mathbf{1}_{w=t+\delta}  ( \delta^{\frac{1}{2}} \nabla_{x} A_{1}^{v}(X^{\delta}_{t}, t)  +\delta \sum_{j \in \EnsN} 2 Z^{\delta,j}_{t+\delta}  \nabla_{x} A_{2}^{v,j}(X^{\delta}_{t}, t,\delta^{\frac{1}{2}}Z^{\delta}_{t+\delta})  \nonumber \\
&+\delta^{\frac{3}{2}} \sum_{i,j \in \EnsN} Z^{\delta,i}_{t+\delta} Z^{\delta,j}_{t+\delta}  \partial_{z^{v}} \nabla_{x} A_{2}^{i,j}(X^{\delta}_{t}, t,\delta^{\frac{1}{2}}Z^{\delta}_{t+\delta}) +  \delta^{\frac{3}{2}}\partial_{z^{v}} \nabla_{x} A_{3}(X^{\delta}_{t}, t,\delta^{\frac{1}{2}}Z^{\delta}_{t+\delta},\delta))  \nonumber.
\end{align}
Since $B^{1}_{1,.}=0$ and $\mathbf{A}_{1}^{\delta}(q+1)$ holds, Lemma \ref{lemme:borne_norm_sob_prod_bis} and (\ref{eq:norme_Sobolev_X_theo}) yield
\begin{align*}
\mathscr{C}^{x}_{\mathcal{R}^{d}_{q},0,p}(&B^{1}_{q,.},0,0) 
\leqslant  \mathscr{C}^{x}_{\mathcal{R}^{d}_{q},0,p}(\delta^{\frac{q}{2}}(\DeriveMal X^{\delta}_{})^{\mbox{\textsc{t}}}  \Hessiennex  A_{1}(X^{\delta}_{}, .)\DeriveMal^{q-1} X^{\delta}_{},0,0)+ \mathscr{C}_{\mathcal{R}^{d}_{q-1},1,p}(B^{1}_{q-1,.},0,0) \\
\leqslant & \sum_{q^{\prime}=1}^{q-1} \mathscr{C}^{x}_{\mathcal{R}^{d}_{q-q^{\prime}+1},q^{\prime}-1,p}( \delta^{\frac{q-q^{\prime}+1}{2}}(\DeriveMal X^{\delta}_{})^{\mbox{\textsc{t}}} \Hessiennex  A_{1}(X^{\delta}_{},.)\DeriveMal^{q-q^{\prime}} X^{\delta}_{},0,0)\\
\leqslant &  C\CstePsiouf_{q+1}  \mathbb{E}_{x}[ \sup_{t \in \EnsT } ( 1+ \vert   X^{\delta}_{t}  \vert_{\mathbb{R}^{d},1,q-1}^{p q})(1+ \vert   X^{\delta}_{t}  \vert_{\mathbb{R}^{d}}^{p \ExpPsiouf  _{q+1} }) ]^{\frac{1}{p}}.
\end{align*}
Moreover
\begin{align*}
\mathscr{C}^{x}_{\mathcal{R}^{d}_{q},0,p}(&0,0,B^{3}_{q,.})  \leqslant  \mathscr{C}^{x}_{\mathcal{R}^{d}_{q},0,p}(0,0,\delta^{\frac{q-1}{2}}  (B^{3,1}_{}+B^{3,2}_{})\DeriveMal^{q-1} X^{\delta}) \\
&+\mathscr{C}^{x}_{\mathcal{R}^{d}_{q-1},1,p}(0,0, B^{3}_{q-1,.})   +\mathscr{C}^{x}_{\mathcal{R}^{d}_{q},0,p}(0,0,\delta \sum_{i \in \EnsN}B^{1,i}
_{q-1,.}\DeriveMal Z^{\delta,i}_{.+\delta}) \\
\leqslant & \sum_{q^{\prime}=1}^{q-1}\mathscr{C}^{x}_{\mathcal{R}^{d}_{q-q^{\prime}+1},q^{\prime}-1,p}(0,0,\delta^{\frac{q-q^{\prime}}{2}}  (B^{3,1}_{}+B^{3,2}_{}) \DeriveMal^{q-q^{\prime}} X^{\delta}) \\
&+\sum_{q^{\prime}=1}^{q-1}\mathscr{C}^{x}_{\mathcal{R}^{d}_{q-q^{\prime}+1},q^{\prime}-1,p}(0,0,\delta \sum_{i \in \EnsN}B^{1,i}
_{q-q^{\prime},.}\DeriveMal Z^{\delta,i}_{.+\delta})  +\mathscr{C}^{x}_{\mathcal{R}^{d}_{1},q-1,p}(0,0,B^{3}_{1,.})  .
\end{align*}

Since $\mathbf{A}_{1}^{\delta}(q+2)$ holds, using a similar approach as for the case $q=1$ yields

\begin{align*}
\mathscr{C}^{x}_{\mathcal{R}_{1}^{d},q-1,p}(&0,0,B^{3}_{1,.}) 
=  1+\mathbb{E}_{x}[  \sup_{t \in \EnsT} \vert \sum_{\underset{w<t}{w\in \pi^{\delta}}} \sum_{i \in \EnsN} \sum_{q^{\prime}=0}^{q-1} \delta^{q^{\prime}} \vert \DeriveMal^{q^{\prime}}_{}  (B^{3}_{1,w})_{w+\delta,i}  \vert_{\mathcal{R}_{q^{\prime}}}^{2} \vert^{\frac{p}{2}}]^{\frac{1}{p}}  \\
\leqslant  &  1+CT^{\frac{1}{2}}  \delta^{-\frac{1}{2}} \sup_{\underset{t<T}{t\in \pi^{\delta}}}  \sup_{q^{\prime} \in \{0,\ldots,q-1\}}   \mathbb{E}_{x}[  \vert \delta^{\frac{q^{\prime}}{2}} \DeriveMal^{q^{\prime}}_{}  (B^{3}_{1,t})_{t+\delta,.}  \vert_{(\mathcal{R}_{q^{\prime}}^{d})^{\EnsN}}^{p} ]^{\frac{1}{p}} \\
\leqslant & 1+C T^{\frac{1}{2}}    \CstePsiouf_{q+2} \MomZ_{C}(Z^{\delta})^{\frac{1}{p}}  \mathbb{E}_{x}[\sup_{t \in \EnsT } \vert 1 + \vert X^{\delta}_{t-\delta} \vert_{\mathbb{R}^{d},1,q-1}^{q-1} \vert^{p} \vert 1+ \vert X^{\delta}_{t-\delta} \vert_{\mathbb{R}^{d}}^{\ExpPsiouf  _{q+2} } \vert^{p} ]^{\frac{1}{p}} .
\end{align*}

Moreover, for $q^{\prime} \in \{1,\ldots,q-1\}$, we deduce from the Cauchy-Schwarz inequality that
\begin{align*}
\mathscr{C}^{x}_{\mathcal{R}^{d}_{q-q^{\prime}+1},q^{\prime}-1,p} & (0,0,\delta^{\frac{q-q^{\prime}}{2}} (B^{3,1}+B^{3,2} ) \DeriveMal^{q-q^{\prime}} X^{\delta}) \\
= & 1+\mathbb{E}_{x}[ \sup_{t \in \EnsT} \vert \sum_{\underset{w<t}{w\in \pi^{\delta}}} \delta^{\frac{q-q^{\prime}}{2}} (B^{3,1}_{w}+B^{3,2}_{w}) \DeriveMal^{q-q^{\prime}} X^{\delta}_{w}  \vert_{\mathcal{R}_{q-q^{\prime}+1}^{d},q^{\prime}-1}^{p}]^{\frac{1}{p}}  \\
\leqslant &1+ \sum_{\underset{t<T}{t\in \pi^{\delta}}}  \mathbb{E}_{x}[\vert  \delta^{\frac{q-q^{\prime}}{2}}  B^{3,1}_{t} \DeriveMal^{q-q^{\prime}} X^{\delta}_{t}  \vert_{\mathcal{R}_{q-q^{\prime}+1}^{d},q^{\prime}-1}^{p}]^{\frac{1}{p}} \\
&+\vert \delta \sum_{\underset{t<T}{t\in \pi^{\delta}}}  \sum_{i \in \EnsN} \mathbb{E}_{x}[ \vert \delta^{\frac{q-q^{\prime}-1}{2}}(B^{3,2}_{t})_{t+\delta,i}  \DeriveMal^{q-q^{\prime}} X^{\delta}_{t}  \vert_{\mathcal{R}_{q-q^{\prime}}^{d},q^{\prime}-1}^{p} ]^{\frac{2}{p}} \vert^{\frac{1}{2}}.
\end{align*}
Since $\mathbf{A}_{1}^{\delta}(q+2)$ (see (\ref{eq:hyp_1_Norme_adhoc_fonction_schema})) holds,  we have also
\begin{align*}
 \mathbb{E}_{x}[\vert  \delta^{\frac{q-q^{\prime}}{2}} B^{3,1}_{t} \DeriveMal^{q-q^{\prime}} X^{\delta}_{t}  \vert_{\mathcal{R}_{q-q^{\prime}+1}^{d},q^{\prime}-1}^{p}  ]^{\frac{1}{p}}  \leqslant  &  C \delta  \MomZ_{C}(Z^{\delta})^{\frac{1}{p}}\CstePsiouf_{q+2}  \mathbb{E}_{x}[\sup_{t \in \EnsT } (1 + \vert   X^{\delta}_{t}  \vert_{\mathbb{R}^{d},1,q-1}^{p q } ) ( 1+ \vert   X^{\delta}_{t}  \vert_{\mathbb{R}^{d}}^{p \ExpPsiouf  _{q+2}} )]^{\frac{1}{p}}\\
\mathbb{E}_{x}[ \vert \delta^{\frac{q-q^{\prime}-1}{2}}(B^{3,2}_{t})_{t+\delta,i}  \DeriveMal^{q-q^{\prime}} X^{\delta}_{t}  \vert_{\mathcal{R}_{q-q^{\prime}}^{d},q^{\prime}-1}^{p} ]^{\frac{2}{p}}   
\leqslant  &  C \MomZ_{C}(Z^{\delta})^{\frac{1}{p}}\CstePsiouf_{q+2}^{2}  \mathbb{E}_{x}[\sup_{t \in \EnsT } (1 + \vert   X^{\delta}_{t}  \vert_{\mathbb{R}^{d},1,q-1}^{p (q-1)} ) ( 1+ \vert   X^{\delta}_{t}  \vert_{\mathbb{R}^{d}}^{p \ExpPsiouf  _{q+2}} )]^{\frac{2}{p}} .
\end{align*}

Finally, for $q^{\prime} \in \{1,\ldots,q-1\}$, using (\ref{eq:derivee_Zkj}) and $\mathbf{A}_{1}^{\delta}(q)$, we obtain
\begin{align*}
\mathscr{C}^{x}_{\mathcal{R}^{d}_{q-q^{\prime}+1},  q^{\prime}-1,p}  (0,0,\delta \sum_{i \in \EnsN}B^{1,i}
_{q-q^{\prime},.}\DeriveMal Z^{\delta,i}_{.+\delta})  
= & 1+ \mathbb{E}_{x}[  \sup_{t \in \EnsT}\ \vert \sum_{\underset{w<t}{w\in \pi^{\delta}}}  \sum_{i \in \EnsN}   \delta^{2}  \vert B^{1,i}
_{q-q^{\prime},w} \DeriveMal_{(w+\delta,i)} Z^{\delta,i}_{w+\delta}  \vert_{\mathcal{R}_{q-q^{\prime}}^{d},q^{\prime}-1}^{2} \vert^{\frac{p}{2}}]^{\frac{1}{p}}   \\
\leqslant &  1+ C T^{\frac{1}{2}}   \CstePsiouf_{q} \mathbb{E}_{x}[ \sup_{t \in \EnsT } (  1 + \vert X^{\delta}_{t}  \vert_{\mathbb{R}^{d},1,q-1}^{p(q-1)} ) ( 1 + \vert   X^{\delta}_{t}  \vert_{\mathbb{R}^{d}}^{p \ExpPsiouf  _{q}} )]^{\frac{1}{p}}.
\end{align*}

More specifically, we have shown that
\begin{align*}
\mathscr{C}^{x}_{\mathcal{R}^{d}_{q} ,0,p}(0,0,B^{3}_{q,.})   \leqslant C(1+T) \MomZ_{C}(Z^{\delta})^{\frac{1}{p}}\CstePsiouf_{q+2} \mathbb{E}_{x}[\sup_{t \in \EnsT } ( 1 + \vert   X^{\delta}_{t}  \vert_{\mathbb{R}^{d},1,q-1}^{pq}  )(1 + \vert   X^{\delta}_{t}  \vert_{\mathbb{R}^{d}}^{p \ExpPsiouf  _{q+2}})]^{\frac{1}{p}}  .
\end{align*}

Applying (\ref{eq:borne_norme_sobolev_generique}) with $q=0$ yields, for $q \geqslant 2$,

\begin{align*}
\mathbb{E}_{x}[\sup_{t \in \EnsT } & \vert   X^{\delta}_{t}  \vert_{\mathbb{R}^{d},1,q}^{p}]^{\frac{1}{p}} 
\leqslant   \CstePsiouf_{q+2}  \exp (C (T+1)\MomZ_{C}(Z^{\delta})^{\frac{2}{p}}\CstePsiouf_{3} ^{2})   \mathbb{E}_{x}[\sup_{t \in \EnsT } (1 + \vert  X^{\delta}_{t}  \vert_{\mathbb{R}^{d},1,q-1}^{pq}) ( 1 + \vert   X^{\delta}_{t}  \vert_{\mathbb{R}^{d}}^{p \ExpPsiouf  _{q+2}} )]^{\frac{1}{p}}  .
\end{align*}
Using the induction hypothesis (i.e. (\ref{eq:norme_Sobolev_X_theo}) at order $q-1$) together with Lemma \ref{lemme:borne:moments_X} and H\"older's inequality to bound the r.h.s. above yields (\ref{eq:norme_Sobolev_X_theo}) at order $q$. This completes the induction.

\textbf{Step 3}.  In this last step, we prove (\ref{eq:borne_norme_sobolev_generique}). For $q \in \mathbb{N}^{\ast}$, we define $\mathcal{H}_{0}=\mathcal{H}$ and $\mathcal{H}_{q}=(\mathcal{H}_{q-1})^{\EnsT  \times \EnsN}$.  For $Y$ satisfying (\ref{eq:processus_general_Y_sobolev}), we have (remember that $\DeriveMal^{q}Y_{t}$, $t \in \pi^{\delta}$, belongs to $\mathcal{H}_{q}^{d}$), for every $t \in \pi^{\delta}$

\begin{align*}
\delta^{\frac{q}{2}}  \DeriveMal^{q}Y_{t+ \delta}=  & \delta^{\frac{q}{2}}  \DeriveMal^{q}Y_{t}+B_{t} \delta^{\frac{q}{2}}  \DeriveMal^{q}Y_{t}+ \delta^{\frac{1}{2}} \sum_{i=1}^{N}  Z^{\delta,i}_{t+\delta} B^{1,i}_{q,t}  +  \delta^{\frac{1}{2}}   \sum_{i=1}^{N}  \OpOrnsteinUlh Z^{\delta,i}_{t+\delta} B^{2,i}_{q,t} +B^{3}_{q,t} ,
\end{align*}
with $B^{1,i}_{q,t} = \delta^{\frac{1}{2}} \DeriveMal B^{1,i}_{q-1,t}  + \delta^{\frac{q}{2}} (\DeriveMal X^{\delta}_{t})^{\mbox{\textsc{t}}} \Hessiennex A_{1}^{i}(X^{\delta}_{t}, t) \DeriveMal^{q-1}Y_{t}$, $
B^{2,i}_{q,t} =\delta^{\frac{1}{2}} \DeriveMal B^{2,i}_{q-1,t}  $ and
\begin{align*}
B^{3}_{q,t} =& \delta^{\frac{1}{2}} \DeriveMal B^{3}_{q-1,t} + \delta^{\frac{q}{2}} \sum_{i=1}^N \nabla
_{x}A_{1}^{i}(X^{\delta}_{t}, t) \DeriveMal ( \delta^{\frac{1}{2}}Z^{\delta,i}_{t+\delta} )\DeriveMal^{q-1}Y_{t} \\
&+ \delta^{\frac{q}{2}} \sum_{i,j \in \EnsN}%
 \DeriveMal (\delta Z^{\delta,i}_{t+\delta} Z^{\delta,j}_{t+\delta}  \nabla_{x}A_{2}^{i,j}(X^{\delta}_{t}, t,\delta^{\frac{1}{2}}Z^{\delta}_{t+\delta}))\DeriveMal^{q-1}Y_{t} \\
&+ \delta^{\frac{q}{2}} \DeriveMal (\delta \nabla_{x}A_{3}(X^{\delta}_{t}, t,\delta^{\frac{1}{2}}Z^{\delta}_{t+\delta},\delta) )\DeriveMal^{q-1}Y_{t} +\delta^{\frac{1}{2}} \sum_{i \in \EnsN}B^{1,i}
_{q-1,t}\DeriveMal (\delta^{\frac{1}{2}}Z^{\delta,i}_{t+\delta}) +B^{2,i}
_{q-1,t}\DeriveMal\OpOrnsteinUlh (\delta^{\frac{1}{2}}Z^{\delta,i}_{t+\delta})  .
\end{align*}
Now, we aim to obtain an upper bound on $\mathscr{C}^{x}_{\mathcal{H}^{d}_{q},0,p}(B^{1}_{q,.},B^{2}_{q,.},B^{3}_{q,.})$. 

We study $\mathscr{C}^{x}_{\mathcal{H}^{d}_{q},0,p}(B^{1}_{q,.},B^{2}_{q,.},B^{3}_{q,.})$. We first remark that, using (\ref{eq:borne_norm_sob_prod_bis}) and $\mathbf{A}_{1}^{\delta}(q+2)$,
\begin{align*}
\mathscr{C}^{x}_{\mathcal{H}^{d}_{q},0,p} & (B^{1}_{q,.},0,0) \leqslant  \mathscr{C}^{x}_{\mathcal{H}^{d}_{q-1},1,p}(B^{1}_{q-1,.},0,0) + \mathscr{C}^{x}_{\mathcal{H}^{d}_{q},0,p}(\delta^{\frac{q}{2}}(\DeriveMal X^{\delta}_{})^{\mbox{\textsc{t}}} \Hessiennex  A_{1}(X^{\delta}_{}, .)\DeriveMal^{q-1} Y,0,0)\\
\leqslant &   \mathscr{C}^{x}_{\mathcal{H}^{d},q,p}(B^{1},0,0)  +  \sum_{q^{\prime}=1}^{q} \mathscr{C}^{x}_{\mathcal{H}_{q-q^{\prime}+1}^{d},q^{\prime}-1,p}(\delta^{\frac{q-q^{\prime}+1}{2}}(\DeriveMal X^{\delta}_{})^{\mbox{\textsc{t}}} \Hessiennex  A_{1}(X^{\delta}_{},.)\DeriveMal^{q-q^{\prime}} Y_{},0,0) \\
\leqslant &  \mathscr{C}^{x}_{\mathcal{H}^{d},q,p}(B^{1},0,0)  +C \CstePsiouf_{q+2} (1+\mathbb{E}_{x}[ \sup_{t \in \EnsT } \vert   Y_{t} \vert_{\mathcal{H}^{d},q-1}^{2p}]^{\frac{1}{2p}})   \mathbb{E}_{x}[ \sup_{t \in \EnsT } (  1 + \vert X^{\delta}_{t}  \vert_{\mathbb{R}^{d},1,q}^{2p q} )( 1 + \vert   X^{\delta}_{t}  \vert_{\mathbb{R}^{d}}^{3p \ExpPsiouf  _{q+2}} )]^{\frac{1}{2p}}  ,
\end{align*}
and similarly $\mathscr{C}^{x}_{\mathcal{H}_{q}^{d},0,p}(0,B^{2}_{q,.},0)  \leqslant \mathscr{C}^{x}_{\mathcal{H}^{d},q,p}(0,B^{2},0) $.  
 Computing the Malliavin derivatives involved in $B^{3}_{q,.}$ (recall that $B^{3,1}$ and $B^{3,2}$ are defined in (\ref{eq:def_B12_B13})), we also derive the following estimate
\begin{align*}
 \mathscr{C}^{x}_{\mathcal{H}^{d}_{q},0,p}(0,0,B^{3}_{q,.}) \leqslant &  \mathscr{C}^{x}_{\mathcal{H}^{d}_{q-1},1,p}(0,0,B^{3}_{q-1,.})  + \mathscr{C}^{x}_{\mathcal{H}^{d}_{q},0,p}(0,0,\delta^{\frac{q-1}{2}}(B^{3,1}_{}+B^{3,2}_{})\DeriveMal^{q-1} Y)  \\
& +\mathscr{C}^{x}_{\mathcal{H}^{d}_{q},0,p}(0,0,\delta^{\frac{1}{2}} \sum_{i \in \EnsN}B^{1,i}
_{q-1,.}\DeriveMal  (\delta^{\frac{1}{2}} Z^{\delta,i}_{.+\delta})+B^{2,i}
_{q-1,.}\DeriveMal \OpOrnsteinUlh (\delta^{\frac{1}{2}} Z^{\delta,i}_{.+\delta})) \\
\leqslant & \mathscr{C}^{x}_{\mathcal{H}^{d},q,p}(0,0,B^{3}) +  \sum_{q^{\prime}=1}^{q}\mathscr{C}^{x}_{\mathcal{H}^{d}_{q-q^{\prime}+1},q^{\prime}-1,p}(0,0,\delta^{\frac{q-q^{\prime}}{2}} (B^{3,1}_{}+B^{3,2}_{}) \DeriveMal^{q-q^{\prime}} Y) \\
&+\sum_{q^{\prime}=1}^{q}\mathscr{C}^{x}_{\mathcal{H}^{d}_{q-q^{\prime}+1},q^{\prime}-1,p}(0,0,\delta^{\frac{1}{2}} \sum_{i \in \EnsN }B^{1,i}
_{q-q^{\prime},.}\DeriveMal  (\delta^{\frac{1}{2}} Z^{\delta,i}_{.+\delta})+B^{2,i}
_{q-q^{\prime},.}\DeriveMal  \OpOrnsteinUlh   (\delta^{\frac{1}{2}} Z^{\delta,i}_{.+\delta}))  .
\end{align*}

We study the first term of the estimate above. As in \textbf{Step 2}, we have, for $q^{\prime} \in \{1,\ldots,q\}$,

\begin{align*}
\mathscr{C}^{x}_{\mathcal{H}^{d}_{q-q^{\prime}+1},q^{\prime}-1,p}  (0,0,\delta^{\frac{q-q^{\prime}}{2}} (B^{3,1}+B^{3,2} ) &  \DeriveMal^{q-q^{\prime}} Y) 
\leqslant 1+ \sum_{\underset{t<T}{t\in \pi^{\delta}}}  \mathbb{E}_{x}[\vert  B^{3,1}_{t} \delta^{\frac{q-q^{\prime}}{2}}  \DeriveMal^{q-q^{\prime}} Y_{t}  \vert_{\mathcal{H}_{q-q^{\prime}+1}^{d},q^{\prime}-1}^{p}]^{\frac{1}{p}} \\
&+\vert \delta \sum_{\underset{t<T}{t\in \pi^{\delta}}}  \sum_{i \in \EnsN} \mathbb{E}_{x}[ \vert \delta^{\frac{q-q^{\prime}-1}{2}}(B^{3,2}_{t})_{t+\delta,i}  \DeriveMal^{q-q^{\prime}} Y_{t}  \vert_{\mathcal{H}_{q-q^{\prime}}^{d},q^{\prime}-1}^{p} ]^{\frac{2}{p}} \vert^{\frac{1}{2}}.
\end{align*}
Moreover, since $\mathbf{A}_{1}^{\delta}(q+3)$ holds, applying (\ref{eq:borne_norm_sob_prod_bis}) yields the estimates
\begin{align*}
 \mathbb{E}_{x}[\vert  B^{3,1}_{t} \delta^{\frac{q-q^{\prime}}{2}}  & \DeriveMal^{q-q^{\prime}} Y_{t}   \vert_{\mathcal{H}_{q-q^{\prime}+1}^{d},q^{\prime}-1}^{p}]^{\frac{1}{p}} 
\leqslant   \delta C \MomZ_{C}(Z^{\delta})^{\frac{1}{p}} \CstePsiouf_{q+3} \\
& \times  \mathbb{E}_{x}[\sup_{t \in \EnsT } (1 + \vert   X^{\delta}_{t}  \vert_{\mathbb{R}^{d},1,q}^{2p q } ) ( 1 +  \vert   X^{\delta}_{t}  \vert_{\mathbb{R}^{d}}^{2p \ExpPsiouf  _{q+3}} )]^{\frac{1}{2p}}  \mathbb{E}_{x}[\sup_{t \in \EnsT } \vert   Y_{t}  \vert_{\mathcal{H}^{d},q-1}^{2p}]^{\frac{1}{2p}}  ,
\end{align*}
and
\begin{align*}
\mathbb{E}_{x}[ \vert & \delta^{-\frac{1}{2}}(B^{3,2}_{t})_{t+\delta,.}  \delta^{\frac{q-q^{\prime}}{2}}  \DeriveMal^{q-q^{\prime}} Y_{t}  \vert_{(\mathcal{H}_{q-q^{\prime}}^{d})^{\EnsN},q^{\prime}-1}^{p} ]^{\frac{1}{p}}   \leqslant   C  \MomZ_{C}(Z^{\delta}) \CstePsiouf_{q+3} \\
& \times  \mathbb{E}_{x}[\sup_{t \in \EnsT \cup \{0\} } (1 + \vert   X^{\delta}_{t}  \vert_{\mathbb{R}^{d},1,q-1}^{2p(q-1)} )(1 + \vert   X^{\delta}_{t}  \vert_{\mathbb{R}^{d}}^{2p \ExpPsiouf  _{q+3}})]^{\frac{1}{2p}}  \mathbb{E}_{x}[1+\sup_{t \in \EnsT \cup \{0\}} \vert   Y_{t}  \vert_{\mathcal{H}^{d},q-1}^{2p}]^{\frac{1}{2p}} .
\end{align*}

Now, we focus on the last two terms of the estimate of $\mathscr{C}_{\mathcal{H}^{d}_{q},0,p}(0,0,B^{3}_{q,.})$. For $q^{\prime} \in \{1,\ldots,q\}$, using the same calculus as a few lines above, we have
\begin{align*}
\mathscr{C}^{x}_{\mathcal{H}^{d}_{q-q^{\prime}+1},q^{\prime}-1,p} & (0,0,\delta^{\frac{1}{2}} \sum_{i \in \EnsN}B^{1,i}
_{q-q^{\prime},.}\DeriveMal(\delta^{\frac{1}{2}} Z^{\delta,i}_{.+\delta})) 
\leqslant  1 + T^{\frac{1}{2}}  \mathscr{C}^{x}_{\mathcal{H}^{d},q-1,p}(B^{1},0,0) \\
&+C T^{\frac{1}{2}}    \CstePsiouf_{q+1}  \mathbb{E}_{x}[ \sup_{t \in \EnsT } ( 1 + \vert   X^{\delta}_{t}  \vert_{\mathbb{R}^{d},1,q-1}^{2p(q-1)} ) ( 1 +  \vert   X^{\delta}_{t}  \vert_{\mathbb{R}^{d}}^{2p \ExpPsiouf  _{q+1}})]^{\frac{1}{2p}}  \mathbb{E}_{x}[ 1+\sup_{t \in \EnsT } \vert   Y_{t} \vert_{\mathcal{H}^{d},q-1}^{2p}]^{\frac{1}{2p}}     .
\end{align*}

Using (\ref{eq:borne_norm_sob_prod_bis}) with the estimate (\ref{eq:deriv_LZ}) from Lemma \ref{lemme:estim_LZ} yields, for every $q^{\prime} \in \{1,\ldots,q\}$,
\begin{align*}
\mathscr{C}^{x}_{\mathcal{H}^{d}_{q-q^{\prime}+1},q^{\prime}-1,p} & (0,0,\delta^{\frac{1}{2}} \sum_{i \in \EnsN}B^{2,i}
_{q-q^{\prime},.}\DeriveMal \OpOrnsteinUlh(\delta^{\frac{1}{2}} Z^{\delta,i}_{.+\delta})) \\
 \leqslant & 1+ C T^{\frac{1}{2}} \sup_{t \in \EnsT } \mathbb{E}_{x}[ \vert B^{2}
_{q-q^{\prime},t-\delta}  \vert_{(\mathcal{H}_{q-q^{\prime}}^{d})^{\EnsN},q^{\prime}-1}^{2p} ]^{\frac{1}{2p}} \Vert \OpOrnsteinUlh Z^{\delta}_{t} \Vert _{\mathbb{R}^{N} ,q^{\prime} ,2p}  
\leqslant  1+C T^{\frac{1}{2}}  \mathscr{C}^{x}_{\mathcal{H}^{d},q-1,2p}(0,B^{2},0) .
\end{align*}
In particular, we have shown  that
\begin{align*}
\mathscr{C}^{x}_{\mathcal{H}^{d}_{q},0,p}( &0,0,B^{3}_{q,.})  
\leqslant   C (1+T^{\frac{1}{2}} )  \mathscr{C}^{x}_{\mathcal{H}^{d},q,2p}(B^{1},B^{2},B^{3})\\
&+  C(1+T) \MomZ_{C}(Z^{\delta})\CstePsiouf_{q+3} \mathbb{E}_{x}[ \sup_{t \in \EnsT \cup \{0\}} (1+  \vert   X^{\delta}_{t}  \vert_{\mathbb{R}^{d},1,q}^{2pq})(1+ \vert   X^{\delta}_{t}  \vert_{\mathbb{R}^{d}}^{2p\ExpPsiouf  _{q+3}})]^{\frac{1}{2p}}    \mathbb{E}_{x}[1+ \sup_{t \in \EnsT \cup \{0\}} \vert   Y_{t}  \vert_{\mathcal{H}^{d},q-1}^{2p}]^{\frac{1}{2p}}  .
\end{align*}
Applying (\ref{eq:borne_norme_sobolev_generique}) with $q=0$, the estimate (\ref{eq:norme_Sobolev_X_theo}), and the induction hypothesis for (\ref{eq:borne_norme_sobolev_generique}) at order $q-1$ with exponent $2p$, we obtain (\ref{eq:borne_norme_sobolev_generique}) at order $q$.
}

\end{proof}

Now, we are in a position to prove Theorem \ref{theo:Norme_Sobolev_borne}.

\begin{proof}[Proof of Theorem \ref{theo:Norme_Sobolev_borne}.]
We do not treat the case $(\ExpPsiouf  _{n})_{n \in \mathbb{N}^{\ast}} \equiv 0$ which is similar but simpler because we do not use Lemma \ref{lemme:borne:moments_X} in this case. Let us focus on the case $(\ExpPsiouf  _{n})_{n \in \mathbb{N}^{\ast}} \not \equiv 0$. Notice also that it is sufficient to prove the estimates for $p\geqslant2$.
We treat the Sobolev norms of $\partial_{\mbox{\textsc{x}}^{\delta}_{0}}^{\alpha }X^{\delta}_{t}$. First, we notice that if $\vert \alpha \vert=0$, Lemma \ref{lemme:borne:moments_X} and (\ref{eq:norme_Sobolev_X_theo}) yield (\ref{eq:borne_Sob_flot_tang}). Moreover, in the case $\vert \alpha \vert =1$, (\ref{eq:borne_Sob_flot_tang}) is a direct consequence of Proposition \ref{prop:borne_Sob_generique}, since
\begin{align*}
\partial_{\mbox{\textsc{x}}^{\delta}_{0}}^{\alpha }X^{\delta}_{t+\delta}=\partial_{\mbox{\textsc{x}}^{\delta}_{0}}^{\alpha }X^{\delta}_{t}+B_{t} \partial_{\mbox{\textsc{x}}^{\delta}_{0}}^{\alpha }X^{\delta}_{t}.
\end{align*}
For $\alpha = (\alpha^{1},\ldots,\alpha^{d})\in \mathbb{N}^{d}$ with $\vert \alpha \vert \in \mathbb{N}^{\ast}$, we consider $i_{0} \in \{1,\ldots,d\}$ such that $\alpha^{i_{0}} \in \mathbb{N}^{\ast}$ and denote $\alpha^{-}=(\alpha^{1},\ldots,\alpha^{i_{0}-1},\alpha^{i_{0}} -1,\alpha^{i_{0}+1},\ldots,\alpha^{d})$. Then
\begin{align*}
\partial_{\mbox{\textsc{x}}^{\delta}_{0}}^{\alpha }X^{\delta}_{t+\delta}=\partial_{\mbox{\textsc{x}}^{\delta}_{0}}^{\alpha }X^{\delta}_{t}+B_{t} \partial_{\mbox{\textsc{x}}^{\delta}_{0}}^{\alpha }X^{\delta}_{t} + \delta^{\frac{1}{2}} \sum_{i \in \EnsN}  Z^{\delta,i}_{t+\delta} B^{1,i}_{\alpha,t}  +B^{3}_{\alpha,t} ,
\end{align*}
with $B^{1}_{\alpha}=B^{3}_{\alpha}=0$ if $\vert \alpha \vert = 1$ and for $\vert \alpha \vert \geqslant 2$,
$B^{1,i}_{\alpha,t}=   ( \partial_{\mbox{\textsc{x}}_{0}^{\delta,i_{0}}}  X^{\delta}_{t})^{\mbox{\textsc{t}}} \Hessiennex  A_{1}^{i}(X^{\delta}_{t}, t) \partial_{\mbox{\textsc{x}}^{\delta}_{0}}^{\alpha^{-} } X^{\delta}_{t}+\partial_{\mbox{\textsc{x}}^{\delta,i_{0}}_{0}}  B^{1,i}_{\alpha^{-},t}$, $B^{3}_{\alpha,t}= \tilde{B}^{i_{0}}_{t}  \partial_{\mbox{\textsc{x}}^{\delta}_{0}}^{\alpha^{-} }X^{\delta}_{t}+\partial_{\mbox{\textsc{x}}_{0}^{\delta,i_{0}}}  B^{3}_{\alpha^{-},t} $, where we have introduced the quantity
\begin{align*}
\tilde{B}^{i_{0}}_{t}=&\delta  \sum_{i,j \in \EnsN}  Z^{\delta,i}_{t+\delta} Z^{\delta,j}_{t+\delta}  (\partial_{\mbox{\textsc{x}}_{0}^{\delta,i_{0}}} X^{\delta}_{t})^{\mbox{\textsc{t}}} \Hessiennex  A_{2}^{i,j}(X^{\delta}_{t},t,\delta^{\frac{1}{2}}Z^{\delta}_{t+\delta}) \\
&+ \delta  (\partial_{\mbox{\textsc{x}}_{0}^{\delta,i_{0}}} X^{\delta}_{t})^{\mbox{\textsc{t}}} \Hessiennex  A_{3}(X^{\delta}_{t},t,\delta^{\frac{1}{2}}Z^{\delta}_{t+\delta},\delta) .
\end{align*}
We prove our estimate using Proposition \ref{prop:borne_Sob_generique}. Therefore, it is sufficient to study the quantity $ \mathscr{C}_{\mathbb{R}^{d},q,2^{q}p}(B_{\alpha}^{1},0,B_{\alpha}^{3}) $. In particular, since $\mathbf{A}_{1}^{\delta}(q+3)$ (see (\ref{eq:hyp_1_Norme_adhoc_fonction_schema})) holds, for every $p \geqslant 1$ and every $i \in \EnsN$, using a recursive argument and the fact that $B^{1}_{\alpha^{\prime}}=0$ if $\vert \alpha^{\prime} \vert = 1$, we obtain (the expectation in the norms below are computed conditionally to $X^{\delta}_{0}=x$)
\begin{align*}
\Vert   B^{1}_{\alpha,t} \Vert _{(\mathbb{R}^{d})^{\EnsN} ,q,2^{q}p}  \leqslant & \Vert   \partial_{\mbox{\textsc{x}}_{0}^{\delta,i_{0}}}   B^{1}_{\alpha^{-},t} \Vert _{(\mathbb{R}^{d})^{\EnsN},q,2^{q}p} \\
&+ C \CstePsiouf_{q+3} \sup_{t \in \EnsT \cup \{0\}} \sum_{\underset{  1 - q^{\prime}  \leqslant \vert \alpha^{\prime} \vert < \vert \alpha \vert}{q^{\prime} \in \{0,1\}, \alpha^{\prime} \in \mathbb{N}^{d}} } \mathbb{E}_{x} [( 1  + \vert \partial_{\mbox{\textsc{x}}^{\delta}_{0}}^{\alpha^{\prime} }  X^{\delta}_{t}\vert_{\mathbb{R}^{d} ,q^{\prime},q}^{2^{q}p(q+ 2)} ) ( 1  + \vert X^{\delta}_{t}\vert_{\mathbb{R}^{d}}^{2^{q}p \ExpPsiouf  _{q+3}}  )]^{\frac{1}{2^{q}p}} \\
\leqslant & C\CstePsiouf_{q+\vert \alpha \vert+1}  \sup_{t \in \EnsT \cup \{0\}} \sum_{\underset{  1 - q^{\prime}  \leqslant \vert \alpha^{\prime} \vert < \vert \alpha \vert }{q^{\prime} \in \{0,1\}, \alpha^{\prime} \in \mathbb{N}^{d}} } \mathbb{E}_{x} [ ( 1  + \vert \partial_{\mbox{\textsc{x}}^{\delta}_{0}}^{\alpha^{\prime} }  X^{\delta}_{t}\vert_{\mathbb{R}^{d} ,q^{\prime},q}^{2^{q}p(q+\vert \alpha \vert) }  ) ( 1  + \vert X^{\delta}_{t}\vert_{\mathbb{R}^{d}}^{2^{q}p \ExpPsiouf  _{q+\vert \alpha \vert+1}} ) ]^{\frac{1}{2^{q}p}} .
\end{align*}
In the same way
\begin{align*}
\Vert \sup_{t \in \EnsT}  \vert \sum_{w\in \pi^{\delta},w<t} B^{3}_{\alpha,w}  \vert_{\mathbb{R}^{d} ,q} \Vert _{2^{q}p} \leqslant & C(1+T)\CstePsiouf_{q+\vert \alpha \vert+2} \MomZ_{C}(Z^{\delta}) \\
& \times \sup_{t \in \EnsT \cup \{0\} } \sum_{\underset{  1 - q^{\prime}  \leqslant \vert \alpha^{\prime} \vert < \vert \alpha \vert}{q^{\prime} \in \{0,1\}, \alpha^{\prime} \in \mathbb{N}^{d}} } \mathbb{E}_{x} [ (1  + \vert \partial_{\mbox{\textsc{x}}^{\delta}_{0}}^{\alpha^{\prime} }  X^{\delta}_{t}\vert_{\mathbb{R}^{d} ,q^{\prime},q}^{2^{q}p(q+\vert \alpha \vert)}  ) ( 1  + \vert X^{\delta}_{t}\vert_{\mathbb{R}^{d}}^{2^{q}p \ExpPsiouf  _{q+\vert \alpha \vert+2}}  ) ]^{\frac{1}{2^{q}p}} .
\end{align*}

Then (\ref{eq:borne_Sob_flot_tang}) follows from Proposition \ref{prop:borne_Sob_generique} combined with a recursive approach.
We now study the Sobolev norms of $\OpOrnsteinUlh X^{\delta}_{t}$. We have
\begin{align*}
\OpOrnsteinUlh X^{\delta}_{t+ \delta}=  & \OpOrnsteinUlh X^{\delta}_{t}+B_{t}\OpOrnsteinUlh X^{\delta}_{t}+ \delta^{\frac{1}{2}} \sum_{i \in \EnsN}  Z^{\delta,i}_{t+\delta} B^{1,i}_{t}  +  \delta^{\frac{1}{2}}   \sum_{i \in \EnsN}  \OpOrnsteinUlh Z^{\delta,i}_{t+\delta} B^{2,i}_{t} +B^{3}_{t} ,
\end{align*}
with $B^{1,i}_{t} 
=\mbox{Tr}(\MatCovMal_{X^{\delta}_{t}} \Hessiennex A_{1}^{i}(X^{\delta}_{t}, t)) $, $B^{2,i}_{t} =A_{1}^{i}(X^{\delta}_{t}, t)  $ 
and
\begin{align*}
B^{3}_{t} =&  \delta  \sum_{i,j \in \EnsN} \{ (Z^{\delta,i}_{t+\delta} \OpOrnsteinUlh Z^{\delta,j}_{t+\delta}  + Z^{\delta,j}_{t+\delta}  \OpOrnsteinUlh Z^{\delta,i}_{t+\delta} +2 \chi^{\delta}_{t+\delta}\mathbf{1}_{i,j} )  A_{2}^{i,j}(X^{\delta}_{t}, t,\delta^{\frac{1}{2}}Z^{\delta}_{t+\delta})  \\
& +
2\delta^{\frac{1}{2}}\chi^{\delta}_{t+\delta}(
Z^{\delta,i}_{t+\delta} \partial_{z^{j}}A_{2}^{i,j}(X^{\delta}_{t},t,\delta^{\frac{1}{2}}Z^{\delta}_{t+\delta}) +Z^{\delta,j}_{t+\delta}
\partial_{z^{i}}A_{2}^{i,j}(X^{\delta}_{t},t,\delta^{\frac{1}{2}}Z^{\delta}_{t+\delta}))\\
&+Z^{\delta,i}_{t+\delta} Z^{\delta,j}_{t+\delta} (\mbox{Tr}(\MatCovMal_{X^{\delta}_{t}} \Hessiennex A_{2}^{i,j}(X^{\delta}_{t}, t,\delta^{\frac{1}{2}}Z^{\delta}_{t+\delta})) + \delta^{\frac{1}{2}}\sum_{l \in \EnsN} \partial_{z^{l}}A_{2}^{i,j}(X^{\delta}_{t},t,\delta^{\frac{1}{2}}Z^{\delta}_{t+\delta}) \OpOrnsteinUlh Z^{\delta,l}_{t+ \delta}  \\
&+ \chi^{\delta}_{t+\delta}  \delta \sum_{l \in \EnsN} \partial_{z^{l}}^{2}A_{2}^{i,j}(X^{\delta}_{t},t,\delta^{\frac{1}{2}}Z^{\delta}_{t+\delta})  ) \} + \delta^{\frac{3}{2}}\sum_{l \in \EnsN} \partial_{z^{l}}A_{3}(X^{\delta}_{t},t,\delta^{\frac{1}{2}}Z^{\delta}_{t+\delta},\delta) \OpOrnsteinUlh Z^{\delta,l}_{t+ \delta}  \\
&+ \chi^{\delta}_{t+\delta}  \delta^{2} \sum_{l \in \EnsN} \partial_{z^{l}}^{2}A_{3}(X^{\delta}_{t},t,\delta^{\frac{1}{2}}Z^{\delta}_{t+\delta},\delta) + \delta  \mbox{Tr}(\MatCovMal_{X^{\delta}_{t}} \Hessiennex A_{3}(X^{\delta}_{t},t,\delta^{\frac{1}{2}}Z^{\delta}_{t+\delta},\delta)) .
\end{align*}
Moreover, for every $p \geqslant 1$, and every $i \in \EnsN$, using $\mathbf{A}_{1}^{\delta}(q+4)$ (see (\ref{eq:hyp_1_Norme_adhoc_fonction_schema})),  we have
\begin{align*}
\Vert B^{1}_{t} \Vert _{(\mathbb{R}^{d})^{\EnsN},q,2^{q}p} + \Vert B^{2}_{t} \Vert _{(\mathbb{R}^{d})^{\EnsN},q,2^{q}p} \leqslant & C \CstePsiouf_{q+3}\sup_{t \in \EnsT \cup \{0\}} \mathbb{E}_{x}[ ( 1 + \vert  X^{\delta}_{t}  \vert_{\mathbb{R}^{d},1,q+1}^{2^{q}p(q+2)} ) ( 1 + \vert   X^{\delta}_{t}  \vert_{\mathbb{R}^{d}}^{2^{q}p \ExpPsiouf  _{q+3}})]^{\frac{1}{2^{q}p}} ,
\end{align*}
and
\begin{align*}
\Vert \sup_{t \in \EnsT}  \vert \sum_{w\in \pi^{\delta},w<t} B^{3}_{w}  \vert_{\mathbb{R}^{d} ,q} \Vert _{2^{q}p}\leqslant & C(1+T)\CstePsiouf_{q+4} \MomZ_{C}(Z^{\delta}) (1+\sup_{t \in \EnsT }\Vert \OpOrnsteinUlh Z^{\delta}_{t} \Vert _{\mathbb{R}^{N} ,q,2^{q}2p}) \\
& \times \sup_{t \in \EnsT \cup \{0\}} \mathbb{E}_{x}[ (1 + \vert  X^{\delta}_{t}  \vert_{\mathbb{R}^{d},1,q+1}^{2^{q}p(q+2)}) ( 1 + \vert   X^{\delta}_{t}  \vert_{\mathbb{R}^{d}}^{2^{q}p\ExpPsiouf  _{q+4}} )]^{\frac{1}{2^{q}p}}   .
\end{align*}
We finally use (\ref{eq:deriv_LZ}) from Lemma \ref{lemme:estim_LZ}, Proposition \ref{prop:borne_Sob_generique} and Lemma \ref{lemme:borne:moments_X} to complete the proof of (\ref{eq:theo_borne_{d}eriv_mall_schema_annex_L}).

\end{proof}

\subsection{Proof of Theorem \ref{th:borne_Lp_inv_cov_Mal}}

\subsubsection{Preliminaries}
Before we focus on the proof of Theorem \ref{th:borne_Lp_inv_cov_Mal}, we provide a representation formula for the Malliavin derivatives using the variation of constants formula and some technical results we will employ in our proof.

\noindent \textbf{Representation formula.} Let $w,t\in \pi^{\delta,\ast}, i \in \EnsN$. Then $\DeriveMal_{(w,i)}X^{\delta}_{t} (x)=0$ for every $w>t$ and for $w \leqslant t$,
\begin{align*}
\DeriveMal_{(w,i)}X^{\delta}_{t} =  \chi^{\delta}_{t} \partial_{z^{i}} \psi(X^{\delta}_{t-\delta},t-\delta,\delta^{\frac{1}{2}} Z^{\delta}_t,\delta) \mathbf{1}_{w=t} +\nabla_{x} \psi(X^{\delta}_{t-\delta},t-\delta,\delta^{\frac{1}{2}} Z^{\delta}_t,\delta) \DeriveMal_{(w,i)}X^{\delta}_{t-\delta} (x).
\end{align*}

We consider the tangent flow process $(\dot{X}_{t})_{t \in \pi^{\delta}}$ defined by $ \dot{X}_{0}=I_{d \times d}$ and
\begin{align*}
\dot{X}_{t} :=   \partial_{\mbox{\textsc{x}}^{\delta}_{0}}X^{\delta}_{t} = \nabla_{x} \psi(X^{\delta}_{t-\delta},t-\delta,\delta^{\frac{1}{2}} Z^{\delta}_t,\delta) \dot{X}_{t-\delta}.
\end{align*}
We now define the inverse tangent flow.  To prove the invertibility, we consider the Hilbert space $(\mathbb{R}^{d \times d}, \langle , \rangle_{\mathbb{R}^{d \times d}})$, with the Frobenius scalar product defined by $\langle M , M^{\prime} \rangle_{\mathbb{R}^{d \times d}} :=    \mbox{Trace}(M^{\prime} M^{\mbox{\textsc{t}}})=\sum_{i=1}^{d}( M^{\prime}M^{\mbox{\textsc{t}}})_{i,i}$, $M,M^{\prime} \in \mathbb{R}^{d \times d}$. Notice that for $M \in \mathbb{R}^{d \times d}$, $\Vert M \Vert_{\mathbb{R}^{d}} \leqslant \vert M \vert_{ \mathbb{R}^{d \times d}} \leqslant d^{\frac{1}{2}}\Vert M \Vert_{\mathbb{R}^{d}}$.  Also, for $k \in \mathbb{N}^{\ast}$, $\vert M^{k} \vert_{ \mathbb{R}^{d \times d}} \leqslant \Vert M\Vert_{ \mathbb{R}^{d}} \vert M^{k-1} \vert_{ \mathbb{R}^{d \times d}}  \leqslant  \vert M \vert_{ \mathbb{R}^{d \times d}}^{k}$ (with $M^{0}=I_{d \times d}$ and $M^{l}=MM^{l-1}$, $l \in\{1,\ldots,k\}$). 

Now, since $\nabla_{x} \psi(x,t,0,0)=I_{d\times d}$ for every $(x,t)\in \mathbb{R}^{d} \times \pi^{\delta}$, it follows from the Taylor expansion of $\nabla_{x} \psi$, that
\begin{align*}
 \nabla_{x} \psi(X^{\delta}_{t-\delta},t-\delta,\delta^{\frac{1}{2}}Z^{\delta}_t,\delta)=& I_{d \times d}+ \delta^{\frac{1}{2}}  \sum_{l \in \EnsN}   Z^{\delta,l}_{t}  \int_{0}^{1}\partial_{z^{l}} \nabla_{x} \psi(X^{\delta}_{t-\delta},t-\delta, \lambda \delta^{\frac{1}{2}}Z^{\delta}_t,0) \mbox{d} \lambda \\
& +  \delta \int_{0}^{1}  \partial_{y} \nabla_{x} \psi(X^{\delta}_{t-\delta},t-\delta,\delta^{\frac{1}{2}}Z^{\delta}_t,\lambda \delta) \mbox{d} \lambda .
\end{align*}

Applying $\mathbf{A}_{1}^{\delta}$ to this expansion, we obtain
\begin{align}
\vert I_{d\times d} - \nabla_{x} \psi (X^{\delta}_{t-\delta},t-\delta,\delta^{\frac{1}{2}}Z^{\delta}_{t},\delta)  \vert_{\mathbb{R}^{d \times d}}  \leqslant & \delta^{\frac{1}{2}}  4  \CstePsiouf_{3}  \max(\vert Z^{\delta}_{t} \vert_{\mathbb{R}^{N}}^{\ExpPsiouf_{3}  +1},1).
\label{borne_grad_psi-Id}
\end{align}
In particular, if we assume that
\begin{align}
\label{hyp:delta_eta_inversibilite}
\delta^{\frac{1}{2}} \eta_{2}^{\ExpPsiouf_{3}  +1} 8  \CstePsiouf_{3}  < 1,
\end{align}
we remark that, on the set $\{\vert Z^{\delta}_{t} \vert_{\mathbb{R}^{N}} \leqslant \eta_{2} \}$, we have
\begin{align*}
\vert \det  \nabla_{x} \psi (X^{\delta}_{t-\delta},t-\delta,\delta^{\frac{1}{2}}Z^{\delta}_{t},\delta)   \vert^{\frac{1}{d}}  \geqslant & \inf_{\xi \in \mathbb{R}^{d}; \vert \xi \vert_{\mathbb{R}^{d}} =1} \vert \nabla_{x} \psi (X^{\delta}_{t-\delta},t-\delta,\delta^{\frac{1}{2}}Z^{\delta}_{t},\delta)  \xi \vert_{\mathbb{R}^{d}} \\
\geqslant & 1-  \Vert I_{d\times d} - \nabla_{x} \psi (X^{\delta}_{t-\delta},t-\delta,\delta^{\frac{1}{2}}Z^{\delta}_{t},\delta)  \Vert_{\mathbb{R}^{d}}  \\
\geqslant & 1- \delta^{\frac{1}{2}} 4    \CstePsiouf_{3}   \eta_{2}^{\ExpPsiouf_{3}  +1} \geqslant \frac{1}{2}.
\end{align*}

The matrix $ \nabla_{x} \psi (X^{\delta}_{t-\delta},t-\delta,\delta^{\frac{1}{2}}Z^{\delta}_{t},\delta)  $ is thus invertible on the set $\{ \vert Z^{\delta}_{t} \vert_{\mathbb{R}^{N}} \leqslant \eta_{2} \}$.  We are now in a position to introduce the inverse tangent flow, namely $(\mathring{X}_{t})_{t \in \pi^{\delta}}$ satisfying $ \mathring{X}_{0}=I_{d \times d}$ and which is well defined for every $t \in \pi^{\delta,\ast}$ whenever we are on the set $\{\Theta_{\eta_{2},\pi^{\delta,\ast},t}>0\}$ and (\ref{hyp:delta_eta_inversibilite}) holds. In this case
\begin{align}
\label{eq:inverse_tangent_flot}
\mathring{X}_{t}  :=    \dot{X}_{t}^{-1}=\mathring{X}_{t-\delta}\nabla_{x} \psi(X^{\delta}_{t-\delta},t-\delta, \delta^{\frac{1}{2}} Z^{\delta}_t,\delta)^{-1}.
\end{align}
In particular we introduce $ \mathring{X}_{\eta_{2},t}  :=    \mathring{X}_{t}  \mathbf{1}_{\Theta_{\eta_{2},\pi^{\delta,\ast},t}>0}$ which is well defined for every $t \in \pi^{\delta}$.

We conclude these preliminaries by observing that we have the so-called variation of constants formula. Under assumption (\ref{hyp:delta_eta_inversibilite}), on the set $\{\Theta_{\eta_{2},\pi^{\delta,\ast},t}>0 \}$, for every $(w,i) \in \pi^{\delta,\ast} \cap (0,t] \times \EnsN$, we have
\begin{align}
\label{eq:Mal_var_of_constant_formula}
 \DeriveMal_{(w,i)}X^{\delta}_{t} = \chi^{\delta}_{w} \dot{X}_{t} \mathring{X}_{w}   \partial_{z^{i}} \psi(X^{\delta}_{w-\delta},w-\delta,\delta^{\frac{1}{2}} Z^{\delta}_w,\delta).
\end{align}

Before we give the proof of Theorem \ref{th:borne_Lp_inv_cov_Mal}, we start with some preliminary results which are crucial in the study of the determinant of the inverse of the Malliavin covariance matrix.  \\

\noindent \textbf{Preliminary results.}
We use the Burkholder inequality (see (\ref{eq:burkholder_inequality})) and an exponential martingale inequality, which we recall below. First, let us introduce some notations. Given an $\mathbb{R}$-valued process $(Y_{t})_{t \in \pi^\delta}$ progressively measurable $w.r.t.$ a filtration $(\mathcal{F}_{t})_{t \in \pi^\delta}$, we denote $\OpeNorrisMart (Y)_{t}= \delta^{-\frac{1}{2}}(Y_{t+\delta} - \mathbb{E}[Y_{t+\delta} \vert \mathcal{F}_{t}])$, $\OpeNorrisDrift (Y)_{t}=  \delta^{-1} \mathbb{E}[Y_{t+\delta}-Y_{t} \vert \mathcal{F}_{t}] $.\\
 Let $(M_t)_{t \in \pi^{\delta}}$ be an $\mathbb{R}$-valued local square integrable $(\mathcal{F}_t)_{t \in \pi^{\delta}}$-martingale with $M_0=0$. We denote $[M]_{t}=\delta \sum_{\underset{w<t}{w \in \pi^{\delta}}}  \OpeNorrisMart({M})_{w}^{2} $ and $\langle M\rangle_{t}=\delta \sum_{\underset{w<t}{w \in \pi^{\delta}}} \mathbb{E}[\OpeNorrisMart({M})_{w}^{2} \vert \mathcal{F}_{w}]$. Then (see \cite{Dzhaparidze_VanZanten_2001} Corollary 3.4 or \cite{Fan_Grama_Liu_2017}), we have the following extension of the Freedman inequality \cite{Freedman_1975}: For $a,b>0$ and $t \in \pi^{\delta}$, 
\begin{align}
\label{eq:Doob_martin_expo_ineq}
\mathbb{P} (\sup_{\underset{w \leqslant t}{w \in \pi^{\delta}}}\vert M_{w} \vert \geqslant a,   [M]_{t}+\langle M\rangle_{t}< b)     \leqslant  2\exp(-\frac{a^{2}}{2b}).
\end{align}

Now, let us give some additional intermediate results which are proved in the Appendix \ref{Append}. The first one is a technical result that is used to bound the probability that the determinant of a random matrix $\Sigma$ falls below a given threshold by studying $ \mathbb{P}( \xi^{\mbox{\textsc{t}}} \Sigma \xi \leqslant \epsilon)$ for $\xi \in \mathbb{R}^{d}$.

 \begin{lemme}
 \label{lemme:inversion_sup_Proba}
Let $\Sigma$ be a random symmetric positive semidefinite
$\mathbb{R}^{d\times d}$-valued matrix and $\epsilon \in (0,2^{\frac{1}{2}})$. Then
 \begin{align}
 \mathbb{P}( \inf_{\xi \in \mathbb{R}^{d}; \vert \xi \vert_{\mathbb{R}^{d}} =1} \xi^{\mbox{\textsc{t}}} \Sigma \xi \leqslant \frac{1}{2} \epsilon) \leqslant & C(d) \epsilon^{-2d} \sup_{\xi \in \mathbb{R}^{d}; \vert \xi \vert_{\mathbb{R}^{d}} =1}  \mathbb{P}( \xi^{\mbox{\textsc{t}}} \Sigma \xi \leqslant \epsilon) +  \mathbb{P}( \Vert \Sigma \Vert_{\mathbb{R}^{d}}> \frac{1}{3 \epsilon}).
 \label{eq:inversion_sup_Proba}
 \end{align}
 \end{lemme}
The second result provides an estimate of the moments of the inverse tangent flow.


\begin{lemme}
\label{lemme:borne_flt_tangent}
Let $T\in \pi^{\delta,\ast}$, $\EnsT =(0,T] \cap \pi^{\delta}$, $x \in \mathbb{R}^{d}$, $p\geqslant 2$ and $\eta_{2}>1$. Assume that (\ref{eq:hyp_3_Norme_adhoc_fonction_schema}) from $\mathbf{A}_{1}^{\delta}$, $\mathbf{A}_{3}^{\delta}(p(q^{\delta}_{\eta_{2}} \vee (2\ExpPsiouf_{3}  +2)))$ with $q^{\delta}_{\eta_{2}}  :=   1+  \lceil -\frac{\ln(\delta)}{2 \ln(\eta_{2})} \rceil$ and that (\ref{hyp:delta_eta_inversibilite}) hold.
Then, 
\begin{align}
\mathbb{E}_{x}[\sup_{t \in \EnsT }\Vert \mathring{X}_{\eta_{2},t}  \Vert ^{p}_{\mathbb{R}^{d}} ]^{\frac{1}{p}}  \leqslant C(d)\exp(C(p)  T \MomZ_{p(q^{\delta}_{\eta_{2}} \vee (2\ExpPsiouf_{3}  +2))}(Z^{\delta})^{\frac{2}{p}} \CstePsiouf_{3} ^{4}) .
\label{eq:borne_flot_tangent_inv}
\end{align}


\end{lemme}

The next result is a discrete-time Lie expansion satisfied by our process $X^{\delta}$ together with a control of the remainder appearing. 
\begin{lemme}[Discrete-time Lie expansion]
\label{lemme:dvpt_Lie}
Let $V \in \mathcal{C}^{3}(\mathbb{R}^{d} \times \mathbb{R}_{+};\mathbb{R}^{d})$ and let $\eta_{2}>1$. Assume that $\psi \in \mathcal{C}^{4}(\mathbb{R}^{d} \times \mathbb{R}_{+} \times \mathbb{R}^{N} \times [0,1])$ and that (\ref{eq:hyp_3_Norme_adhoc_fonction_schema}) and (\ref{hyp:delta_eta_inversibilite}) hold. Then, for every $t \in \pi^{\delta,\ast}$, we have the Lie expansion
\begin{align*}
  \mathring{X}_{\eta_{2},t}  V(X^{\delta}_{t},t)  =&  \mathring{X}_{\eta_{2},t-\delta}V(X^{\delta}_{t-\delta},t-\delta)  + \delta^{\frac{1}{2}}\sum_{i \in \EnsN} Z^{\delta,i}_{t} \mathring{X}_{\eta_{2},t-\delta} V^{[i]}(X^{\delta}_{t-\delta},t-\delta) \\
&+\delta \mathring{X}_{\eta_{2},t-\delta} V^{[0]}(X^{\delta}_{t-\delta},t-\delta) + \mathring{X}_{\eta_{2},t-\delta} \mathscr{R}V(X^{\delta}_{t-\delta},t-\delta,Z^{\delta}_{t}).
\end{align*}
Moreover, let $\alpha^{x} \in \mathbb{N}^{d}$ and let us introduce the $\mathbb{R}^{d}$-valued functions defined for every $(x,t,z) \in \mathbb{R}^{d} \times \pi^{\delta,\ast} \times \mathbb{R}^{N}$ by 
\begin{align*}
\tilde{\mathscr{R}}V(x,t-\delta,z)=& \mathscr{R}V(x,t-\delta,z)- \mathbb{E}[ \mathscr{R}V(x,t-\delta,Z^{\delta}_{t})] \\
\bar{\mathscr{R}}V(x,t-\delta)=& \mathbb{E}[ \mathscr{R}V(x,t-\delta,Z^{\delta}_{t}) ] .
\end{align*}
Assume that $\mathbf{A}_{1}^{\delta}(\vert \alpha^{x} \vert +4 )$ and $\mathbf{A}_{3}^{\delta}(+\infty)$ hold and that $
V \in \mathcal{C}^{\vert \alpha^{x} \vert+3}_{\mbox{pol}}(\mathbb{R}^{d} \times \mathbb{R}_{+};\mathbb{R}^{d} )  $.  Then, for every $(x,t,z) \in \mathbb{R}^{d} \times \pi^{\delta,\ast} \times  \mathbb{R}^{N}$,

 \begin{align}
\label{eq:borne_reste_mart_dvpt_Lie}
\vert \tilde{\mathscr{R}} V (x,t-\delta,z) \vert_{\mathbb{R}^{d}} \leqslant &  \delta C \MomZ_{C}(Z^{\delta}) \CstePsiouf_{4}^{C}  \CstePsiouf_{V,3}^{C} (1+ \vert x \vert_{\mathbb{R}^{d}}^{C}+ \vert z \vert_{\mathbb{R}^{N}}^{C}) ,
\end{align}
with $C=C(\ExpPsiouf  _{4}, \ExpPsiouf  _{V, 3 },q^{\delta}_{\eta_{2}} ) $ where $q^{\delta}_{\eta_{2}}$ is introduced in Lemma \ref{lemme:borne_flt_tangent}. Moreover
\begin{align}
\label{eq:borne_reste_moy_dvpt_Lie}
\vert \partial^{\alpha^{x}}_{x}  \bar{\mathscr{R}}V (x,t-\delta) \vert_{\mathbb{R}^{d}} \leqslant &  \delta^{\frac{3}{2}} C \MomZ_{C}(Z^{\delta})   \CstePsiouf_{\vert \alpha^{x} \vert+4}^{C}  \CstePsiouf_{V,\vert \alpha^{x} \vert +3}^{C} (1+ \vert x \vert_{\mathbb{R}^{d}}^{C}) ,
\end{align}
with $C=C(\ExpPsiouf  _{\vert \alpha^{x} \vert+4}, \ExpPsiouf  _{V,\vert \alpha^{x} \vert + 3 },q^{\delta}_{\eta_{2}} ,\vert \alpha^{x} \vert)$. 

\end{lemme}

We point out that Lemma \ref{lemme:borne_flt_tangent} and Lemma \ref{lemme:dvpt_Lie} will be used in the specific situation where $\eta_{2}=\funceta_{2}(\delta)$ under assumption $\mathbf{A}_{5}^{\delta}(x,T)$. In this case $\delta \leqslant(2^{\ExpPsiouf_{3}  +1}8\CstePsiouf_{3} )^{-4}$ so that $\funceta_{2}(\delta) \geqslant \min(\delta^{-\frac{3}{182}}, \delta^{-\frac{1}{4(1+\ExpPsiouf_{3}  )}})$ and then $q^{\delta}_{\funceta_{2}(\delta)}   \leqslant C(\ExpPsiouf_{3}  )$. 

The last result is a Norris Lemma adapted to discrete-time processes. In the continuous case, this lemma can be found in \cite{Nualart_2006}, Lemma 2.3.2. Before giving this result, we introduce some notations.  Let $q>0$ and $\EnsT  \subset  \pi^{\delta,\ast}$. Given an $\mathbb{R}$-valued process $(Y_{t})_{t \in \pi^\delta}$ progressively measurable $w.r.t.$ a filtration $(\mathcal{F}_{t})_{t \in \pi^\delta}$, we denote, 
\begin{align}
\label{eq:Quantite_lemme_Norris}
 \FuncNorris_{Y,\EnsT }(q)  :=    &1+ \sup_{t \in \EnsT } \mathbb{E}[\vert Y_{t-\delta} \vert^{q} ] + \mathbb{E}[\sup_{t \in \EnsT } \vert \OpeNorrisDrift (Y)_{t-\delta}\vert ^{q}] + \mathbb{E}[\sup_{t \in \EnsT }  \mathbb{E}[ \vert \OpeNorrisMart (Y)_{t-\delta} \vert^{q}\vert \mathcal{F}_{t-\delta}]] \\
& +  \mathbb{E}[\sup_{t \in \EnsT }  \vert \OpeNorrisDrift (\OpeNorrisDrift (Y))_{t-\delta}\vert ^{q}] + \mathbb{E}[\sup_{t \in \EnsT }  \mathbb{E}[ \vert \OpeNorrisMart(\OpeNorrisDrift (Y))_{t-\delta} \vert^{q} \vert \mathcal{F}_{t-\delta}]] . \nonumber
\end{align}

  \begin{lemme}[Discrete-time Norris Lemma]
 \label{lem:Norris}

Let $T\in \pi^{\delta,\ast}$, $\EnsT =(0,T] \cap \pi^{\delta}$.  Let $(Y_{t})_{t \in \pi^\delta}$ be an $\mathbb{R}$-valued random process progressively measurable with respect to a filtration $(\mathcal{F}_{t})_ {t \in \pi^\delta}$, let $r \in (0,\frac{1}{12})$ and let $ p  >0$.  Let us introduce $q(r,p)=4\max(4,\frac{44p}{1-12r})$ and assume that $\FuncNorris_{Y,\EnsT }(q(r,p)) < + \infty$. 
Then, for every $\epsilon \in [( 2^{10} (1 \vee T^{3}) \delta )^{\frac{44}{91-36r}},( 2^{20} (1 \vee T^6) )^{-\frac{11}{1-12r}} ]$, 
  \begin{align}
\label{eq:ineg_Norris}
  \mathbb{P}(\delta  \sum_{t \in \EnsT }  Y_t^{2} < \epsilon,  \delta\sum_{t \in \EnsT } & \mathbb{E}[ \OpeNorrisMart (Y)_{t-\delta}^{2} \vert  \mathcal{F}_{t-\delta}] +  \OpeNorrisDrift (Y)_{t-\delta}^{2} \geqslant \epsilon^r)   \\
  \leqslant  &  \epsilon^{p} (1 \vee T^{2q(r,p)}) 2^{5q(r,p)+3}  \FuncNorris_{Y,\EnsT }(q(r,p)) +    12 \exp(-\frac{\epsilon^{-\frac{1-12r}{22}}}{2^{12} (1 \vee T^{2})}) . \nonumber
  \end{align}
\end{lemme}

\subsubsection{Proof of Theorem \ref{th:borne_Lp_inv_cov_Mal}}

\begin{proof}

The strategy of this proof consists in showing that, for every $p \geqslant 0$,
  \begin{align}
 \label{eq:borne_Lp_inv_cov_Mal_tilde}
 \mathbb{E}_{x}[ \vert  \det \tilde{\InvMatCovMal}
_{X^{\delta}_{T}}  \vert^{p} \mathbf{1}_{\Theta^{\ast}_{T}>0} ] \leqslant   \frac{1+ \mathbf{1}_{\ExpPsiouf  _{2 L +5}>0}  \vert x \vert_{\mathbb{R}^{d}}^{C }     }{  \HormV_{L}(x)^{13^{L} 3 d(p+4)} T^{13^{L}dp}}   \CstePsiouf_{2 L + 5}^{C}    \exp(C (1+T) \MomZ_{C}(Z^{\delta}) \CstePsiouf_{3}^{4}),
 \end{align}
 where $\tilde{\InvMatCovMal}_{X^{\delta}_{T}}=\dot{X}_{T}^{\mbox{\textsc{t}}} \InvMatCovMal_{X^{\delta}_{T}}\dot{X}_{T}$ and (\ref{eq:borne_Lp_inv_cov_Mal}) follows from the H\"older inequality together with Lemma \ref{lemme:borne_flt_tangent}.  To show this estimate, we first establish a sufficient condition (see (\ref{eq:borne:mall_matrice_mom_1}) and (\ref{eq:borne:mall_matrice_mom_2})) to derive (\ref{eq:borne_Lp_inv_cov_Mal_tilde}) (see \textbf{Step 1.}). Then we prove that (\ref{eq:borne:mall_matrice_mom_1}) (resp. (\ref{eq:borne:mall_matrice_mom_2})) holds in \textbf{Step 2} (resp. \textbf{Step 3}). \textbf{Step 4} is dedicated to the proof of (\ref{eq:th_cov_mal_proba_espace_comp}).
We begin with some notations. For every $i \in \EnsN$, we introduce, when it is well defined, the $\mathbb{R}^{d}$-valued process $(\TermStepInitCovMal_{i,t})_{t \in \EnsT }$ defined for every $t \in \EnsT =(0,T] \cap \pi^{\delta}$ by $\TermStepInitCovMal_{i,t}=\mathring{X}_{t-\delta} \nabla_{x} \psi^{-1} \partial_{z^{i}} \psi (X^{\delta}_{t-\delta},t-\delta,\delta^{\frac{1}{2}} Z^{\delta}_{t},\delta)$. Notice that we have also $\TermStepInitCovMal_{i,t}=
\mathring{X}_{t}  \partial_{z^{i}} \psi(X^{\delta}_{t-\delta},t-\delta,\delta^{\frac{1}{2}} Z^{\delta}_{t},\delta)  .$
We now introduce the notation $v^{2}=v v^{\mbox{\textsc{t}}} \in \mathbb{R}^{d \times d}$ for a vector $v \in \mathbb{R}^{d}$. In particular, on the set $\{\Theta_{\funceta_{2}(\delta),\EnsT}>0\}$, using the variation of constants formula (\ref{eq:Mal_var_of_constant_formula}), and denoting $\tilde{\MatCovMal}_{X^{\delta}_{T}} = \delta \sum_{(t,i) \in \EnsT  \times \EnsN} \chi^{\delta}_{t}(\TermStepInitCovMal_{i,t})^{2}$, we have
  \begin{align*}
\MatCovMal_{X^{\delta}_{T}} = & \delta \sum_{(t,i) \in \EnsT  \times \EnsN}  (\DeriveMal_{(t,i)}X^{\delta}_{T} )^{2} = \delta \sum_{(t,i) \in \EnsT  \times \EnsN}\chi^{\delta}_{t}(\dot{X}_T  \mathring{X}_{t}   \partial_{z^{i}} \psi(X^{\delta}_{t-\delta},t-\delta,\delta^{\frac{1}{2}} Z^{\delta}_{t},\delta))^{2} 
=\dot{X}_T\tilde{\MatCovMal}_{X^{\delta}_{T}}  \dot{X}_T^{\mbox{\textsc{t}}}.
 \end{align*}
In particular, on the set $\{\Theta^{\ast}_{T}>0\}$, $\MatCovMal_{X^{\delta}_{T}}$ and $\tilde{\MatCovMal}_{X^{\delta}_{T}}$ are invertible with inverse given by  $\InvMatCovMal_{X^{\delta}_{T}}$ and $ \tilde{\InvMatCovMal}_{X^{\delta}_{T}}$, respectively. Our objective is to derive quantitative bounds on the localized moments of $ \tilde{\InvMatCovMal}_{X^{\delta}_{T}}$.

 \textbf{Step 1.} We first show that the proof of (\ref{eq:borne_Lp_inv_cov_Mal_tilde}) boils down to proving that there exist  $\bar{\boldsymbol{\epsilon}} \in (\funceta_{1}(\delta)^{-\frac{1}{d}}, 2^{-\frac{1}{2}}]$ and $\boldsymbol{\kappa}  \geqslant 1$ (which do not depend on $\delta$ and will be made explicit in the sequel) such that, for every $\epsilon \in (\funceta_{1}(\delta)^{-\frac{1}{d}}, \bar{\boldsymbol{\epsilon}})$,
 \begin{align}
 \label{eq:borne:mall_matrice_mom_1}
\sup_{\xi \in \mathbb{R}^{d}; \vert \xi \vert_{\mathbb{R}^{d}} =1} \mathbb{P}_{x}( \xi^{\mbox{\textsc{t}}} \tilde{\MatCovMal}_{X^{\delta}_{T}}  \xi \leqslant 2 \epsilon, \Theta_{\funceta_{2}(\delta),\EnsT }>0) \leqslant \boldsymbol{\kappa} \epsilon^{d(p+4)} ,
 \end{align}
 and
 \begin{align}
 \label{eq:borne:mall_matrice_mom_2}
  \mathbb{P}_{x}(\Vert\tilde{\MatCovMal}_{X^{\delta}_{T}}  \Vert_{\mathbb{R}^{d}}> \frac{1}{6 \epsilon}, \Theta_{\funceta_{2}(\delta),\EnsT }>0) \leqslant \boldsymbol{\kappa} \epsilon^{d(p+2)}.
 \end{align}
In this case 
\begin{align*}
\mathbb{E}_{x}[\vert \det \tilde{\InvMatCovMal}
_{X^{\delta}_{T}} \vert^{p}  \mathbf{1}_{\Theta^{\ast}_{T}>0} ]\leqslant & C \boldsymbol{\kappa}+C\lfloor \bar{\boldsymbol{\epsilon}}^{-d} \rfloor^{p} .
\end{align*}
The result of \textbf{Step 1} is mainly a consequence of Lemma \ref{lemme:inversion_sup_Proba}. We begin by noticing that 
\begin{align*}
\mathbb{P}_{x}(\vert \det \tilde{\InvMatCovMal}
_{X^{\delta}_{T}} \vert  \mathbf{1}_{\Theta^{\ast}_{T}>0} \geqslant  \epsilon^{-d}) =&\mathbb{P}_{x}(\vert \det \tilde{\MatCovMal}
_{X^{\delta}_{T}} \vert  \leqslant \epsilon^{d} ,\Theta^{\ast}_{T}>0) .
\end{align*}

Since $\vert \det  \tilde{\MatCovMal}_{X^{\delta}_{T}} \vert > \funceta_{1}(\delta)^{-1}$ on $\{\Theta_{X^{\delta}_{T},\det(\dot{X}^{\delta}_{T})^{2},\funceta_{1}(\delta)}>0\}$, the quantity above is equal to zero whenever $\epsilon^{d} \leqslant \funceta_{1}(\delta)^{-1}$ and for every $\epsilon^{d} >\funceta_{1}(\delta)^{-1}$,
 \begin{align*}
\mathbb{P}_{x}(\vert \det \tilde{\MatCovMal}_{X^{\delta}_{T}} \vert  \leqslant \epsilon^{d},\Theta^{\ast}_{T}>0)    \leqslant & \mathbb{P}_{x}(\vert \det \tilde{\MatCovMal}_{X^{\delta}_{T}} \vert  \leqslant \epsilon^{d},\Theta_{\funceta_{2}(\delta),\EnsT }>0)  \\
\leqslant & \mathbb{P}_{x}(\inf_{\xi \in \mathbb{R}^{d}; \vert \xi \vert_{\mathbb{R}^{d}} =1} \xi^{\mbox{\textsc{t}}} \tilde{\MatCovMal}_{X^{\delta}_{T}}  \xi \leqslant  \epsilon ,\Theta_{\funceta_{2}(\delta),\EnsT }>0) .
\end{align*}

Applying Lemma \ref{lemme:inversion_sup_Proba} (with (\ref{eq:borne:mall_matrice_mom_1}) and (\ref{eq:borne:mall_matrice_mom_2})), for every $\epsilon \in (\funceta_{1}(\delta)^{-\frac{1}{d}}, \bar{\boldsymbol{\epsilon}})$,
 \begin{align*}
\mathbb{P}_{x}(\vert \det \tilde{\MatCovMal}
_{X^{\delta}_{T}} \vert  \leqslant \epsilon^{d},\Theta^{\ast}_{T}>0)   
\leqslant & C \boldsymbol{\kappa} \epsilon^{d(p+2)}.
\end{align*}
To conclude the proof of \textbf{Step 1},  we observe that, by decomposing $\vert \det \tilde{\InvMatCovMal}
_{X^{\delta}_{T}} \vert$ on the intervals $(k,k+1]$, $k \in \mathbb{N}$, and exploiting the localization yields
\begin{align*}
\mathbb{E}_{x}[\vert \det \tilde{\InvMatCovMal}
_{X^{\delta}_{T}} \vert^{p}  \mathbf{1}_{\Theta^{\ast}_{T}>0} ]\leqslant & C \boldsymbol{\kappa} \sum_{k=\lfloor \bar{\boldsymbol{\epsilon}}^{-d} \rfloor +1}^{\lceil \funceta_{1}(\delta) \rceil -1}\frac{(k+1)^{p}}{k^{p+2}}+(\lfloor \bar{\boldsymbol{\epsilon}}^{-d} \rfloor+1)^{p} 
\leqslant C \boldsymbol{\kappa} 2^{p} \frac{\pi^{2}}{6}+(\lfloor \bar{\boldsymbol{\epsilon}}^{-d} \rfloor+1)^{p} .
\end{align*}

\textbf{Step 2.}
 In this part, we focus on the proof of $(\ref{eq:borne:mall_matrice_mom_1})$.  More specifically, we demonstrate that, setting
\begin{align*}
\bar{\boldsymbol{\epsilon}} \coloneqq \min(2^{-\frac{1}{2}},(\frac{T \HormV_{L}(x) }{4(L+1) \BornsupNl})^{13^{L}} ,\mathbf{1}_{L=0}+\mathbf{1}_{L>0}\frac {m_{\ast}}{10}  \vert  \frac{( 2^{20} (1 \vee T^{6}) )^{-143}}{4^{\frac{13}{6}}(N+1)^{\frac{13}{12}(L-1)}} \vert^{13^{L-1}} ),
\end{align*}
for every $\epsilon \in [\funceta_{1}(\delta)^{-\frac{1}{d}},\bar{\boldsymbol{\epsilon}})$,
\begin{align}
\label{req:step_2_preuvermatMal}
\sup_{\xi \in \mathbb{R}^{d}; \vert \xi \vert_{\mathbb{R}^{d}} =1} \mathbb{P}_{x}  ( \xi^{\mbox{\textsc{t}}} \tilde{\MatCovMal}_{X^{\delta}_{T}}  \xi \leqslant & 2 \epsilon, \Theta_{\funceta_{2}(\delta),\EnsT }>0)  \leqslant   \epsilon^{d(p+4)} \HormV_{L}(x)^{-13^{L} 3 d(p+4)}   \\
&\times (1+\mathbf{1}_{\ExpPsiouf  _{2 L +5}>0}  \vert x \vert_{\mathbb{R}^{d}}^{C }   )    \CstePsiouf_{2 L + 5}^{C}      \exp(C(1+T) \MomZ_{C}(Z^{\delta}) \CstePsiouf_{3} ^{4} ) . \nonumber
\end{align}
Note that $\bar{\boldsymbol{\epsilon}}$ is admissible in \textbf{Step 1} since, by (\ref{hyp:hyp_5_loc_var_eq}), it belongs to $(\funceta_{1}(\delta)^{-\frac{1}{d}},2^{-\frac{1}{2}}]$.

 \textbf{Step 2.1.} 
   For every $l \in \{0,\ldots,L\}$ and $\xi \in \mathbb{R}^{d}$, we introduce the $\mathbb{R}_{+}$-valued process $(\TermStepDeuxCovMal^{\xi}_{l,t})_{t \in \pi^{\delta} }$ defined by $\TermStepDeuxCovMal^{\xi}_{l,0}=0$ and for every $t \in \pi^{\delta,\ast} $ by $\TermStepDeuxCovMal^{\xi}_{l,t}= \sum_{\alpha \in \overline{\EnsN}^l} \sum_{i \in \EnsN} \langle \xi, \mathring{X}_{t-\delta}V_{i}^{[\alpha]}(X^{\delta}_{t-\delta},t-\delta)\rangle_{\mathbb{R}^{d}}^{2}$ where we define $\overline{\EnsN}=\EnsN \cup  \{0\}$. Let us denote $N_{l}= (\frac{10}{m_{\ast}})^{13^{-l}}(4^{\frac{14}{13}}N)^{\sum_{j=1}^{l}13^{j-l}}\prod_{j=1}^{l}(N+1)^{(j-1)13^{j-l}}$ for every $l \in \mathbb{N}$.
  Then,  for every $\xi \in \mathbb{R}^{d}$ with $\vert \xi \vert =1$ and every $\epsilon \in [\funceta_{1}(\delta)^{-\frac{1}{d}},1)$,
 \begin{align}
\label{eq:concl_Step22}
 \mathbb{P}_{x}  ( \xi^{\mbox{\textsc{t}}} \tilde{\MatCovMal}_{X^{\delta}_{T}}  \xi \leqslant  &2 \epsilon, \Theta_{\funceta_{2}(\delta),\EnsT }>0)  
\leqslant  \epsilon^{d(p+4)}  (1+\mathbf{1}_{\ExpPsiouf  _{3}>0} \vert x \vert_{\mathbb{R}^{d}}^{C}   ) \CstePsiouf_{3}^{C } \exp(C (1+T) \MomZ_{C}(Z^{\delta}) \CstePsiouf_{3} ^{4})   \\ 
 &+ \sum_{l=0}^{L-1} \mathbb{P}_{x}( \delta \sum_{t \in \EnsT } \TermStepDeuxCovMal^{\xi}_{l,t}< N_{l} \epsilon^{13^{-l}},  \delta \sum_{t \in \EnsT }\TermStepDeuxCovMal^{\xi}_{l+1,t} \geqslant  N_{l+1}\epsilon^{13^{-l-1}}, \Theta_{\funceta_{2}(\delta),\EnsT }>0 ) \nonumber \\
&+  \mathbb{P}_{x} (\delta \sum_{t \in \EnsT } \sum_{l=0}^{L} \TermStepDeuxCovMal^{\xi}_{l,t} <  \epsilon_{L},   \Theta_{\funceta_{2}(\delta),\EnsT }>0) ,\nonumber 
 \end{align}
where $\epsilon_{L}=(L+1)\BornsupNl \epsilon^{13^{-L}}$. 
First, we note that $\xi^{\mbox{\textsc{t}}} \tilde{\MatCovMal}_{X^{\delta}_{T}}  \xi= \delta \sum_{(t,i) \in \EnsT  \times \EnsN}  \chi^{\delta}_{t}\langle  \xi  , \TermStepInitCovMal_{i,t}\rangle_{\mathbb{R}^{d}}^{2}$. Moreover, applying the standard inequality $a^2\geqslant \frac{1}{2}b^{2} - (a-b)^{2}$ with $a= \langle\xi, \TermStepInitCovMal_{i,t} \rangle_{\mathbb{R}^{d}}$ and $b= \langle\xi,   \mathring{X}_{t-\delta}V_{i}(X^{\delta}_{t-\delta},t-\delta)  \rangle_{\mathbb{R}^{d}}$ for $(t,i) \in \EnsT  \times \EnsN$ and distinguishing the cases $(a-b)^{2} \leqslant 2 \epsilon$ and $(a-b)^{2} > 2 \epsilon$, we obtain
   \begin{align*}
\mathbb{P}_{x}  (\delta \sum_{(t,i) \in \EnsT  \times \EnsN}  \chi^{\delta}_{t}\langle  \xi & , \TermStepInitCovMal_{i,t}\rangle_{\mathbb{R}^{d}}^{2} \leqslant   2\epsilon,\Theta_{\funceta_{2}(\delta),\EnsT }>0) 
\leqslant  \mathbb{P}_{x}(\delta \sum_{t \in \EnsT }\chi^{\delta}_{t}   \TermStepDeuxCovMal^{\xi}_{0,t} \leqslant 8 \epsilon, \Theta_{\funceta_{2}(\delta),\EnsT }>0) \\ 
&+  \mathbb{P}_{x}( \delta  \sum_{(t,i) \in \EnsT  \times \EnsN} \langle\xi, \TermStepInitCovMal_{i,t}-  \mathring{X}_{t-\delta}V_{i}(X^{\delta}_{t-\delta},t-\delta)  \rangle_{\mathbb{R}^{d}}^{2} >  2 \epsilon,\Theta_{\funceta_{2}(\delta),\EnsT }>0)  ,
 \end{align*}
 with, by decomposing again into two estimates,
 \begin{align*}
\mathbb{P}_{x}(\delta \sum_{t \in \EnsT }\chi^{\delta}_{t}   \TermStepDeuxCovMal^{\xi}_{0,t} \leqslant 8 \epsilon,\Theta_{\funceta_2(\delta),\EnsT }>0)   \leqslant &  \mathbb{P}_{x}  (\delta \vert \sum_{t \in \EnsT }(\chi^{\delta}_{t}-m_{\ast})   \TermStepDeuxCovMal^{\xi}_{0,t}  \vert \geqslant   2 \epsilon,\Theta_{\funceta_{2}(\delta),\EnsT }>0)   \\
  &+ \mathbb{P}_{x}(\delta \sum_{t \in \EnsT }  \TermStepDeuxCovMal^{\xi}_{0,t}  < \frac{10}{m_{\ast}} \epsilon,\Theta_{\funceta_{2}(\delta),\EnsT }>0) .
 \end{align*}
We focus on the estimate of the second term of the $r.h.s.$ above. Our strategy is to handle this term by combining Lemma \ref{lem:Norris} (which accounts for the appearance of the exponent $13^{-1}$ corresponding to the choice of the parameter $r \in (0,12^{-1})$ in the application of the lemma)  for a process related to $(\TermStepDeuxCovMal^{\xi}_{l,t})_{t \in \pi^{\delta}}$ and assumption $\mathbf{A}_{5}^{\delta}(x,T)$. Recalling that $N_{0}=\frac{10}{m_{\ast}}$,  we derive 
\begin{align*}
\mathbb{P}_{x}&(\delta \sum_{t \in \EnsT }  \TermStepDeuxCovMal^{\xi}_{0,t}  < \frac{10}{m_{\ast}} \epsilon,\Theta_{\funceta_{2}(\delta),\EnsT }>0)  \leqslant  \mathbb{P}_{x}(\bigcap_{l=0}^{L} \delta \sum_{t \in \EnsT }  \TermStepDeuxCovMal^{\xi}_{l,t}<N_{l} \epsilon^{13^{-l}},   \Theta_{\funceta_{2}(\delta),\EnsT }>0)  \\
 & + \sum_{l=0}^{L-1} \mathbb{P}_{x}( \delta \sum_{t \in \EnsT } \TermStepDeuxCovMal^{\xi}_{l,t} < N_{l} \epsilon^{13^{-l}},  \delta \sum_{t \in \EnsT }\TermStepDeuxCovMal^{\xi}_{l+1,t} \geqslant  N_{l+1}\epsilon^{13^{-l-1}}, \Theta_{\funceta_{2}(\delta),\EnsT }>0 ) .
\end{align*}
Moreover,  since $\sup_{l \in \{0,\ldots,L\}} N_{l} \leqslant \BornsupNl$, the first term of the $r.h.s.$ above is bounded by
\begin{align*}
 \mathbb{P} _{x}(\delta \sum_{t \in \EnsT } \sum_{l=0}^{L}\TermStepDeuxCovMal^{\xi}_{l,t} < \epsilon_{L},   \Theta_{\funceta_{2}(\delta),\EnsT }>0) ,
 \end{align*}
which provides the expected bounding terms on $\mathbb{P}_{x}(\delta \sum_{t \in \EnsT }  \TermStepDeuxCovMal^{\xi}_{0,t} < \frac{10}{m_{\ast}} \epsilon,\Theta_{\funceta_{2}(\delta),\EnsT }>0)$. A bound on those terms will be established in the next step.
We now focus on the study of  $\mathbb{P} _{x} (\delta \vert \sum_{t \in \EnsT }(\chi^{\delta}_{t}-m_{\ast})   \TermStepDeuxCovMal^{\xi}_{0,t}  \vert \geqslant  2 \epsilon,\Theta_{\funceta_{2}(\delta),\EnsT }>0)  $.  Our idea is to employ the exponential martingale inequality so, keeping in mind that $\chi^{\delta}_{t}$ follows a Bernoulli distribution with mean $m_{\ast}$, we write

\begin{align*}
\mathbb{P}_{x}  (\delta \vert  \sum_{t \in \EnsT }(\chi^{\delta}_{t}-m_{\ast})   \TermStepDeuxCovMal^{\xi}_{0,t}  \vert \geqslant  2 \epsilon, & \Theta_{\funceta_{2}(\delta),\EnsT }>0)   \leqslant   \mathbb{P}_{x}(\delta^{2} \sum_{t \in \EnsT } \mathbf{1}_{\Theta_{\funceta_{2}(\delta),\EnsT ,t-\delta}>0}  \vert  \TermStepDeuxCovMal^{\xi}_{0,t}  \vert^{2}  \geqslant 2  \epsilon^{2+\frac{1}{22}})  \\
& + \mathbb{P}_{x}  (\delta \vert \sum_{t \in \EnsT }  (\chi^{\delta}_{t}-m_{\ast}) \mathbf{1}_{\Theta_{\funceta_{2}(\delta),\EnsT ,t-\delta}>0} \TermStepDeuxCovMal^{\xi}_{0,t} \vert \geqslant 2 \epsilon ,\\
& \qquad\delta^{2} \sum_{t \in \EnsT }( m_{\ast}(1-m_{\ast}) + (\chi^{\delta}_{t}-m_{\ast}) ^{2})  \mathbf{1}_{\Theta_{\funceta_{2}(\delta),\EnsT ,t-\delta}>0}  \vert  \TermStepDeuxCovMal^{\xi}_{0,t}  \vert^{2} <  2  \epsilon^{2+\frac{1}{22}})    .
 \end{align*}
Notice that the choice of the exponent $\frac{1}{22}$ is specific to our approach in order to ensure that we can estimate the quantity above with the expected bounds when $\epsilon  \geqslant \funceta_{1}(\delta)^{-\frac{1}{d}}$ as described thereafter. Using (\ref{eq:Doob_martin_expo_ineq}), with the martingale defined by $M_{t}=\sum_{w \in (0,t] \cap \pi^{\delta,\ast}} \delta  (\chi^{\delta}_{w}-m_{\ast}) \mathbf{1}_{\Theta_{\funceta_{2}(\delta),\EnsT ,w-\delta}>0} \TermStepDeuxCovMal^{\xi}_{0,w}$ for $t \in \pi^{\delta,\ast}$ and $M_{0}=0$, the second term of the $r.h.s.$ of the inequality above is bounded by $2\exp(-\epsilon^{-\frac{1}{22}})$, itself being bounded by $C \epsilon^{d(p+4)}$. 
To estimate the remaining term, note that $\TermStepDeuxCovMal^{\xi}_{0,t}=\sum_{i \in \EnsN}  \langle \xi ,\mathring{X}_{t-\delta}V_{i}(X^{\delta}_{t-\delta},t-\delta) \rangle_{\mathbb{R}^{d}}^{2}$. Applying Markov's inequality, we obtain, for every $a>0$,
  \begin{align*}
\mathbb{P}_{x} (\delta^{2}  \sum_{t \in \EnsT } \vert &  \sum_{i \in \EnsN} \langle \xi ,\mathring{X}_{t-\delta}V_{i}(X^{\delta}_{t-\delta},t-\delta) \rangle_{\mathbb{R}^{d}}^{2} \vert^{2} \mathbf{1}_{\Theta_{\funceta_{2}(\delta),\EnsT ,t-\delta}>0} \geqslant  2 \epsilon^{\frac{45}{22}}  )  \\
\leqslant & \delta^{a} \epsilon^{-a\frac{45}{22}} C(a) \CstePsiouf_{1}^{4 a} T^{a} \mathbb{E}_{x}[\sup_{t \in \EnsT } \Vert \mathring{X}_{\funceta_{2}(\delta),t-\delta} \Vert_{\mathbb{R}^{d}}^{4a} (1+ \sup_{t \in \EnsT } \vert X^{\delta}_{t-\delta} \vert_{\mathbb{R}^{d}}^{\ExpPsiouf  _{1}})^{4a}  ] .
 \end{align*}
In particular we choose $a=\frac{d(p+4) \ln(\funceta_{1}(\delta))}{-\frac{45}{22}\ln(\funceta_{1}(\delta))-d\ln(\delta)}$ (remember that $\delta \leqslant  \funceta_{1}(\delta)^{-\frac{91}{44d}}$ so that $a \in (0,  44 d(p+4)]$) and apply Lemma \ref{lemme:borne_flt_tangent} (see (\ref{eq:borne_flot_tangent_inv}) that is combined with the H\"older inequality if $4a<2$) and Lemma \ref{lemme:borne:moments_X} to conclude this estimate.
   

Now,  we study $\mathbb{P}_{x}( \delta  \sum_{(t,i) \in \EnsT  \times \EnsN} \langle\xi, \TermStepInitCovMal_{i,t}-  \mathring{X}_{t-\delta}V_{i}(X^{\delta}_{t-\delta},t-\delta)  \rangle_{\mathbb{R}^{d}}^{2} >  2 \epsilon,\Theta_{\funceta_{2}(\delta),\EnsT }>0) $. Recall that, on the set $\{\Theta_{\funceta_{2}(\delta),\EnsT }>0\}$, we have $\vert Z^{\delta,j}_{t} \vert \leqslant \funceta_{2}(\delta)$. Moreover, for $\eta_{2}>1$, let $\mathcal{D}_{\eta_{2}}=\{z\in \mathbb{R}^N, \vert z\vert_{\mathbb{R}^{N}} <\delta^{\frac{1}{2}} \eta_{2}\}$. We fix $(x,t,z,y) \in \mathbb{R}^{d} \times \EnsT  \times \mathcal{D}_{\eta_{2}} \times [0,\delta]$. Using the Taylor expansion thus yields
\begin{align*}
\vert \nabla_{x} \psi^{-1} \partial_{z^{i}} \psi (x,t-\delta,z,y)- V_{i}(x,t-\delta)   \vert_{\mathbb{R}^{d}} \leqslant & \delta^{\frac{1}{2}} \eta_{2} \sup_{\lambda \in [0,1]} \sum_{j \in \EnsN} \vert \partial_{z^{j}}(\nabla_{x} \psi^{-1} \partial_{z^{i}} \psi)(x,t-\delta,\lambda z,\lambda y)  \vert_{\mathbb{R}^{d}} \\
&+y  \sup_{\lambda \in [0,1]} \vert \partial_{y}(\nabla_{x} \psi^{-1} \partial_{z^{i}} \psi )(x,t-\delta,\lambda z,\lambda y) \vert_{\mathbb{R}^{d}} ,
\end{align*}
with $ \partial_{z^{j}}(\nabla_{x} \psi^{-1} \partial_{z^{i}} \psi ) = -\nabla_{x} \psi^{-1} \partial_{z^{j}}\nabla_{x} \psi \nabla_{x} \psi^{-1} \partial_{z^{i}} \psi +  \nabla_{x} \psi^{-1} \partial_{z^{j},z^{i}} \psi .$
Moreover, a similar formula holds for $ \partial_{y}(\nabla_{x} \psi^{-1} \partial_{z^{i}} \psi ) $. The study of this last term is similar to that of $ \partial_{z^{j}}(\nabla_{x} \psi^{-1} \partial_{z^{i}} \psi ) $ and is left to the reader.  Owing to the estimate
\begin{align*}
 \sum_{i,j \in \EnsN} \vert \partial_{z^{j}}(\nabla_{x} \psi^{-1} \partial_{z^{i}} \psi ) \vert_{\mathbb{R}^{d}} \leqslant &\Vert \nabla_{x} \psi^{-1} \Vert_{\mathbb{R}^{d}}^{2}  \sum_{i,j \in \EnsN}   \Vert \partial_{z^{j}}\nabla_{x} \psi  \Vert_{\mathbb{R}^{d}}  \vert \partial_{z^{i}} \psi  \vert_{\mathbb{R}^{d}} + \Vert \nabla_{x} \psi^{-1} \Vert_{\mathbb{R}^{d}}  \sum_{i,j \in \EnsN} \vert \partial_{z^{j},z^{i}} \psi \vert_{\mathbb{R}^{d}} ,
\end{align*}
we derive our estimate by showing that the function $\Vert \nabla_{x} \psi^{-1} \Vert_{\mathbb{R}^{d}}$ is bounded on $\mathbb{R}^{d} \times \EnsT  \times \mathcal{D}_{\eta_{2}} \times [0,\delta]$.  We consider the following decomposition
\begin{align*}
\nabla_{x} \psi^{-1}(x,t-\delta,z,y)=I_{d \times d}-(\nabla_{x} \psi(x,t-\delta,z,y)-I_{d \times d})\nabla_{x} \psi^{-1}(x,t-\delta,z,y).
\end{align*}
Now, assumption $\mathbf{A}_{1}^{\delta}$ (see (\ref{eq:hyp_3_Norme_adhoc_fonction_schema})) implies that (\ref{borne_grad_psi-Id}) holds. It follows that, under the assumption (\ref{hyp:delta_eta_inversibilite}),  for every $(x,t,z,y) \in \mathbb{R}^{d} \times \EnsT  \times \mathcal{D}_{\eta_{2}} \times [0,\delta]$, $ \Vert \nabla_{x} \psi(x,t-\delta,z,y) -I_{d \times d} \Vert_{\mathbb{R}^{d}} \leqslant \frac{1}{2} $ and then $\Vert \nabla_{x} \psi^{-1} \Vert_{\mathbb{R}^{d}}\leqslant 2$. The expected estimate thus follows by also observing that
\begin{align*}
 \sum_{j \in \EnsN}   \Vert \partial_{z^{j}}\nabla_{x} \psi  \Vert_{\mathbb{R}^{d}} \leqslant  & \sum_{j \in \EnsN}  (  \sum_{l=1}^{d}  \vert \partial_{z^{j}}\partial_{x^{l}} \psi \vert_{\mathbb{R}^{d}}^{2} )^{\frac{1}{2}} 
\leqslant   \sum_{j \in \EnsN}  \sum_{l=1}^{d}  \vert \partial_{z^{j}}\partial_{x^{l}} \psi \vert_{\mathbb{R}^{d}} .
\end{align*}
Using similar estimates for the term $ \partial_{y}(\nabla_{x} \psi^{-1} \partial_{z^{i}} \psi ) $ together with $\mathbf{A}_{1}^{\delta}(3)$ (see (\ref{eq:hyp_1_Norme_adhoc_fonction_schema})) and since (\ref{hyp:delta_eta_inversibilite}) holds for $\eta_{2}=\funceta_{2}(\delta)$, we obtain, for every $a \geqslant \frac{1}{2}$,  (with $C$ that may also depend on $a$ in the following)
 \begin{align*}
\mathbb{P}_{x}( & \delta  \sum_{(t,i) \in \EnsT  \times \EnsN} \langle\xi, \TermStepInitCovMal_{i,t}-  \mathring{X}_{t-\delta}V_{i}(X^{\delta}_{t-\delta},t-\delta)  \rangle_{\mathbb{R}^{d}}^{2} >  2 \epsilon,\Theta_{\funceta_{2}(\delta),\EnsT }>0)  \\
\leqslant & C \delta^{a} \funceta_{2}(\delta)^{2a} \epsilon^{-a} \CstePsiouf_{3}^{4 a} T^{a}  \mathbb{E}_{x}[\sup_{t \in \EnsT } \Vert \mathring{X}_{\funceta_{2}(\delta),t-\delta} \Vert_{\mathbb{R}^{d}}^{2a} (1+ \sup_{t \in \EnsT } \vert X^{\delta}_{t-\delta} \vert_{\mathbb{R}^{d}}^{4a \ExpPsiouf  _{3}} )] \\
& +  C \delta^{a} \funceta_{2}(\delta)^{2a} \epsilon^{-a}  \mathbb{E}_{x}[\sup_{t \in \EnsT } \Vert \mathring{X}_{\funceta_{2}(\delta),t-\delta} \Vert_{\mathbb{R}^{d}}^{2a} (\delta  \sum_{t \in \EnsT } \vert Z^{\delta}_{t} \vert_{\mathbb{R}^{N}}^{4 \ExpPsiouf  _{3}} )^{a} ] .
 \end{align*}
Moreover, the H\"older inequality (since $2a \geqslant 1$) yields
\begin{align*}
\mathbb{E}[( \delta  \sum_{t \in \EnsT } \vert Z^{\delta}_{t} \vert_{\mathbb{R}^{N}}^{4 \ExpPsiouf  _{3}} )^{2 a} ] \leqslant  T^{2a-1} \mathbb{E}[\delta  \sum_{t \in \EnsT } \vert Z^{\delta}_{t} \vert_{\mathbb{R}^{N}}^{8 a \ExpPsiouf  _{3}} ] \leqslant T^{2a} \MomZ_{8 a \ExpPsiouf  _{3}} (Z^{\delta}).
\end{align*}
We choose $a=\lceil \frac{d(p+4)\ln(\funceta_{1}(\delta))}{-\ln(\funceta_{1}(\delta))-d\ln(\delta)-2d\ln(\funceta_{2}(\delta))} \rceil $ (remember that $\delta \leqslant  \funceta_{2}(\delta)^{-2}\funceta_{1}(\delta)^{-\frac{2}{d}}$ so that $a \leqslant d(p+4)+1$) and conclude using Cauchy-Schwarz inequality, Lemma \ref{lemme:borne_flt_tangent} (see (\ref{eq:borne_flot_tangent_inv})) and Lemma \ref{lemme:borne:moments_X}. Gathering all the upper bounds together, we obtain (\ref{eq:concl_Step22}).

 \textbf{Step 2.2.} 
Let us show that, for every $\epsilon \in (0,(\frac{T \HormV_{L}(x)}{4(L+1) \BornsupNl})^{13^{L}} ]$,
\begin{align*}
 \mathbb{P}_{x} (\delta  \sum_{t \in \EnsT } \sum_{l=0}^{L} &\TermStepDeuxCovMal^{\xi}_{l,t}  < \epsilon_{L}  , \Theta_{\funceta_{2}(\delta),\EnsT }>0)    
 \leqslant  \epsilon^{d(p+4)} \frac{1+ \mathbf{1}_{\ExpPsiouf  _{4 + 2 L}>0} \vert x \vert_{\mathbb{R}^{d}}^{C } }{ \HormV_{L}(x)^{3d(p+4)13^{L}}}   \CstePsiouf_{2 L +4}^{C }   \exp(C (1+T) \MomZ_{C}(Z^{\delta}) \CstePsiouf_{3} ^{4}) .\nonumber  
\end{align*}
This result is mainly obtained by exploiting the local weak H\"ormander property $\mathbf{A}_{2}(x,L)$. It is based on small-time estimates of $\TermStepDeuxCovMal^{\xi}_{l,t}$ and the fact that $\inf_{\xi \in \mathbb{R}^{d}, \vert \xi \vert_{\mathbb{R}^{d}} =1}  \sum_{l=0}^{L} \TermStepDeuxCovMal^{\xi}_{l,\delta} \geqslant  \HormV_{L}(x)$ when $X^{\delta}_{0}=x$. 

Now let us focus on the proof of \textbf{Step 2.2}. 
If $\epsilon_{L} \leqslant \delta \HormV_L(x)  $, we estimate the probability of an empty set,  and the result is immediate.  We consider $\epsilon_{L} > \delta \HormV_L(x)  $. Let $\EnsLocMatMall := \{\delta,\ldots, \lceil  \frac{4 \epsilon_{L}}{\delta \HormV_L(x)} \rceil\delta \}$, which is well defined under $\mathbf{A}_{2}(x,L)$. Since $\epsilon \leqslant (\frac{T \HormV_{L}(x)} {4(L+1)  \BornsupNl})^{13^{L}}$, we have $\EnsLocMatMall \subset \EnsT $. Therefore,
\begin{align*}
 \mathbb{P}_{x} (\delta \sum_{t \in \EnsT } & \sum_{l=0}^{L} \TermStepDeuxCovMal^{\xi}_{l,t} < \epsilon_{L}  , \Theta_{\funceta_{2}(\delta),\EnsT }>0)  
 \leqslant    \mathbb{P}_{x} (\delta \sum_{t \in \EnsLocMatMall}  \sum_{l=0}^{L}  \TermStepDeuxCovMal^{\xi}_{l,t} < \epsilon_{L} , \Theta_{\funceta_{2}(\delta),\EnsT }>0)  \\
\leqslant & \mathbb{P}_{x} (\frac{1}{2}\delta \vert \EnsLocMatMall \vert \sum_{\Vert \alpha \Vert \leqslant L} \sum_{i \in \EnsN}  \langle \xi ,V_{i}^{[\alpha]}(x,0)\rangle_{\mathbb{R}^{d}}^{2} - \epsilon_{L} \\
 & \qquad  <  \delta \sum_{t \in \EnsLocMatMall} \sum_{\Vert \alpha \Vert \leqslant L} \sum_{i \in \EnsN}   \langle \xi ,\mathring{X}_{t-\delta}V_{i}^{[\alpha]}(X^{\delta}_{t-\delta},t-\delta) - V_{i}^{[\alpha]}(x,0)\rangle_{\mathbb{R}^{d}}^{2},   \Theta_{\funceta_{2}(\delta),\EnsT }>0) \\
  \leqslant & \mathbb{P}_{x} (  \sup_{t \in \EnsLocMatMall}  \sum_{\Vert \alpha \Vert \leqslant L} \sum_{i \in \EnsN}     \langle \xi ,\mathring{X}_{\funceta_{2}(\delta),t-\delta}V_{i}^{[\alpha]}(X^{\delta}_{t-\delta},t-\delta) - V_{i}^{[\alpha]}(x,0)\rangle_{\mathbb{R}^{d}}  ^{2} >  \frac{\HormV_L(x)}{4} ) \\
\leqslant & \mathbb{P}_{x} ( \sup_{t \in \EnsLocMatMall}  \sum_{\Vert \alpha \Vert \leqslant L} \sum_{i \in \EnsN}   M_{\alpha,i,t-\delta} ^{2}  >   \frac{\HormV_L(x)}{8} - \sup_{t \in \EnsLocMatMall}\sum_{\Vert \alpha \Vert \leqslant L} \sum_{i \in \EnsN}     B_{\alpha,i,t-\delta}  ^{2} ) \\
\leqslant &  \sum_{\Vert \alpha \Vert \leqslant L} \sum_{i \in \EnsN} \mathbb{P}_{x} ( \sup_{t \in \EnsLocMatMall}  M_{\alpha,i,t-\delta}^{2}  >  \frac{\HormV_L(x)}{8 ((N+1)^{L+1}-1)} - \sup_{t \in \EnsLocMatMall}   B_{\alpha,i,t-\delta}^{2} )   ,
\end{align*}
with for every $t \in \EnsT $, and using notations introduced in (\ref{eq:Doob_martin_expo_ineq}),
\begin{align*}
M_{\alpha,i,t}=&\delta^{\frac{1}{2}}\sum_{w \in \EnsT ;0<w\leqslant t} \OpeNorrisMart(Y_{\alpha,i})_{w-\delta}, \qquad B_{\alpha,i,t}=\delta\sum_{w \in \EnsT ;0<w\leqslant t} \OpeNorrisDrift(Y_{\alpha,i})_{w-\delta},
\end{align*}
where $Y_{\alpha,i,0}=0$ and for every $t \in \EnsT $,
$
Y_{\alpha,i,t}=   \langle \xi,\mathring{X}_{\funceta_{2}(\delta),t} V_{i}^{[\alpha]}(X^{\delta}_{t},t)-V_{i}^{[\alpha]}(x,0)\rangle_{\mathbb{R}^{d}} .
$
Now we decompose our estimate in the following way
\begin{align*}
 \mathbb{P}_{x} (\delta \sum_{t \in \EnsT } & \sum_{l=0}^{L} \TermStepDeuxCovMal^{\xi}_{l,t} < (L+1) \BornsupNl  \epsilon^{13^{-L}},   \Theta_{\funceta_{2}(\delta),\EnsT }>0) \\
\leqslant &  \sum_{\Vert \alpha \Vert \leqslant L} \sum_{i \in \EnsN}\mathbb{P}_{x}( \sup_{t \in \EnsLocMatMall}   \vert B_{\alpha,i,t-\delta}  \vert^{2} > \frac{\HormV_{L}(x)}{16((N+1)^{L+1}-1)} ,   \Theta_{\funceta_{2}(\delta),\EnsT }>0)  \\
&+ \mathbb{P}_{x} ( \sup_{t \in \EnsLocMatMall}   \vert  M_{\alpha,i,t-\delta} \vert^{2}  >  \frac{\HormV_{L}(x)}{16 ((N+1)^{L+1}-1)}  ,  \Theta_{\funceta_{2}(\delta),\EnsT }>0).
\end{align*}
We study the first term of the $r.h.s.$ above. Using Markov's and H\"older's inequalities,
for every $a\geqslant 1$, we have
\begin{align*}
\mathbb{P}_{x}( \sup_{t \in \EnsLocMatMall} &  B_{\alpha,i,t-\delta} ^{2} > \frac{\HormV_{L}(x)}{16 ((N+1)^{L+1}-1)},\Theta_{\funceta_{2}(\delta),\EnsT }>0 ) \\
\leqslant & \mathbb{P}_{x}(  ( \delta \sum_{t \in \EnsLocMatMall} \vert \OpeNorrisDrift(Y_{\alpha,i})_{t-\delta}\vert )^{2} > \frac{\HormV_{L}(x)}{16 ((N+1)^{L+1}-1)} ,   \Theta_{\funceta_{2}(\delta),\EnsT }>0)  \\
\leqslant & C \delta^{a} \vert \EnsLocMatMall \vert^{a} \sup_{t \in \EnsLocMatMall} \mathbb{E}_{x}[ \vert \OpeNorrisDrift(Y_{\alpha,i})_{t-\delta}\vert^{a} ]   \vert \frac{(N+1)^{L+1}-1}{\HormV_{L}(x)} \vert^{\frac{a}{2}}.
\end{align*}
At this point, we choose $a= d(p+4)13^{L}$ so that $\delta^{a} \vert \EnsLocMatMall \vert^{a} \leqslant C \HormV_{L}(x)^{-a} \epsilon^{d(p+4)}$. Moreover, as a consequence of Lemma \ref{lemme:dvpt_Lie} with $V=V_{i}^{[\alpha]}$ (recall that owing to $\mathbf{A}_{5}^{\delta}(x,T)$ we have $q^{\delta}_{\funceta_{2}(\delta)}   \leqslant C$), using the Cauchy-Schwarz inequality and remarking that $\CstePsiouf_{V_{i}^{[\alpha]},3} \leqslant C(\Vert \alpha \Vert) \CstePsiouf_{4+2 \Vert \alpha \Vert}^{C(\Vert \alpha \Vert)}$ and $\ExpPsiouf  _{V_{i}^{[\alpha]},3} \leqslant C(\Vert \alpha \Vert) \ExpPsiouf  _{4+2 \Vert \alpha \Vert}$,  we derive
\begin{align*}
\mathbb{E}_{x}[ \vert \OpeNorrisDrift(Y_{\alpha,i})_{t-\delta}\vert^{a} ] \leqslant & C \CstePsiouf_{4+2 \Vert \alpha \Vert}^{C} \MomZ_{C}(Z^{\delta}) \mathbb{E}_{x}[  \Vert \mathring{X}_{\funceta_{2}(\delta),t-\delta} \Vert_{\mathbb{R}^{d}}^{2a} ]^{\frac{1}{2}}  (1+ \mathbb{E}_{x}[\vert X^{\delta}_{t-\delta} \vert_{\mathbb{R}^{d}}^{C}]^{\frac{1}{2}}) .
\end{align*}

 To complete our estimate, we bound the $r.h.s.$ above using Lemma \ref{lemme:borne_flt_tangent} (see (\ref{eq:borne_flot_tangent_inv})) and Lemma \ref{lemme:borne:moments_X}.
We consider now the term involving $M_{\alpha,i,t-\delta}$ and first observe that we have the following inequality
\begin{align*}
\mathbb{P}_{x} &( \sup_{t \in \EnsLocMatMall}   M_{\alpha,i,t-\delta}^{2}  >  \frac{\HormV_L(x)}{16 ((N+1)^{L+1}-1)}  ) \leqslant \mathbb{P}_{x}(\delta \sum_{t \in \EnsLocMatMall} \mathbb{E}[  \OpeNorrisMart(M_{\alpha,i})_{t-\delta} ^{2} \vert \mathcal{F}^{Z^{\delta}}_{t-\delta}]+ \OpeNorrisMart(M_{\alpha,i})_{t-\delta}^{2}  \geqslant  \epsilon^{\frac{13^{-L}}{3}}) \\
&+ \mathbb{P}_{x} ( \sup_{t \in \EnsLocMatMall}    M_{\alpha,i,t-\delta}^{2} >  \frac{\HormV_L(x)}{16 ((N+1)^{L+1}-1)}  , \delta \sum_{t \in \EnsLocMatMall} \mathbb{E}_{x}[ \OpeNorrisMart(M_{\alpha,i})_{t-\delta}^{2} \vert \mathcal{F}^{Z^{\delta}}_{t-\delta}]+ \OpeNorrisMart(M_{\alpha,i})_{t-\delta}^{2} < \epsilon^{\frac{13^{-L}}{3}})  .
\end{align*}
Using the exponential martingale inequality (\ref{eq:Doob_martin_expo_ineq}), the second term in the $r.h.s.$ above is bounded by $ 2\exp(-\frac{\HormV_L(x)}{32\epsilon^{\frac{13^{-L}}{3}} ((N+1)^{L+1}-1)})$.
To bound the first term, fix $a\geqslant 1$. Since $\OpeNorrisMart(Y_{\alpha,i})=\OpeNorrisMart(M_{\alpha,i})$, Markov's and H\"older's inequalities yield
\begin{align*}
\mathbb{P}_{x}(\delta \sum_{t \in \EnsLocMatMall} \mathbb{E}[ \OpeNorrisMart(M_{\alpha,i})_{t-\delta}^{2} \vert \mathcal{F}^{Z^{\delta}}_{t-\delta}] + \OpeNorrisMart(M_{\alpha,i})_{t-\delta}^{2}  \geqslant \epsilon^{\frac{13^{-L}}{3}}) \leqslant \delta^{a} \vert \EnsLocMatMall \vert^{a} \epsilon^{-a\frac{13^{-L}}{3}} 2^{a+1} \sup_{t \in \EnsLocMatMall} \mathbb{E}_{x}[\vert \OpeNorrisMart(Y_{\alpha,i})_{t-\delta} \vert^{2a}] .
\end{align*}

At this point, we choose $a =2^{-1}3d(p+4)13^{L} $ so that $ \delta^{a} \vert \EnsLocMatMall \vert^{a}  \epsilon^{-a\frac{13^{-L}}{3}} \leqslant C \HormV_{L}(x)^{-a}  \epsilon^{d(p+4)}$.  In order to bound the $r.h.s.$ above we use Lemma \ref{lemme:dvpt_Lie}  and the Cauchy-Schwarz inequality to derive
\begin{align*}
\mathbb{E}_{x}[ \vert \OpeNorrisMart(Y_{\alpha,i})_{t-\delta} \vert ^{2a} ] \leqslant & C \CstePsiouf_{4+2 \Vert \alpha \Vert}^{C} \MomZ_{C}(Z^{\delta})^{2}   \mathbb{E}_{x}[  \Vert \mathring{X}_{\funceta_{2}(\delta),t-\delta} \Vert_{\mathbb{R}^{d}}^{4a}]^{\frac{1}{2}}  (1+ \mathbb{E}_{x}[\vert X^{\delta}_{t-\delta} \vert_{\mathbb{R}^{d}}^{C}]^{\frac{1}{2}}) ,
\end{align*}
and then use Lemma \ref{lemme:borne_flt_tangent} (see (\ref{eq:borne_flot_tangent_inv})) and Lemma \ref{lemme:borne:moments_X}.  Gathering all the estimates and using standard inequalities between exponential and polynomial functions concludes the proof of \textbf{Step 2.2}.

 \textbf{Step 2.3.} Consider the case $L \in \mathbb{N}^{\ast}$. Let $l \in \{0,\ldots,L-1\}$. 
Let us show that for every 
$
\epsilon \in [ \funceta_{1}(\delta)^{-\frac{1}{d}}, \vert ( 2^{20} (1 \vee T^{6}) )^{-143} (4 N_{l})^{-1}\vert^{13^{l}} ],
$
then  
   \begin{align*}
 \mathbb{P}_{x}( \delta \sum_{t \in \EnsT } \TermStepDeuxCovMal^{\xi}_{l,t} < N_{l} \epsilon^{13^{-l}}, & \delta \sum_{t \in \EnsT }\TermStepDeuxCovMal^{\xi}_{l+1,t} \geqslant N_{l+1}\epsilon^{13^{-l-1}}, \Theta_{\funceta_{2}(\delta),\EnsT }>0 )  \\ 
& \leqslant  \epsilon^{d(p+4)}  (1+ \mathbf{1}_{\ExpPsiouf  _{2 l +7}>0} \vert x \vert_{\mathbb{R}^{d}}^{C }   )  \CstePsiouf_{2 l + 7}^{C}  \exp(C(1+T) \MomZ_{C }(Z^{\delta}) \CstePsiouf_{3} ^{4}) .
 \end{align*}
 First, for $\alpha \in \overline{\EnsN}^{l}$ and $i \in \EnsN$, we introduce the $\mathbb{R}$-valued process $(\tilde{Y}_{\alpha,i,t})_{t \in \pi^{\delta}}$ such that $\tilde{Y}_{\alpha,i,0}=\langle \xi, (V_{i}^{[\alpha]}-\delta V_{i}^{[(\alpha,0)]}-\bar{\mathscr{R}}V_{i}^{[\alpha]})(x,0) \rangle_{\mathbb{R}^{d}}$ and $ \tilde{Y}_{\alpha,i,t}=\langle \xi, \mathring{X}_{\funceta_{2}(\delta),t-\delta}V_{i}^{[\alpha]}(X^{\delta}_{t-\delta},t-\delta)\rangle_{\mathbb{R}^{d}}$, $ t \in \pi^{\delta,\ast}$, and define also
 \begin{align*}
 \hat{Y}_{\alpha,i,t}= \tilde{Y}_{\alpha,i,t}-A_{\alpha,i,t},  \quad t \in \pi^{\delta},
 \end{align*}
where $A_{\alpha,i,t}=\sum_{w \in \pi^{\delta},w < t} a_{\alpha,i,w},$ with $a_{\alpha,i,0}=\langle \xi, \bar{\mathscr{R}}V_{i}^{[\alpha]}(x,0) \rangle_{\mathbb{R}^{d}}$ and $a_{\alpha,i,w}=\langle \xi, \mathring{X}_{\funceta_{2}(\delta),w-\delta} \bar{\mathscr{R}} V_{i}^{[\alpha]}(X^{\delta}_{w-\delta},w-\delta)\rangle_{\mathbb{R}^{d}}$, $w \in \pi^{\delta,\ast}$. In particular, on the set $\{\Theta_{\funceta_{2}(\delta),\EnsT }>0  \}$, we have $ \TermStepDeuxCovMal^{\xi}_{l,t}= \sum_{\alpha \in  \overline{\EnsN}^l} \sum_{i \in \EnsN}  \tilde{Y}_{\alpha,i,t} ^{2}$, $t \in \pi^{\delta,\ast}$. Furthermore, Lemma \ref{lemme:dvpt_Lie}, applied with $V=V_{i}^{[\alpha]}$, yields 
for $t \in \pi^{\delta,\ast}$,
\begin{align*}
\hat{Y}_{\alpha,i,t+\delta}-\hat{Y}_{\alpha,i,t} 
=&  \delta^{\frac{1}{2}}\sum_{j \in \EnsN}Z^{\delta,j}_{t} \tilde{Y}_{(\alpha,j),i,t} +\delta \tilde{Y}_{(\alpha,0),i,t}+ \langle \xi , \mathring{X}_{\funceta_{2}(\delta),t-\delta} \tilde{\mathscr{R}}V_{i}^{[\alpha]}(X^{\delta}_{t-\delta},t-\delta,Z^{\delta}_{t}) \rangle_{\mathbb{R}^{d}}  .
\end{align*}
It follows that, $\OpeNorrisMart(\hat{Y}_{\alpha,i})_{0}=0$ and $ \OpeNorrisDrift(\hat{Y}_{\alpha,i})_{0} =\langle \xi, V_{i}^{[(\alpha,0)]}(x,0) \rangle_{\mathbb{R}^{d}}$, for $t \in \pi^{\delta,\ast}$, $\OpeNorrisMart(\hat{Y}_{\alpha,i})_{t} =\sum_{j \in \EnsN}Z^{\delta,j}_{t} \tilde{Y}_{(\alpha,j),i,t} + \delta^{-\frac{1}{2}} \langle \xi , \mathring{X}_{\funceta_{2}(\delta),t-\delta} \tilde{\mathscr{R}}V_{i}^{[\alpha]}(X^{\delta}_{t-\delta},t-\delta,Z^{\delta}_{t}) \rangle_{\mathbb{R}^{d}}$,  $ \OpeNorrisDrift(\hat{Y}_{\alpha,i})_{t} = \tilde{Y}_{(\alpha,0),i,t}$ and then,
\begin{align*} \TermStepDeuxCovMal^{\xi}_{l+1,t}= & \sum_{\alpha \in  \overline{\EnsN}^l} \sum_{i \in \EnsN} \mathbb{E}_{x}[ ( \OpeNorrisMart(\hat{Y}_{\alpha,i})_{t} - \delta^{-\frac{1}{2}} \langle \xi , \mathring{X}_{\funceta_{2}(\delta),t-\delta} \tilde{\mathscr{R}}V_{i}^{[\alpha]}(X^{\delta}_{t-\delta},t-\delta,Z^{\delta}_{t}) \rangle_{\mathbb{R}^{d}} )^{2} \vert \mathcal{F}^{Z^{\delta}}_{t-\delta}] + \OpeNorrisDrift(\hat{Y}_{\alpha,i})_{t}^{2}.
\end{align*}
Therefore, noting that $(\hat{Y}_{\alpha,i,t})_{t \in \pi^\delta}$ is progressively measurable $w.r.t.$ the filtration $(\mathcal{G}_{t})_{t \in \pi^\delta}$ with $\mathcal{G}_{t}=\mathcal{F}^{Z^{\delta}}_{t-\delta}$, for $t \in \pi^{\delta,\ast}$ and $\mathcal{G}_{0}=\{\emptyset,\Omega\}$, we decompose our estimate in the following way 
 \begin{align*}
 \mathbb{P}_{x}( &\delta \sum_{t \in \EnsT } \TermStepDeuxCovMal^{\xi}_{l,t} < N_{l} \epsilon^{13^{-l}},  \delta \sum_{t \in \EnsT }\TermStepDeuxCovMal^{\xi}_{l+1,t} \geqslant N_{l+1}\epsilon^{13^{-l-1}}, \Theta_{\funceta_{2}(\delta),\EnsT }>0 ) \\
  \leqslant &  \mathbb{P}_{x}(  \delta \sum_{t \in \EnsT} \sum_{\alpha \in \overline{\EnsN}^l} \sum_{i \in \EnsN}   A_{\alpha,i,t}^{2} \geqslant   N_{l} \epsilon^{13^{-l}}) + \mathbb{P}_{x}( \delta \sum_{t \in \EnsT} \sum_{\alpha \in \overline{\EnsN}^l} \sum_{i \in \EnsN}  \hat{Y}_{\alpha,i,t}^{2} < 4 N_{l} \epsilon^{13^{-l}},\\
&\delta \sum_{t \in \EnsT }\sum_{\alpha \in  \overline{\EnsN}^l} \sum_{i \in \EnsN} \mathbb{E}_{x}[(\OpeNorrisMart(\hat{Y}_{\alpha,i})_{t} - \delta^{-\frac{1}{2}} \langle \xi , \mathring{X}_{\funceta_{2}(\delta),t-\delta} \tilde{\mathscr{R}}V_{i}^{[\alpha]}(X^{\delta}_{t-\delta},t-\delta,Z^{\delta}_{t}) \rangle_{\mathbb{R}^{d}} )^{2} \vert \mathcal{F}^{Z^{\delta}}_{t-\delta}] \\
& \qquad \qquad  +  \OpeNorrisDrift(\hat{Y}_{\alpha,i})_{t}^{2} \geqslant N_{l+1}\epsilon^{13^{-l-1}})    \\
  \leqslant &  \mathbb{P}_{x}(  \delta \sum_{t \in \EnsT} \sum_{\alpha \in \overline{\EnsN}^l} \sum_{i \in \EnsN}   A_{\alpha,i,t}^{2} \geqslant  N_{l} \epsilon^{13^{-l}})  \\
  &+ \sum_{\alpha \in  \overline{\EnsN}^l} \sum_{i \in \EnsN}  \mathbb{P}_{x}( \delta \sum_{t \in \EnsT }  \hat{Y}_{\alpha,i,t}^{2} < 4 N_{l} \epsilon^{13^{-l}},  \delta \sum_{t \in \EnsT } \mathbb{E}_{x}[\OpeNorrisMart(\hat{Y}_{\alpha,i})_{t-\delta} ^{2}\vert   \mathcal{G}_{t-\delta}] + \OpeNorrisDrift(\hat{Y}_{\alpha,i})_{t-\delta}  ^{2} \geqslant  \frac{ N_{l+1}}{4 N(N+1)^{l}} \epsilon^{13^{-l-1}})  \\
& +   \sum_{\alpha \in  \overline{\EnsN}^l} \sum_{i \in \EnsN}  \mathbb{P}_{x}(\delta \sum_{t \in \EnsT } \mathbb{E}_{x}[\vert \delta^{-\frac{1}{2}} \langle \xi , \mathring{X}_{\funceta_{2}(\delta),t-\delta} \tilde{\mathscr{R}}V_{i}^{[\alpha]}(X^{\delta}_{t-\delta},t-\delta,Z^{\delta}_{t}) \rangle_{\mathbb{R}^{d}} \vert^{2} \vert   \mathcal{F}^{Z^{\delta}}_{t-\delta} ]+ \delta  \mathbb{E}_{x}[\OpeNorrisMart(\hat{Y}_{\alpha,i})_{T} ^{2}\vert     \mathcal{F}^{Z^{\delta}}_{T-\delta} ]\\
& \qquad \qquad   + \delta \OpeNorrisDrift(\hat{Y}_{\alpha,i})_{T}^{2}> (4N (N+1)^{l})^{-1}  N_{l+1}\epsilon^{13^{-l-1}}) .
 \end{align*}
We bound the second term of the $r.h.s.$ above. Since $N_{l+1} = 4N(N+1)^{l} (4N_{l})^{\frac{1}{13}}$ and $N_{l}\in [1,\BornNl]$, we have $4 N_{l} \epsilon^{13^{-l}} \in  [ ( 2^{10} (1 \vee T^{3}) \delta ) ^{\frac{44}{91-\frac{36}{13}}},( 2^{20} (1\vee T^{6}) )^{-143} ]$ and the expected bound is obtained by applying Lemma \ref{lem:Norris} with $Y=\hat{Y}_{\alpha,i}$, $(\mathcal{F}_{t})_{t \in \pi^{\delta}}=(\mathcal{G}_{t})_{t \in \pi^{\delta}}$, $r=13^{-1}$, and $\epsilon$ and $p$ replaced by $4 N_{l} \epsilon^{13^{-l}}$ and $13^{l}d(p+4)$, respectively. 
Standard inequalities for polynomial and exponential functions yield the expected factor $\epsilon^{d(p+4)}$. It remains to show that $\FuncNorris_{\hat{Y}_{\alpha,i},\EnsT}(q(d,l,p))$, defined in (\ref{eq:Quantite_lemme_Norris}), is bounded, where $q(d,l,p)=13^{l+1}176d(p+4)$. 
Moreover, we notice that $\OpeNorrisMart( \OpeNorrisDrift(\hat{Y}_{\alpha,i}))_{0}= \OpeNorrisDrift( \OpeNorrisDrift(\hat{Y}_{\alpha,i}))_{0} =0$. For $t \in \pi^{\delta,\ast}$,  as a consequence of Lemma \ref{lemme:dvpt_Lie},
\begin{align*}
\OpeNorrisDrift(\OpeNorrisDrift(\hat{Y}_{\alpha,i}))_{t}= &\tilde{Y}_{(\alpha,0,0),i,t} + \delta^{-1}\langle \xi,  \mathring{X}_{\funceta_{2}(\delta),t-\delta} \bar{\mathscr{R}}V_{i}^{[(\alpha,0)]}(X^{\delta}_{t-\delta},t-\delta)  \rangle_{\mathbb{R}^{d}}
\end{align*}
and
\begin{align*}
\OpeNorrisMart(\OpeNorrisDrift(\hat{Y}_{\alpha,i}))_{t}= &\sum_{j \in \EnsN} Z^{\delta,j}_{t}\tilde{Y}_{(\alpha,0,j),i,t} + \delta^{-\frac{1}{2}}\langle \xi,  \mathring{X}_{\funceta_{2}(\delta),t-\delta} \tilde{\mathscr{R}}V_{i}^{[(\alpha,0)]}(X^{\delta}_{t-\delta},t-\delta,Z^{\delta}_{t})  \rangle_{\mathbb{R}^{d}} 
\end{align*}

Applying (\ref{eq:borne_reste_mart_dvpt_Lie}) and (\ref{eq:borne_reste_moy_dvpt_Lie}) together with $\mathbf{A}^{\delta}_{1}(2 L + 5)$ and $\mathbf{A}_{3}^{\delta}(+\infty)$, we obtain after computations,
\begin{align*}
 \FuncNorris_{\hat{Y}_{\alpha,i},\EnsT}(q(d,l,p)) \leqslant & C \MomZ_{C}(Z^{\delta})  \CstePsiouf_{2 l + 7}^{C}   \mathbb{E}_{x}[\sup_{t \in \EnsT } \Vert \mathring{X}_{\funceta_{2}(\delta),t-\delta} \Vert_{\mathbb{R}^{d}}^{2q(d,l,p)}]^{\frac{1}{2}} (1+ \mathbb{E}_{x}[\sup_{t \in \EnsT } \vert X^{\delta}_{t-\delta} \vert_{\mathbb{R}^{d}}^{C}]^{\frac{1}{2}}).
\end{align*}
Finally, using the Markov and Cauchy-Schwarz inequalities together with Lemma \ref{lemme:dvpt_Lie} gives also, for every $a,b \in (0,2d(p+4)]$,
\begin{align*}
\mathbb{P}_{x}(\delta & \sum_{t \in \EnsT } \mathbb{E}_{x}[ \delta^{-1} \langle \xi , \mathring{X}_{\funceta_{2}(\delta),t-\delta} \tilde{\mathscr{R}}V_{i}^{[\alpha]}(X^{\delta}_{t-\delta},t-\delta,Z^{\delta}_{t}) \rangle_{\mathbb{R}^{d}} ^{2} \vert   \mathcal{F}^{Z^{\delta}}_{t-\delta} ]+ \delta  \mathbb{E}_{x}[\OpeNorrisMart(\hat{Y}_{\alpha,i})_{T} ^{2}\vert   \mathcal{F}^{Z^{\delta}}_{T-\delta}  ]\\
& \quad  +  \delta  \OpeNorrisDrift(\hat{Y}_{\alpha,i})_{T}  ^{2} \geqslant (4N(N+1)^{l})^{-1}  N_{l+1}\epsilon^{13^{-l-1}})  \\
+ \mathbb{P}_{x} & (   \delta \sum_{t \in \EnsT} \sum_{\alpha \in \overline{\EnsN}^l} \sum_{i \in \EnsN}   A_{\alpha,i,t}^{2} \geqslant  N_{l} \epsilon^{13^{-l}}) \\
\leqslant & C\delta^{\frac{a}{2}} \epsilon^{-a 13^{-l-1}} (1\vee T^{a}) \MomZ_{C}(Z^{\delta})   \CstePsiouf_{2l+ 4}^{C }  
  \mathbb{E}_{x}[\sup_{t \in \EnsT } \Vert \mathring{X}_{\funceta_{2}(\delta),t-\delta} \Vert_{\mathbb{R}^{d}}^{2 a}]^{\frac{1}{2}} (1+ \mathbb{E}_{x}[\sup_{t \in \EnsT } \vert X^{\delta}_{t-\delta} \vert_{\mathbb{R}^{d}}^{2 a C}]^{\frac{1}{2}})  \\
 & +C\delta^{\frac{b}{2}} \epsilon^{-\frac{b}{2} 13^{-l}} (1\vee T^{\frac{3}{2}b}) \MomZ_{C}(Z^{\delta})   \CstePsiouf_{2l+ 4}^{C }  
  \mathbb{E}_{x}[\sup_{t \in \EnsT } \Vert \mathring{X}_{\funceta_{2}(\delta),t-\delta} \Vert_{\mathbb{R}^{d}}^{2 b}]^{\frac{1}{2}} (1+ \mathbb{E}_{x}[\sup_{t \in \EnsT } \vert X^{\delta}_{t-\delta} \vert_{\mathbb{R}^{d}}^{2 b C}]^{\frac{1}{2}})   ,
 \end{align*}
In particular, we choose $a = \frac{d(p+4) \ln(\funceta_{1}(\delta))}{-13^{-l-1}\ln(\funceta_{1}(\delta))-\frac{d}{2}\ln(\delta)}$ so that $\delta^{\frac{a}{2}} \epsilon^{-a 13^{-l-1}} \leqslant \epsilon^{d(p+4)}$ (remember that $\delta \leqslant \funceta_{1}(\delta)^{-\frac{2}{d}}$ so $a \in (0,  2d(p+4) ]$) and $b=\frac{2d(p+4) \ln(\funceta_{1}(\delta))}{-13^{-l}\ln(\funceta_{1}(\delta))-d\ln(\delta)}$ (in particular $b \in (0,  2d(p+4) ]$) and then apply Lemma \ref{lemme:borne_flt_tangent} (see (\ref{eq:borne_flot_tangent_inv}) and apply first the H\"older inequality if $a<1$ or $b<1$) and Lemma \ref{lemme:borne:moments_X} to conclude the proof of \textbf{Step 2.3}.

\textbf{Step 2.4}. To conclude the proof of \textbf{Step 2} we simply gather the estimates obtained in \textbf{Step 2.1}, \textbf{2.2} and \textbf{2.3}.\\

 \textbf{Step 3.} We now focus on the proof of (\ref{eq:borne:mall_matrice_mom_2}). In particular, we show that for every $\epsilon >0$,
  \begin{align}
\label{eq:Step3CovMal}
   \mathbb{P}_{x}(\Vert \tilde{\MatCovMal}_{X^{\delta}_{T}}  \Vert_{\mathbb{R}^{d}}> \frac{1}{6 \epsilon},\Theta_{\funceta_{2}(\delta),\EnsT }>0)  
   \leqslant & \epsilon^{d(p+2)} (1+ \mathbf{1}_{\ExpPsiouf  _{3} > 0} \vert x \vert_{\mathbb{R}^{d}}^{C})\CstePsiouf_{3} ^{C}    \exp (C(T+1)\MomZ_{C}(Z^{\delta})\CstePsiouf_{3}^{4})  .
 \end{align}
First, the Cauchy-Schwarz inequality gives
\begin{align*}
\Vert \tilde{\MatCovMal}_{X^{\delta}_{T}}  \Vert_{\mathbb{R}^{d}} \leqslant & \Vert \MatCovMal_{X^{\delta}_{T}} \Vert_{\mathbb{R}^{d}}  \Vert \mathring{X}_{T} \Vert_{\mathbb{R}^{d}}^{2}
\leqslant  \vert X^{\delta}_{T} \vert_{\mathbb{R}^{d} ,1,1}^{2} \Vert \mathring{X}_{T} \Vert_{\mathbb{R}^{d}}^{2}.
\end{align*}
Applying Markov's and Cauchy-Schwarz inequalities, we obtain
  \begin{align*}
   \mathbb{P}_{x}(\Vert \tilde{\MatCovMal}_{X^{\delta}_{T}}  \Vert_{\mathbb{R}^{d}}> \frac{1}{6 \epsilon}, \Theta_{\funceta_{2}(\delta),\EnsT }>0)  
   \leqslant & C \epsilon^{d(p+2)}  \mathbb{E}_{x}[\vert X^{\delta}_{T} \vert_{\mathbb{R}^{d},1,1}^{4d(p+2)}]^{\frac{1}{2}}\mathbb{E}_{x}[\Vert \mathring{X}_{T}  \Vert ^{4d(p+2)}_{\mathbb{R}^{d}} \mathbf{1}_{\Theta_{\funceta_{2}(\delta),\EnsT}>0}]^{\frac{1}{2}} .
 \end{align*}

To conclude the proof of \textbf{Step 3}, we then apply Proposition \ref{prop:borne_Sob_generique} (see (\ref{eq:norme_Sobolev_X_theo})) and Lemma \ref{lemme:borne_flt_tangent} and obtain (\ref{eq:Step3CovMal}).\\

 \textbf{Step 4.} In order to complete the proof of Theorem \ref{th:borne_Lp_inv_cov_Mal}, it remains to show that (\ref{eq:th_cov_mal_proba_espace_comp}) holds. As in \textbf{Step 1}, we have 
 \begin{align*}
\mathbb{P}_{x}(\vert \det \tilde{\InvMatCovMal}
_{X^{\delta}_{T}} \vert  \geqslant \frac{\funceta_{1}(\delta)}{2},\Theta^{\ast}_{T}>0)   
\leqslant & \mathbb{P}_{x}(\vert \det \tilde{\MatCovMal}
_{X^{\delta}_{T}} \vert  \leqslant 2 \funceta_{1}(\delta)^{-1},\Theta_{\funceta_{2}(\delta),\EnsT }>0)  \\
\leqslant & \mathbb{P}_{x}(\inf_{\xi \in \mathbb{R}^{d}; \vert \xi \vert_{\mathbb{R}^{d}} =1} \xi^{\mbox{\textsc{t}}} \tilde{\MatCovMal}_{X^{\delta}_{T}}  \xi \leqslant 2^{\frac{1}{d}} \funceta_{1}(\delta)^{-\frac{1}{d}},\Theta_{\funceta_{2}(\delta),\EnsT }>0) .
\end{align*}

Applying Lemma \ref{lemme:inversion_sup_Proba} with $\Sigma=\tilde{\MatCovMal}_{X^\delta_T}$ and $\epsilon=2^{\frac{d+1}{d}}\funceta_{1}(\delta)^{-\frac{1}{d}}$ which remains valid when all probabilities are intersected with the localization event $\{\Theta_{\funceta_{2}(\delta),\EnsT}>0\}$ by the same covering argument, and then using the estimates obtained in \textbf{Step 2} (see (\ref{req:step_2_preuvermatMal})) and \textbf{Step 3} (with $p$ replaced by $p-4$ and $p-2$, respectively), we obtain, for $p \geqslant 4$, 
 \begin{align*}
\mathbb{P}_{x}(\vert \det \tilde{\MatCovMal}
_{X^{\delta}_{T} } \vert  \leqslant  \frac{2}{\funceta_{1}(\delta)}, & \Theta_{\funceta_{2}(\delta),\EnsT }>0)  
\leqslant   \funceta_{1}(\delta)^{-p+2} (1+ \HormV_{L}(x)^{-13^{L} 3 dp}  )\\
& \times  (1+ \mathbf{1}_{\ExpPsiouf  _{2 L +5}>0}  \vert x \vert_{\mathbb{R}^{d}}^{C }   )    \CstePsiouf_{2 L + 5}^{C}  \exp(C(1+ T) \MomZ_{C}(Z^{\delta}) \CstePsiouf_{3} ^{4}) .\nonumber
\end{align*}  
To conclude the proof, we simply observe that,  since $ \funceta_{1}(\delta), \funceta_{2}(\delta)>1$,

\begin{align*}
\mathbb{P}_{x}(\Theta^{\ast}_{T}< 1  ) \leqslant &  \mathbb{P}_{x}(\vert \det \tilde{\MatCovMal}_{X^{\delta}_{T}}  \vert \leqslant (\funceta_{1}(\delta)-\frac{1}{2})^{-1}, \Theta_{\funceta_{2}(\delta),\EnsT }>0)+ \sum_{t \in \EnsT }\mathbb{P}(\vert Z^{\delta}_{t} \vert_{\mathbb{R}^{N}}  \geqslant \funceta_{2}(\delta)-\frac{1}{2})\\
\leqslant & \mathbb{P}_{x}(\vert \det \tilde{\MatCovMal}_{X^{\delta}_{T}}  \vert \leqslant \frac{2}{\funceta_{1}(\delta)}  , \Theta_{\funceta_{2}(\delta),\EnsT }>0)+  \sum_{t \in \EnsT }\mathbb{P}(\vert Z^{\delta}_{t} \vert_{\mathbb{R}^{N}}  \geqslant \frac{\funceta_{2}(\delta)}{2}) ,
\end{align*}
and apply Markov's inequality.

\end{proof}

\begin{appendix}
%

\section{Proof of technical lemmas}
\label{Append}
\subsection{Proof of Lemma \ref{lemme:inversion_sup_Proba}}
 \begin{proof}
 First notice that, since $\epsilon \in (0,\sqrt{2})$, there exists $\{\xi_1,\ldots,\xi_{N(\epsilon)}\}$, $\xi_i \in \mathbb{R}^{d}$ with $\vert \xi_{i} \vert_{\mathbb{R}^{d}}=1$, $N(\epsilon) \leqslant (1+\frac{4}{\epsilon^{2}})^{d} \leqslant 6^{d}\epsilon^{-2d}$ (see $e.g.$ Corollary 4.2.13 in \cite{Vershynin_2018}) such that $\{ \xi \in \mathbb{R}^{d}, \vert \xi \vert_{\mathbb{R}^{d}} =1 \}   \subset \cup_{i=1}^{N(\epsilon)} \{ \xi \in \mathbb{R}^{d}, \vert \xi_i-\xi \vert_{\mathbb{R}^{d}} \leqslant \frac{\epsilon^{2}}{2} \}$. Moreover 
 \begin{align*}
  \mathbb{P}( \inf_{\xi \in \mathbb{R}^{d}; \vert \xi \vert_{\mathbb{R}^{d}} =1} \xi^{\mbox{\textsc{t}}} \Sigma \xi \leqslant \frac{1}{2} \epsilon)  \leqslant  & \mathbb{P}( \inf_{\xi \in \mathbb{R}^{d}; \vert \xi \vert_{\mathbb{R}^{d}} =1} \xi^{\mbox{\textsc{t}}} \Sigma \xi \leqslant \frac{1}{2} \epsilon, \Vert \Sigma \Vert_{\mathbb{R}^{d}}  \leqslant \frac{1}{3 \epsilon}) +\mathbb{P}( \Vert \Sigma \Vert_{\mathbb{R}^{d}}  > \frac{1}{3 \epsilon}).
 \end{align*}

 In particular for every $\xi \in \mathbb{R}^{d}, \vert \xi \vert_{\mathbb{R}^{d}} =1$,
 \begin{align*}
 \xi^{\mbox{\textsc{t}}} \Sigma \xi =& \xi_{i}^{\mbox{\textsc{t}}} \Sigma \xi_{i} + (\xi-\xi_{i})^{\mbox{\textsc{t}}} \Sigma( \xi_{i} +\xi  ) \\
\geqslant & \xi_{i}^{\mbox{\textsc{t}}} \Sigma \xi_{i} - 2\vert \xi_i-\xi \vert_{\mathbb{R}^{d}}  \Vert \Sigma \Vert_{\mathbb{R}^{d}} -\vert \xi_i-\xi \vert_{\mathbb{R}^{d}}^{2}  \Vert \Sigma   \Vert_{\mathbb{R}^{d}}. 
 \end{align*}
 Therefore, choosing $i$ such that $\vert \xi_i-\xi \vert_{\mathbb{R}^{d}} \leqslant \frac{\epsilon^{2}}{2}$, on the set $\{\Vert \Sigma\Vert_{\mathbb{R}^{d}} \leqslant \frac{1}{3 \epsilon}\}$, $\xi^{\mbox{\textsc{t}}} \Sigma \xi \leqslant \frac{1}{2}\epsilon$ implies $\xi_{i}^{\mbox{\textsc{t}}} \Sigma \xi_{i} \leqslant \frac{\epsilon}{2}+\frac{\epsilon}{3}+\frac{\epsilon^{3}}{12} \leqslant \epsilon$ since $\epsilon^{2} \leqslant 2$, so that

 \begin{align*}
 \mathbb{P}( \inf_{\xi \in \mathbb{R}^{d}; \vert \xi \vert_{\mathbb{R}^{d}} =1} \xi^{\mbox{\textsc{t}}} \Sigma \xi \leqslant \frac{1}{2} \epsilon,   \Vert \Sigma\Vert_{\mathbb{R}^{d}} \leqslant \frac{1}{3 \epsilon}) \leqslant &  \mathbb{P}( \cup_{i=1}^{N(\epsilon)} \xi_{i}^{\mbox{\textsc{t}}} \Sigma \xi_{i} \leqslant \epsilon) ,
 \end{align*}
 and the proof of (\ref{eq:inversion_sup_Proba}) is completed by taking $C(d)=6^{d}$.
 \end{proof}

\subsection{Proof of Lemma \ref{lemme:borne_flt_tangent}}

In this proof, we are going to use the Burkholder inequality (see (\ref{eq:burkholder_inequality})) on the Hilbert space $(\mathbb{R}^{d \times d}, \langle , \rangle_{\mathbb{R}^{d \times d}})$, with the scalar product defined by $\langle M , M^{\prime} \rangle_{\mathbb{R}^{d \times d}} :=    \mbox{Trace}(M^{\prime} M^{\mbox{\textsc{t}}})=\sum_{i=1}^{d}( M^{\prime}M^{\mbox{\textsc{t}}})_{i,i}$, $M,M^{\prime} \in \mathbb{R}^{d \times d}$. Recall that for $M \in \mathbb{R}^{d \times d}$, $\Vert M \Vert_{\mathbb{R}^{d}} \leqslant \vert M \vert_{ \mathbb{R}^{d \times d}}$.
 
\begin{proof}

\textbf{Step 1.}
First we show that
\begin{align*}
\mathbb{E}_{x}[\sup_{t \in \EnsT }  \vert \mathring{X}_{\eta_{2},t}  \vert_{\mathbb{R}^{d \times d}}^{p} ]^{\frac{1}{p}} \leqslant & \sqrt{d}+ \mathbb{E}_{x}[\sup_{t \in \EnsT } \vert \sum_{w \in \pi^{\delta} \cap (0,t]} \hat{\Upsilon}_w \vert_{\mathbb{R}^{d \times d}}^{p}]^{\frac{1}{p}} + \mathbb{E}_{x}[\sup_{t \in \EnsT } \vert \sum_{w \in \pi^{\delta} \cap (0,t]} \tilde{\Upsilon}_w \vert_{\mathbb{R}^{d \times d}}^{p} ]^{\frac{1}{p}} ,
\end{align*}
where we have introduced $\Upsilon_t =  \mathbf{1}_{\Theta_{\eta_{2},\EnsT ,t}>0} \mathring{X}_{t-\delta} (I_{d \times d}-\nabla_{x}\psi^{-1})(X^{\delta}_{t-\delta},t-\delta,\delta^{\frac{1}{2}}Z^{\delta}_{t},\delta)$, $\hat{\Upsilon}_t= \mathbb{E}_{x}[ \Upsilon_{t} \vert \mathcal{F}^{Z^{\delta}}_{t-\delta}]$ and $\tilde{\Upsilon}_{t}=\Upsilon_{t}-\hat{\Upsilon}_{t}$, $t \in \pi^{\delta,\ast}$. On the set $\{\Theta_{\eta_{2},\EnsT ,t}>0\}$, owing to (\ref{eq:inverse_tangent_flot}), we have
\begin{align*}
\mathring{X}_{t}=&I_{d \times d} - \sum_{w \in \pi^{\delta} \cap (0,t]}\mathring{X}_{w-\delta} (I_{d \times d}-\nabla_{x} \psi^{-1}(X^{\delta}_{w-\delta},w-\delta,\delta^{\frac{1}{2}}Z^{\delta}_w,\delta)) .
\end{align*}

Now, the triangle inequality yields
\begin{align*}
\mathbb{E}_{x}[\sup_{t \in \EnsT } \vert \mathring{X}_{\eta_{2},t}  \vert_{\mathbb{R}^{d \times d}}^{p}  ]^{\frac{1}{p}}
 \leqslant & \sqrt{d}+\mathbb{E}[\sup_{t \in \EnsT } \vert \sum_{w \in \pi^{\delta} \cap (0,t]} \Upsilon_w \vert_{\mathbb{R}^{d \times d}}^{p}  ]^{\frac{1}{p}} ,
\end{align*}
and, using the triangle inequality once again, the proof of \textbf{Step 1} is completed.

\textbf{Step 2.} Let us show that, for $t \in \EnsT $,
\begin{align*}
\vert \hat{\Upsilon}_t \vert_{\mathbb{R}^{d \times d}}\leqslant &\delta \vert\mathring{X}_{\eta_{2},t-\delta} \vert_{\mathbb{R}^{d \times d}} C\CstePsiouf_{3} ^{2} \MomZ_{q^{\delta}_{\eta_{2}} \vee (2\ExpPsiouf_{3}  +2)}(Z^{\delta}) ,
\end{align*}

with $C$ a universal constant. We begin by noticing that, since $ \mathbf{1}_{\Theta_{\eta_{2},\EnsT ,t}>0}=\mathbf{1}_{\delta^{\frac{1}{2}}Z^{\delta}_{t} \in \mathcal{D}_{\eta_{2}}} \mathbf{1}_{\Theta_{\eta_{2},\EnsT ,t-\delta}>0}$ (with  $\mathcal{D}_{\eta_{2}}=\{z\in \mathbb{R}^N, \vert z\vert_{\mathbb{R}^{N}} < \delta^{\frac{1}{2}} \eta_{2}\}$ introduced in the proof of Theorem \ref{th:borne_Lp_inv_cov_Mal}), for every $t \in \pi^{\delta,\ast}$,
\begin{align*}
\vert \hat{\Upsilon}_t \vert_{\mathbb{R}^{d \times d}}= & \vert\mathring{X}_{t-\delta} \mathbf{1}_{\Theta_{\eta_{2},\EnsT ,t-\delta}>0} \mathbb{E}[( I_{d \times d}-\nabla_{x}\psi^{-1})(X^{\delta}_{t-\delta},t-\delta,\delta^{\frac{1}{2}}Z^{\delta}_{t},\delta)  \mathbf{1}_{\delta^{\frac{1}{2}}Z^{\delta}_{t} \in \mathcal{D}_{\eta_{2}}}\vert \mathcal{F}^{Z^{\delta}}_{t-\delta}] \vert_{\mathbb{R}^{d \times d}} .
\end{align*}

Now, we remark that
\begin{align*}
\mathbb{E}_{x}[\delta^{\frac{1}{2}}  \sum_{l \in \EnsN}   Z^{\delta,l}_{t} \nabla_{x} V_{l}(X^{\delta}_{t-\delta},t-\delta) (\mathbf{1}_{\delta^{\frac{1}{2}}Z^{\delta}_{t} \in \mathcal{D}_{\eta_{2}}} +  \mathbf{1}_{\delta^{\frac{1}{2}}Z^{\delta}_{t} \notin \mathcal{D}_{\eta_{2}}}) \vert \mathcal{F}^{Z^{\delta}}_{t-\delta}]=0.
\end{align*}
Moreover, Markov's inequality, together with (\ref{eq:hyp_3_Norme_adhoc_fonction_schema}), implies that 
\begin{align*}
\mathbb{E}_{x}[ \delta^{\frac{1}{2}} \vert  \sum_{l \in \EnsN}  Z^{\delta,l}_{t} \nabla_{x} V_{l}(X^{\delta}_{t-\delta},t-\delta)  \mathbf{1}_{\delta^{\frac{1}{2}}Z^{\delta}_{t} \notin \mathcal{D}_{\eta_{2}}} \vert_{\mathbb{R}^{d \times d}}\  \vert \mathcal{F}^{Z^{\delta}}_{t-\delta}] \leqslant &  \delta 
  \CstePsiouf_{3}  \mathbb{E}[\vert Z^{\delta}_{t} \vert_{\mathbb{R}^{N}}^{q^{\delta}_{\eta_{2}}}].
\end{align*}
In particular, applying the Taylor expansion to $\nabla_{x} \psi$, we can show that
\begin{align*}
\vert \mathbb{E}_{x}[ (I_{d\times d}-\nabla_{x} \psi)  (X^{\delta}_{t-\delta},t-\delta,\delta^{\frac{1}{2}}Z^{\delta}_{t},\delta)  \mathbf{1}_{\delta^{\frac{1}{2}}Z^{\delta}_{t} \in \mathcal{D}_{\eta_{2}}}\vert \mathcal{F}^{Z^{\delta}}_{t-\delta}] \vert_{\mathbb{R}^{d \times d}} 
\leqslant &\delta    \CstePsiouf_{3} C \MomZ_{q^{\delta}_{\eta_{2}} \vee (\ExpPsiouf_{3}  +2)}(Z^{\delta}) .
\end{align*}


On the other hand, using (\ref{borne_grad_psi-Id}), for every $k \in \mathbb{N},k\geqslant 2$, we have
\begin{align*}
\mathbb{E}_{x}[\vert I_{d\times d} - \nabla_{x} \psi (X^{\delta}_{t-\delta},t-\delta,\delta^{\frac{1}{2}}Z^{\delta}_{t},\delta) &  \vert_{\mathbb{R}^{d \times d}}^k   \mathbf{1}_{\delta^{\frac{1}{2}}Z^{\delta}_{t} \in \mathcal{D}_{\eta_{2}}} \vert \mathcal{F}^{Z^{\delta}}_{t-\delta}]  \leqslant  \delta^{\frac{k}{2}} \eta_{2}^{(k-2)(\ExpPsiouf_{3}  +1)}  4^{k}  \CstePsiouf_{3}  ^{k} 2 \MomZ_{2(\ExpPsiouf_{3}  +1)}(Z^{\delta}) .
\end{align*}
Since $\delta^{\frac{1}{2}} \eta_{2}^{\ExpPsiouf_{3}  +1} 4 \CstePsiouf_{3} < \frac{1}{2}$, the following geometric series converges and satisfies
\begin{align*}
\mathbb{E}_{x}[ \sum_{k=2}^{\infty}  \vert I_{d\times d} - \nabla_{x} \psi (X^{\delta}_{t-\delta},t-\delta,\delta^{\frac{1}{2}}Z^{\delta}_{t},\delta)  \vert_{\mathbb{R}^{d \times d}}^k & \mathbf{1}_{\delta^{\frac{1}{2}}Z^{\delta}_{t} \in \mathcal{D}_{\eta_{2}}}  \vert \mathcal{F}^{Z^{\delta}}_{t-\delta}] 
\leqslant  \delta C \CstePsiouf_{3}^{2} \MomZ_{2(\ExpPsiouf_{3}  +1)}(Z^{\delta})  .
\end{align*}

In particular, using the Neumann series to write $\nabla_{x} \psi^{-1}=\sum_{k=0}^{\infty}(I_{d \times d} - \nabla_{x} \psi )^k$, we have
\begin{align*}
 \mathbb{E}_{x}[ \vert (\nabla_{x} \psi^{-1} -2I_{d\times d}+ & \nabla_{x} \psi)(X^{\delta}_{t-\delta},t-\delta, \delta^{\frac{1}{2}}Z^{\delta}_{t},\delta) \vert_{\mathbb{R}^{d \times d}}  \mathbf{1}_{\delta^{\frac{1}{2}}Z^{\delta}_{t} \in \mathcal{D}_{\eta_{2}}} \vert \mathcal{F}^{Z^{\delta}}_{t-\delta}]  \\
 \leqslant  &  \mathbb{E}_{x}[ \sum_{k=2}^{\infty} \vert I_{d\times d} - \nabla_{x} \psi (X^{\delta}_{t-\delta},t-\delta,\delta^{\frac{1}{2}}Z^{\delta}_{t},\delta) \vert_{\mathbb{R}^{d \times d}}^k  \mathbf{1}_{\delta^{\frac{1}{2}}Z^{\delta}_{t} \in \mathcal{D}_{\eta_{2}}}  \vert \mathcal{F}^{Z^{\delta}}_{t-\delta}]  ,
\end{align*}

so that
\begin{align*}
\vert \hat{\Upsilon}_t \vert_{\mathbb{R}^{d \times d}}\leqslant & \vert\mathring{X}_{t-\delta} \vert_{\mathbb{R}^{d \times d}}  \mathbf{1}_{\Theta_{\eta_{2},\EnsT ,t-\delta}>0} ( \vert \mathbb{E}_{x}[ (I_{d\times d}-\nabla_{x} \psi)(X^{\delta}_{t-\delta},t-\delta,\delta^{\frac{1}{2}}Z^{\delta}_{t},\delta)  \mathbf{1}_{\delta^{\frac{1}{2}}Z^{\delta}_{t} \in \mathcal{D}_{\eta_{2}}}\vert \mathcal{F}^{Z^{\delta}}_{t-\delta}] \vert_{\mathbb{R}^{d \times d}} \\
&+ \mathbb{E}_{x}[ \sum_{k=2}^{\infty} \vert I_{d\times d} - \nabla_{x} \psi (X^{\delta}_{t-\delta},t-\delta,\delta^{\frac{1}{2}}Z^{\delta}_{t},\delta)  \vert_{\mathbb{R}^{d \times d}}^k  \mathbf{1}_{\delta^{\frac{1}{2}}Z^{\delta}_{t} \in \mathcal{D}_{\eta_{2}}} \vert \mathcal{F}^{Z^{\delta}}_{t-\delta}] ). 
\end{align*}

We already obtained an estimate of the second term of the $r.h.s.$ above.  To bound the other term, we compute $ \nabla_{x} \psi(X^{\delta}_{t-\delta},t-\delta,  \delta^{\frac{1}{2}}Z^{\delta}_t,\delta)$ using the decomposition (\ref{eq:representation_psi_A}) and by standard estimates we complete  the proof of \textbf{Step 2}.
%
%

\textbf{Step 3.} 
Let us show that
\begin{align*}
 \mathbb{E}_{x}[\vert \tilde{\Upsilon}_{t} \vert_{\mathbb{R}^{d \times d}}^{p} \mathbf{1}_{\Theta_{\eta_{2},\EnsT ,t-\delta}>0}]^{\frac{1}{p}} \leqslant \delta^{\frac{1}{2}} \mathbb{E}_{x}[ \vert\mathring{X}_{\eta_{2},t-\delta}  \vert_{\mathbb{R}^{d \times d}}^{p}   ]^{\frac{1}{p}} C \CstePsiouf_{3}^{2} \MomZ_{p(q^{\delta}_{\eta_{2}} \vee (2\ExpPsiouf_{3}  +2))}(Z^{\delta})^{\frac{1}{p}},
\end{align*}

with $C$ a universal constant. First, we remark that $
\vert \tilde{\Upsilon}_{t} \vert_{\mathbb{R}^{d \times d}} \leqslant \vert \Upsilon_{t} \vert_{\mathbb{R}^{d \times d}}+\vert \hat{\Upsilon}_{t} \vert_{\mathbb{R}^{d \times d}}.$
We have already studied the second term of the $r.h.s.$ in \textbf{Step 2} so we focus on the first one. Proceeding as in \textbf{Step 2}, we have
\begin{align*}
\vert \Upsilon_t \vert_{\mathbb{R}^{d \times d}}\leqslant & \vert\mathring{X}_{\eta_{2},t-\delta} \vert_{\mathbb{R}^{d \times d}}  (\vert (I_{d\times d}-\nabla_{x} \psi)(X^{\delta}_{t-\delta},t-\delta,\delta^{\frac{1}{2}}Z^{\delta}_{t},\delta)  \mathbf{1}_{\delta^{\frac{1}{2}}Z^{\delta}_{t} \in \mathcal{D}_{\eta_{2}}}\vert_{\mathbb{R}^{d \times d}} \\
& \qquad +\sum_{k=2}^{\infty} \vert I_{d\times d} - \nabla_{x} \psi (X^{\delta}_{t-\delta},t-\delta,\delta^{\frac{1}{2}}Z^{\delta}_{t},\delta)  \vert_{\mathbb{R}^{d \times d}}^k  \mathbf{1}_{\delta^{\frac{1}{2}}Z^{\delta}_{t} \in \mathcal{D}_{\eta_{2}}} ).
\end{align*}

Using (\ref{borne_grad_psi-Id}), it follows that
\begin{align*}
\mathbb{E}_{x}[ \vert\mathring{X}_{\eta_{2},t-\delta}  \vert_{\mathbb{R}^{d \times d}}^{p} &  \vert I_{d\times d} - \nabla_{x} \psi (X^{\delta}_{t-\delta},t-\delta,\delta^{\frac{1}{2}}Z^{\delta}_{t},\delta)  \vert_{\mathbb{R}^{d \times d}}^{p}  \mathbf{1}_{\delta^{\frac{1}{2}}Z^{\delta}_{t} \in \mathcal{D}_{\eta_{2}}}   ]  \\
\leqslant & \delta^{\frac{p}{2}}\mathbb{E}_{x}[ \vert\mathring{X}_{t-\delta}  \vert_{\mathbb{R}^{d \times d}}^{p}  \mathbf{1}_{\Theta_{\eta_{2},\EnsT ,t-\delta}>0} ]  \CstePsiouf_{3}^{p} 4^{p} C \MomZ_{p(\ExpPsiouf_{3}  +1)}(Z^{\delta}).
\end{align*}
Moreover, since $\delta^{\frac{1}{2}} \eta_{2}^{\ExpPsiouf_{3}  +1}4\CstePsiouf_{3} < \frac{1}{2}$, on the set $\{\delta^{\frac{1}{2}}Z^{\delta}_{t} \in \mathcal{D}_{\eta_{2}} \}$, we have (\ref{borne_grad_psi-Id}) and
\begin{align*}
\sum_{k=2}^{\infty} \vert I_{d\times d} - \nabla_{x} \psi (X^{\delta}_{t-\delta},t-\delta,\delta^{\frac{1}{2}}Z^{\delta}_{t},\delta)  \vert_{\mathbb{R}^{d \times d}}^k \leqslant \delta C \CstePsiouf_{3} ^{2} ( 1 \vee  \vert Z^{\delta}_{t} \vert_{\mathbb{R}^{N}}^{2(\ExpPsiouf_{3}  +1)} )   ,
\end{align*}
and 
\begin{align*}
\mathbb{E}_{x}[ \vert\mathring{X}_{\eta_{2},t-\delta}  \vert_{\mathbb{R}^{d \times d}}^{p} & \vert \sum_{k=2}^{\infty} \vert I_{d\times d} - \nabla_{x} \psi (X^{\delta}_{t-\delta},t-\delta,\delta^{\frac{1}{2}}Z^{\delta}_{t},\delta)  \vert_{\mathbb{R}^{d \times d}}^k  \vert^{p}  \mathbf{1}_{\delta^{\frac{1}{2}}Z^{\delta}_{t} \in \mathcal{D}_{\eta_{2}}}  ]^{\frac{1}{p}} \\
\leqslant & \delta  \mathbb{E}[ \vert\mathring{X}_{\eta_{2},t-\delta}  \vert_{\mathbb{R}^{d \times d}}^{p}  ]^{\frac{1}{p}} C \CstePsiouf_{3} ^{2} \MomZ_{2p(\ExpPsiouf_{3}  +1)}(Z^{\delta})^{\frac{1}{p}} .
\end{align*}

Gathering all the terms concludes the proof of \textbf{Step 3}.

\textbf{Step 4.}
We are now in a position to conclude the proof. First, employing the Burkholder inequality (\ref{eq:burkholder_inequality}) and (\ref{eq:burholder_inequality_bis}), we have for every $p \geqslant 2$,
\begin{align*}
\mathbb{E}_{x}[\sup_{t \in \EnsT }\vert \sum_{w \in \pi^{\delta} \cap (0,t]} \tilde{\Upsilon}_{w} \vert_{\mathbb{R}^{d \times d}}^{p} ] \leqslant & \CsteBDH_{p}^{p} \mathbb{E}_{x}[(\sum_{t \in \EnsT } \vert \tilde{\Upsilon}_{t} \vert_{\mathbb{R}^{d \times d}}^{2})^{\frac{p}{2}} ] \leqslant C(p) (\sum_{t \in \EnsT } \mathbb{E}_{x}[\vert \tilde{\Upsilon}_{t} \vert_{\mathbb{R}^{d \times d}}^{p}]^{\frac{2}{p}} )^{\frac{p}{2}}.
\end{align*}
We deduce from \textbf{Step 1,2,3} that
\begin{align*}
\mathbb{E}_{x}[\sup&_{t \in \EnsT  \cup \{ 0 \}}  \vert \mathring{X}_{\eta_{2},t}  \vert_{\mathbb{R}^{d \times d}}^{p}   ]^{\frac{1}{p}} \leqslant \sqrt{d}+ \mathbb{E}_{x}[\sup_{t \in \EnsT } \vert \sum_{w \in \pi^{\delta} \cap (0,t]} \hat{\Upsilon}_w \vert_{\mathbb{R}^{d \times d}}^{p}]^{\frac{1}{p}} +C (\sum_{t \in \EnsT } \mathbb{E}_{x}[\vert \tilde{\Upsilon}_{t} \vert_{\mathbb{R}^{d \times d}}^{p}]^{\frac{2}{p}} )^{\frac{1}{2}}\\
  \leqslant & \sqrt{d}+C(p) \CstePsiouf_{3} ^{2} \MomZ_{q^{\delta}_{\eta_{2}} \vee (2\ExpPsiouf_{3}  +2)}(Z^{\delta}) \sum_{t \in \EnsT } \delta \mathbb{E}_{x}[ \vert\mathring{X}_{\eta_{2},t-\delta} \vert_{\mathbb{R}^{d \times d}}^{p}  ]^{\frac{1}{p}} \\
  &+C(p) \CstePsiouf_{3} ^{2}\MomZ_{p(q^{\delta}_{\eta_{2}} \vee (2\ExpPsiouf_{3}  +2))}(Z^{\delta})^{\frac{1}{p}}(\sum_{t \in \EnsT } \delta \mathbb{E}_{x}[ \vert\mathring{X}_{\eta_{2},t-\delta}  \vert_{\mathbb{R}^{d \times d}}^{p}  ]^{\frac{2}{p}} )^{\frac{1}{2}} .
\end{align*}
Therefore,  (\ref{eq:borne_flot_tangent_inv}) follows from the discrete Gronwall lemma.

\end{proof}

\subsection{Proof of Lemma \ref{lemme:dvpt_Lie}}

\begin{proof}

In this proof we first derive the expansion formula and give an explicit value for $\mathscr{R}V$ (\textbf{Step 1,2,3}) and then establish the estimates on this remainder in \textbf{Step 4}.\\
\textbf{Step 1.} Let us show that for every $t \in \pi^{\delta,\ast}$,
\begin{align*}
 V(X^{\delta}_{t},t) & -  V(X^{\delta}_{t-\delta},t-\delta) = \delta^{\frac{1}{2}}\sum_{l \in \EnsN}   Z^{\delta,l}_{t}  \nabla_{x} V V_{l}(X^{\delta}_{t-\delta},t-\delta)   \\
 &+ \delta  \nabla_{x} V V_{0}(X^{\delta}_{t-\delta},t-\delta)  +\delta\partial_{t} V(X^{\delta}_{t-\delta},t-\delta)  +  \delta \frac{1}{2}\sum_{l \in \EnsN}  V_{l}^{\mbox{\textsc{t}}} \Hessiennex   VV_{l}(X^{\delta}_{t-\delta}, t-\delta)   \\
&+R^{\delta,1}(X^{\delta}_{t-\delta},t-\delta,Z^{\delta}_{t}),
\end{align*}
with for every $(x,t,z) \in \mathbb{R}^{d}\times \pi^{\delta} \times \mathbb{R}^{N}$,
\begin{align*}
R^{\delta,1}(x,t,z) =& \frac{1}{2}  \delta \sum_{i,l \in \EnsN}(z^{i} z^{l} - \mathbf{1}_{i=l})V_{i}^{\mbox{\textsc{t}}} \Hessiennex   V V_{l}(x,t)  \\
 & +\frac{1}{2}(R^{\delta,1,1}(x,t,z)+2  \delta^{\frac{1}{2}}  \sum_{l \in \EnsN}   z^{l}  V_{l}(x,t))^{\mbox{\textsc{t}}} \Hessiennex  V(x,t)R^{\delta,1,1}(x,t,z)    \\
&+  \nabla_{x}V(x,t) R^{\delta,1,2}(x,t,z)  +  R^{\delta,1,3}(x,t,z) ,
\end{align*}

with (remembering notation (\ref{def:notation_{d}vpt_X})), $R^{\delta,1,1}(x,t,z)=  \delta A_{3} (x,t,\delta^{\frac{1}{2}}z, \delta)+ \delta \sum_{i,j \in \EnsN}z^{i} z^{j} A_{2}^{i,j}(x,t,\delta^{\frac{1}{2}}z)  $,
\begin{align*}
R^{\delta,1,2}(x,t,z)=& \delta \frac{1}{2} \sum_{i,j  \in \EnsN}(z^{i} z^{j}-\mathbf{1}_{i=j})  \partial_{z^{i}} \partial_{z^{j}} \psi(x,t,0,0)  + \delta^{2} \int_{0}^{1} (1-\lambda) \partial_{y}^{2} \psi(x,t,\delta^{\frac{1}{2}}z,\lambda \delta) \mbox{d} \lambda \\
&+  \delta^{\frac{3}{2}} \sum_{i  \in \EnsN}  z^{i} \int_{0}^{1}  \partial_{z^{i}} \partial_{y} \psi(x,t,\lambda \delta^{\frac{1}{2}}z,0) \mbox{d} \lambda \\
&+\delta^{\frac{3}{2}} \frac{1}{2} \sum_{i,j,l  \in \EnsN} z^{i} z^{j} z^{l} \int_{0}^{1} (1-\lambda)^{2} \partial_{z^{i}} \partial_{z^{j}} \partial_{z^{l}} \psi(x,t,\lambda \delta^{\frac{1}{2}}z,0) \mbox{d} \lambda, \\
R^{\delta,1,3}(x,t,z)=& \delta^{2} \int_{0}^{1}(1-\lambda)  \partial_{t}^{2}V(x,t+\lambda \delta) \mbox{d}\lambda \\
&+ \sum_{i=1}^{d}  R^{\delta,1,0}(x,t,z)_{i}  \int_{0}^{1} \partial_{x^{i}} \OperPreuveLie  V(x+ \lambda R^{\delta,1,0}(x,t,z),t) \mbox{d} \lambda  \\
&+\frac{1}{2} \sum_{i,j,k=1}^{d} R^{\delta,1,0}(x,t,z)_{i ,j,k}^{\otimes 3}  \int_{0}^{1} (1-\lambda)^{2}  \partial_{x}^{(i,j,k)} V(x+ \lambda 
R^{\delta,1,0}(x,t,z),t) \mbox{d} \lambda ,
\end{align*}
with $R^{\delta,1,0}(x,t,z)= \delta A_{3}(x,t,\delta^{\frac{1}{2}}z,\delta)+   \delta^{\frac{1}{2}}  \sum_{i \in \EnsN}  z^{i}  \int_{0}^{1} \partial_{z^{i}} \psi(x,t, \lambda \delta^{\frac{1}{2}} z,0) \mbox{d} \lambda $ and
\begin{align*}
\OperPreuveLie V(x,t) :=   \delta \int_{0}^{1} \partial_{t}V(x,t+\lambda \delta) \mbox{d}\lambda=\delta \partial_{t}V(x,t) +\delta^{2} \int_{0}^{1}(1-\lambda) \partial_{t}^{2}V(x,t+\lambda \delta) \mbox{d}\lambda.
\end{align*}

We begin by noticing that, using the Taylor expansion of $\psi$ $w.r.t.$ its fourth and then third variables, we have
\begin{align}
\label{eq:represe_X_preuve_Lie}
X^{\delta}_{t}=& \psi(X^{\delta}_{t-\delta},t-\delta,  \delta^{\frac{1}{2}}Z^{\delta}_t,\delta)=  X^{\delta}_{t-\delta}+  R^{\delta,1,0}(X^{\delta}_{t-\delta},t-\delta,Z^{\delta}_{t}) \\
 =& X^{\delta}_{t-\delta}+ \delta^{\frac{1}{2}}  \sum_{l  \in \EnsN}   Z^{\delta,l}_{t}  V_{l}(X^{\delta}_{t-\delta},t-\delta) + R^{\delta,1,1}(X^{\delta}_{t-\delta},t-\delta,Z^{\delta}_{t}) \nonumber    \\
 =& X^{\delta}_{t-\delta}+ \delta^{\frac{1}{2}}  \sum_{l   \in \EnsN}   Z^{\delta,l}_{t}  V_{l}(X^{\delta}_{t-\delta},t-\delta) + \delta  V_{0}(X^{\delta}_{t-\delta},t-\delta)   
+ R^{\delta,1,2}(X^{\delta}_{t-\delta},t-\delta,Z^{\delta}_{t}). \nonumber
\end{align}

Now, using again the Taylor expansion on the function $V$ $w.r.t.$ its second variable,
\begin{align*}
 V(X^{\delta}_{t},t)  -  V(X^{\delta}_{t-\delta},t-\delta) =& (\OperPreuveLie V+V)(X^{\delta}_{t},t-\delta)  -  V(X^{\delta}_{t-\delta},t-\delta) .
\end{align*} 

The Taylor expansion on the function $\OperPreuveLie V$ $w.r.t$ its first variable combined with (\ref{eq:represe_X_preuve_Lie}) yields
\begin{align*}
 \OperPreuveLie V(X^{\delta}_{t},t-\delta)  =& \OperPreuveLie V(X^{\delta}_{t-\delta},t-\delta)  \\
 &+   \sum_{i=1}^{d}  R^{\delta,1,0}(X^{\delta}_{t-\delta},t-\delta,Z^{\delta}_{t})_{i}  \int_{0}^{1} \partial_{x^{i}} \OperPreuveLie  V(X^{\delta}_{t-\delta}+ \lambda R^{\delta,1,0}(X^{\delta}_{t-\delta},t-\delta,Z^{\delta}_{t}),t-\delta) \mbox{d} \lambda.
\end{align*} 

To complete \textbf{Step 1}, it remains to study $ V(X^{\delta}_{t},t-\delta)  - V(X^{\delta}_{t-\delta},t-\delta) $. To do so, we employ the third order Taylor expansion of $V$ $w.r.t.$ its first variable together with the representation formulas given in (\ref{eq:represe_X_preuve_Lie}). Gathering all the terms together completes the proof.

\textbf{Step 2.} Let us show that for every $t \in \pi^{\delta,\ast}$, on the set $\{\delta^{\frac{1}{2}}Z^{\delta}_{t} \in \mathcal{D}_{\eta_{2}}\}$ (with  $\mathcal{D}_{\eta_{2}}=\{z\in \mathbb{R}^N, \vert z \vert < \delta^{\frac{1}{2}} \eta_{2}\}$ introduced in the proof of Theorem \ref{th:borne_Lp_inv_cov_Mal}), we have
\begin{align*}
  \nabla_{x} \psi^{-1}(X^{\delta}_{t-\delta},t-\delta,&\delta^{\frac{1}{2}}Z^{\delta}_t,\delta)=I_{d \times d}  - \delta^{\frac{1}{2}}  \sum_{i \in \EnsN}   Z^{\delta,i}_{t} \nabla_{x} V_{i}(X^{\delta}_{t-\delta},t-\delta)   - \delta\nabla_{x} V_{0}(X^{\delta}_{t-\delta},t-\delta) \\ &+ \delta \sum_{i \in \EnsN} \nabla_{x} V_{i}(X^{\delta}_{t-\delta},t-\delta)^{2} +\hat{R}^{\delta,2}(X^{\delta}_{t-\delta},t-\delta,Z^{\delta}_{t}) ,
\end{align*}

with, for every $(x,t,z) \in \mathbb{R}^{d}\times \pi^{\delta} \times \mathbb{R}^{N}$,
\begin{align*}
\hat{R}^{\delta,2}(&x,t,z) = \hat{R}^{\delta,2,3}(x,t,z)-\hat{R}^{\delta,2,2}(x,t,z)  + \delta \sum_{i,l \in \EnsN}(z^{i} z^{l} - \mathbf{1}_{i=l}) \nabla_{x} V_{i} \nabla_{x} V_{l}(x,t)  \\
&+\delta^{\frac{1}{2}}  \sum_{l \in \EnsN}   z^{l}( \nabla_{x} V_{l}(x,t)\hat{R}^{\delta,2,1}(x,t,z)  + \hat{R}^{\delta,2,1}(x,t,z)  \nabla_{x} V_{l}(x,t) ) +\hat{R}^{\delta,2,1}(x,t,z)^{2} ,
\end{align*}

where $\hat{R}^{\delta,2,1}(x,t,z)=\nabla_{x} R^{\delta,1,1}(x,t,z)$, $\hat{R}^{\delta,2,2}(x,t,z)=\nabla_{x}R^{\delta,1,2}(x,t,z)$ and, recalling that $ \nabla_{x} \psi (X^{\delta}_{t-\delta},t-\delta,\delta^{\frac{1}{2}}Z^{\delta}_{t},\delta)  $ is invertible on the set $\{ \vert Z^{\delta}_{t} \vert_{\mathbb{R}^{N}} \leqslant \eta_{2} \}$ under (\ref{hyp:delta_eta_inversibilite}), $\hat{R}^{\delta,2,3}(x,t,z)=(\nabla_{x} \psi^{-1} - I_{d\times d}-(I_{d\times d}-\nabla_{x} \psi )-(I_{d\times d}-\nabla_{x} \psi)^{2})(x,t,\delta^{\frac{1}{2}}z,\delta) $ if $\vert z \vert_{\mathbb{R}^{N}} \leqslant \eta_2 $ and is zero otherwise.  
The proof simply boils down to using the identities
\begin{align*}
 \nabla_{x} \psi( &X^{\delta}_{t-\delta},t-\delta,  \delta^{\frac{1}{2}}Z^{\delta}_t,\delta)=  I_{d \times d}+ \delta^{\frac{1}{2}}  \sum_{l \in \EnsN}   Z^{\delta,l}_{t} \nabla_{x} V_{l}(X^{\delta}_{t-\delta},t-\delta)  + \hat{R}^{\delta,2,1}(X^{\delta}_{t-\delta},t-\delta,Z^{\delta}_{t}) ,\\
=& I_{d \times d}+ \delta^{\frac{1}{2}}  \sum_{l \in \EnsN}   Z^{\delta,l}_{t} \nabla_{x} V_{l}(X^{\delta}_{t-\delta},t-\delta)  +  \delta \nabla_{x} V_{0}(X^{\delta}_{t-\delta},t-\delta)  + \hat{R}^{\delta,2,2}(X^{\delta}_{t-\delta},t-\delta,Z^{\delta}_{t}) ,
\end{align*}
and rearranging the terms together.

\textbf{Step 3.}
 Let us show that for every $t \in \pi^{\delta,\ast}$,  on the set $\{\Theta_{\eta_{2},\EnsT ,t}>0\}$, we have
 \begin{align*}
\mathring{X}_{t}V(X^{\delta}_{t},t)  =&  \mathring{X}_{t-\delta}V(X^{\delta}_{t-\delta},t-\delta)  + \delta^{\frac{1}{2}}\sum_{i \in \EnsN} Z^{\delta,i}_{t} \mathring{X}_{t-\delta} V^{[i]}(X^{\delta}_{t-\delta},t-\delta) \\
&+\delta \mathring{X}_{t-\delta} V^{[0]}(X^{\delta}_{t-\delta},t-\delta) +\mathring{X}_{t-\delta} R^{\delta}V(X^{\delta}_{t-\delta},t-\delta,Z^{\delta}_{t}),
\end{align*}

with, for every $(x,t,z) \in \mathbb{R}^{d}\times \pi^{\delta} \times \mathbb{R}^{N}$, $R^{\delta}V(x,t,z)=R^{\delta,1}(x,t,z) +R^{\delta,2}(x,t,z) +R^{\delta,3}(x,t,z)$, where $R^{\delta,2}(x,t,z) =\hat{R}^{\delta,2}(x,t,z) V(x,t)$ and 
\begin{align*}
R^{\delta,3}&(x,t,z)=
- \delta  \sum_{i, l \in \EnsN}  (z^{i} z^{l}-\mathbf{1}_{i=l}) \nabla_{x} V_{i}\nabla_{x} V V_{l}(x,t) \\
&+(\delta(-   \nabla_{x} V_{0} + \sum_{l \in \EnsN} \nabla_{x} V_{l}^{2})(x,t) +\hat{R}^{\delta,2}(x,t,z) )    \\
& \quad \times (\delta^{\frac{1}{2}}\sum_{i \in \EnsN}   z^{i}  \nabla_{x} VV_{i}(x,t) + \delta (  \nabla_{x} V V_{0} +\partial_{t} V   +\frac{1}{2}\sum_{i \in \EnsN}  V_{i}^{\mbox{\textsc{t}}}  \Hessiennex   V  V_{i})(x,t)  +R^{\delta,1}(x,t,z)) \\
&- \delta^{\frac{1}{2}}  \sum_{i \in \EnsN}  z^{i} \nabla_{x} V_{i}(x,t)    ( \delta  (\nabla_{x} VV_{0}   +\partial_{t} V +   \frac{1}{2}\sum_{j \in \EnsN}  V_{j}^{\mbox{\textsc{t}}}  \Hessiennex   V  V_{j})(x,t)  +R^{\delta,1}(x,t,z)).
\end{align*}
First, owing to (\ref{eq:inverse_tangent_flot}), we write the following decomposition on the set $\{\delta^{\frac{1}{2}}Z^{\delta}_{t} \in \mathcal{D}_{\eta_{2}}\}$,
\begin{align*}
\mathring{X}_{t}   V(X^{\delta}_{t},t) - \mathring{X}_{t-\delta}  V(X^{\delta}_{t-\delta},t-\delta)   = &  \mathring{X}_{t-\delta}  \nabla_{x} \psi^{-1}(X^{\delta}_{t-\delta},t-\delta,\delta^{\frac{1}{2}}Z^{\delta}_t,\delta) \left(  V(X^{\delta}_{t},t)  -  V(X^{\delta}_{t-\delta},t-\delta)  \right) \\
& + \mathring{X}_{t-\delta}  \left(  \nabla_{x} \psi(X^{\delta}_{t-\delta},t-\delta,\delta^{\frac{1}{2}}Z^{\delta}_t,\delta)^{-1} -   I_{d \times d} \right)   V(X^{\delta}_{t-\delta},t-\delta) .
\end{align*}

Then, using \textbf{Step 1} and \textbf{Step 2}, a direct computation yields
\begin{align*}
\nabla_{x}  \psi^{-1} & (X^{\delta}_{t-\delta}, t-\delta,\delta^{\frac{1}{2}} Z^{\delta}_t  ,\delta) (  V(X^{\delta}_{t},t)  -  V(X^{\delta}_{t-\delta},t-\delta) ) =\delta^{\frac{1}{2}}\sum_{l \in \EnsN}   Z^{\delta,l}_{t}  \nabla_{x} V V_{l}(X^{\delta}_{t-\delta},t-\delta)   \\
 &+ \delta ( \nabla_{x} V V_{0}+ \partial_{t} V +  \sum_{l \in \EnsN}   (  \frac{1}{2} V_{l}^{\mbox{\textsc{t}}}  \Hessiennex   V-\nabla_{x} V_{l} \nabla_{x} V )  V_{l})(X^{\delta}_{t-\delta},t-\delta) \\
&+R^{\delta,3}(X^{\delta}_{t-\delta},t-\delta,Z^{\delta}_{t}) +R^{\delta,1}(X^{\delta}_{t-\delta},t-\delta,Z^{\delta}_{t}) .
\end{align*}

The study of the other term was done in \textbf{Step 2} and the proof of \textbf{Step 3} is completed. In particular, we have, for $(x,t,z)\in \mathbb{R}^{d} \times \pi^{\delta} \times \mathbb{R}^{N}$,
\begin{align*}
\mathscr{R}V(x,t,z)=R^{\delta}V(x,t,z) \mathbf{1}_{\delta^{\frac{1}{2}}z\in \mathcal{D}_{\eta_{2}}}-(V(x,t)+\delta^{\frac{1}{2}} \sum_{i \in\EnsN} z^{i} V^{[i]}(x,t)+\delta V^{[0]}(x,t)) \mathbf{1}_{\delta^{\frac{1}{2}}z \notin \mathcal{D}_{\eta_{2}}}
\end{align*}

\VersionLong{
\textbf{Step 4.}
Let us prove (\ref{eq:borne_reste_moy_dvpt_Lie}) and (\ref{eq:borne_reste_mart_dvpt_Lie}) beginning with the proof of (\ref{eq:borne_reste_moy_dvpt_Lie}). 
}
\VersionCourt{
\textbf{Step 4.} Let us prove (\ref{eq:borne_reste_moy_dvpt_Lie}). The proof of (\ref{eq:borne_reste_mart_dvpt_Lie}) follows from a similar but simpler method (since it does not involve derivatives) used to study $\bar{\mathscr{R}}$.
}
For $i \in \{1,2,3\}$, $t \in \pi^{\delta,\ast}$, we introduce the functions defined for every $x \in \mathbb{R}^{d}$ by $\bar{R}_{t}^{i}(x)= \mathbb{E}[ R^{\delta,i}(x,t-\delta,Z^{\delta}_{t}) \mathbf{1}_{\delta^{\frac{1}{2}}Z^{\delta}_{t} \in \mathcal{D}_{\eta_{2}}}]$ and for $i \in \{1,2\},j\in \{1,2,3\}$, $\bar{R}_{t}^{i,j}(x)=\mathbb{E}[ R^{\delta,i,j}(x,t-\delta,Z^{\delta}_{t}) \mathbf{1}_{\delta^{\frac{1}{2}}Z^{\delta}_{t} \in \mathcal{D}_{\eta_{2}}}]$ (with the notation $R^{\delta,2,j}=\hat{R}^{\delta,2,j}V$). In particular, since $\{\Theta_{\eta_{2},\EnsT ,t}>0\}=\{\Theta_{\eta_{2},\EnsT ,t-\delta}>0\} \cap \{ \delta^{\frac{1}{2}}Z^{\delta}_{t} \in \mathcal{D}_{\eta_{2}} \}$,  $\bar{\mathscr{R}}V(x,t-\delta) = \sum_{i=1}^{6}\bar{R}_{t}^{i}(x)= \mathbb{E}[ R^{\delta}V(x,t-\delta,Z^{\delta}_{t}) \mathbf{1}_{\delta^{\frac{1}{2}}Z^{\delta}_{t} \in \mathcal{D}_{\eta_{2}}}]+\bar{R}_{t}^{4}(x)+\bar{R}_{t}^{5}(x)+\bar{R}_{t}^{6}(x)$ with
\begin{align*}
\bar{R}_{t}^{4}(x)=& - \delta^{\frac{1}{2}}\sum_{i \in \EnsN} V^{[i]}(x,t-\delta)  \mathbb{E}[ Z^{\delta,i}_{t} \mathbf{1}_{\delta^{\frac{1}{2}}Z^{\delta}_{t} \notin \mathcal{D}_{\eta_{2}}} ] , \quad \bar{R}_{t}^{5}(x)= -\delta V^{[0]}(x,t-\delta) \mathbb{P}( \delta^{\frac{1}{2}}Z^{\delta}_{t} \notin \mathcal{D}_{\eta_{2}} ) , \\
\bar{R}_{t}^{6}(x)=& - V(x,t-\delta) \mathbb{P}( \delta^{\frac{1}{2}}Z^{\delta}_{t} \notin \mathcal{D}_{\eta_{2}} ).
\end{align*}


We first study $\partial^{\alpha^{x}}_{x}  \bar{R}_{t}^{1} $ for $\alpha^{x} \in \mathbb{N}^{d}$.  We first remark that $\bar{R}_{t}^{1}   =  \bar{R}_{t}^{1,3}  + \nabla_{x}V(\cdot,t-\delta) \bar{R}_{t}^{1,2}  + \bar{R}_{t}^{1,4}$,
 with, for every $x \in \mathbb{R}^{d}$ and $t \in \pi^{\delta,\ast}$,
  \begin{align*}
 \bar{R}_{t}^{1,4} (x)=&\frac{1}{2} \mathbb{E}[ ( 2  \delta^{\frac{1}{2}}  \sum_{l \in \EnsN}   Z^{\delta,l}_{t}  V_{l}(x,t-\delta)+R^{\delta,1,1}(x,t-\delta,Z^{\delta}_{t}))^{\mbox{\textsc{t}}} \\
 & \quad \times\Hessiennex  V(x,t-\delta)R^{\delta,1,1}(x,t-\delta,Z^{\delta}_{t}) \mathbf{1}_{\delta^{\frac{1}{2}}Z^{\delta}_{t} \in \mathcal{D}_{\eta_{2}}}]   \\
 & + \frac{\delta}{2}\sum_{i,l \in \EnsN} \mathbb{E}[(Z^{\delta,i}_{t} Z^{\delta,l}_{t}-\mathbf{1}_{i=l}) V_{i}^{\mbox{\textsc{t}}} \nabla_{x}^{2}V V_{l}(x,t-\delta)\mathbf{1}_{\delta^{\frac{1}{2}}Z^{\delta}_{t} \in \mathcal{D}_{\eta_{2}}} ]
 \end{align*}

We begin with the study of $\bar{R}^{1,2}_{t}(x)$. We observe that, for every $t \in \pi^{\delta,\ast}$,
\begin{align*}
\sum_{i,l \in \EnsN}& \mathbb{E}[(Z^{\delta,i}_{t} Z^{\delta,l}_{t}-\mathbf{1}_{i=l})   \partial_{z^{i}} \partial_{z^{l}} \psi(x,t-\delta,0,0)\mathbf{1}_{\delta^{\frac{1}{2}}Z^{\delta}_{t} \in \mathcal{D}_{\eta_{2}}} ]\\
=&-\sum_{i,l \in \EnsN} \mathbb{E}[(Z^{\delta,i}_{t} Z^{\delta,l}_{t}-\mathbf{1}_{i=l})   \partial_{z^{i}} \partial_{z^{l}} \psi(x,t-\delta,0,0)\mathbf{1}_{\delta^{\frac{1}{2}}Z^{\delta}_{t} \notin \mathcal{D}_{\eta_{2}}} ],
\end{align*}
with $\vert \sum_{i,l\in \EnsN} \mathbb{E}[ \vert Z^{\delta,i}_{t} Z^{\delta,l}_{t} -\mathbf{1}_{i=l} \vert \mathbf{1}_{\delta^{\frac{1}{2}}Z^{\delta}_{t} \notin \mathcal{D}_{\eta_{2}}}] \vert \leqslant N \eta_{2}^{-a} \mathbb{E}[ \vert Z^{\delta}_{t} \vert_{\mathbb{R}^{N}}^{a}+ \vert Z^{\delta}_{t} \vert_{\mathbb{R}^{N}}^{2+a} ] $, for every $a >0$. In particular we take $a= \lceil- \frac{3 \ln(\delta)}{2 \ln(\eta_{2})} \rceil$. Combining this estimate with standard calculus together with hypothesis $\mathbf{A}_{1}^{\delta}(\vert \alpha^{x} \vert+3)$ (see (\ref{eq:hyp_1_Norme_adhoc_fonction_schema})) and $\mathbf{A}_{3}^{\delta}(+\infty)$ (see (\ref{eq:hyp:moment_borne_Z})),  we obtain, for every $x \in \mathbb{R}^{d}$,
\begin{align*}
\vert \partial^{\alpha^{x}}_{x}  \bar{R}^{1,2}_{t}(x) \vert_{\mathbb{R}^{d}} \leqslant &  \delta^{\frac{3}{2}} C \MomZ_{C}(Z^{\delta}) \CstePsiouf_{\vert \alpha^{x} \vert+3}(1+ \vert x \vert_{\mathbb{R}^{d}}^{C}).
\end{align*}
Similarly, $\mathbf{A}_{1}^{\delta}(\vert \alpha^{x} \vert+2)$ implies that
\begin{align*}
\vert \partial^{\alpha^{x}}_{x}  \bar{R}^{1,3}_{t}(x) \vert_{\mathbb{R}^{d}}+\vert \partial^{\alpha^{x}}_{x} \bar{R}_{t}^{1,4}(x) \vert_{\mathbb{R}^{d}}  \leqslant &  \delta^{\frac{3}{2}} C \MomZ_{C}(Z^{\delta})   \CstePsiouf_{\vert \alpha^{x} \vert+2}^{C}  \CstePsiouf_{V,\vert \alpha^{x} \vert+3}(1+ \vert x \vert_{\mathbb{R}^{d}}^{C}).
\end{align*}
We conclude that, under the assumptions $\mathbf{A}_{1}^{\delta}(\vert \alpha^{x} \vert+3)$ and $\mathbf{A}_{3}^{\delta}(+\infty)$, for every $x \in \mathbb{R}^{d}$, 
\begin{align*}
\vert \partial^{\alpha^{x}}_{x}  \bar{R}^{1}_{t}(x) \vert_{\mathbb{R}^{d}} \leqslant &  \delta^{\frac{3}{2}} C\MomZ_{C}(Z^{\delta})  \CstePsiouf_{\vert \alpha^{x} \vert+3}^{C}  \CstePsiouf_{V,\vert \alpha^{x} \vert+3} (1+ \vert x \vert_{\mathbb{R}^{d}}^{C}) .
\end{align*}
Now, we focus on the study of $\partial^{\alpha^{x}}_{x}  \bar{R}_{t}^{2} $. We first observe that $
\bar{R}^{2}_{t}  = \bar{R}^{2,3}_{t}-\bar{R}^{2,2}_{t} + \bar{R}^{2,4}_{t}$,
where we have introduced the function $\bar{R}^{2,4}_{t}$ defined for every $x \in \mathbb{R}^{d}$ by 
\begin{align*}
\bar{R}^{2,4}_{t}(x)=&  \mathbb{E}[\mathbf{1}_{\delta^{\frac{1}{2}}Z^{\delta}_{t} \in \mathcal{D}_{\eta_{2}}}(\hat{R}^{\delta,2,1}(x,t-\delta,Z^{\delta}_{t})^{2} \\
&+  \delta^{\frac{1}{2}}  \sum_{l \in \EnsN}   Z^{\delta,l}_{t} (\nabla_{x} V_{l}(x,t-\delta)\hat{R}^{\delta,2,1}(x,t-\delta,Z^{\delta}_{t}) + \hat{R}^{\delta,2,1}(x,t-\delta,Z^{\delta}_{t})  \nabla_{x} V_{l}(x,t-\delta) ) \\
 & + \delta \sum_{i,l \in \EnsN}(Z^{\delta,i}_{t} Z^{\delta,l}_{t}-\mathbf{1}_{i=l}) \nabla_{x} V_{i} \nabla_{x} V_{l}(x,t-\delta)) ]V(x,t-\delta)
\end{align*}

We begin with the study of $\bar{R}^{2,2}_{t}$. Remembering that $\bar{R}^{2,2}_{t}=\nabla_{x}   \bar{R}^{1,2}_{t} V(\cdot,t-\delta)$ and using arguments similar to those in the study of $  \bar{R}^{1,2}_{t}$, under the assumptions $\mathbf{A}_{1}^{\delta}(\vert \alpha^{x} \vert+4)$ (see (\ref{eq:hyp_1_Norme_adhoc_fonction_schema})) and $\mathbf{A}_{3}^{\delta}(+\infty)$ (see (\ref{eq:hyp:moment_borne_Z})), we obtain, for every $x \in \mathbb{R}^{d}$,
\begin{align*}
\vert \partial^{\alpha^{x}}_{x}  \bar{R}^{2,2}_{t}(x) \vert_{\mathbb{R}^{d}} \leqslant &  \delta^{\frac{3}{2}} C \MomZ_{C}(Z^{\delta}) \CstePsiouf_{\vert \alpha^{x} \vert+4}  \CstePsiouf_{V,\vert \alpha^{x} \vert} (1+ \vert x \vert_{\mathbb{R}^{d}}^{C})  .
\end{align*}
We then bound the derivatives of $\bar{R}_{t}^{2,3}$. We first remark that, for $\vert z \vert_{\mathbb{R}^{N}}\leqslant\eta_{2}$,  $\hat{R}^{\delta,2,3}(x,t,z)= \sum_{k=3}^{\infty} ( I_{d\times d} - \nabla_{x} \psi(x,t,\delta^{\frac{1}{2}}z,\delta))^k$. 
If $\vert \alpha^{x} \vert =1$, it follows from standard calculus that
\begin{align*}
\partial^{\alpha^{x}}_{x}  \bar{R}^{2,3}_{t}(x) =& \mathbb{E}[\mathbf{1}_{\delta^{\frac{1}{2}}Z^{\delta}_{t} \in \mathcal{D}_{\eta_{2}}} \sum_{k=3}^{\infty} ( I_{d\times d} - \nabla_{x} \psi(x,t-\delta,\delta^{\frac{1}{2}}Z^{\delta}_{t},\delta))^{k} ] \partial^{\alpha^{x}}_{x} V(x,t-\delta) \\
 & -\mathbb{E}[\mathbf{1}_{\delta^{\frac{1}{2}}Z^{\delta}_{t} \in \mathcal{D}_{\eta_{2}}} \sum_{k=3}^{\infty} \sum_{l=1}^{k}( ( I_{d\times d} - \nabla_{x} \psi)^{l-1}  \partial^{\alpha^{x}}_{x} \nabla_{x} \psi (I_{d\times d} -\nabla_{x} \psi)^{k-l})(x,t-\delta,\delta^{\frac{1}{2}}Z^{\delta}_{t},\delta)  ]  V (x,t-\delta) .
\end{align*}
When $\vert \alpha^{x} \vert >1$, we iterate the above formula. Combining this decomposition with the following estimate
\begin{align*}
\vert \partial^{\alpha^{x}}_{x}  \nabla_{x} \psi(x,t-\delta,  \delta^{\frac{1}{2}}Z^{\delta}_t,\delta) \vert_{\mathbb{R}^{d \times d}}= &\vert  \delta^{\frac{1}{2}}  \sum_{l \in \EnsN}   Z^{\delta,l}_{t} \partial^{\alpha^{x}}_{x}  \nabla_{x} V_{l}(x,t-\delta)  + \partial^{\alpha^{x}}_{x} \hat{R}^{\delta,2,1}(x,t-\delta,Z^{\delta}_{t}) \vert_{\mathbb{R}^{d \times d}}\\
\leqslant & 2 \delta^{\frac{1}{2}}  \CstePsiouf_{\vert \alpha^{x} \vert+3}(1+\vert x \vert_{\mathbb{R}^{d}}^{\ExpPsiouf  _{\vert \alpha^{x} \vert+3}}+\vert Z^{\delta}_{t} \vert_{\mathbb{R}^{N}} ^{\ExpPsiouf  _{\vert \alpha^{x} \vert+3}}) (1+\vert Z^{\delta}_{t} \vert_{\mathbb{R}^{N}}^{2}),
\end{align*}

we deduce that
\begin{align*}
\vert \partial^{\alpha^{x}}_{x}  \bar{R}^{2,3}_{t}(x) \vert_{\mathbb{R}^{d}} \leqslant & C \CstePsiouf_{\vert \alpha^{x} \vert + 3}^{C} \CstePsiouf_{V,\vert \alpha^{x} \vert} \mathbb{E}[(1+\vert x \vert_{\mathbb{R}^{d}}^{C}+\vert Z^{\delta}_{t} \vert_{\mathbb{R}^{N}}^{C}) \\
& \times  \sum_{k=0}^{\infty}  \delta^{\frac{\max(3-k,0)}{2}}  (k+1)^{\vert \alpha^{x} \vert} \vert I_{d\times d} - \nabla_{x} \psi \vert_{\mathbb{R}^{d \times d}}^{k}   (x,t-\delta,\delta^{\frac{1}{2}}Z^{\delta}_{t},\delta)  \mathbf{1}_{\delta^{\frac{1}{2}}Z^{\delta}_{t} \in \mathcal{D}_{\eta_{2}}}  ]  .
\end{align*}
Using $\mathbf{A}_{1}^{\delta}$ (see (\ref{eq:hyp_3_Norme_adhoc_fonction_schema})), we have (\ref{borne_grad_psi-Id}). Moreover, when $k \geqslant 3$, we use $\vert Z^{\delta}_{t} \vert_{\mathbb{R}^{N}}^{k} \mathbf{1}_{\delta^{\frac{1}{2}}Z^{\delta}_{t} \in \mathcal{D}_{\eta_{2}}}\leqslant \vert Z^{\delta}_{t} \vert_{\mathbb{R}^{N}}^{3}\eta_{2}^{k-3} $ and we obtain (where the constant $C$ in the first line below is the same as in the previous estimate)

\begin{align*}
\mathbb{E}[  \delta^{\frac{\max(3-k,0)}{2}} &  (k+1)^{\vert \alpha^{x} \vert} (1+\vert x \vert_{\mathbb{R}^{d}}^{C}+\vert Z^{\delta}_{t} \vert_{\mathbb{R}^{N}}^{C})\vert I_{d\times d} - \nabla_{x} \psi \vert_{\mathbb{R}^{d \times d}}^{k}   (x,t-\delta,\delta^{\frac{1}{2}}Z^{\delta}_{t},\delta)  \mathbf{1}_{\delta^{\frac{1}{2}}Z^{\delta}_{t} \in \mathcal{D}_{\eta_{2}}}     ] \\
 \leqslant & \delta^{\frac{\max(k,3)}{2}} \eta_{2}^{\max(k-3,0)(\ExpPsiouf_{3}  +1)}   (k+1)^{\vert \alpha^{x} \vert}  4^{k}  \CstePsiouf_{3} ^{k} \MomZ_{C}(Z^{\delta})(1+\vert x \vert_{\mathbb{R}^{d}}^{C}) .
\end{align*}
Since $\delta^{\frac{1}{2}} \eta_{2}^{\ExpPsiouf_{3}  +1} 4 \CstePsiouf_{3} < \frac{1}{2}$ (see (\ref{hyp:delta_eta_inversibilite})), we also obtain the estimate
\begin{align*}
\vert \partial^{\alpha^{x}}_{x}  \bar{R}^{2,3}_{t}(x) \vert_{\mathbb{R}^{d}}+\vert \partial^{\alpha^{x}}_{x} \bar{R}^{2,4}_{t}(x) \vert_{\mathbb{R}^{d}}  \leqslant &  \delta^{\frac{3}{2}} C\MomZ_{C}(Z^{\delta} ) \CstePsiouf_{\vert \alpha^{x} \vert + 3}^{C}  \CstePsiouf_{V,\vert \alpha^{x} \vert}(1+\vert x \vert_{\mathbb{R}^{d}}^{C}).
\end{align*}
We conclude that, under the assumptions $\mathbf{A}_{1}^{\delta}(\vert \alpha^{x} \vert+4)$ and (\ref{eq:hyp_3_Norme_adhoc_fonction_schema}), $\mathbf{A}_{3}^{\delta}(+\infty)$, and $\delta^{\frac{1}{2}} \eta_{2}^{\ExpPsiouf_{3}  +1} 4 \CstePsiouf_{3} < \frac{1}{2}$, for every $x \in \mathbb{R}^{d}$,
\begin{align*}
\vert \partial^{\alpha^{x}}_{x}  \bar{R}^{2}_{t}(x) \vert_{\mathbb{R}^{d}} \leqslant &  \delta^{\frac{3}{2}} C \MomZ_{C}(Z^{\delta})    \CstePsiouf_{\vert \alpha^{x} \vert+4}^{C}  \CstePsiouf_{V,\vert \alpha^{x} \vert}(1+ \vert x \vert_{\mathbb{R}^{d}}^{C}) .
\end{align*}
We now focus on the study of $\bar{R}_{t}^{3} $. We have
\begin{align*}
\bar{R}^{3}_{t}=&\bar{R}^{3,1}_{t}-\bar{R}^{3,2}_{t}+\bar{R}^{3,3}_{t}-\bar{R}^{3,4}_{t}+\bar{R}^{3,5}_{t},
\end{align*}
where we have introduced 
\begin{align*}
\bar{R}^{3,1}_{t}(x)=&\delta^{\frac{1}{2}}\sum_{l \in \EnsN}   \mathbb{E}[ Z^{\delta,l}_{t}  \hat{R}^{\delta,2}(x,t-\delta,Z^{\delta}_{t})  \nabla_{x} VV_{l}(x,t-\delta) \mathbf{1}_{\delta^{\frac{1}{2}}Z^{\delta}_{t} \in \mathcal{D}_{\eta_{2}}} ] \\
\bar{R}^{3,2}_{t}(x)=&\delta^{\frac{1}{2}}\sum_{l \in \EnsN}   \mathbb{E}[  Z^{\delta,l}_{t} \nabla_{x} V_{l}(x,t-\delta)R^{\delta,1}(x,t-\delta,Z^{\delta}_{t}) \mathbf{1}_{\delta^{\frac{1}{2}}Z^{\delta}_{t} \in \mathcal{D}_{\eta_{2}}} ] \\
\bar{R}^{3,3}_{t}(x)=&   \mathbb{E}[\mathbf{1}_{\delta^{\frac{1}{2}}Z^{\delta}_{t} \in \mathcal{D}_{\eta_{2}}}  \hat{R}^{\delta,2}(x,t-\delta,Z^{\delta}_{t})  (  \delta ( \nabla_{x} V V_{0}+  \partial_t V+ \frac{1}{2}\sum_{l \in \EnsN}  V_{l}^{\mbox{\textsc{t}}}  \Hessiennex   V  V_{l})(x,t-\delta)  +R^{\delta,1}(x,t-\delta,Z^{\delta}_{t}) ) ] \\
\bar{R}^{3,4}_{t}(x)=& \delta^{2}    \mathbb{P}( \delta^{\frac{1}{2}}Z^{\delta}_{t} \in \mathcal{D}_{\eta_{2}} ) ( \nabla_{x} V_{0}- \sum_{l \in \EnsN} \nabla_{x} V_{l}^{2})(x,t-\delta)  (   \nabla_{x} V V_{0}+  \partial_t V  +  \frac{1}{2}\sum_{l \in \EnsN}  V_{l}^{\mbox{\textsc{t}}} \Hessiennex   V  V_{l})(x,t-\delta)   ,\\
\bar{R}^{3,5}_{t}(x)=&
-\delta \sum_{i,l \in \EnsN}
\mathbb{E}[ (Z^{\delta,i}_{t}Z^{\delta,l}_{t}-\mathbf{1}_{i=l})
\nabla_{x}V_{i} \nabla_{x}V V_{l}(x,t-\delta)
\mathbf{1}_{\delta^{\frac{1}{2}}Z^{\delta}_{t}
\in \mathcal{D}_{\eta_{2}}} ]
\\
&-\delta
\mathbb{E}[
(\nabla_{x}V_{0}
-\sum_{l \in \EnsN}\nabla_{x}V_{l}^{2})(x,t-\delta)(
\sum_{i \in \EnsN}
\delta^{\frac{1}{2}} Z^{\delta,i}_{t}
\nabla_{x}V
V_{i}(x,t-\delta) +R^{\delta,1}(x,t-\delta,Z^{\delta}_{t}))
\mathbf{1}_{\delta^{\frac{1}{2}}Z^{\delta}_{t}
\in \mathcal{D}_{\eta_{2}}}
] \\
&-\delta^{\frac{3}{2}}
\sum_{i \in \EnsN}
\mathbb{E}[\mathbf{1}_{\delta^{\frac{1}{2}}Z^{\delta}_{t}
\in \mathcal{D}_{\eta_{2}}}
Z^{\delta,i}_{t}
\nabla_{x}V_{i}
(\nabla_{x}VV_{0}
+\partial_{t}V
+\frac{1}{2}\sum_{j \in \EnsN}
V_{j}^{\mbox{\textsc{t}}}
\Hessiennex V
V_{j})(x,t-\delta)
].
\end{align*}

Using a similar approach as in the study of $\bar{R}^{1}_{t}$ and $\bar{R}^{2}_{t}$, we deduce from (\ref{hyp:delta_eta_inversibilite}), $\mathbf{A}_{1}^{\delta}(\vert \alpha^{x} \vert +4 )$ (see (\ref{eq:hyp_1_Norme_adhoc_fonction_schema}) and (\ref{eq:hyp_3_Norme_adhoc_fonction_schema})) and $\mathbf{A}_{3}^{\delta}(+\infty)$ that
\begin{align*}
\vert \partial^{\alpha^{x}}_{x}  \bar{R}^{3}_{t}(x) \vert_{\mathbb{R}^{d}} \leqslant &  \delta^{\frac{3}{2}} C \MomZ_{C}(Z^{\delta})   \CstePsiouf_{\vert \alpha^{x} \vert+4}^{C}  \CstePsiouf_{V,\vert \alpha^{x} \vert +3}^{C} (1+ \vert x \vert_{\mathbb{R}^{d}}^{C}) .
\end{align*}
To complete the proof, it remains to study $\bar{R}_{t}^{4} $, $\bar{R}_{t}^{5} $ and $\bar{R}_{t}^{6} $. As a direct consequence of the Markov  inequality, we have 
\begin{align*}
\mathbb{E}[ \sum_{i \in \EnsN} \vert Z^{\delta,i}_{t} \vert  \mathbf{1}_{\delta^{\frac{1}{2}}Z^{\delta}_{t} \notin \mathcal{D}_{\eta_{2}}} ] +
 \mathbb{P}( \delta^{\frac{1}{2}}Z^{\delta}_{t} \notin \mathcal{D}_{\eta_{2}} )\leqslant \delta^{\frac{3}{2}} \MomZ_{C}(Z^{\delta}).
\end{align*}
Consequently
\begin{align*}
\vert \partial^{\alpha^{x}}_{x}  \bar{R}^{4}_{t}(x) \vert_{\mathbb{R}^{d}} \leqslant & \delta^{\frac{3}{2}}  C\MomZ_{C}(Z^{\delta})   \CstePsiouf_{\vert \alpha^{x} \vert+2} \CstePsiouf_{V,\vert \alpha^{x} \vert+1}(1+ \vert x \vert_{\mathbb{R}^{d}}^{\ExpPsiouf  _{\vert \alpha^{x} \vert + 2} +\ExpPsiouf  _{V,\vert \alpha^{x} \vert + 1}}),
\end{align*}
\begin{align*}
\vert \partial^{\alpha^{x}}_{x}  \bar{R}^{5}_{t}(x) \vert_{\mathbb{R}^{d}} \leqslant & \delta^{\frac{3}{2}}  C\MomZ_{C} (Z^{\delta})   \CstePsiouf_{\vert \alpha^{x} \vert+3}^{2} \CstePsiouf_{V,\vert \alpha^{x} \vert+2}(1+ \vert x \vert_{\mathbb{R}^{d}}^{2 \ExpPsiouf  _{\vert \alpha^{x} \vert + 3} +\ExpPsiouf  _{V,\vert \alpha^{x} \vert + 2}}),
\end{align*}
and
\begin{align*}
\vert \partial^{\alpha^{x}}_{x}  \bar{R}^{6}_{t}(x) \vert_{\mathbb{R}^{d}} \leqslant & \delta^{\frac{3}{2}}  C\MomZ_{C} (Z^{\delta}) \CstePsiouf_{V,\vert \alpha^{x} \vert}(1+ \vert x \vert_{\mathbb{R}^{d}}^{\ExpPsiouf  _{V,\vert \alpha^{x} \vert }}),
\end{align*}
and this completes the proof of (\ref{eq:borne_reste_moy_dvpt_Lie}).
\VersionLong{
To conclude, it remains to study $\tilde{\mathscr{R}}V$. We first observe that
$\tilde{\mathscr{R}}V(x,t-\delta,z) =(R^{\delta}V(x,t-\delta,z) \mathbf{1}_{\delta^{\frac{1}{2}}z \in \mathcal{D}_{\eta_{2}}}- \mathbb{E}[ R^{\delta}V(x,t-\delta,Z^{\delta}_{t}) \mathbf{1}_{\delta^{\frac{1}{2}}Z^{\delta}_{t} \in \mathcal{D}_{\eta_{2}}}])+\tilde{R}_{t}^{4}(x,z)+\tilde{R}_{t}^{5}(x,z)+\tilde{R}_{t}^{6}(x,z)$, with
\begin{align*}
\tilde{R}_{t}^{4}(x,z)=& - \delta^{\frac{1}{2}}\sum_{i \in \EnsN} V^{[i]}(x,t-\delta) ( z^{i} \mathbf{1}_{\delta^{\frac{1}{2}}z \notin \mathcal{D}_{\eta_{2}}}-\mathbb{E}[ Z^{\delta,i}_{t} \mathbf{1}_{\delta^{\frac{1}{2}} Z^{\delta}_{t}  \notin \mathcal{D}_{\eta_{2}}} ] ),\\
\tilde{R}_{t}^{5}(x,z)=& -\delta V^{[0]}(x,t-\delta) (\mathbf{1}_{\delta^{\frac{1}{2}}z \notin \mathcal{D}_{\eta_{2}}}-\mathbb{P}( \delta^{\frac{1}{2}}Z^{\delta}_{t} \notin \mathcal{D}_{\eta_{2}}) ) \\
\tilde{R}_{t}^{6}(x,z)=& -V(x,t-\delta) (\mathbf{1}_{\delta^{\frac{1}{2}}z \notin \mathcal{D}_{\eta_{2}}}-\mathbb{P}( \delta^{\frac{1}{2}}Z^{\delta}_{t} \notin \mathcal{D}_{\eta_{2}}) ). 
\end{align*}
Using $\mathbf{A}_{1}^{\delta}(2)$ (see (\ref{eq:hyp_1_Norme_adhoc_fonction_schema})) and $\mathbf{A}_{3}^{\delta}(+\infty)$ (see (\ref{eq:hyp:moment_borne_Z})), we obtain
\begin{align*}
\vert \tilde{R}_{t}^{4}(x,z) \vert_{\mathbb{R}^{d}}=& \vert \delta^{\frac{1}{2}}\sum_{i \in \EnsN} V^{[i]}(x,t-\delta)(z^{i} \mathbf{1}_{\delta^{\frac{1}{2}} z \notin \mathcal{D}_{\eta_{2}}} -  \mathbb{E}[ Z^{\delta,i}_{t} \mathbf{1}_{\delta^{\frac{1}{2}} Z^{\delta}_{t}  \notin \mathcal{D}_{\eta_{2}}} ] ) \vert_{\mathbb{R}^{d}} \\
\leqslant &  \delta C\CstePsiouf_{2}\CstePsiouf_{V,1} \MomZ_{C}(Z^{\delta}) (1+\vert x\vert_{\mathbb{R}^{d}}^{2(\ExpPsiouf  _{2}+\ExpPsiouf  _{V,1})} + \vert z \vert_{\mathbb{R}^{N}}^{C }) ,
\end{align*}
and using $\mathbf{A}_{1}^{\delta}(3)$ (see (\ref{eq:hyp_1_Norme_adhoc_fonction_schema})), we get $
\vert \tilde{R}_{t}^{5}(x,z) \vert_{\mathbb{R}^{d}}
\leqslant   \delta C \CstePsiouf_{3}^{2}\CstePsiouf_{V,2}  (1+\vert x\vert_{\mathbb{R}^{d}}^{2\ExpPsiouf  _{3}+\ExpPsiouf  _{V,2}} ) $ and $
\vert \tilde{R}_{t}^{6}(x,z) \vert_{\mathbb{R}^{d}}
\leqslant   \delta C \CstePsiouf_{V,0}   \MomZ_{C}(Z^{\delta}) (1+\vert x\vert_{\mathbb{R}^{d}}^{\ExpPsiouf  _{V,0}} +\vert z \vert_{\mathbb{R}^{N}}^{C})$. 
The remaining terms can be treated as in the estimate of $\bar{\mathscr{R}}$, but more simply, since no derivatives are involved. This completes the proof of (\ref{eq:borne_reste_mart_dvpt_Lie}).
}
\end{proof}

\subsection{Proof of Lemma \ref{lem:Norris}}

\begin{proof}
The proof relies on estimates under suitable localization events, together with bounds on the probabilities of their complements.\\
\textbf{Step 1.} To begin we show that for every $\epsilon \in [ \underline{\epsilon}_{1}(\delta),\bar{\epsilon}_{1}(\delta) ]$, every $s \in (3r,\frac{1}{2})$, $u\in(0,\frac{1}{2}-s)$, every $p,v,v^{\prime}>0$, and every $q\geqslant 4$,
 \begin{align*}
  \mathbb{P}(\delta & \sum_{t \in \EnsT }  Y_t^{2} < \epsilon,  \delta\sum_{t \in \EnsT } \mathbb{E}[\OpeNorrisMart (Y)_{t-\delta}^{2} \vert \mathcal{F}_{t-\delta}] + \OpeNorrisDrift (Y)_{t-\delta}^{2} \geqslant \epsilon^r, \mathcal{A}^{Y}_{u,q})     \\
 \leqslant  &  \epsilon^{p} \mathbb{E}[\vert Y_{0} \vert^{\frac{p}{v}}] +\mathbb{P}(\delta Y_{0}^{2} \geqslant \epsilon) \\
&+ \delta^{\frac{q}{4}} ( \delta^{\frac{q}{4}} \epsilon^{-q(s+2u)} +\epsilon^{-q(s+u)}+\epsilon^{-q\frac{(2+v^{\prime})}{4}})2^{5q} (1 \vee T^{2q})(1+\sup_{t \in \EnsT } \mathbb{E}[\vert Y_{t-\delta} \vert^{q} ])  \\
& + 2\exp(-\frac{\epsilon^{-2s}}{16})+2\exp(-2 \epsilon^{-v^{\prime}})+  2 \exp(-\frac{\epsilon^{2(s+u)-1}}{2^{12} T^{2}})  \\
 &+   \mathbb{P}(\delta\sum_{t \in \EnsT } Y_t ^{2} < \epsilon,  \delta\sum_{t \in \EnsT } \mathbb{E}[\OpeNorrisMart (Y)_{t-\delta} ^{2} \vert \mathcal{F}_{t-\delta}] + \OpeNorrisDrift (Y)_{t-\delta} ^{2} \geqslant \epsilon^r, \mathcal{A}_{1}, Y_{0}^{2} < \frac{\epsilon^{s}}{T} ,\mathcal{A}^{Y}_{u,q})  ,
  \end{align*}
with $\underline{\epsilon}_{1}(\delta)=  \max( ( 16 \delta T^{2} )^{\frac{1}{s+2u}}, (2^{10} \delta T^{3} )^{\frac{1}{2u+2s+2v}} ),  \bar{\epsilon}_{1}(\delta)= \min(2^{-\frac{1}{1-s}},( 2^{10} T^{3} )^{-\frac{1}{1-2s-2u}}) $ and
\begin{align*}
   \mathcal{A}^{Y}_{u,q} := & \{  \sup_{t \in \EnsT } \vert \OpeNorrisDrift (Y)_{t-\delta}\vert  \leqslant \epsilon^{-u}\} \cap  \{  \sup_{t \in \EnsT }  \mathbb{E}[ \vert \OpeNorrisMart (Y)_{t-\delta} \vert^{q} \vert \mathcal{F}_{t-\delta}] \leqslant \epsilon^{-qu} \} \\
   & \cap  \{  \sup_{t \in \EnsT } \vert \OpeNorrisDrift (\OpeNorrisDrift (Y))_{t-\delta}\vert \leqslant \epsilon^{-u}  \} \cap   \ \{  \sup_{t \in \EnsT } \mathbb{E}[ \vert \OpeNorrisMart(\OpeNorrisDrift (Y))_{t-\delta} \vert^q \vert \mathcal{F}_{t-\delta}] \leqslant \epsilon^{-qu} \} ,\\
	\mathcal{A}_{1}: =& \{ \delta^{2} \sum_{(w,t) \in \EnsT^{2}_{\mathsmaller{\leqslant}}} \mathbb{E}[ \OpeNorrisMart (Y)_{w-\delta} ^{2}\vert \mathcal{F}_{w-\delta}]< \epsilon^{s} \} \cap \{ \delta^{2}\sum_{(w,t) \in \EnsT^{2}_{\mathsmaller{\leqslant}}}  \OpeNorrisMart (Y)_{w-\delta} ^{2} < \epsilon^{s} \},
\end{align*}
where $ \EnsT^{2}_{\mathsmaller{\leqslant}} =\{(w,t)\in \EnsT^{2}, w \leqslant t\}$.
We begin by writing, with notation introduced in (\ref{eq:Doob_martin_expo_ineq}), for every $t \in \EnsT $, 
  \begin{align*}
 Y_{t}^{2}=&Y_{t-\delta}^{2}+\delta^{\frac{1}{2}} 2 \OpeNorrisMart (Y)_{t-\delta}Y_{t-\delta}+\delta (2 \OpeNorrisDrift (Y)_{t-\delta} Y_{t-\delta}+\OpeNorrisMart (Y)_{t-\delta} ^{2}) \\
 &+\delta^{\frac{3}{2}} 2\OpeNorrisMart (Y)_{t-\delta} \OpeNorrisDrift (Y)_{t-\delta} +\delta^{2} \OpeNorrisDrift (Y)_{t-\delta}^{2} \\
 =& Y_{0}^{2} + \sum_{w \in \EnsT,w\leqslant t } \delta^{\frac{1}{2}} 2 \OpeNorrisMart (Y)_{w-\delta}Y_{w-\delta}+\delta (2\OpeNorrisDrift (Y)_{w-\delta}Y_{w-\delta}+ \OpeNorrisMart (Y)_{w-\delta} ^{2})  \\
 &+\delta^{\frac{3}{2}}2\OpeNorrisMart (Y)_{w-\delta} \OpeNorrisDrift (Y)_{w-\delta}+\delta^{2}  \OpeNorrisDrift (Y)_{w-\delta} ^{2} .
 \end{align*} 
We now introduce the following event, on which we can control $\delta\sum_{t \in \EnsT }  Y_t ^{2} $,
\begin{align*}
\mathcal{A}_{2}  :=    & \{  \delta^{\frac{3}{2}} \vert \sum_{(w,t) \in \EnsT^{2}_{\mathsmaller{\leqslant}}} 2\OpeNorrisMart (Y)_{w-\delta}Y_{w-\delta}\vert <  \frac{\epsilon^{s}}{8}  \} \cap  \{  \vert \delta^{2} \sum_{(w,t) \in \EnsT^{2}_{\mathsmaller{\leqslant}}} 2 \OpeNorrisDrift (Y)_{w-\delta}Y_{w-\delta}\vert <  \frac{\epsilon^{s}}{8} \} \\
& \cap  \{ \vert \delta^{2} \sum_{(w,t) \in \EnsT^{2}_{\mathsmaller{\leqslant}}}  \OpeNorrisMart (Y)_{w-\delta} ^{2} - \mathbb{E}[ \OpeNorrisMart (Y)_{w-\delta} ^{2}\vert \mathcal{F}_{w-\delta}]\vert <  \frac{\epsilon^{s}}{8} \} \\
& \cap \{ \delta^{2} \sum_{(w,t) \in \EnsT^{2}_{\mathsmaller{\leqslant}}} \vert \delta^{\frac{1}{2}} 2\OpeNorrisMart (Y)_{w-\delta} \OpeNorrisDrift (Y)_{w-\delta}+\delta  \OpeNorrisDrift (Y)_{w-\delta}  ^{2} \vert <  \frac{\epsilon^{s}}{8} \}.
\end{align*}
Our idea is to distinguish the cases according to the event $\mathcal{A}_{2}$. We first estimate $\mathbb{P}(\delta\sum_{t \in \EnsT } \vert Y_t \vert^{2} <  \epsilon,\mathcal{A}_{2} ^c,\mathcal{A}^{Y}_{u,q})$. We then treat each event appearing in the definition of $\mathcal{A}_{2}$  separately. In the sequel, for $t \in \pi^{\delta}$ we will denote $\mathfrak{n}_{t} = \vert \EnsT  \vert -t \delta^{-1}$ and recall also that $\vert \EnsT \vert \delta = T$. 
Concerning the first term, since $ \sup_{t \in \EnsT }  \mathbb{E}[ \OpeNorrisMart (Y)_{t-\delta} ^{2} \vert \mathcal{F}_{t-\delta}] \leqslant \epsilon^{-2u}$ on $ \mathcal{A}^{Y}_{u,q}$, for every $s \in (3r,\frac{1}{2})$, $u\in(0,\frac{1}{2}-s)$,  we have
\begin{align*}
\mathbb{P}( \vert \delta^{\frac{3}{2}}  & \sum_{(w,t) \in \EnsT^{2}_{\mathsmaller{\leqslant}}} 2\OpeNorrisMart (Y)_{w-\delta}Y_{w-\delta}\vert \geqslant \frac{\epsilon^{s}}{8},  \delta  \sum_{t \in \EnsT } Y_t ^{2}< \epsilon, \mathcal{A}^{Y}_{u,q})    \\
\leqslant& \mathbb{P}( \vert \delta^{\frac{3}{2}}  \sum_{t \in \EnsT }  \mathfrak{n}_{t-\delta} \OpeNorrisMart (Y)_{t-\delta}Y_{t-\delta}\vert \geqslant \frac{\epsilon^{s}}{16}, \delta^{3}  \sum_{t \in \EnsT }  \mathfrak{n}_{t-\delta} ^{2}(\OpeNorrisMart (Y)_{t-\delta} ^{2}+\mathbb{E}[ \OpeNorrisMart (Y)_{t-\delta}^{2}\vert\mathcal{F}_{t-\delta}] ) Y_{t-\delta}^{2}< 8  T^{2} \epsilon^{1-2u}) \\
& + \mathbb{P}(\delta^{3}  \sum_{t \in \EnsT }  \mathfrak{n}_{t-\delta} ^{2}( \OpeNorrisMart (Y)_{t-\delta}^{2}+\mathbb{E}[ \OpeNorrisMart (Y)_{t-\delta}^{2}\vert\mathcal{F}_{t-\delta}] ) Y_{t-\delta} ^{2}\geqslant 8 T^{2} \epsilon^{1-2u},   \delta  \sum_{t \in \EnsT } Y_{t-\delta} ^{2}< 2 \epsilon, \mathcal{A}^{Y}_{u,q}) \\
& +\mathbb{P}(\delta Y_{0} ^{2} \geqslant \epsilon) .
\end{align*}
Using the martingale exponential inequality (\ref{eq:Doob_martin_expo_ineq}), the first term above is bounded by $ 2 \exp(-\frac{\epsilon^{2(s+u)-1}}{2^{12} ( \delta\vert  \EnsT  \vert )^{2}}) $. We now study the second term of the $r.h.s.$ above. Let us denote $H_{t}= \OpeNorrisMart (Y)_{t}^{2} - \mathbb{E}[\OpeNorrisMart (Y)_{t} ^{2}\vert \mathcal{F}_{t}] $, $t \in \pi^{\delta}$. Then the aforementioned term is bounded by 
\begin{align*}
 & \mathbb{P}  (  \delta^{3}  \sum_{t \in \EnsT }  \mathfrak{n}_{t-\delta} ^{2} H_{t-\delta} Y_{t-\delta}^{2} \geqslant  4 T^{2}  \epsilon^{1-2u} , \mathcal{A}^{Y}_{u,q}) \\
& +  \mathbb{P}  (  \delta^{3}  \sum_{t \in \EnsT }  \mathfrak{n}_{t-\delta} ^{2} \mathbb{E}[ \OpeNorrisMart (Y)_{t-\delta}^{2}\vert\mathcal{F}_{t-\delta}]  Y_{t-\delta}^{2} \geqslant 2 T^{2}  \epsilon^{1-2u}, \delta  \sum_{t \in \EnsT } Y_{t-\delta}^{2}< 2 \epsilon , \mathcal{A}^{Y}_{u,q})  .
\end{align*}

Since $\mathfrak{n}_{t} \leqslant \vert \EnsT  \vert $ for every $t \in \EnsT $ and $  \sup_{t \in \EnsT }  \mathbb{E}[ \OpeNorrisMart (Y)_{t-\delta}^{2} \vert \mathcal{F}_{t-\delta}] \leqslant \epsilon^{-2u}$ on $ \mathcal{A}^{Y}_{u,q}$, the second term of the $r.h.s.$ above is equal to zero. We then focus on the first term of the $r.h.s.$ above. Let $v^{\prime}>0$. In preparation for using the martingale exponential inequality, we observe that the first term above is bounded by
\begin{align*}
 \mathbb{P}  &(  \delta^{3}  \sum_{t \in \EnsT }  \mathfrak{n}_{t-\delta} ^{2}H_{t-\delta} Y_{t-\delta}^{2} \geqslant  2 T^{2}  \epsilon^{1-2u} ,     \delta^{6}  \sum_{t \in \EnsT }  \mathfrak{n}_{t-\delta} ^{4} (H_{t-\delta}^{2}+ \mathbb{E}[H_{t-\delta}^{2}\vert\mathcal{F}_{t-\delta}] )  Y_{t-\delta}^{4} <  T^{4} \epsilon^{2+v^{\prime}-4u} )  \\
& +  \mathbb{P}  (  \delta^{6}  \sum_{t \in \EnsT }  \mathfrak{n}_{t-\delta} ^{4} (H_{t-\delta} ^{2}+ \mathbb{E}[ H_{t-\delta}^{2}\vert\mathcal{F}_{t-\delta}] ) Y_{t-\delta} ^{4} \geqslant   T^{4}   \epsilon^{2+v^{\prime}-4u} , \mathcal{A}^{Y}_{u,q}) .
\end{align*}
Using (\ref{eq:Doob_martin_expo_ineq}), the first term above is bounded by $2\exp(-2 \epsilon^{-v^{\prime}})$. To bound the second term, we apply  H\"older's  and Markov's inequalities. Thus, for every $q^{\prime} \in [1, \frac{q}{4}]$, using also the triangle inequality when $q^{\prime}=1$, we deduce that it is bounded by
\begin{align*}
 \mathbb{P} & (  T^{q^{\prime}-1}  \delta^{q^{\prime}+1} \sum_{t \in \EnsT }( H_{t-\delta}^{2}+ \mathbb{E}[ H_{t-\delta}^{2}\vert\mathcal{F}_{t-\delta}] )^{q^{\prime}} \vert Y_{t-\delta} \vert^{4q^{\prime}} \geqslant  \epsilon^{q^{\prime}(2+v^{\prime}-4u)},\mathcal{A}^{Y}_{u,q})  \\
\leqslant & \delta^{q^{\prime}}\epsilon^{-q^{\prime}(2+v^{\prime}-4u)}T^{q^{\prime}-1} \delta\sum_{t \in \EnsT } 2^{q^{\prime}}\mathbb{E}[ \vert H_{t-\delta}\vert^{2q^{\prime}} \vert Y_{t-\delta}\vert^{4q^{\prime}} \mathbf{1}_{\mathbb{E}[ \vert \OpeNorrisMart (Y)_{t-\delta}\vert^{q} \vert\mathcal{F}_{t-\delta}] \leqslant\epsilon^{-qu}} ] \\
\leqslant &  \delta^{q^{\prime}}  \epsilon^{-q^{\prime}(2+v^{\prime}-4u)}2^{3q^{\prime}-1} T^{q^{\prime}-1} (\delta  \sum_{t \in \EnsT }  \mathbb{E}[ \vert Y_{t-\delta} \vert^{4q^{\prime}} \mathbb{E}[ \vert \OpeNorrisMart (Y)_{t-\delta}\vert^{4q^{\prime}}\vert\mathcal{F}_{t-\delta}]  \mathbf{1}_{\mathbb{E}[ \vert \OpeNorrisMart (Y)_{t-\delta} \vert^{q} \vert \mathcal{F}_{t-\delta}] \leqslant \epsilon^{-qu} }  ] \\
& +  \delta  \sum_{t \in \EnsT }  \mathbb{E}[ \vert Y_{t-\delta} \vert^{4q^{\prime}} \vert \OpeNorrisMart (Y)_{t-\delta}\vert^{4q^{\prime}} \mathbf{1}_{\mathbb{E}[ \vert \OpeNorrisMart (Y)_{t-\delta} \vert^{q} \vert \mathcal{F}_{t-\delta}] \leqslant \epsilon^{-qu} }  ])\\
\leqslant &  \delta^{q^{\prime}}  \epsilon^{-(2+v^{\prime})q^{\prime}}2^{3q^{\prime}} T^{q^{\prime}}  \sup_{t \in \EnsT } \mathbb{E}[ \vert Y_{t-\delta} \vert^{4q^{\prime}} ],
\end{align*}

where for the last inequality, we note that $q^{\prime} \leqslant \frac{q}{4}$ and use the tower property to deduce that
\begin{align*}
\mathbb{E}[\vert  Y_{t-\delta} & \OpeNorrisMart (Y)_{t-\delta} \vert^{4q^{\prime}} \mathbf{1}_{\mathbb{E}[ \vert \OpeNorrisMart (Y)_{t-\delta} \vert^{q} \vert \mathcal{F}_{t-\delta}] \leqslant \epsilon^{-qu} }  ]\\
 \leqslant & \mathbb{E}[ \vert Y_{t-\delta} \vert^{4q^{\prime}} \mathbb{E}[ \vert \OpeNorrisMart (Y)_{t-\delta} \vert^{4q^{\prime}}\vert\mathcal{F}_{t-\delta}]  \mathbf{1}_{\mathbb{E}[ \vert \OpeNorrisMart (Y)_{t-\delta} \vert^{4q^{\prime}} \vert \mathcal{F}_{t-\delta}] \leqslant \epsilon^{-4q^{\prime}u} } ]  \leqslant  \epsilon^{-4q^{\prime}u}  \mathbb{E}[ \vert Y_{t-\delta} \vert^{4q^{\prime}}  ].
\end{align*}


\VersionLong{Gathering all the previous bounds, we complete the estimation of $\mathbb{P}( \vert \delta^{\frac{3}{2}}   \sum_{w, t \in \EnsT ;w\leqslant t} 2\OpeNorrisMart (Y)_{w-\delta}Y_{w-\delta}\vert \geqslant \frac{\epsilon^{s}}{8},  \delta  \sum_{t \in \EnsT }  Y_t ^{2}< \epsilon, \mathcal{A}^{Y}_{u,q}) $.  Using a similar approach we also obtain
   \begin{align*}
\mathbb{P}(\vert & \delta^{2} \sum_{(w,t) \in \EnsT^{2}_{\mathsmaller{\leqslant}}}   \OpeNorrisMart (Y)_{w-\delta} ^{2} -  \mathbb{E}[ \OpeNorrisMart (Y)_{w-\delta} ^{2}\vert \mathcal{F}_{w-\delta}]\vert \geqslant \frac{\epsilon^{s}}{8},\mathcal{A}^{Y}_{u,q}) =  \mathbb{P}  (  \delta^{2} \vert  \sum_{t \in \EnsT } \mathfrak{n}_{t-\delta} H_{t-\delta} \vert \geqslant \frac{\epsilon^{s}}{8}, \mathcal{A}^{Y}_{u,q}) \\
\leqslant & \mathbb{P}  (  \delta^{2} \vert \sum_{t \in \EnsT }   \mathfrak{n}_{t-\delta} H_{t-\delta} \vert \geqslant  \frac{\epsilon^{s}}{8},    \delta^{4}  \sum_{t \in \EnsT }  \mathfrak{n}_{t-\delta} ^{2} (H_{t-\delta}^{2}+ \mathbb{E}[ H_{t-\delta}^{2}\vert\mathcal{F}_{t-\delta}] ) <  \frac{\epsilon^{4s}}{8}, \mathcal{A}^{Y}_{u,q})  \\
& +  \mathbb{P}  (  \delta^{4}  \sum_{t \in \EnsT }  \mathfrak{n}_{t-\delta} ^{2} (H_{t-\delta}^{2}+ \mathbb{E}[ H_{t-\delta}^{2}\vert\mathcal{F}_{t-\delta}] )  \geqslant   \frac{\epsilon^{4s}}{8}, \mathcal{A}^{Y}_{u,q}) ,
  \end{align*}
where, as a consequence of the exponential martingale inequality (\ref{eq:Doob_martin_expo_ineq}), the first term is bounded by $2\exp(-\frac{1}{16}\epsilon^{-2s})$. To complete the upper bound on the $l.h.s.$ above, we notice that, owing to the H\"older and Markov's inequalities, for every $q^{\prime} \in [1,\frac{q}{4}]$,
\begin{align*}
\mathbb{P}  (  \delta^{4}  \sum_{t \in \EnsT }  \mathfrak{n}_{t-\delta} ^{2} (  H_{t-\delta}^{2}+  \mathbb{E}[ H_{t-\delta}^{2}\vert\mathcal{F}_{t-\delta}] )  \geqslant   \frac{\epsilon^{4s}}{8}, \mathcal{A}^{Y}_{u,q}) \leqslant   & \delta^{q^{\prime}}  \epsilon^{-4(s+u)q^{\prime}}2^{6q^{\prime}} T^{3q^{\prime}}  .
\end{align*}
}
\VersionCourt{
Gathering all the previous bounds, we complete the estimation of $\mathbb{P}( \vert \delta^{\frac{3}{2}}   \sum_{w, t \in \EnsT ;w\leqslant t} 2\OpeNorrisMart (Y)_{w-\delta}Y_{w-\delta}\vert \geqslant \frac{\epsilon^{s}}{8},  \delta  \sum_{t \in \EnsT }  Y_t ^{2}< \epsilon, \mathcal{A}^{Y}_{u,q}) $. Using a similar approach by identifying the martingale decomposition, using (\ref{eq:Doob_martin_expo_ineq}) and bounding the remaining term by Markov's inequality, we also obtain for every $q^{\prime} \in [1,\frac{q}{4}]$
   \begin{align*}
\mathbb{P}(\vert  \delta^{2} \sum_{(w,t) \in \EnsT^{2}_{\mathsmaller{\leqslant}}}   \OpeNorrisMart (Y)_{w-\delta} ^{2} -  \mathbb{E}[ \OpeNorrisMart (Y)_{w-\delta} ^{2}\vert \mathcal{F}_{w-\delta}]\vert \geqslant \frac{\epsilon^{s}}{8},\mathcal{A}^{Y}_{u,q})  
\leqslant & 2\exp(-\frac{\epsilon^{-2s}}{16}) +  \delta^{q^{\prime}}  \epsilon^{-4(s+u)q^{\prime}}2^{6q^{\prime}} T^{3q^{\prime}} .
  \end{align*}
}

For the next term, we remark that, since $\sup_{t \in \EnsT } \vert \OpeNorrisDrift (Y)_{t-\delta}\vert \mathbf{1}_{\mathcal{A}^{Y}_{u,q}}  \leqslant \epsilon^{-u}$, it follows from the Cauchy-Schwarz inequality that
  \begin{align*}
  \vert  \sum_{(w,t) \in \EnsT^{2}_{\mathsmaller{\leqslant}}}  \OpeNorrisDrift (Y)_{w-\delta} Y_{w-\delta}\vert \mathbf{1}_{\mathcal{A}^{Y}_{u,q}}  \mathbf{1}_{\vert Y_{0} \vert <\epsilon^{-v} } < \vert \EnsT  \vert^{\frac{3}{2}} \epsilon^{-u}(\epsilon^{-2v}+ \sum_{t \in \EnsT } Y_{t}^{2})^{\frac{1}{2}},
  \end{align*}
  and, for $v>0$, whenever $\epsilon \in [ ( 32 T^{\frac{3}{2}} \delta^{\frac{1}{2}} )^{\frac{1}{u+s+v}}, ( 2^{10} T^{3} )^{-\frac{1}{1-2s-2u}} ]$,
   \begin{align*}
  \mathbb{P}&(\delta\sum_{t \in \EnsT } Y_{t} ^{2} < \epsilon,\vert \delta^{2} \sum_{(w,t) \in \EnsT^{2}_{\mathsmaller{\leqslant}}} 2 \OpeNorrisDrift (Y)_{w-\delta} Y_{w-\delta}\vert \geqslant \frac{\epsilon^{s}}{8},\mathcal{A}^{Y}_{u,q},\vert Y_{0} \vert < \epsilon^{-v})=0.
  \end{align*}
Moreover, from Markov's inequality, we deduce that for every $p>0$, $\mathbb{P}(\vert Y_{0} \vert \geqslant \epsilon^{-v})\leqslant \epsilon^{p} \mathbb{E}[\vert Y_{0} \vert^{\frac{p}{v}}]$.\\

Finally, for every $\epsilon \geqslant  (16 \delta T^{2} )^{\frac{1}{s+2u}}$, the estimates satisfied on $\mathcal{A}^{Y}_{u,q}$, together with the Markov and H\"older inequalities, yield 
\begin{align*}
\mathbb{P}( \delta^{2} \sum_{(w,t) \in \EnsT^{2}_{\mathsmaller{\leqslant}}} &\vert \delta^{\frac{1}{2}} 2\OpeNorrisMart (Y)_{w-\delta} \OpeNorrisDrift (Y)_{w-\delta}+\delta  \OpeNorrisDrift (Y)_{w-\delta}^{2} \vert \geqslant \frac{\epsilon^{s}}{8},\mathcal{A}^{Y}_{u,q}) \\
\leqslant & \mathbb{P}( \delta^{5/2} \sum_{(w,t) \in \EnsT^{2}_{\mathsmaller{\leqslant}}} \vert 2\OpeNorrisMart (Y)_{w-\delta} \OpeNorrisDrift (Y)_{w-\delta} \vert \geqslant \frac{\epsilon^{s}}{8}-\delta T^{2} \epsilon^{-2u},\mathcal{A}^{Y}_{u,q})  \\
\leqslant &32^q \vert \EnsT  \vert^{2q-2} \delta^{5q/2}\epsilon^{-q(s+u)} \sum_{(w,t) \in \EnsT^{2}_{\mathsmaller{\leqslant}}}  \mathbb{E}[  \vert \OpeNorrisMart (Y)_{w-\delta} \vert^q \mathbf{1}_{\mathcal{A}^{Y}_{u,q}}]
\leqslant  32^q \delta^{\frac{q}{2}} \epsilon^{-q(s+2u)} T^{2q} .
\end{align*}

In particular,  taking $q^{\prime}=\frac{q}{4}$, we have proved that for every $\epsilon \in [ \underline{\epsilon}_{1}(\delta),  ( 2^{10} T^{3} )^{-\frac{1}{1-2s-2u}}]$,
\begin{align*}
\mathbb{P}(\delta\sum_{t \in \EnsT } Y_t^{2} < & \epsilon,\mathcal{A}_{2} ^c,\mathcal{A}^{Y}_{u,q}) \leqslant  \epsilon^{p} \mathbb{E}[\vert Y_{0} \vert^{\frac{p}{v}}] +\mathbb{P}(\delta Y_{0}^{2} \geqslant \epsilon) + \delta^{\frac{q}{4}}  \epsilon^{-q(s+u)}2^{\frac{3q}{2}} T^{\frac{3q}{4}} \\
&+2^{5q} \delta^{\frac{q}{2}} \epsilon^{-q(s+2u)} T^{2q}  +\delta^{\frac{q}{4}}  \epsilon^{-\frac{(2+v^{\prime})q}{4}}2^{\frac{3q}{4}} T^{\frac{q}{4}}  \sup_{t \in \EnsT } \mathbb{E}[\vert Y_{t-\delta} \vert^{q} ]  \\
& + 2\exp(-\frac{\epsilon^{-2s}}{16})+2\exp(-2\epsilon^{-v^{\prime}})  +  2 \exp(-\frac{\epsilon^{2(s+u)-1}}{2^{12}  T^{2}}) .
\end{align*}

Moreover,  we remark that
\begin{align*}
\mathcal{A}_{2}  \subset &  \{ \delta\sum_{t \in \EnsT }  Y_{0}^{2}+ \delta^{2}\sum_{(w,t) \in \EnsT^{2}_{\mathsmaller{\leqslant}}}  \OpeNorrisMart (Y)_{w-\delta} ^{2} <  \delta\sum_{t \in \EnsT }  Y_t ^{2}+\frac{\epsilon^s}{2}  \} \\
& \cap  \{ \delta\sum_{t \in \EnsT } Y_{0}^{2}+ \delta^{2}\sum_{(w,t) \in \EnsT^{2}_{\mathsmaller{\leqslant}}} \mathbb{E}[ \OpeNorrisMart (Y)_{w-\delta} ^{2}\vert \mathcal{F}_{w-\delta}] <  \delta\sum_{t \in \EnsT }  Y_t ^{2}+\frac{\epsilon^s}{2}  \}.
\end{align*}
In particular, for every $\epsilon  \leqslant 2^{-\frac{1}{1-s}} $, we have $\epsilon \leqslant \frac{\epsilon^{s}}{2}$ and
\begin{align*}
 \{ \delta\sum_{t \in \EnsT } Y_t ^{2} < \epsilon  \} \cap \mathcal{A}_{2}   \subset  \{ \delta\sum_{t \in \EnsT } Y_t^{2} < \epsilon  \} \cap \mathcal{A}_{1} \cap \{ Y_{0}^{2} < \frac{\epsilon^{s}}{T} \},
  \end{align*}
and the proof of \textbf{Step 1} is completed.

  \textbf{Step 2.} We show that for every $\epsilon \in(0,\bar{\epsilon}_{2}(\delta)]$  and $u \in (0,\frac{s}{4}-\frac{3r}{4})$,
 \begin{align*}
  \mathbb{P}(\delta\sum_{t \in \EnsT } Y_t^{2} < \epsilon, & \delta\sum_{t \in \EnsT } \mathbb{E}[\OpeNorrisMart (Y)_{t-\delta}^{2} \vert \mathcal{F}_{t-\delta}] \geqslant \frac{\epsilon^{r}}{2},\mathcal{A}_{1},\mathcal{A}^{Y}_{u,q})  =0.
  \end{align*}
with $\bar{\epsilon}_{2}(\delta) =(2^{7} T )^{-\frac{1}{s-3r-4u}} $. 
First, we recall that $ \mathcal{A}^{Y}_{u,q} \subset \{  \sup_{t \in \EnsT }  \mathbb{E}[  \OpeNorrisMart (Y)_{t-\delta} ^{2} \vert \mathcal{F}_{t-\delta}] \leqslant \epsilon^{-2u} \} $ and we observe that 
  \begin{align*}
   \frac{\epsilon^{r+2u}}{4} \epsilon^{-2u} + T \frac{\epsilon^{r}}{4 T}
    =  \frac{\epsilon^{r}}{2}.
  \end{align*} 
In particular, on the set $\{\delta \sum_{t \in \EnsT } \mathbb{E}[ \OpeNorrisMart (Y)_{t-\delta}^{2} \vert \mathcal{F}_{t-\delta}] \geqslant \frac{\epsilon^{r}}{2} \} \cap \{  \sup_{t \in \EnsT }  \mathbb{E}[  \OpeNorrisMart (Y)_{t-\delta}^{2} \vert \mathcal{F}_{t-\delta}] \leqslant \epsilon^{-2u} \}  $,  the following lower bound holds
    \begin{align*}
\delta  \sum_{t \in \EnsT } \mathbf{1}_{\mathbb{E}[  \OpeNorrisMart (Y)_{t-\delta} ^{2} \vert \mathcal{F}_{t-\delta}] \geqslant \frac{\epsilon^{r}}{4 T}}  \geqslant \frac{\epsilon^{r+2u}}{4} .
  \end{align*}
To show this assertion, we simply note that if the reverse (strict) inequality holds, we have
  \begin{align*}
\delta  \sum_{t \in \EnsT } \mathbb{E}[ \OpeNorrisMart (Y)_{t-\delta} ^{2} \vert \mathcal{F}_{t-\delta}]  \leqslant& \delta  \sum_{t \in \EnsT } \mathbf{1}_{\mathbb{E}[  \OpeNorrisMart (Y)_{t-\delta}^{2} \vert \mathcal{F}_{t-\delta}] \geqslant \frac{\epsilon^{r}}{4 T}}  \epsilon^{-2u} + \mathbf{1}_{\mathbb{E}[\OpeNorrisMart (Y)_{t-\delta} ^{2} \vert \mathcal{F}_{t-\delta}] <\frac{\epsilon^{r}}{4 T}}  \frac{\epsilon^{r}}{4 T} < \frac{\epsilon^{r}}{2}.
  \end{align*}
 It follows that
\begin{align*}
 \delta^{2} \sum_{(w,t) \in \EnsT^{2}_{\mathsmaller{\leqslant}}}& \mathbb{E}[  \OpeNorrisMart (Y)_{w-\delta} ^{2} \vert \mathcal{F}_{w-\delta}]    \geqslant  \delta^{2}  \sum_{(w,t) \in \EnsT^{2}_{\mathsmaller{\leqslant}}} \mathbb{E}[  \OpeNorrisMart (Y)_{w-\delta} ^{2} \vert \mathcal{F}_{w-\delta}]   \mathbf{1}_{\mathbb{E}[  \OpeNorrisMart (Y)_{w-\delta} ^{2} \vert \mathcal{F}_{w-\delta}]  \geqslant \frac{\epsilon^{r}}{4 T}}  \\
 \geqslant & \frac{\epsilon^{r}}{4 T} \delta^{2} \sum_{(w,t) \in \EnsT^{2}_{\mathsmaller{\leqslant}}}  \mathbf{1}_{\mathbb{E}[  \OpeNorrisMart (Y)_{w-\delta} ^{2} \vert \mathcal{F}_{w-\delta}]  \geqslant \frac{\epsilon^{r}}{4 T}}  \geqslant    \frac{\epsilon^{r}}{4 T}  \frac{1}{2}  \frac{\epsilon^{r+2u}}{4}(\frac{\epsilon^{r+2u}}{4}+\delta)   
 \geqslant  \frac{\epsilon^{3r+4u}}{2^{7} T} .
 \end{align*}
More specifically, since $\epsilon  \leqslant ( 2^{7} T )^{-\frac{1}{s-3r-4u}}  $, we have
 \begin{align*}
  \{ \delta\sum_{t \in \EnsT } \mathbb{E}[  \OpeNorrisMart (Y)_{t-\delta}^{2} \vert \mathcal{F}_{t-\delta}] \geqslant \frac{\epsilon^r}{2}  \} \cap  \{  \delta^{2}  \sum_{(w,t) \in \EnsT^{2}_{\mathsmaller{\leqslant}}}  \mathbb{E}[  \OpeNorrisMart (Y)_{w-\delta} ^{2} \vert \mathcal{F}_{w-\delta}] <\epsilon^s  \} \cap \mathcal{A}^{Y}_{u,q}=\emptyset.
 \end{align*}

Recalling that we have $\mathcal{A}_{1} \subset  \{ \delta^{2} \sum_{(w,t) \in \EnsT^{2}_{\mathsmaller{\leqslant}}} \mathbb{E}[ \OpeNorrisMart (Y)_{w-\delta} ^{2}\vert \mathcal{F}_{w-\delta}]< \epsilon^{s} \} $ the proof of \textbf{Step 2} is completed.

\textbf{Step 3.}
In this part we obtain that for every $\epsilon \in [\underline{\epsilon}_{3}(\delta),\bar{\epsilon}_{3}(\delta)]$, every $h,s \in (3r,\frac{1}{2})$ with $2h<s$, $u \in(0,\min(\frac{s}{2}-h,\frac{h}{4}-\frac{3r}{4}))$,
 \begin{align*}
  \mathbb{P}(\delta\sum_{t \in \EnsT } & Y_t^{2} < \epsilon,  \delta\sum_{t \in \EnsT } \mathbb{E}[\OpeNorrisMart (Y)_{t-\delta} ^{2} \vert \mathcal{F}_{t-\delta}] < \frac{\epsilon^{r}}{2}, \\
  &\delta\sum_{t \in \EnsT }  \OpeNorrisDrift (Y)_{t-\delta}^{2}\geqslant \frac{\epsilon^{r}}{2} ,\mathcal{A}_{1}, Y_{0}^{2} < \frac{\epsilon^{s}}{T} ,\mathcal{A}^{Y}_{u,q}, \vert \OpeNorrisDrift (Y)_{T} \vert \leqslant \epsilon^{-u}) \\
& \qquad \leqslant  \delta^{\frac{q}{4}}( \delta^{\frac{q}{4}} \epsilon^{-q(h+2u)} +  \epsilon^{-\frac{(2+v^{\prime})q}{4}}) 2^{5q}(1 \vee T^{2q}) (1+ \sup_{t \in \EnsT } \mathbb{E}[\vert Y_{t-\delta} \vert^{q} ]  )+  \epsilon^{\frac{q(s-2h-2u)}{2}}    2^{4q}  T^{q}\\
& \qquad \quad  +2\mathbb{P}(\delta Y_{0} ^{2} \geqslant \epsilon)   +2 \exp(-\frac{\epsilon^{2h+2u-1}}{2^{10} T^{2}}) +2\exp(-2 \epsilon^{-v^{\prime}})+   2 \exp(-\frac{\epsilon^{2h+2u-s}}{ 2^{8}  T}) ,
  \end{align*}
 with
\begin{align*}
\underline{\epsilon}_{3}(\delta) = & ( 16 \delta T^{2} )^{\frac{1}{h+2u}} \quad \mbox{and} \quad \bar{\epsilon}_{3}(\delta) = \min ( ( 2^{7} T)^{-\frac{1}{h-3r-4u}} ,( 16 T)^{-\frac{1}{s-2h-2u}},( 2^{9} T^3)^{-\frac{1}{1-2h-2u}},1).
\end{align*}
 
 
 \VersionCourt{The proof of this result relies on the decomposition: For $t \in \EnsT $ ,
\begin{align*}
 Y_{t}  \OpeNorrisDrift (Y)_{t}= & Y_0 \OpeNorrisDrift (Y)_{0}+ \sum_{\underset{w\leqslant t}{w \in \EnsT }} \delta^{\frac{1}{2}} (\OpeNorrisMart (Y)_{w-\delta} \OpeNorrisDrift (Y)_{w-\delta}+\OpeNorrisMart(\OpeNorrisDrift (Y))_{w-\delta} Y_{w-\delta}) \\
 & + \delta (\OpeNorrisDrift (Y)_{w-\delta} ^{2}+\OpeNorrisMart(\OpeNorrisDrift (Y))_{w-\delta}\OpeNorrisMart (Y)_{w-\delta} +\OpeNorrisDrift(\OpeNorrisDrift (Y))_{w-\delta} Y_{w-\delta} ) \\
 &+\delta^{\frac{3}{2}}(\OpeNorrisMart(\OpeNorrisDrift (Y))_{w-\delta}\OpeNorrisDrift (Y)_{w-\delta}+\OpeNorrisMart (Y)_{w-\delta}  \OpeNorrisDrift (\OpeNorrisDrift (Y))_{w-\delta})+\delta^{2} \OpeNorrisDrift (\OpeNorrisDrift (Y))_{w-\delta}\OpeNorrisDrift (Y)_{w-\delta} ,
\end{align*}
and uses arguments similar to those in \textbf{Step 1} and \textbf{Step 2}. The remaining details are left to the reader.
}
 \VersionLong{
In the same way as in the proof of \textbf{Step 1}, we begin by writing for every $t \in \EnsT $ ,
\begin{align*}
 Y_{t}  \OpeNorrisDrift (Y)_{t}= & Y_0 \OpeNorrisDrift (Y)_{0}+ \sum_{w \in \EnsT,w\leqslant t} \delta^{\frac{1}{2}} (\OpeNorrisMart (Y)_{w-\delta} \OpeNorrisDrift (Y)_{w-\delta}+\OpeNorrisMart(\OpeNorrisDrift (Y))_{w-\delta} Y_{w-\delta}) \\
 & + \delta (\vert \OpeNorrisDrift (Y)_{w-\delta} \vert^{2}+\OpeNorrisMart(\OpeNorrisDrift (Y))_{w-\delta}\OpeNorrisMart (Y)_{w-\delta} +\OpeNorrisDrift(\OpeNorrisDrift (Y))_{w-\delta} Y_{w-\delta} ) \\
 &+\delta^{\frac{3}{2}}(\OpeNorrisMart(\OpeNorrisDrift (Y))_{w-\delta}\OpeNorrisDrift (Y)_{w-\delta}+\OpeNorrisMart (Y)_{w-\delta}  \OpeNorrisDrift (\OpeNorrisDrift (Y))_{w-\delta})+\delta^{2} \OpeNorrisDrift (\OpeNorrisDrift (Y))_{w-\delta}\OpeNorrisDrift (Y)_{w-\delta} ,
\end{align*}
and  we define for $h \in(3r,\frac{s}{2})$
\begin{align*}
\mathcal{A}_{3}   :=    & \{  \vert \sum_{(w,t) \in \EnsT^{2}_{\mathsmaller{\leqslant}}} \delta^{\frac{3}{2}} \OpeNorrisMart(\OpeNorrisDrift (Y))_{w-\delta} Y_{w-\delta} \vert < \frac{\epsilon^{h}}{8}  \} \cap  \{ \delta^{\frac{3}{2}} \vert \sum_{(w,t) \in \EnsT^{2}_{\mathsmaller{\leqslant}}} \OpeNorrisMart (Y)_{w-\delta} \OpeNorrisDrift (Y)_{w-\delta}\vert  < \frac{\epsilon^{h}}{8}  \} \\
&\cap  \{  \vert  \delta^{2} \sum_{(w,t) \in \EnsT^{2}_{\mathsmaller{\leqslant}}} \OpeNorrisMart(\OpeNorrisDrift (Y))_{w-\delta}\OpeNorrisMart (Y)_{w-\delta}  + \OpeNorrisDrift(\OpeNorrisDrift (Y))_{w-\delta}Y_{w-\delta} \vert < \frac{\epsilon^{h}}{8}  \}    \\
& \cap  \{ \delta^{2} \sum_{(w,t) \in \EnsT^{2}_{\mathsmaller{\leqslant}}} \vert \delta^{\frac{1}{2}}(\OpeNorrisMart(\OpeNorrisDrift (Y))_{w-\delta}\OpeNorrisDrift (Y)_{w-\delta}+\OpeNorrisMart (Y)_{w-\delta}  \OpeNorrisDrift (\OpeNorrisDrift (Y))_{w-\delta})+\delta \OpeNorrisDrift (\OpeNorrisDrift (Y))_{w-\delta}\OpeNorrisDrift (Y)_{w-\delta} \vert <  \frac{\epsilon^{h}}{8}  \}.
\end{align*}
We take $u \in (0,\frac{s}{2}-h)$.  As in \textbf{Step 1}, the exponential martingale inequality (\ref{eq:Doob_martin_expo_ineq}), together with Markov's and H\"older's inequalities, yields, for every $v^{\prime}>0$, the estimate
 \begin{align*}
  \mathbb{P}(\vert \sum_{(w,t) \in \EnsT^{2}_{\mathsmaller{\leqslant}}}  \delta^{\frac{3}{2}} \OpeNorrisMart(\OpeNorrisDrift (Y))_{w-\delta} Y_{w-\delta} \vert \geqslant \frac{\epsilon^{h}}{8}, \delta\sum_{t \in \EnsT  } & Y_{t}^{2} <  \epsilon,\mathcal{A}^{Y}_{u,q})  
 \leqslant  \delta^{\frac{q}{4}}  \epsilon^{-\frac{(2+v^{\prime})q}{4}}2^{\frac{3q}{4}} T^{\frac{q}{4}}  \sup_{t \in \EnsT } \mathbb{E}[\vert Y_{t-\delta} \vert^{q} ]  \\
&+\mathbb{P}(\delta Y_{0}^{2} \geqslant \epsilon)   +2 \exp(-\frac{\epsilon^{2h+2u-1}}{2^{10}T^{2}}) +2\exp(-2\epsilon^{-v^{\prime}}) .
 \end{align*}
Moreover, since $\mathcal{A}_{1} \subset \{ \delta^{2}  \sum_{(w,t) \in \EnsT^{2}_{\mathsmaller{\leqslant}}}  \mathbb{E}[  \OpeNorrisMart (Y)_{w-\delta} ^{2} \vert \mathcal{F}_{w-\delta}] +  \OpeNorrisMart (Y)_{w-\delta} ^{2} <2 \epsilon^{s}\}$ and $  \sup_{t \in \EnsT } \vert \OpeNorrisDrift (Y)_{t-\delta} \vert \leqslant \epsilon^{-u}$ on $ \mathcal{A}^{Y}_{u,q}$, the inequality (\ref{eq:Doob_martin_expo_ineq}) yields
     \begin{align*}
    \mathbb{P}( \delta^{\frac{3}{2}} \vert \sum_{(w,t) \in \EnsT^{2}_{\mathsmaller{\leqslant}}} \OpeNorrisMart (Y)_{w-\delta} \OpeNorrisDrift (Y)_{w-\delta}\vert  \geqslant \frac{\epsilon^{h}}{8},\mathcal{A}_{1},\mathcal{A}^{Y}_{u,q})   \leqslant &  2 \exp(-\frac{\epsilon^{2h+2u-s}}{ 2^{8} T}) .
 \end{align*}
%
%
%
It remains to study the last term stemming from the definition of $\mathcal{A}_{3}$.  Using the Cauchy-Schwarz inequality, we have
  \begin{align*}
 \mathbb{P}( \vert \delta^{2} \sum_{(w,t) \in \EnsT^{2}_{\mathsmaller{\leqslant}}} & \OpeNorrisMart(\OpeNorrisDrift (Y))_{w-\delta}\OpeNorrisMart (Y)_{w-\delta} \vert  \geqslant \frac{\epsilon^{h}}{16},   \delta^{2}  \sum_{(w,t) \in \EnsT^{2}_{\mathsmaller{\leqslant}}}  \OpeNorrisMart (Y)_{w-\delta} ^{2} <\epsilon^s  , \mathcal{A}^{Y}_{u,q}) \\
 \leqslant &  \mathbb{P}( \delta^{2} \sum_{(w,t) \in \EnsT^{2}_{\mathsmaller{\leqslant}}}  \OpeNorrisMart(\OpeNorrisDrift (Y))_{w-\delta} ^{2} \geqslant \frac{\epsilon^{2h-s}}{2^{8}},  \delta^{2}  \sum_{(w,t) \in \EnsT^{2}_{\mathsmaller{\leqslant}}} \OpeNorrisMart (Y)_{w-\delta}^{2} <\epsilon^s  , \mathcal{A}^{Y}_{u,q})  ,
 \end{align*}
In addition, by the Markov and H\"older inequalities, we derive
  \begin{align*}
   \mathbb{P}( \delta^{2} \sum_{(w,t) \in \EnsT^{2}_{\mathsmaller{\leqslant}}} \OpeNorrisMart(\OpeNorrisDrift (Y))_{w-\delta} ^{2}  \geqslant  \frac{\epsilon^{2h-s}}{2^{8}} ,  \mathcal{A}^{Y}_{u,q})  \leqslant & \epsilon^{\frac{q(s-2h)}{2}}  2^{4q}    \mathbb{E}[\vert \delta^{2} \sum_{(w,t) \in \EnsT^{2}_{\mathsmaller{\leqslant}}}  \OpeNorrisMart(\OpeNorrisDrift (Y))_{w-\delta}^{2} \vert^{\frac{q}{2}} \mathbf{1}_{\mathcal{A}^{Y}_{u,q}} ]\\
   \leqslant & \epsilon^{\frac{q(s-2h-2u)}{2}}    2^{4q}   T^{q} .
  \end{align*}
Moreover, using the Cauchy-Schwarz inequality and the fact that $\epsilon \leqslant (2^{9} T^3)^{-\frac{1}{1-2h-2u}}$, we have
 \begin{align*}
   \mathbb{P}(\vert \sum_{(w,t) \in \EnsT^{2}_{\mathsmaller{\leqslant}}} & \delta^{2} \OpeNorrisDrift(\OpeNorrisDrift (Y))_{w-\delta} Y_{w-\delta} \vert \geqslant \frac{\epsilon^{h}}{16}, \delta\sum_{t \in \EnsT  }  Y_{t} ^{2} <  \epsilon, \delta  Y_{0}^{2} <  \epsilon,\mathcal{A}^{Y}_{u,q})  \\
   \leqslant &    \mathbb{P}(  \delta T^{3} \sum_{t\in \EnsT } Y_{t-\delta} ^{2} \geqslant \frac{\epsilon^{2h+2u}}{2^{8}}, \delta\sum_{t \in \EnsT  } Y_{t-\delta}^{2} <  2\epsilon, \mathcal{A}^{Y}_{u,q})  =0
 \end{align*}
Besides, for every $\epsilon 
\geqslant   ( 16 \delta^{3} \vert \EnsT  \vert^{2} )^{\frac{1}{h+2u}}$, we have $ \delta^{3} \sum_{w,t \in \EnsT ;w\leqslant t}\vert \OpeNorrisDrift (\OpeNorrisDrift (Y))_{w-\delta}\OpeNorrisDrift (Y)_{w-\delta} \vert \leqslant \frac{\epsilon^{h}}{16}$ on the set $\mathcal{A}^{Y}_{u,q}$. The triangle inequality, Markov's inequality,
and H\"older's inequality then yield
\begin{align*}
\mathbb{P}( \delta^{2} \sum_{(w,t) \in \EnsT^{2}_{\mathsmaller{\leqslant}}}\vert \delta^{\frac{1}{2}} &(\OpeNorrisMart(\OpeNorrisDrift (Y))_{w-\delta}\OpeNorrisDrift (Y)_{w-\delta}  +\OpeNorrisMart (Y)_{w-\delta}  \OpeNorrisDrift (\OpeNorrisDrift (Y))_{w-\delta})+\delta \OpeNorrisDrift (\OpeNorrisDrift (Y))_{w-\delta}\OpeNorrisDrift (Y)_{w-\delta} \vert \geqslant \frac{\epsilon^{h}}{8}, \mathcal{A}^{Y}_{u,q} ) \\
\leqslant & \mathbb{P}( \delta^{5/2} \sum_{(w,t) \in \EnsT^{2}_{\mathsmaller{\leqslant}}} \vert \OpeNorrisMart(\OpeNorrisDrift (Y))_{w-\delta}\OpeNorrisDrift (Y)_{w-\delta} +\OpeNorrisMart (Y)_{w-\delta}  \OpeNorrisDrift (\OpeNorrisDrift (Y))_{w-\delta} \vert \geqslant \frac{\epsilon^{h}}{16}, \mathcal{A}^{Y}_{u,q}) 
 \\
\leqslant & 2^{5q}\delta^{\frac{q}{2}} \epsilon^{-q(h+2u)} T^{2q}  .
\end{align*}
In particular, for every $\epsilon \in [ \underline{\epsilon}_{3}(\delta), (2^{9} T^3)^{-\frac{1}{1-2h-2u}}]$,  we have just shown that
\begin{align*}
\mathbb{P}(\delta\sum_{t \in \EnsT } Y_t ^{2} < & \epsilon, \delta\sum_{t \in \EnsT } \mathbb{E}[\OpeNorrisMart (Y)_{t-\delta}^{2} \vert \mathcal{F}_{t-\delta}] < \frac{\epsilon^{r}}{2} ,  \mathcal{A}_{1},\mathcal{A}_{3} ^c,\mathcal{A}^{Y}_{u,q}) \leqslant  \\
&\delta^{\frac{q}{4}}( 2^{5q}\delta^{\frac{q}{4}} \epsilon^{-q(h+2u)} T^{2q}  +  \epsilon^{-\frac{(2+v^{\prime})q}{4}}2^{\frac{3q}{4}} T^{\frac{q}{4}}  \sup_{t \in \EnsT } \mathbb{E}[\vert Y_{t-\delta} \vert^{q} ]  )\\
&+ \epsilon^{\frac{q(s-2h-2u)}{2}}    2^{4q}  T^{q} +\mathbb{P}(\delta Y_{0}^{2} \geqslant \epsilon)   \\
&+2 \exp(-\frac{\epsilon^{2h+2u-1}}{2^{10} T^{2}}) +2\exp(-2 \epsilon^{-v^{\prime}})+  2 \exp(-\frac{\epsilon^{2h+2u-s}}{ 2^{8} T }) .
\end{align*}
We notice that, as in \textbf{Step 1},  
\begin{align*}
\mathcal{A}_{3}  \subset \{ \delta\sum_{t \in \EnsT }  Y_0 \OpeNorrisDrift (Y)_{0} + \delta^{2} \sum_{(w,t) \in \EnsT^{2}_{\mathsmaller{\leqslant}}} \OpeNorrisDrift (Y)_{w-\delta} ^{2} <  \delta\vert  \sum_{t \in \EnsT } Y_t  \OpeNorrisDrift (Y)_{t} \vert  +\frac{\epsilon^{h}}{2} \}.
\end{align*}
It follows from the Cauchy-Schwarz inequality, applied to $\sum_{t \in \EnsT } Y_t  \OpeNorrisDrift (Y)_{t} $, and the triangle inequality that,  for every $\epsilon \leqslant  1 \wedge ( 16 T )^{-\frac{1}{s-2u-2h}}$, 
   \begin{align*}
\mathbb{P}(\delta\sum_{t \in \EnsT } Y_t ^{2} < & \epsilon, \delta\sum_{t \in \EnsT }  \OpeNorrisDrift (Y)_{t-\delta}^{2} \geqslant \frac{\epsilon^{r}}{2},  \mathcal{A}_{1} ,\mathcal{A}_{3}, Y_{0}^{2} < \frac{\epsilon^{s}}{T}  ,\mathcal{A}^{Y}_{u,q},  \vert \OpeNorrisDrift (Y)_{T} \vert \leqslant \epsilon^{-u})  \\
\leqslant &  \mathbb{P}( \delta\sum_{t \in \EnsT }  \OpeNorrisDrift (Y)_{t-\delta}^{2}\geqslant  \frac{\epsilon^{r}}{2}, \delta^{2}  \sum_{(w,t) \in \EnsT^{2}_{\mathsmaller{\leqslant}}}\OpeNorrisDrift (Y)_{w-\delta}^{2} <\frac{\epsilon^h}{2}+T^{\frac{1}{2}} \epsilon^{-u} (\epsilon^{\frac{s}{2}}+\epsilon^{\frac{1}{2}}), \mathcal{A}^{Y}_{u,q} )\\
\leqslant &  \mathbb{P}( \delta\sum_{t \in \EnsT } \OpeNorrisDrift (Y)_{t-\delta}^{2}\geqslant  \frac{\epsilon^{r}}{2}, \delta^{2}  \sum_{(w,t) \in \EnsT^{2}_{\mathsmaller{\leqslant}}} \OpeNorrisDrift (Y)_{w-\delta}^{2} < \epsilon^h, \mathcal{A}^{Y}_{u,q} ).
  \end{align*}
 As in \textbf{Step 2}, we notice that, on the set $\{\delta\sum_{t \in \EnsT }  \OpeNorrisDrift (Y)_{t-\delta}^{2} \geqslant  \frac{\epsilon^{r}}{2}\} \cap \mathcal{A}^{Y}_{u,q}$, 
    \begin{align*}
 \delta  \sum_{t \in \EnsT } \mathbf{1}_{ \OpeNorrisDrift (Y)_{t-\delta}^{2} \geqslant  \frac{\epsilon^{r}}{4 T}} \geqslant \frac{\epsilon^{r+2u}}{4} \quad \mbox{and} \quad 
 \delta^{2} \sum_{(w,t) \in \EnsT^{2}_{\mathsmaller{\leqslant}}} \OpeNorrisDrift (Y)_{w-\delta}^{2} 
 \geqslant & \frac{\epsilon^{3r+4u}}{2^{7} T} .
 \end{align*}
 In particular, for every $\epsilon  \leqslant  ( 2^{7} T )^{-\frac{1}{h-3r-4u}} $,
 \begin{align*}
\{ \delta\sum_{t \in \EnsT } \vert \OpeNorrisDrift (Y)_{t-\delta} \vert^{2} \geqslant \frac{\epsilon^{r}}{2} \} \cap \{  \delta^{2}  \sum_{(w,t) \in \EnsT^{2}_{\mathsmaller{\leqslant}}} \vert \OpeNorrisDrift (Y)_{w-\delta} \vert^{2} < \epsilon^h \} \cap \mathcal{A}^{Y}_{u,q} =\emptyset.
 \end{align*}
This completes the proof of \textbf{Step 3}.
}

\textbf{Step 4.} We now show (\ref{eq:ineg_Norris}). In the first three steps we have proved that, for every $h,s \in (3r,\frac{1}{2})$ with $2h<s$, $u\in(0,\min(\frac{1}{2}-s,\frac{s}{2}-h,\frac{h}{4}-\frac{3r}{4}))$, every $p,v,v^{\prime}>0$,  every $q \geqslant 4$ and every $\epsilon \in [\max(\underline{\epsilon}_{1}(\delta),\underline{\epsilon}_{3}(\delta),8 \delta),\min(1,\bar{\epsilon}_{1}(\delta),\bar{\epsilon}_{2}(\delta),\bar{\epsilon}_{3}(\delta))]$,
 \begin{align*}
  \mathbb{P}(\delta & \sum_{t \in \EnsT }  Y_t ^{2} < \epsilon,  \delta\sum_{t \in \EnsT } \mathbb{E}[\OpeNorrisMart (Y)_{t-\delta}^{2} \vert \mathcal{F}_{t-\delta}] +  \OpeNorrisDrift (Y)_{t-\delta}^{2} \geqslant \epsilon^r, \mathcal{A}^{Y}_{u,q},\vert \OpeNorrisDrift (Y)_{T} \vert \leqslant \epsilon^{-u})     \\
 \leqslant  &  \epsilon^{p} \mathbb{E}[\vert Y_{0} \vert^{\frac{p}{v}}] +3 \mathbb{P}(\delta Y_{0}^{2} \geqslant \epsilon) +\epsilon^{\frac{q(s-2h-2u)}{2}}    2^{4q}  T^{q} \\
&+ \delta^{\frac{q}{4}} ( \delta^{\frac{q}{4}} \epsilon^{-q(s+2u)}+ \epsilon^{-q(s+u)} +3\epsilon^{-q\frac{(2+v^{\prime})}{4}})2^{5q} (1 \vee T^{2q})(1+\sup_{t \in \EnsT } \mathbb{E}[\vert Y_{t-\delta} \vert^{q} ] )\\
&+  2\exp(-\frac{\epsilon^{-2s}}{16})
+4\exp(-2 \epsilon^{-v^{\prime}}) +2 \exp(-\frac{\epsilon^{2s+2u-1}}{2^{12}T^{2}}) +2 \exp(-\frac{\epsilon^{2h+2u-1}}{2^{10} T^{2}}) +   2 \exp(-\frac{\epsilon^{2h+2u-s}}{ 2^{8}  T}) . 
  \end{align*}
  
  We first observe that, as a direct consequence of Markov's inequality, for $p>0$,
%
\begin{align*}
\mathbb{P}((\mathcal{A}^{Y}_{u,q})^{c}) \leqslant & \epsilon^{p}( \mathbb{E}[\sup_{t \in \EnsT } \vert \OpeNorrisDrift (Y)_{t-\delta}\vert ^{\frac{p}{u}}] + \mathbb{E}[\sup_{t \in \EnsT }  \mathbb{E}[ \vert \OpeNorrisMart (Y)_{t-\delta} \vert^{q} \vert \mathcal{F}_{t-\delta}]^{\frac{p}{qu}}] \\
& +  \mathbb{E}[\sup_{t \in \EnsT }  \vert \OpeNorrisDrift (\OpeNorrisDrift (Y))_{t-\delta}\vert ^{\frac{p}{u}}] + \mathbb{E}[\sup_{t \in \EnsT }  \mathbb{E}[ \vert \OpeNorrisMart(\OpeNorrisDrift (Y))_{t-\delta} \vert^q \vert \mathcal{F}_{t-\delta}]^{\frac{p}{qu}}]),
\end{align*}
and, using $\OpeNorrisDrift (Y)_{T}=\OpeNorrisDrift (Y)_{T-\delta} + \delta^{\frac{1}{2}} \OpeNorrisMart(\OpeNorrisDrift (Y))_{T-\delta}  + \delta \OpeNorrisDrift(\OpeNorrisDrift (Y))_{T-\delta} $ together with $\delta\leqslant1$ and assuming $q \geqslant \frac{p}{u}$, it follows from Markov's and H\"older's inequalities that
\begin{align*}
\mathbb{P}(\vert & \OpeNorrisDrift (Y)_{T} \vert > \epsilon^{-u}) \\
 \leqslant & \mathbb{P}(\vert \OpeNorrisDrift (Y)_{T-\delta} \vert >\frac{1}{3} \epsilon^{-u})+ \mathbb{P}(\vert \OpeNorrisMart(\OpeNorrisDrift (Y))_{T-\delta} \vert > \frac{1}{3} \epsilon^{-u})+ \mathbb{P}(\vert \OpeNorrisDrift(\OpeNorrisDrift (Y))_{T-\delta} \vert >\frac{1}{3} \epsilon^{-u}) \\
\leqslant & 3^{\frac{p}{u}} \epsilon^{p}( \mathbb{E}[\sup_{t \in \EnsT } \vert \OpeNorrisDrift (Y)_{t-\delta}\vert ^{\frac{p}{u}}] +  \mathbb{E}[\sup_{t \in \EnsT }  \vert \OpeNorrisDrift (\OpeNorrisDrift (Y))_{t-\delta}\vert ^{\frac{p}{u}}] + \mathbb{E}[\sup_{t \in \EnsT }  \mathbb{E}[ \vert \OpeNorrisMart(\OpeNorrisDrift (Y))_{t-\delta} \vert^q \vert \mathcal{F}_{t-\delta}]^{\frac{p}{qu}}]).
\end{align*} 
At this point, we also assume that $q \geqslant \frac{2p}{s-2h-2u}$.  Then, if $\epsilon \geqslant \max(\delta^{\frac{q}{2p+2q(s+2u)}},\delta^{\frac{q}{4p+4q(s+u)}}, \delta^{\frac{1}{2+v^{\prime}+4\frac{p}{q}}})$ (which holds with the parameters fixed below),  we have
 \begin{align*}
  \mathbb{P}(\delta & \sum_{t \in \EnsT } Y_t^{2} < \epsilon,  \delta\sum_{t \in \EnsT } \mathbb{E}[\OpeNorrisMart (Y)_{t-\delta} ^{2} \vert \mathcal{F}_{t-\delta}] + \OpeNorrisDrift (Y)_{t-\delta}^{2} \geqslant \epsilon^r, \mathcal{A}^{Y}_{u,q},\vert \OpeNorrisDrift (Y)_{T} \vert \leqslant \epsilon^{-u})     \\
 \leqslant  &  \epsilon^{p} \mathbb{E}[\vert Y_{0} \vert^{\frac{p}{v}}] +3\mathbb{P}(\delta Y_{0}^{2} \geqslant \epsilon) + \epsilon^{p} 2^{5q} 6(1 \vee T^{2q})(1+\sup_{t \in \EnsT } \mathbb{E}[\vert Y_{t-\delta} \vert^{q} ] )\\
&+  2\exp(-\frac{\epsilon^{-2s}}{16})
+4\exp(-2 \epsilon^{-v^{\prime}})+  2 \exp(-\frac{\epsilon^{2s+2u-1}}{2^{12}T^{2}}) +2 \exp(-\frac{\epsilon^{2h+2u-1}}{2^{10} T^{2}}) +   2 \exp(-\frac{\epsilon^{2h+2u-s}}{ 2^{8}  T})  .
  \end{align*}
Moreover, for every $q^{\prime}>0$ such that $\epsilon \geqslant \delta^{\frac{q^{\prime}}{q^{\prime}+2p}}$, $\mathbb{P}(\delta Y_{0}^{2} \geqslant \epsilon) \leqslant \epsilon^{p} \mathbb{E}[\vert Y_{0} \vert^{q^{\prime}}]$. In particular, we take $q^{\prime}=\frac{2pq}{q(1+v^{\prime})+p}$ so that this inequality is satisfied when $\epsilon \geqslant \delta^{\frac{1}{2+v^{\prime}+\frac{p}{q}}}$.

Now we fix $s=s(r) :=   \frac{5}{11}+\frac{6}{11} r$, $h=h(r) :=   \frac{2}{11}+\frac{9}{11}r$ and take $u<\frac{1}{22}-\frac{6}{11}r$. Notice that, since $r \in (0,\frac{1}{12})$, we have $s(r) \in (6r,\frac{1}{2})$, $h(r) \in(3r,\frac{s(r)}{2})$. Therefore, taking $v=\frac{6}{11}-\frac{6}{11}r-u+\frac{v^{\prime}}{2}+\frac{p}{2q}$, and $q\geqslant \max(4,\frac{2p}{\frac{1}{11}-\frac{12}{11}r-2u})$, we have, for every $\epsilon \in [\vert 2^{10} (1 \vee T^{3})  \delta \vert ^{\frac{1}{2+v^{\prime}+4\frac{p}{q}}},\min (\vert 2^{10}( 1 \vee T^{3}) \vert^{-\frac{1}{\frac{1}{11}-\frac{12}{11}r-2u}} , 2^{-\frac{1}{\frac{6}{11}-\frac{6}{11} r}} )]$ ,
 \begin{align*}
  \mathbb{P}(\delta  \sum_{t \in \EnsT } Y_t^{2} < & \epsilon,  \delta\sum_{t \in \EnsT } \mathbb{E}[\OpeNorrisMart (Y)_{t-\delta}^{2} \vert \mathcal{F}_{t-\delta}] + \OpeNorrisDrift (Y)_{t-\delta}^{2} \geqslant \epsilon^r)    \\
 \leqslant  &  \epsilon^{p} \mathbb{E}[\vert Y_{0} \vert^{\frac{p}{\frac{6}{11}-\frac{6}{11}r-u+\frac{v^{\prime}}{2}+\frac{p}{2q}}}] +3  \epsilon^{p} \mathbb{E}[\vert Y_{0} \vert^{\frac{2pq}{q(1+v^{\prime})+p}}] + \epsilon^{p} 2^{5q} 6 (1 \vee T^{2q})(1+\sup_{t \in \EnsT } \mathbb{E}[\vert Y_{t-\delta} \vert^{q} ] )\\
&+\epsilon^{p}(1+3^{\frac{p}{u}})( \mathbb{E}[\sup_{t \in \EnsT } \vert \OpeNorrisDrift (Y)_{t-\delta}\vert ^{\frac{p}{u}}] + \mathbb{E}[\sup_{t \in \EnsT }  \mathbb{E}[ \vert \OpeNorrisMart (Y)_{t-\delta} \vert^{q} \vert \mathcal{F}_{t-\delta}]^{\frac{p}{qu}}] \\
& \quad +  \mathbb{E}[\sup_{t \in \EnsT }  \vert \OpeNorrisDrift (\OpeNorrisDrift (Y))_{t-\delta}\vert ^{\frac{p}{u}}] + \mathbb{E}[\sup_{t \in \EnsT }  \mathbb{E}[ \vert \OpeNorrisMart(\OpeNorrisDrift (Y))_{t-\delta} \vert^q \vert \mathcal{F}_{t-\delta}]^{\frac{p}{qu}}])\\
&+4\exp(-2\epsilon^{-v^{\prime}})  +  8 \exp(-\frac{\epsilon^{-\frac{1}{11}+\frac{12}{11}r+2u}}{2^{12} (1 \vee T^{2})})  .
  \end{align*}

Now we take $u=\frac{1}{44}-\frac{3}{11}r$ and $q=q(r,p)=4\max(4,\frac{2p}{\frac{1}{11}-\frac{12}{11}r-2u},\frac{p}{u})=4\max(4,\frac{44p}{1-12r})$ (in particular $q(r,p) \geqslant \frac{2p}{s-2h-2u}$). It follows that, for every $\epsilon \in [\vert 2^{10} (1 \vee T^{3})  \delta \vert ^{\frac{1}{2+v^{\prime}+4\frac{p}{q(r,p)}}},( 2^{20} (1 \vee T^{6}) )^{-\frac{11}{1-12r}} ]$, 
 \begin{align*}
  \mathbb{P}(\delta & \sum_{t \in \EnsT } Y_t^{2} < \epsilon,  \delta\sum_{t \in \EnsT } \mathbb{E}[\OpeNorrisMart (Y)_{t-\delta}^{2} \vert \mathcal{F}_{t-\delta}] +  \OpeNorrisDrift (Y)_{t-\delta}^{2} \geqslant \epsilon^r)  \\
 \leqslant  &    \epsilon^{p}(1+4 \mathbb{E}[\vert Y_{0} \vert^{\frac{2pq(r,p)}{q(r,p)(1+v^{\prime})+p}}]) + \epsilon^{p} 2^{5q(r,p)} 6(1\vee T^{2q(r,p)})(1+\sup_{t \in \EnsT } \mathbb{E}[\vert Y_{t-\delta} \vert^{q(r,p)} ] )\\
&+ \epsilon^{p}(1+3^{q(r,p)}) (3+ \FuncNorris_{Y,\EnsT }(q(r,p))) +4\exp(-2\epsilon^{-v^{\prime}}) +  8 \exp(-\frac{\epsilon^{-\frac{1}{22}+\frac{6}{11}r}}{2^{12} (1 \vee T^{2})})  .
  \end{align*}
Since $q(r,p)>p$ and $v^{\prime}>0$, $ \mathbb{E}[\vert Y_{0} \vert^{\frac{2pq(r,p)}{q(r,p)(1+v^{\prime})+p}}] \leqslant 1+ \mathbb{E}[\vert Y_{0} \vert^{q(r,p)}]$. We fix $v^{\prime}=\frac{1}{22}-\frac{6}{11}r$ and the proof of (\ref{eq:ineg_Norris}) is completed.
\end{proof}

\end{appendix}

\bibliography{Biblio}
\bibliographystyle{plain}

\end{document}